\documentclass{amsbook}
\usepackage{amssymb}
\usepackage{amsfonts}
\usepackage{geometry}
\usepackage{hyperref}
\usepackage{graphicx}

\newtheorem{theorem}{Theorem}
\theoremstyle{plain}

\newtheorem{conclusion}[theorem]{Conclusion}

\newtheorem{corollary}[theorem]{Corollary}

\newtheorem{definition}[theorem]{Definition}

\newtheorem{lemma}[theorem]{Lemma}
\newtheorem{notation}[theorem]{Notation}

\newtheorem{proposition}[theorem]{Proposition}
\newtheorem{remark}[theorem]{Remark}

\newtheorem{summary}[theorem]{Summary}
\numberwithin{equation}{chapter}
\input{tcilatex}

\begin{document}
\frontmatter
\title[Sobolev inequalities]{Sharp local boundedness and maximum principle in
	the infinitely degenerate regime via DeGiorgi iteration\thanks{\textit{2010 Mathematics Subject Classification}.
Primary 35B65, 35D30, 35H99. Secondary 51F99, 46E35.}}
\author{Lyudmila Korobenko}
\address{Reed College\\
Portland, Oregon, USA}
\email{korobenko@reed.edu}
\author{Cristian Rios}
\address{University of Calgary\\
Calgary, Alberta, Canada}
\email{crios@ucalgary.ca}
\author{Eric Sawyer}
\address{McMaster University\\
Hamilton, Ontario, Canada}
\email{sawyer@mcmaster.ca}
\author{Ruipeng Shen}
\address{Center for Applied Mathematics\\
Tianjin University \\
Tianjin 300072, P.R.China}
\email{srpgow@163.com}
\date{\today }

\begin{abstract}
	We obtain local boundedness and maximum principles for weak subsolutions to
	certain infinitely degenerate elliptic divergence form inhomogeneous
	equations.
	For example, we consider the family $\left\{ f_{\sigma }\right\} _{\sigma
		>0} $ with 
	\begin{equation*}
	f_{\sigma }\left( x\right) =e^{-\left( \frac{1}{\left\vert x\right\vert }%
		\right) ^{\sigma }},\ \ \ \ \ -\infty <x<\infty ,
	\end{equation*}%
	of infinitely degenerate functions at the origin, and show that all weak
	solutions to the associated infinitely degenerate quasilinear equations of
	the form 
	\begin{equation*}
	\func{div}A\left( x,u\right) \func{grad}u=\phi \left( x\right) ,\ \ \
	A\left( x,z\right) \sim \left[ 
	\begin{array}{cc}
	I_{n-1} & 0 \\ 
	0 & f\left( x_{1}\right) ^{2}%
	\end{array}%
	\right] ,
	\end{equation*}%
	with rough data $A$ and $\phi $, are locally bounded for admissible $\phi $
	provided $0<\sigma <1$. We also show that these conditions are \textit{%
		necessary} for local boundedness in dimension $n\geq 3$, thus paralleling
	the known theory for the smooth Kusuoka-Strook operators $\frac{\partial ^{2}%
	}{\partial x_{1}^{2}}+\frac{\partial ^{2}}{\partial x_{2}^{2}}+f_{\sigma
	}\left( x\right) ^{2}\frac{\partial ^{2}}{\partial x_{3}^{2}}$. We also show
	that subsolutions satisfy a maximum principle for admissible $\phi $ under the same restriction on the degeneracy.
	
	In order to prove these theorems, we first establish abstract results in
	which certain Poincar\'{e} and Orlicz Sobolev inequalities are assumed to
	hold. We then develop subrepresentation inequalities for control geometries
	in order to obtain the needed Poincar\'{e} and Orlicz Sobolev inequalities.
\end{abstract}

\maketitle
\tableofcontents

\chapter*{Preface}

There is a large and well-developed theory of elliptic and subelliptic
equations with rough data, beginning with work of DeGiorgi-Nash-Moser, and
also a smaller theory still in its infancy of infinitely degenerate elliptic
equations with smooth data, beginning with work of Fedii and Kusuoka-Strook,
and continued by Morimoto and Christ. Our purpose here is to initiate a
study of the DeGiorgi regularity theory, as presented by Caffarelli-Vasseur,
in the context of equations that are both infinitely degenerate elliptic and
have rough data. This monograph can be viewed as taking the first steps in
such an investigation and more specifically, in identifying a number of
surprises encountered in the implementation of DeGiorgi iteration in the
infinitely degenerate regime. The similar approach of Moser in the
infinitely degenerate regime is initiated in our paper \cite{KoRiSaSh1}, but
is both technially more complicated and more demanding of the underlying
geometry. As a consequence, the results in \cite{KoRiSaSh1} for local
boundedness are considerably weaker than the sharp results obtained here
with the DeGiorgi approach. On the other hand, the method of Moser does
apply to obtain continuity for solutions to inhomogeneous equations, but at
the expense of a much more elaborate proof strategy. The parallel approach
of Nash seems difficult to adapt to the infinitely degenerate case, but
remains a possibility for future research.

\mainmatter

\part{Overview}

The regularity theory of subelliptic linear equations with smooth
coefficients is well established, as evidenced by the results of H\"{o}%
rmander \cite{Ho} and Fefferman and Phong \cite{FePh}. In \cite{Ho}, H\"{o}%
rmander obtained hypoellipticity of sums of squares of smooth vector fields
whose Lie algebra spans at every point. In \cite{FePh}, Fefferman and Phong
considered general nonnegative semidefinite smooth linear operators, and
characterized subellipticity in terms of a containment condition involving
Euclidean balls and \textquotedblright subunit\textquotedblright\ balls
related to the geometry of the nonnegative semidefinite form associated to
the operator.

The theory in the infinite regime however, has only had its surface
scratched so far, as evidenced by the results of Fedii \cite{Fe} and Kusuoka
and Strook \cite{KuStr}. In \cite{Fe}, Fedii proved that the two-dimensional
operator $\frac{\partial }{\partial x^{2}}+f\left( x\right) ^{2}\frac{%
\partial }{\partial y^{2}}$ is hypoelliptic merely under the assumption that 
$f$ is smooth and positive away from $x=0$. In \cite{KuStr}, Kusuoka and
Strook showed that under the same conditions on $f\left( x\right) $, the
three-dimensional analogue $\frac{\partial ^{2}}{\partial x^{2}}+\frac{%
\partial ^{2}}{\partial y^{2}}+f\left( x\right) ^{2}\frac{\partial ^{2}}{%
\partial z^{2}}$ of Fedii's operator is hypoelliptic \emph{if and only if} $%
\lim_{x\rightarrow 0}x\ln f\left( x\right) =0$. These results, together with
some further refinements of Christ \cite{Chr}, illustrate the complexities
associated with regularity in the infinite regime, and point to the fact
that the theory here is still in its infancy.

The problem of extending these results to include quasilinear operators
requires an understanding of the corresponding theory for linear operators
with \emph{nonsmooth} coefficients, generally as rough as the weak solution
itself. In the elliptic case this theory is well-developed and appears for
example in Gilbarg and Trudinger \cite{GiTr} and many other sources. The key
breakthrough there was the H\"{o}lder \emph{apriori} estimate of DeGiorgi,
and its later generalizations independently by Nash and Moser. The extension
of the DeGiorgi-Nash-Moser theory to the subelliptic or finite type setting,
was initiated by Franchi \cite{Fr}, and then continued by many authors,
including one of the present authors with Wheeden \cite{SaWh4}.

The subject of the present monograph is the extension of DeGiorgi-Nash-Moser
theory to the infinitely degenerate regime, and more specifically the
techniques of DeGiorgi\footnote{%
We thank Pablo Ra\'{u}l Stinga for bringing DeGiorgi's method to our
attention.}. Our theorems fall into two broad categories. First, there is
the \emph{abstract theory} in all dimensions, in which we assume appropriate 
\emph{Orlicz} Sobolev inequalities, as opposed to the familiar $L^{p}$
Sobolev inequalities, and deduce local boundedness and maximum principles
for weak subsolutions. This theory relies heavily on extensions of a lemma of DeGiorgi
to the infinitely degenerate regime. Second, there is the \emph{geometric
theory}, in which we establish the required Orlicz Sobolev inequalities for
large families of infinitely degenerate geometries, obtaining sharp results
in dimension $n\geq 3$ for local boundedness. For this we need
subrepresentation theorems with kernels of the form $K\left( x,y\right) =%
\frac{\widehat{d}\left( x,y\right) }{V\left( x,y\right) }$ where 
\begin{equation*}
\widehat{d}\left( x,y\right) =\min \left\{ d\left( x,y\right) ,\frac{1}{%
\left\vert F^{\prime }\left( x_{1}+d\left( x,y\right) \right) \right\vert }%
\right\} ,
\end{equation*}%
$F=\ln \frac{1}{f}$ and $d\left( x,y\right) $ is the control distance
associated with $f$, and where $V\left( x,y\right) $ is the volume of a
control ball centered at $x$ with radius $d\left( x,y\right) $. Typically, $%
\widehat{d}\left( x,y\right) $ is much smaller than $d\left( x,y\right) $ in
the infinitely degenerate regime.

Finally, the contributions of Nash to the classical DeGiorgi-Nash-Moser
theory revolve around moment estimates for solutions, and we have been
unable to extend these to the infinitely degenerate regime, leaving a
tantalizing loose end. We now turn to a more detailed description of these
results and questions in the introduction that follows.

\chapter{Introduction}

In 1971 Fedii proved in \cite{Fe} that the linear second order partial
differential operator%
\begin{equation*}
Lu\left( x,y\right) \equiv \left\{ \frac{\partial }{\partial x^{2}}+f\left(
x\right) ^{2}\frac{\partial }{\partial y^{2}}\right\} u\left( x,y\right)
\end{equation*}%
is \emph{hypoelliptic}, i.e. every distribution solution $u\in \mathcal{D}%
^{\prime }\left( \mathbb{R}^{2}\right) $ to the equation $Lu=\phi \in
C^{\infty }\left( \mathbb{R}^{2}\right) $ in $\mathbb{R}^{2}$ is smooth,
i.e. $u\in C^{\infty }\left( \mathbb{R}^{2}\right) $, provided:

\begin{itemize}
\item $f\in C^{\infty }\left( \mathbb{R}\right) $,

\item $f\left( 0\right) =0$ and $f$ is positive on $\left( -\infty ,0\right)
\cup \left( 0,\infty \right) $.
\end{itemize}

The main feature of this remarkable theorem is that the order of vanishing
of $f$ at the origin is unrestricted, in particular it can vanish to
infinite order. If we consider the analogous (special form) quasilinear
operator,%
\begin{equation*}
\mathcal{L}u\left( x,y\right) \equiv \left\{ \frac{\partial }{\partial x^{2}}%
+f\left( x,u\left( x,y\right) \right) ^{2}\frac{\partial }{\partial y^{2}}%
\right\} u\left( x,y\right) ,
\end{equation*}%
then of course $f\left( x,u\left( x,y\right) \right) $ makes no sense for $u$
a distribution, but in the special case where $f\left( x,z\right) \approx
f\left( x,0\right) $, the appropriate notion of hypoellipticity for $%
\mathcal{L}$ becomes that of $W_{A}^{1,2}\left( \mathbb{R}^{2}\right) $%
-hypoellipticity with $A\equiv \left[ 
\begin{array}{cc}
1 & 0 \\ 
0 & f\left( x,0\right) ^{2}%
\end{array}%
\right] $, which when $A$ is understood, we refer to as simply \emph{weak
hypoellipticity}.

We say that $\mathcal{L}$ is $W_{A}^{1,2}\left( \mathbb{R}^{2}\right) $%
-hypoelliptic if every $W_{A}^{1,2}\left( \mathbb{R}^{2}\right) $-weak
solution $u$ to the equation $\mathcal{L}u=\phi $ is smooth for all smooth
data $\phi \left( x,y\right) $. Here $u\in W_{A}^{1,2}\left( \mathbb{R}%
^{2}\right) $ is a $W_{A}^{1,2}\left( \mathbb{R}^{2}\right) $-weak solution
to $L_{\limfunc{quasi}}u=\phi $ if%
\begin{equation*}
-\int \left( \nabla w\right) ^{\limfunc{tr}}\left[ 
\begin{array}{cc}
1 & 0 \\ 
0 & f\left( x,u\left( x,y\right) \right) ^{2}%
\end{array}%
\right] \nabla u=\int \phi w,\ \ \ \ \ \text{for all }w\in W_{A}^{1,2}\left( 
\mathbb{R}^{2}\right) _{0}\ .
\end{equation*}%
See below for a precise definition of the degenerate Sobolev space $%
W_{A}^{1,2}\left( \mathbb{R}^{2}\right) $, that informally consists of all $%
w\in L^{2}\left( \mathbb{R}^{2}\right) $ for which $\int \left( \nabla
w\right) ^{\limfunc{tr}}A\nabla w<\infty $.

There is apparently no known $W_{A}^{1,2}\left( \mathbb{R}^{2}\right) $%
-hypoelliptic quasilinear operator $\mathcal{L}$ with coefficient $f\left(
x,z\right) $ that vanishes to \emph{infinite} order when $x=0$, despite the
abundance of results when $f$ vanishes to \emph{finite} order. 

Our method for proving regularity of weak solutions $u$ to $\mathcal{L}%
u=\phi $ is to view $u$ as a weak solution to the linear equation 
\begin{equation*}
Lu\left( x,y\right) \equiv \left\{ \frac{\partial }{\partial x^{2}}+g\left(
x,y\right) ^{2}\frac{\partial }{\partial y^{2}}\right\} u\left( x,y\right)
=\phi \left( x,y\right) ,
\end{equation*}%
where both $g\left( x,y\right) =f\left( x,u\left( x,y\right) \right) $ and $%
\phi \left( x,y\right) $ need no longer be smooth, but $g\left( x,y\right) $
satisfies the estimate%
\begin{equation*}
\frac{1}{C}f\left( x,0\right) \leq g\left( x,y\right) \leq Cf\left(
x,0\right) ,\ \ \ \ x\in \mathbb{R},
\end{equation*}%
and $\phi \left( x,y\right) $ is measurable and admissible - see below for
definitions. The method we employ is an adaptation of DeGiorgi iteration.
The infinite degeneracy of $L$\ forces our adaptation of DeGiorgi iteration
to use Young functions in place of power functions.

Another motivation for this approach is the following three dimensional
analogue of Fedii's equation, which Kusuoka and Strook \cite{KuStr}
considered in 1985 
\begin{equation*}
L_{1}\equiv \frac{\partial ^{2}}{\partial x_{1}^{2}}+\frac{\partial ^{2}}{%
\partial x_{2}^{2}}+f\left( x_{1}\right) ^{2}\frac{\partial ^{2}}{\partial
x_{3}^{2}}\ ,
\end{equation*}%
and showed the surprising result that when $f\left( x_{1}\right) $ is smooth
and positive away from the origin, the smooth linear operator $L_{1}$ is
hypoelliptic \emph{if and only if}%
\begin{equation*}
\lim_{r\rightarrow 0}r\ln f\left( r\right) =0.
\end{equation*}%
This is precisely the condition we show to be necessary and sufficient for
local boundedness of weak solutions to our rough homogeneous equations. Thus
we will begin with an abstract approach in higher dimensions, where we
assume certain Orlicz Sobolev inequalities hold, and then specialize to
geometries that are sufficient to prove the required Orlicz Sobolev
inequalities.

More generally, we consider the divergence form equation 
\begin{equation*}
Lu=\nabla ^{\limfunc{tr}}A\left( x\right) \nabla u=\phi ,\ \ \ \ \ x\in
\Omega ,
\end{equation*}%
and the corresponding second order special quasilinear equation (where only $%
u$, and not $\nabla u$, appears nonlinearly),%
\begin{equation}
\mathcal{L}u\equiv \nabla ^{\limfunc{tr}}\mathcal{A}\left( x,u\left(
x\right) \right) \nabla u=\phi ,\ \ \ \ \ x\in \Omega ,  \label{eq_0}
\end{equation}%
and we assume the following quadratic form condition on the `quasilinear'
matrix $\mathcal{A}(x,z)$, 
\begin{equation}
k\,\xi ^{T}A(x)\xi \leq \xi ^{T}\mathcal{A}(x,z)\xi \leq K\,\xi ^{T}A(x)\xi
\ ,  \label{struc_0}
\end{equation}%
for a.e. $x\in \Omega $ and all $z\in \mathbb{R}$, $\xi \in \mathbb{R}^{n}$.
Here $k,K$ are positive constants and $A(x)=B\left( x\right) ^{\limfunc{tr}%
}B\left( x\right) $ where $B\left( x\right) $ is a Lipschitz continuous $%
n\times n$ real-valued matrix defined for $x\in \Omega $. We define the $A$%
-gradient by%
\begin{equation}
\nabla _{A}=B\left( x\right) \nabla \ ,  \label{def A grad}
\end{equation}%
and the associated degenerate Sobolev space $W_{A}^{1,2}\left( \Omega
\right) $ to have norm%
\begin{equation*}
\left\Vert v\right\Vert _{W_{A}^{1,2}}\equiv \sqrt{\int_{\Omega }\left(
\left\vert v\right\vert ^{2}+\nabla v^{\func{tr}}A\nabla v\right) }=\sqrt{%
\int_{\Omega }\left( \left\vert v\right\vert ^{2}+\left\vert \nabla
_{A}v\right\vert ^{2}\right) }.
\end{equation*}

\begin{notation}
Somewhat informally, we use normal font $Lu=\nabla ^{\limfunc{tr}}A\left(
x\right) \nabla u$ for divergence\ form linear operators with nonnegative
matrix $A\left( x\right) $, and we use calligraphic font $\mathcal{L}%
u=\nabla ^{\limfunc{tr}}\mathcal{A}\left( x,u\right) \nabla u$ to denote the
corresponding special form quasilinear operators with matrix $\mathcal{A}%
\left( x,u\right) $ comparable to $A\left( x\right) $.
\end{notation}

In order to define the notion of weak solution to an inhomogeneous equation $%
\mathcal{L}u=\phi $ will assume that either $\phi \in L_{\limfunc{loc}%
}^{2}\left( \Omega \right) $ or that $\phi \in X\left( \Omega \right) $,
where $X\left( \Omega \right) $ is the space of $A$-admissible functions
defined as follows.

\begin{definition}
\label{def A admiss new}Let $\Omega $ be a bounded domain in $\mathbb{R}^{n}$%
. Fix $x\in \Omega $ and $\rho >0$. We say $\phi $ is $A$\emph{-admissible}
at $\left( x,\rho \right) $ if 
\begin{equation*}
\Vert \phi \Vert _{X\left( B\left( x,\rho \right) \right) }\equiv \sup_{v\in
\left( W_{A}^{1,1}\right) _{0}(B\left( x,\rho \right) )}\frac{\int_{B\left(
x,\rho \right) }\left\vert v\phi \right\vert \,dy}{\int_{B\left( x,\rho
\right) }\Vert \nabla _{A}v\Vert \,dy}<\infty .
\end{equation*}%
We say $\phi $ is $A$\emph{-admissible} in $\Omega $, written $\phi \in
X\left( \Omega \right) $, if $\phi $ is $A$\emph{-admissible} at $\left(
x,\rho \right) $ for all $B\left( x,\rho \right) $ contained in $\Omega $.
\end{definition}

\begin{definition}
Let $\Omega $ be a bounded domain in $\mathbb{R}^{n}$ with $A$ and $\mathcal{%
A}$ as above. Assume that $\phi \in L_{\limfunc{loc}}^{2}\left( \Omega
\right) \cup X\left( \Omega \right) $. We say that $u\in W_{A}^{1,2}\left(
\Omega \right) $ is a \emph{weak solution} to $\mathcal{L}u=\phi $ provided%
\begin{equation*}
-\int_{\Omega }\nabla w\left( x\right) ^{\limfunc{tr}}\mathcal{A}\left(
x,u(x)\right) \nabla u=\int_{\Omega }\phi w
\end{equation*}%
for all $w\in \left( W_{A}^{1,2}\right) _{0}\left( \Omega \right) $, where $%
\left( W_{A}^{1,2}\right) _{0}\left( \Omega \right) $ denotes the closure in 
$W_{A}^{1,2}\left( \Omega \right) $ of the subspace of Lipschitz continuous
functions with compact support in $\Omega $. Similarly, $u\in
W_{A}^{1,2}\left( \Omega \right) $ is a \emph{weak solution} to $Lu=\phi $
provided%
\begin{equation*}
-\int_{\Omega }\nabla w\left( x\right) ^{\limfunc{tr}}A\left( x\right)
\nabla u=\int_{\Omega }\phi w
\end{equation*}%
for all $w\in \left( W_{A}^{1,2}\right) _{0}\left( \Omega \right) $.
\end{definition}

Note that our structural condition (\ref{struc_0}) implies that the integral
on the left above is absolutely convergent, and our assumption that $\phi
\in L_{\limfunc{loc}}^{2}\left( \Omega \right) \cup X\left( \Omega \right) $
implies that the integral on the right above is absolutely convergent: 
\begin{eqnarray*}
\int_{\Omega }\left\vert \phi w\right\vert &\leq &\left( \int_{\Omega \cap 
\limfunc{supp}w}\left\vert \phi \right\vert ^{2}\right) ^{\frac{1}{2}}\left(
\int_{\Omega }\left\vert w\right\vert ^{2}\right) ^{\frac{1}{2}}, \\
\int_{\Omega }\left\vert \phi w\right\vert &\leq &\Vert \phi \Vert _{X\left(
\Omega \right) }\int_{\Omega }\Vert \nabla _{A}w\Vert \leq \Vert \phi \Vert
_{X\left( \Omega \right) }\left( \int_{\Omega }\Vert \nabla _{A}w\Vert
^{2}\right) ^{\frac{1}{2}}.
\end{eqnarray*}%
Weak sub and super solutions are defined by replacing $=$ with $\geq $ and $%
\leq $ respectively in the display above. We can define the gradient $\nabla
_{\mathcal{A}}$ more generally for nonnegative semidefinite matrices $%
\mathcal{A}\left( x,u(x)\right) $, where for convenience in notation we
suppress the dependence on $u$.

\begin{definition}
Given a real symmetric nonnegative semidefinite $n\times n\,\ $matrix $A$,
we can write $A=B^{\limfunc{tr}}B$ with $B=D\left( \left\{ \sqrt{\lambda _{j}%
}\right\} _{j=1}^{n}\right) P$ where $D\left( \left\{ d_{j}\right\}
_{j=1}^{n}\right) =\left[ 
\begin{array}{cccc}
d_{1} & 0 & \cdots & 0 \\ 
0 & d_{2} & \cdots & 0 \\ 
\vdots & \vdots & \ddots & \vdots \\ 
0 & 0 & \cdots & d_{n}%
\end{array}%
\right] $ is the diagonal matrix having $d_{j}$ along the diagonal, and $P$
is the orthogonal matrix that diagonalizes $A$ by the spectral theorem, i.e. 
$PAP^{-1}=D\left( \left\{ \lambda _{j}\right\} _{j=1}^{n}\right) $. This
representation is unique up to permutation of the eigenvalues $\left\{
\lambda _{j}\right\} _{j=1}^{n}$ of $A$. Then we define $\nabla _{A}=B\nabla 
$. In the case that $\mathcal{A}\left( x\right) =A\left( x,u\left( x\right)
\right) $, this gives us a definition of $\nabla _{\mathcal{A}}$ for which%
\begin{equation*}
\int_{\Omega }\left( \nabla w\right) ^{\limfunc{tr}}\mathcal{A}\left(
x,u(x)\right) \nabla w=\int_{\Omega }\left( \nabla w\right) ^{\limfunc{tr}}%
\mathcal{A}\left( x\right) \nabla w=\int_{\Omega }\left\vert \nabla _{%
\mathcal{A}}w\right\vert ^{2}.
\end{equation*}
\end{definition}

We will first obtain \emph{abstract} local boundedness and  maximum principles, in which we \emph{assume} appropriate Orlicz-Sobolev
inequalities hold. Then we will apply our study of degenerate geometries to
prove that these Orlicz-Sobolev inequalities hold in specific situations,
thereby obtaining our \emph{geometric} local boundedness and maximum
principles, in which we only assume information on the \emph{size} of the degenerate geometries.

\chapter{DeGiorgi iteration, local boundedness, and maximum principle}

Recall that the methods of DeGiorgi and Moser iteration play off a Sobolev
inequality, that holds for all compactly supported functions, against a
Cacciopoli inequality, that holds only for subsolutions or supersolutions of
the equation. First, from results of Korobenko, Maldonado and Rios in \cite%
{KoMaRi}, it is known that if there exists a Sobolev bump inequality of the
form%
\begin{equation*}
\left\Vert u\right\Vert _{L^{q}\left( \mu _{B}\right) }\leq Cr\left(
B\right) \left\Vert \left\vert \nabla _{A}u\right\vert \right\Vert
_{L^{p}\left( \mu _{B}\right) },\ \ \ \ \ u\in Lip_{\limfunc{compact}}\left(
B\right) ,
\end{equation*}%
for some pair of exponents $1\leq p<q\leq \infty $, and where the balls $B$
are the Carnot-Carath\'{e}odory control balls for the degenerate vector
field $\nabla _{A}=\left( \frac{\partial }{\partial x},f\frac{\partial }{%
\partial y}\right) $ with radius $r\left( B\right) $, and $d\mu _{B}\left(
x,y\right) =\frac{dxdy}{\left\vert B\right\vert }$ is normalized Lebesgue
measure on $B$, then Lebesgue measure must be \emph{doubling} on control
balls. As a consequence, the function $f$ \emph{cannot} vanish to infinite
order. Thus in order to have any hope of implementing either DeGiorgi or
Moser iteration in the infinitely degenerate regime, we must search for a
weaker Sobolev bump inequality, and the natural setting for this is an
Orlicz Sobolev bump inequality%
\begin{equation*}
\left\Vert w\right\Vert _{L^{\Phi }\left( B\right) }\leq C\varphi \left(
r\left( B\right) \right) \left\Vert \nabla _{A}w\right\Vert _{L^{1}\left(
\mu _{B}\right) },\ \ \ \ \ w\in Lip_{\limfunc{compact}}\left( B\right) ,
\end{equation*}%
where the Young function $\Phi \left( t\right) $ is increasing to $\infty $
and convex on $\left( 0,\infty \right) $, but asymptotically closer to the
identity $t$ than any power function $t^{1+\sigma }$, $\sigma >0$. The
`superradius' $\varphi \left( r\right) $ here is nondecreasing and $\varphi
\left( r\right) \geq r$.

We now recall the definition of Orlicz norms. Suppose that $\mu $ is a $%
\sigma $-finite measure on a set $X$, and $\Phi :\left[ 0,\infty \right)
\rightarrow \left[ 0,\infty \right) $ is a Young function, which for our
purposes is a convex piecewise differentiable (meaning there are at most
finitely many points where the derivative of $\Phi $ may fail to exist, but
right and left hand derivatives exist everywhere) function such that $\Phi
\left( 0\right) =0$ and

\begin{equation*}
\frac{\Phi \left( x\right) }{x}\rightarrow \infty \text{ as }x\rightarrow
\infty .
\end{equation*}%
Let $L_{\ast }^{\Phi }$ be the set of measurable functions $f:X\rightarrow 
\mathbb{R}$ such that the integral%
\begin{equation*}
\int_{X}\Phi \left( \left\vert f\right\vert \right) d\mu ,
\end{equation*}%
is finite, where as usual, functions that agree almost everywhere are
identified. Since the set $L_{\ast }^{\Phi }$ may not be closed under scalar
multiplication, we define $L^{\Phi }$ to be the linear span of $L_{\ast
}^{\Phi }$, and then define%
\begin{equation*}
\left\Vert f\right\Vert _{L^{\Phi }\left( \mu \right) }\equiv \inf \left\{
k\in \left( 0,\infty \right) :\int_{X}\Phi \left( \frac{\left\vert
f\right\vert }{k}\right) d\mu \leq 1\right\} .
\end{equation*}%
The Banach space $L^{\Phi }\left( \mu \right) $ is precisely the space of
measurable functions $f$ for which the norm $\left\Vert f\right\Vert
_{L^{\Phi }\left( \mu \right) }$ is finite. The conjugate Young function $%
\widetilde{\Phi }$ is defined by $\widetilde{\Phi }^{\prime }=\left( \Phi
^{\prime }\right) ^{-1}$ and can be used to give an equivalent norm%
\begin{equation*}
\left\Vert f\right\Vert _{L_{\ast }^{\Phi }\left( \mu \right) }\equiv \sup
\left\{ \int_{X}\left\vert fg\right\vert d\mu :\int_{X}\widetilde{\Phi }%
\left( \left\vert g\right\vert \right) d\mu \leq 1\right\} .
\end{equation*}%
The above considerations motivate our approach, to which we now turn.

Let $\Omega $ be a bounded domain in $\mathbb{R}^{n}$. There is a quadruple $%
\left( \mathcal{A},d,\varphi ,\Phi \right) $ of objects of interest in our
abstract local boundedness, and maximum principle theorems in $%
\Omega $, namely

\begin{enumerate}
\item the matrix $\mathcal{A}=\mathcal{A}\left( x,z\right) $ associated with
our equation and the $A$-gradient,

\item a Young function $\Phi $ appearing in our Orlicz Sobolev inequality,

\item a metric $d$ giving rise to the balls $B\left( x,r\right) $ that
appear in our Orlicz Sobolev inequality, and also in our sequence $\left\{
\psi _{j}\right\} _{j=1}^{\infty }$ of accumulating Lipschitz functions, and

\item a positive function $\varphi \left( r\right) $ for $r\in \left(
0,R\right) $ that appears in place of the radius $r$ in our Orlicz Sobolev
inequality.
\end{enumerate}

For the abstract theory we will assume two connections between these
objects, namely

\begin{itemize}
\item the existence of an appropriate sequence $\left\{ \psi _{j}\right\}
_{j=1}^{\infty }$ of accumulating Lipschitz functions that connects two of
the objects of interest $\mathcal{A}$ and $d$, and

\item an Orlicz Sobolev bump inequality, 
\begin{equation*}
\left\Vert w\right\Vert _{L^{\Phi }\left( B\right) }\leq \varphi \left(
r\left( B\right) \right) \left\Vert \nabla _{A}w\right\Vert _{L^{1}\left(
B\right) },\ \ \ \ \ \limfunc{supp}w\subset B,
\end{equation*}%
that connects all four objects of interest $\mathcal{A}$, $d$, $\varphi $
and $\Phi $.
\end{itemize}

We now describe these matters in more detail.

\begin{definition}[Standard sequence of accumulating Lipschitz functions]
\label{def_cutoff}Let $\Omega $ be a bounded domain in $\mathbb{R}^{n}$. Fix 
$r>0$ and $x\in \Omega $. We define an $\left( \mathcal{A},d\right) $-\emph{%
standard} sequence of Lipschitz cutoff functions $\left\{ \psi _{j}\right\}
_{j=1}^{\infty }$ at $\left( x,r\right) $, along with sets $%
B(x,r_{j})\supset \limfunc{supp}\psi _{j}$, to be a sequence satisfying $%
\psi _{j}=1$on $B(x,r_{j+1})$, $r_{1}=r$, $r_{\infty }\equiv
\lim_{j\rightarrow \infty }r_{j}=\frac{1}{2}r$, $r_{j}-r_{j+1}=\frac{c}{%
j^{\gamma }}r$ for a uniquely determined constant $c$ and $\gamma >1$, and $%
\left\Vert \nabla _{A}\psi _{j}\right\Vert _{\infty }\lesssim \frac{%
j^{\gamma }}{r}$ with $\nabla _{A}$ as in (\ref{def A grad}) (see e.g. \cite%
{SaWh4}). %   \begin{equation}\label{cutoff}
%    \end{equation}
\end{definition}

We will need to assume the following single scale $\left( \Phi ,\varphi
\right) $-Orlicz Sobolev bump inequality:

\begin{definition}
\label{single scale sob}Let $\Omega $ be a bounded domain in $\mathbb{R}^{n}$%
. Fix $x\in \Omega $ and $\rho >0$. Then the single scale $\left( \Phi
,\varphi \right) $-Orlicz Sobolev bump inequality at $\left( x,\rho \right) $
is: 
\begin{equation}
\left\Vert w\right\Vert _{L^{\Phi }\left( B\right) }\leq \varphi \left( \rho
\right) \left\Vert \nabla _{A}w\right\Vert _{L^{1}\left( B\right) },\ \ \ \
\ w\in Lip_{0}\left( B\left( x,\rho \right) \right) ,  \label{Phi bump' new}
\end{equation}%
\end{definition}

A particular family of Orlicz bump functions that is crucial for our theorem
is the family 
\begin{equation*}
\Phi _{N}\left( t\right) \equiv \left\{ 
\begin{array}{ccc}
t(\ln t)^{N} & \text{ if } & t\geq E=E_{N}=e^{2N} \\ 
\left( \ln E\right) ^{N}t & \text{ if } & 0\leq t\leq E=E_{N}=e^{2N}%
\end{array}%
\right. .
\end{equation*}%
The bump $\Phi _{N}$ is submultiplicative for each $N\geq 1$, i.e.%
\begin{equation*}
\Phi _{N}\left( st\right) \leq \Phi _{N}\left( s\right) \Phi _{N}\left(
t\right) ,\ \ \ \ \ s,t>0,
\end{equation*}%
which can be easily seen using that for $s,t\geq e^{2N}$, 
\begin{eqnarray*}
st\left[ \ln \left( st\right) \right] ^{N} &=&st\left[ \ln s+\ln t\right]
^{N}\leq s\left[ \ln s\right] ^{N}t\left[ \ln t\right] ^{N} \\
&\Longleftrightarrow &\ln s+\ln t\leq \left[ \ln s\right] \left[ \ln t\right]
,
\end{eqnarray*}%
and then that $a+b\leq ab$ if $a,b\geq 2$. Submultiplicativity plays a
critical role in proving our Orlicz Sobolev inequalities below.

For the inhomogeneous equation $\mathcal{L}u=\phi $ we will assume the
forcing function $\phi $ is\ $A$-admissible in $\Omega $. 

\section{Local boundedness and maximum principle for subsolutions}

Recall that a measurable function $u$ in $\Omega $ is \emph{locally bounded
above} at $x$ if $u$ can be modified on a set of measure zero so that the
modified function $\widetilde{u}$ is bounded above in some neighbourhood of $%
x$.

\begin{theorem}[abstract local boundedness]
\label{bound_gen_thm}Let $\Omega $ be a bounded domain in $\mathbb{R}^{n}$
with $n\geq 2$. Suppose that $\mathcal{A}(x,z)$ is a nonnegative
semidefinite matrix in $\Omega \times \mathbb{R}$ that satisfies the
structural condition (\ref{struc_0}). Let $d(x,y)$ be a symmetric metric in $%
\Omega $, and suppose that $B(x,r)=\{y\in \Omega :d(x,y)<r\}$ with $x\in
\Omega $ are the corresponding metric balls. Fix $x\in \Omega $. Then every
weak subsolution of (\ref{eq_0}) is \emph{locally bounded above} at $x$
provided there is $r_{0}>0$ such that:

\begin{enumerate}
\item the function $\phi $ is $A$-admissible at $\left( x,r_{0}\right) $,

\item the single scale $\left( \Phi ,\varphi \right) $-Orlicz Sobolev bump
inequality (\ref{Phi bump' new}) holds at $\left( x,r_{0}\right) $ with $%
\Phi =\Phi _{N}$ for some $N>1$,

\item there exists an $\left( \mathcal{A},d\right) $-\emph{standard}
accumulating sequence of Lipschitz cutoff functions at $\left(
x,r_{0}\right) $.
\end{enumerate}
\end{theorem}

\begin{remark}
The hypotheses required for local boundedness of weak solutions to $Lu=\phi $
at a single \emph{fixed} point $x$ in $\Omega $ are quite weak; namely we
only need that the inhomogeneous term $\phi $ is $A$\emph{-admissible} at 
\textbf{just one} point $\left( x,r_{0}\right) $ for some $r_{0}>0$, and
that there are two \emph{single} \emph{scale} conditions relating the
geometry to the equation at \textbf{the one} point $\left( x,r_{0}\right) $.
\end{remark}

\begin{remark}
For the purposes of this paper we could simply take the metric $d$ to be the
Carnot-Carath\'{e}odory metric associated with $A$, but the present
formulation allows for additional flexibility in the choice of balls used
for DeGiorgi iteration in other situations.
\end{remark}

In the special case that a weak subsolution $u$ to (\ref{eq_0}) is \emph{%
nonpositive} on the boundary of a ball $B\left( x,r_{0}\right) $, we can
obtain a global boundedness inequality $\left\Vert u\right\Vert _{L^{\infty
}\left( B\left( x,r_{0}\right) \right) }\lesssim \Vert \phi \Vert _{X\left(
B\left( x,r_{0}\right) \right) }$ from the arguments used for Theorem \ref%
{bound_gen_thm}, simply by noting that integration by parts no longer
requires premultiplication by a Lipschitz cutoff function. Moreover, the
ensuing arguments work just as well for an arbitrary bounded open set $%
\Omega $ in place of the ball $B\left( x,r_{0}\right) $, provided only that
we assume our Sobolev inequality for $\Omega $ instead of for the ball $%
B\left( x,r_{0}\right) $. Of course there is no role played here by a
superradius $\varphi $. This type of result is usually referred to as a 
\emph{maximum principle}, and we now formulate our theorem precisely.

\begin{definition}
Fix a bounded domain $\Omega \subset \mathbb{R}^{n}$. Then the $\Phi $%
-Orlicz Sobolev bump inequality for $\Omega $ is: 
\begin{equation}
\left\Vert w\right\Vert _{L^{\Phi }\left( \Omega \right) }\leq C\left\Vert
\nabla _{A}w\right\Vert _{L^{1}\left( \Omega \right) },\ \ \ \ \ w\in
Lip_{0}\left( \Omega \right) .  \label{Phi bump' max}
\end{equation}
\end{definition}

\begin{definition}
\label{def A admiss max}Fix a bounded domain $\Omega \subset \mathbb{R}^{n}$%
. We say $\phi $ is $A$\emph{-admissible} for $\Omega $ if 
\begin{equation*}
\Vert \phi \Vert _{X\left( \Omega \right) }\equiv \sup_{v\in \left(
W_{A}^{1,1}\right) _{0}(\Omega )}\frac{\int_{\Omega }\left\vert v\phi
\right\vert \,dy}{\int_{\Omega }\Vert \nabla _{A}v\Vert \,dy}<\infty .
\end{equation*}
\end{definition}

We say a function $u\in W_{A}^{1,2}\left( \Omega \right) $ is \emph{bounded
by a constant }$\ell \in \mathbb{R}$ on the boundary $\partial \Omega $ if $%
\left( u-\ell \right) ^{+}=\max \left\{ u-\ell ,0\right\} \in \left(
W_{A}^{1,2}\right) _{0}\left( \Omega \right) $. We define $\sup_{x\in
\partial \Omega }u\left( x\right) $ to be $\inf \left\{ \ell \in \mathbb{R}%
:\left( u-\ell \right) ^{+}\in \left( W_{A}^{1,2}\right) _{0}\left( \Omega
\right) \right\} $.

\begin{theorem}[abstract maximum principle]
\label{max}Let $\Omega $ be a bounded domain in $\mathbb{R}^{n}$ with $n\geq
2$. Suppose that $\mathcal{A}(x,z)$ is a nonnegative semidefinite matrix in $%
\Omega \times \mathbb{R}$ that satisfies the structural condition (\ref%
{struc_0}). Let $u$ be a nonnegative subsolution of (\ref{eq_0}). Then the
following maximum principle holds, 
\begin{equation*}
\limfunc{esssup}_{x\in \Omega }u\left( x\right) \leq \sup_{x\in \partial
\Omega }u\left( x\right) +C\left\Vert \phi \right\Vert _{X(\Omega )}\ ,
\end{equation*}%
where the constant $C$ depends only on $\Omega $, provided that:

\begin{enumerate}
\item the function $\phi $ is $A$-admissible for $\Omega $,

\item the $\Phi $-Orlicz Sobolev bump inequality (\ref{Phi bump' max}) for $%
\Omega $ holds with $\Phi =\Phi _{N}$ for some $N>1$.
\end{enumerate}
\end{theorem}

In order to obtain a \emph{geometric} local boundedness theorem, as well as
a \emph{geometric} maximum principle, we will take the metric $d$ in Theorem %
\ref{bound_gen_thm} to be the Carnot-Caratheodory metric associated with the
vector field $\nabla _{A}$, and we will replace the hypotheses (2) and (3)
in Theorem \ref{bound_gen_thm} with a geometric description of appropriate
balls. For this we need to introduce a family of infinitely degenerate
geometries that are simple enough that we can compute the balls, prove the
required Sobolev Orlicz bump inequality, and define an appropriate
accumulating sequence of Lipschitz cutoff functions.

We consider special quasilinear operators%
\begin{equation*}
\mathcal{L}u\left( x,y\right) \equiv \nabla ^{\func{tr}}\mathcal{A}\left(
x,u\left( x\right) \right) \nabla u\left( x\right) ,\ \ \ \ \ x\in \Omega ,
\end{equation*}%
where $\Omega \subset \mathbb{R}^{n}$ and the $n\times n$ matrix $\mathcal{A}%
\left( x,z\right) $ has bounded measurable coefficients and is comparable to
an $n$-dimensional matrix%
\begin{equation*}
A\left( x\right) \equiv \left[ 
\begin{array}{ccccc}
1 & 0 & \cdots & 0 & 0 \\ 
0 & \ddots &  &  & \vdots \\ 
\vdots &  & \ddots & 0 & 0 \\ 
0 &  & 0 & 1 & 0 \\ 
0 & \cdots & 0 & 0 & f\left( x_{1}\right) ^{2}%
\end{array}%
\right] ,\ \ \ \ \ f\left( s\right) =e^{-F\left( s\right) },
\end{equation*}%
by which we mean that 
\begin{eqnarray}
&&\frac{1}{C}\left( \xi _{1}^{2}+...+\xi _{n-1}^{2}+f\left( x_{1}\right)
^{2}\xi _{n}^{2}\right)  \label{form bound'} \\
&\leq &\left( \xi _{1},...,\xi _{n}\right) \mathcal{A}\left( x,z\right)
\left( 
\begin{array}{c}
\xi _{1} \\ 
\vdots \\ 
\xi _{n}%
\end{array}%
\right)  \notag \\
&\leq &C\left( \xi _{1}^{2}+...+\xi _{n-1}^{2}+f\left( x_{1}\right) ^{2}\xi
_{n}^{2}\right) ,\ \ \ \ a.e.x\in \Omega ,\ z\in \mathbb{R}\text{ and }\xi
\in \mathbb{R}^{n},  \notag
\end{eqnarray}%
and where the degeneracy function $f\left( x\right) =e^{-F\left( x\right) }$
is even and there is $R>0$ such that $F$ satisfies the following five
structure conditions for some constants $C\geq 1$ and $\varepsilon >0$.

\begin{definition}[structure conditions]
\label{structure conditions}A function $F:\left( 0,R\right) \rightarrow 
\mathbb{R}$ is said to satisfy \emph{geometric structure conditions} if:

\begin{enumerate}
\item $\lim_{x\rightarrow 0^{+}}F\left( x\right) =+\infty $;

\item $F^{\prime }\left( x\right) <0$ and $F^{\prime \prime }\left( x\right)
>0$ for all $x\in (0,R)$;

\item $\frac{1}{C}\left\vert F^{\prime }\left( r\right) \right\vert \leq
\left\vert F^{\prime }\left( x\right) \right\vert \leq C\left\vert F^{\prime
}\left( r\right) \right\vert $ for $\frac{1}{2}r<x<2r<R$;

\item $\frac{1}{-xF^{\prime }\left( x\right) }$ is increasing in the
interval $\left( 0,R\right) $ and satisfies $\frac{1}{-xF^{\prime }\left(
x\right) }\leq \frac{1}{\varepsilon }\,$for $x\in (0,R)$;

\item $\frac{F^{\prime \prime }\left( x\right) }{-F^{\prime }\left( x\right) 
}\approx \frac{1}{x}$ for $x\in (0,R)$.
\end{enumerate}
\end{definition}

\begin{remark}
We make no smoothness assumption on $f$ other than the existence of the
second derivative $f^{\prime \prime }$ on the open interval $(0,R)$. Note
also that at one extreme, $f$ can be of finite type, namely $f\left(
x\right) =x^{\alpha }$ for any $\alpha >0$, and at the other extreme, $f$
can be of strongly degenerate type, namely $f\left( x\right) =e^{-\frac{1}{%
x^{\alpha }}}$ for any $\alpha >0$. Assumption (1) rules out the elliptic
case $f\left( 0\right) >0$.
\end{remark}

Using these geometric structure conditions, we can show that standard
sequences of Lipschitz cutoff functions always exist for our geometries.

\begin{lemma}
\label{cutoff exist}If $\gamma >1$ and $A\left( x\right) $ is a \emph{%
continuous} nonnegative semidefinite $n\times n$ matrix valued function on a
bounded domain $\Omega \subset \mathbb{R}^{n}$ as above, and if $d$ is the
associated control metric, then for every $r>0$ and $x\in \Omega $, there is
an $\left( A,d\right) $-\emph{standard} sequence of Lipschitz cutoff
functions $\left\{ \psi _{j}\right\} _{j=1}^{\infty }$ at $\left( x,r\right) 
$, associated with balls $B(x,r_{j})\supset \limfunc{supp}\psi _{j}$ as in
Definition \ref{def_cutoff}.
\end{lemma}

\begin{proof}
This follows immediately from Proposition 68 on page 90 of \cite{SaWh4},
once we observe that in the proof of Proposition 68, we can take $N$ to be
any real number greater than $1$ (so that $\sum_{j=1}^{\infty }j^{-N}<\infty 
$), and that the assumption of the containment condition in Proposition 68
there was \emph{only} used in the proof to conclude that the annuli $%
B(x,t)\setminus B(x,s)$ have positive Euclidean thickness for $0<s<t<\infty $
- i.e. that the boundaries $\partial B(x,t)$ and $\partial B(x,s)$ are
pairwise disjoint. This is certainly the case for the control balls $B(x,r)$
associated with our geometries $F$ satisfying Definition \ref{structure
conditions}, and so the proof of Proposition 68 of \cite{SaWh4} applies to
prove Lemma \ref{cutoff exist}.
\end{proof}

In the next theorem we will consider the geometry of balls defined by%
\begin{equation*}
F_{\sigma }\left( r\right) =r^{-\sigma },
\end{equation*}%
where $\sigma >0$. Note that $f_{\sigma }=e^{-F_{\sigma }}$ vanishes to
infinite order at $r=0$, and that $f_{\sigma }$ vanishes to a faster order
than $f_{\sigma ^{\prime }}$ if $\sigma >\sigma ^{\prime }$. We also define
the simpler linear operator%
\begin{equation*}
Lu\equiv \func{div}A\left( x\right) \nabla u.
\end{equation*}%
with $A(x)$ as in (\ref{struc_0}).

\begin{theorem}
\label{loc}Suppose that $\Omega \subset \mathbb{R}^{n}$ is a domain in $%
\mathbb{R}^{n}$ with $n\geq 2$ and that 
\begin{equation*}
\mathcal{L}u\equiv \func{div}\mathcal{A}\left( x,u\right) \nabla u,\ \ \ \ \
x=\left( x_{1},...,x_{n}\right) \in \Omega ,
\end{equation*}%
where$\ \mathcal{A}\left( x,z\right) \sim \left[ 
\begin{array}{cc}
I_{n-1} & 0 \\ 
0 & f\left( x_{1}\right) ^{2}%
\end{array}%
\right] $, $I_{n-1}$ is the $\left( n-1\right) \times \left( n-1\right) $
identity matrix, $\mathcal{A}$ has bounded measurable components, and the
geometry $F=-\ln f$ satisfies the geometric structure conditions in
Definition \ref{structure conditions}.

\begin{enumerate}
\item If $F\leq F_{\sigma }$ for some $0<\sigma <1$, then every weak
subsolution to $\mathcal{L}u=\phi $ with $A$-admissible $\phi $ is locally
bounded in $\Omega $.

\item On the other hand, if $n\geq 3$ and $\sigma \geq 1$, then there exists
a locally unbounded weak solution $u$ in a neighbourhood of the origin in $%
\mathbb{R}^{n}$ to the equation $Lu=0$ with geometry $F=F_{\sigma }$.
\end{enumerate}
\end{theorem}

\begin{theorem}
	\label{max copy(1)}Suppose that $F$ satisfies the geometric structure
	conditions in Definition \ref{structure conditions} and $F\leq F_{\sigma }$ for some $0<\sigma <1$. Assume that $u$ is a
	weak subsolution to $\mathcal{L}u=\phi $ in a domain $\Omega \subset \mathbb{%
		R}^{n}$ with $n\geq 2$, where $\mathcal{L}$ has degeneracy $F$ and $\phi $
	is $A$-admissible. Moreover, suppose that $u$ is bounded in the weak sense
	on the boundary $\partial \Omega $. Then $u$ is globally bounded in $\Omega $
	and satisfies%
	\begin{equation*}
		\sup_{\Omega }u\leq \sup_{\partial \Omega }u+C\left\Vert \phi \right\Vert
		_{X\left( \Omega \right) }\ 
	\end{equation*}
	with the constant $C$ depending only on $\Omega$.
\end{theorem}

\chapter{Organization of the proofs}

In Part 2 we use DeGiorgi iteration to prove the abstract local boundedness and
maximum principle theorems. Then in Part 3 we first calculate
geodesics and volumes of balls in our geometries $F$ satisfying the
geometric structure conditions in Definition \ref{structure conditions}, and
second establish subrepresentation and Orlicz Sobolev inequalities. Finally
then we can prove the geometric theorems. 

\part{Abstract theory}

There are two main ingredients needed to prove local boundedness and maximum
principle, namely the Orlicz Sobolev inequality for
compactly supported functions, and the Caccioppoli inequality for subsolutions of the
degenerate equation. We start with the Orlicz Sobolev inequality we need.

Let $\Phi $ be a Young function on $\left( 0,\infty \right) $ and let $F$ be
a geometry satisfying the geometric structure conditions in Definition \ref%
{structure conditions}. We will assume initially that we have an Orlicz
Sobolev norm inequality for the control balls $B$ in some domain $\Omega
\subset \mathbb{R}^{n}$: 
\begin{equation}
\left\Vert w\right\Vert _{L^{\Phi }\left( \mu _{B}\right) }\leq \varphi
\left( r_{0}\right) \ \left\Vert \nabla _{A}w\right\Vert _{L^{1}\left( \mu
_{B}\right) },\ \ \ \ \ w\in \left( W_{A}^{1,1}\right) _{0}\left( B\right) ,
\label{OSN}
\end{equation}%
for some increasing `superradius' function $\varphi \left( r_{0}\right) \geq
r_{0}$, where $r_{0}$ is the control radius of a control ball $B$ and $\mu
_{B}$ is normalized Lebesgue measure on $B$. We prove this inequality for
appropriate geometries and superradii below in Proposition \ref{sob_nd} and
Lemma \ref{norm half}.

Next, we establish a Caccioppoli inequality for weak subsolutions that holds
independently of any geometric considerations.

\begin{proposition}\label{Cacc_prop}
If $u$ is a weak subsolution to $\mathcal{L}u=\phi $ with $\mathcal{L}$ as
in (\ref{eq_0}), (\ref{struc_0}), and $A$-admissible $\phi $, then the the
following Caccioppoli inequality holds for $u_{+}$ on a ball $B$ with
Lipschitz $\psi $ supported in $B$: 
\begin{equation}
\int |\nabla _{A}(\psi u_{+})|^{2}d\mu \leq C\left( \left\Vert \psi
\right\Vert _{L^{2}\left( \mu \right) }^{2}\left\Vert \phi \right\Vert
_{X}^{2}+\left\Vert \nabla _{A}\psi \right\Vert _{L^{\infty
}}^{2}\int_{B}u_{+}^{2}d\mu \right) ,  \label{standard Cacc}
\end{equation}%
where $d\mu =\frac{dx}{\left\vert B\right\vert }$, and where the constant $C$
depends only on the constants in (\ref{struc_0}) and not on $r(B)$.
\end{proposition}

\begin{proof}
Recall we use $C$ and $c$ to denote constants that may change from line to
line. For convenience in notation we denote the matrix function $\mathcal{A}%
\left( x,u\left( x\right) \right) $ by $\mathcal{A}$, and its associated
gradient by $\nabla _{\mathcal{A}}$. Since our conclusion involves the
matrix $A$, while the definition of $u$ being a subsolution involves the
matrix $\mathcal{A}$, we must apply care in using the comparability of the
matrices $A$ and $\mathcal{A}$ in the positive definite sense. We have the\
pointwise vector identity%
\begin{eqnarray*}
&&|\nabla _{\mathcal{A}}(\psi u_{+})|^{2}-\nabla \left( \psi
^{2}u_{+}\right) \mathcal{A}\nabla u_{+} \\
&=&\left\vert \psi \nabla _{\mathcal{A}}u_{+}+u_{+}\nabla _{\mathcal{A}}\psi
\right\vert ^{2}-\left[ \psi ^{2}\nabla u_{+}+u_{+}\nabla \psi ^{2}\right] 
\mathcal{A}\nabla u_{+} \\
&=&\left\vert \psi \nabla _{\mathcal{A}}u_{+}\right\vert ^{2}+2\left\langle
\psi \nabla _{\mathcal{A}}u_{+},u_{+}\nabla _{\mathcal{A}}\psi \right\rangle
+\left\vert u_{+}\nabla _{\mathcal{A}}\psi \right\vert ^{2} \\
&&-\left\vert \psi \nabla _{\mathcal{A}}u_{+}\right\vert ^{2}-2\left\langle
u_{+}\nabla _{\mathcal{A}}\psi ,\psi \nabla _{\mathcal{A}}u_{+}\right\rangle
\\
&=&\left\vert u_{+}\nabla _{\mathcal{A}}\psi \right\vert ^{2}.
\end{eqnarray*}%
Then using (\ref{struc_0}), we obtain%
\begin{equation*}
c|\nabla _{A}(\psi u_{+})|^{2}\leq \nabla \left( \psi ^{2}u_{+}\right) 
\mathcal{A}\nabla u_{+}+C\left\vert u_{+}\nabla _{A}\psi \right\vert ^{2},
\end{equation*}%
and using the fact that $d\mu \left( x\right) $ is a constant multiple of
Lebesgue measure $dx$, we integrate to obtain that%
\begin{equation}
\left\Vert \nabla _{A}(\psi u_{+})\right\Vert _{L^{2}\left( \mu \right)
}^{2}\leq C\int \nabla \left( \psi ^{2}u_{+}\right) \mathcal{A}\nabla
u_{+}d\mu +C\left\Vert u_{+}\nabla _{A}\psi \right\Vert _{L^{2}\left( \mu
\right) }^{2}.  \label{one}
\end{equation}

Next, since $u$ is a weak subsolution to $\mathcal{L}u=\phi $, we have $%
\nabla \mathcal{A}\nabla u\geq \phi $ in the weak sense, which implies 
\begin{equation}
\int \nabla \left( \psi ^{2}u_{+}\right) \mathcal{A}\nabla u_{+}d\mu =\int
\nabla \left( \psi ^{2}u_{+}\right) \mathcal{A}\nabla ud\mu \leq -\int
\left( \psi ^{2}u_{+}\right) \phi d\mu \leq \left\Vert \phi \right\Vert
_{X}\int \left\vert \nabla _{\mathcal{A}}\left( \psi ^{2}u_{+}\right)
\right\vert d\mu \ .  \label{two}
\end{equation}%
Now we compute that 
\begin{eqnarray*}
\int \left\vert \nabla _{\mathcal{A}}\left( \psi ^{2}u_{+}\right)
\right\vert d\mu &=&\int \left\vert \psi u_{+}\nabla _{\mathcal{A}}\psi
+\psi \nabla _{\mathcal{A}}(\psi u_{+})\right\vert d\mu \\
&\leq &\left\Vert \psi u_{+}\right\Vert _{L^{2}\left( \mu \right)
}\left\Vert \nabla _{\mathcal{A}}\psi \right\Vert _{L^{2}\left( \mu \right)
}+\left\Vert \nabla _{\mathcal{A}}(\psi u_{+})\right\Vert _{L^{2}\left( \mu
\right) }\left\Vert \psi \right\Vert _{L^{2}\left( \mu \right) } \\
&\leq &\left\Vert \nabla _{\mathcal{A}}\psi \right\Vert _{L^{2}\left( d\mu
\right) }\left\Vert \psi u_{+}\right\Vert _{L^{2}\left( \mu \right)
}+\left\Vert \psi \right\Vert _{L^{2}\left( d\mu \right) }\left\Vert \nabla
_{\mathcal{A}}(\psi u_{+})\right\Vert _{L^{2}\left( \mu \right) },
\end{eqnarray*}%
and combining this with (\ref{one}) and (\ref{two}), and using (\ref{struc_0}%
) again, we obtain 
\begin{eqnarray}
&&  \label{three} \\
\left\Vert \nabla _{A}(\psi u_{+})\right\Vert _{L^{2}\left( \mu \right)
}^{2} &\leq &C\left\Vert \phi \right\Vert _{X}\left\{ \left\Vert \nabla
_{A}\psi \right\Vert _{L^{2}\left( d\mu \right) }\left\Vert \psi
u_{+}\right\Vert _{L^{2}\left( \mu \right) }+\left\Vert \psi \right\Vert
_{L^{2}\left( d\mu \right) }\left\Vert \nabla _{A}(\psi u_{+})\right\Vert
_{L^{2}\left( \mu \right) }\right\}  \notag \\
&&\ \ \ \ \ \ \ \ \ \ \ \ \ \ \ \ \ \ \ \ +C\left\Vert \nabla _{A}\psi
\right\Vert _{L^{\infty }}^{2}\int_{B}u_{+}^{2}d\mu \ .  \notag
\end{eqnarray}%
To estimate the first term on the right hand side of (\ref{three}), we apply
Young's inequality twice to get 
\begin{eqnarray*}
&&\left\Vert \phi \right\Vert _{X}\left( \left\Vert \nabla _{A}\psi
\right\Vert _{L^{2}\left( d\mu \right) }\left\Vert \psi u_{+}\right\Vert
_{L^{2}\left( \mu \right) }+\left\Vert \psi \right\Vert _{L^{2}\left( d\mu
\right) }\left\Vert \nabla _{A}(\psi u_{+})\right\Vert _{L^{2}\left( \mu
\right) }\right) \\
&\leq &\frac{1}{\varepsilon }\left\Vert \phi \right\Vert
_{X}^{2}+\varepsilon \left( \left\Vert \nabla _{A}\psi \right\Vert
_{L^{2}\left( d\mu \right) }\left\Vert \psi u_{+}\right\Vert _{L^{2}\left(
\mu \right) }+\left\Vert \psi \right\Vert _{L^{2}\left( d\mu \right)
}\left\Vert \nabla _{A}(\psi u_{+})\right\Vert _{L^{2}\left( \mu \right)
}\right) ^{2} \\
&\leq &\frac{1}{\varepsilon }\left\Vert \phi \right\Vert
_{X}^{2}+C\varepsilon \left( \left\Vert \nabla _{A}\psi \right\Vert
_{L^{2}\left( d\mu \right) }^{2}\left\Vert \psi u_{+}\right\Vert
_{L^{2}\left( \mu \right) }^{2}+\left\Vert \psi \right\Vert _{L^{2}\left(
d\mu \right) }^{2}\left\Vert \nabla _{A}(\psi u_{+})\right\Vert
_{L^{2}\left( \mu \right) }^{2}\right) ,
\end{eqnarray*}%
and combining this with (\ref{three}), we obtain%
\begin{eqnarray}
\left\Vert \nabla _{A}(\psi u_{+})\right\Vert _{L^{2}\left( \mu \right)
}^{2} &\leq &C\frac{1}{\varepsilon }\left\Vert \phi \right\Vert
_{X}^{2}+C\varepsilon \left\Vert \nabla _{A}\psi \right\Vert _{L^{2}\left(
d\mu \right) }^{2}\left\Vert \psi u_{+}\right\Vert _{L^{2}\left( \mu \right)
}^{2}  \label{we obtain} \\
&&+C\varepsilon \left\Vert \psi \right\Vert _{L^{2}\left( d\mu \right)
}^{2}\left\Vert \nabla _{A}(\psi u_{+})\right\Vert _{L^{2}\left( \mu \right)
}^{2}+C\left\Vert \nabla _{A}\psi \right\Vert _{L^{\infty
}}^{2}\int_{B}u_{+}^{2}d\mu \ .  \notag
\end{eqnarray}%
Now choose $\varepsilon $ so small that 
\begin{equation*}
C\varepsilon \left\Vert \psi \right\Vert _{L^{2}\left( d\mu \right)
}^{2}=C\varepsilon \frac{1}{\left\vert B\right\vert }\int_{B}\psi ^{2}\leq
C\varepsilon \left\Vert \psi \right\Vert _{L^{\infty }}^{2}<\frac{1}{2},
\end{equation*}%
and then absorb the third term on the right hand side of (\ref{we obtain})
into its left hand side to obtain 
\begin{equation*}
\left\Vert \nabla _{A}(\psi u_{+})\right\Vert _{L^{2}\left( \mu \right)
}^{2}\leq C\left( \left\Vert \psi \right\Vert _{L^{2}\left( d\mu \right)
}^{2}\left\Vert \phi \right\Vert _{X}^{2}+\left\Vert \nabla _{A}\psi
\right\Vert _{L^{\infty }}^{2}\int_{B}u_{+}^{2}d\mu \right)
\end{equation*}%
upon using $\left\Vert \nabla _{A}\psi \right\Vert _{L^{2}\left( d\mu
\right) }^{2}\leq \left\Vert \nabla _{A}\psi \right\Vert _{L^{\infty }}^{2}$%
. This completes the proof of (\ref{standard Cacc}).
\end{proof}

\begin{remark}
It is important to note that (\ref{standard Cacc}) holds for $u_{+}$
whenever $u$ is a weak subsolution \emph{without} assuming that $u_{+}$ is
also a subsolution.
\end{remark}

\begin{corollary}\label{Cacc_cor}[of the proof]
Let $u$ be a weak subsolution to $\mathcal{L}u=\phi $ with $\mathcal{L}$ as
in (\ref{eq_0}), (\ref{struc_0}), and $A$-admissible $\phi $, and suppose
for some ball $B$, $P>0$, and a nonnegative function $v\in W_{A}^{1,2}$
there holds 
\begin{equation}
\left\Vert \phi \right\Vert _{X}^{2}\leq Pv(x),\quad \text{a.e.}\ x\in
\{u>0\}\cap B.  \label{P_bound}
\end{equation}%
Then 
\begin{equation}
\int |\nabla _{A}(\psi u_{+})|^{2}d\mu \leq C\left( \left\Vert \nabla
_{A}\psi \right\Vert _{L^{\infty }}+P\right) ^{2}\int \left(
u_{+}^{2}+v^{2}\right) d\mu ,  \label{Cacc_inhomog}
\end{equation}%
where $d\mu =\frac{dx}{\left\vert B\right\vert }$, and where $\psi $ is a
Lipschitz cutoff function supported in $B$.
\end{corollary}

\begin{proof}
First, recall that from (\ref{one}) we have 
\begin{equation}
\int |\nabla _{A}(\psi u_{+})|^{2}d\mu \leq C\int \nabla \left( \psi
^{2}u_{+}\right) \mathcal{A}\nabla u_{+}d\mu +C\left\Vert \nabla _{A}\psi
\right\Vert _{L^{\infty }}^{2}\int_{B}u_{+}^{2}d\mu \ ,  \label{one'}
\end{equation}%
where the constant $C$ depends only on constants in (\ref{struc_0}). Now
using (\ref{two}) and (\ref{P_bound}) we have 
\begin{align*}
\int \nabla \left( \psi ^{2}u_{+}\right) \mathcal{A}\nabla u_{+}d\mu & \leq
\left\Vert \phi \right\Vert _{X}\int \left\vert \nabla _{\mathcal{A}}\left(
\psi ^{2}u_{+}\right) \right\vert d\mu \leq P\int v\left\vert \nabla _{%
\mathcal{A}}\left( \psi ^{2}u_{+}\right) \right\vert d\mu \\
& =P\int \left\vert \psi u_{+}v\nabla _{\mathcal{A}}\psi d\mu +\psi v\nabla
_{\mathcal{A}}(\psi u_{+})\right\vert d\mu \\
& \leq P\left\Vert \psi u_{+}\right\Vert _{L^{2}\left( \mu \right)
}\left\Vert v\nabla _{\mathcal{A}}\psi \right\Vert _{L^{2}\left( \mu \right)
}+\varepsilon \left\Vert \nabla _{\mathcal{A}}(\psi u_{+})\right\Vert
_{L^{2}\left( \mu \right) }^{2}+\frac{P^{2}}{\varepsilon }\left\Vert \psi
v\right\Vert _{L^{2}\left( \mu \right) }^{2} \\
& \leq CP\left\Vert \psi u_{+}\right\Vert _{L^{2}\left( \mu \right)
}\left\Vert v\nabla _{A}\psi \right\Vert _{L^{2}\left( \mu \right)
}+\varepsilon C\left\Vert \nabla _{A}(\psi u_{+})\right\Vert _{L^{2}\left(
\mu \right) }^{2}+\frac{P^{2}}{\varepsilon }\left\Vert \psi v\right\Vert
_{L^{2}\left( \mu \right) }^{2} \\
& \leq CP\left\Vert \nabla _{A}\psi \right\Vert _{L^{\infty }}\int \left(
u_{+}^{2}+v^{2}\right) d\mu +C\varepsilon \left\Vert \nabla _{A}(\psi
u_{+})\right\Vert _{L^{2}\left( \mu \right) }^{2}+\frac{P^{2}}{\varepsilon }%
\int (\psi v)^{2}d\mu ,
\end{align*}%
where in passing from the third line to the fourth line above, we have
replaced $\mathcal{A}$ with $A$ at the expense of multiplying by the
constant $C$ in (\ref{struc_0}). Combining this with (\ref{one'}), and
choosing $\varepsilon $ small enough to absorb the second summand on the
right, we obtain 
\begin{equation*}
\int |\nabla _{A}(\psi u_{+})|^{2}d\mu \leq C\left( \left\Vert \nabla
_{A}\psi \right\Vert _{L^{\infty }}+P\right) ^{2}\int \left(
u_{+}^{2}+v^{2}\right) d\mu .
\end{equation*}
\end{proof}

\chapter{Local Boundedness}

\label{Chapter local boundedness}

We begin with a short review of that part of the theory of Orlicz norms that
is relevant for us.

\section{Orlciz norms}

Recall that if $\Phi :\left[ 0,\infty \right) \rightarrow \left[ 0,\infty
\right) $ is a Young function, we define $L^{\Phi }$ to be the linear span
of $L_{\ast }^{\Phi }$, the set of measurable functions $f:X\rightarrow 
\mathbb{R}$ such that the integral $\int_{X}\Phi \left( \left\vert
f\right\vert \right) d\mu $ is finite, and then define%
\begin{equation*}
	\left\Vert f\right\Vert _{L^{\Phi }\left( \mu \right) }\equiv \inf \left\{
	k\in \left( 0,\infty \right) :\int_{X}\Phi \left( \frac{\left\vert
		f\right\vert }{k}\right) d\mu \leq 1\right\} .
\end{equation*}%
In our application to DeGiorgi iteration the convex bump function $\Phi
\left( t\right) $ will satisfy in addition:

\begin{itemize}
	\item The function $\frac{\Phi (t)}{t}$ is positive, nondecreasing and tends
	to $\infty $ as $t\rightarrow \infty $;
	
	\item $\Phi$ is submultiplicative on an interval $\left( E,\infty \right) $
	for some $E>1$: 
	\begin{equation}
	\Phi \left( ab\right) \leq \Phi \left( a\right) \Phi \left( b\right) ,\ \ \
	\ \ a,b>E.  \label{submult}
	\end{equation}
\end{itemize}

Note that if we consider more generally the quasi-submultiplicative
condition,%
\begin{equation}
\Phi \left( ab\right) \leq K\Phi \left( a\right) \Phi \left( b\right) ,\ \ \
\ \ a,b>E,  \label{submult quasi}
\end{equation}%
for some constant $K$, then $\Phi \left( t\right) $ satisfies (\ref{submult
	quasi}) if and only if $\Phi _{K}\left( t\right) \equiv K\Phi \left(t\right) 
$ satisfies (\ref{submult}). Thus we can alway rescale a
quasi-submultiplicative function to be submultiplicative.

Now let us consider the \emph{linear extension} $\Phi _{\func{ext}}$ of a
function $\Phi :\left[ E,\infty \right) \rightarrow \left[ 0,\infty \right) $
to the entire positive real axis $\left( 0,\infty \right) $ given by%
\begin{equation*}
	\Phi _{\func{ext}}\left( t\right) =\left\{ 
	\begin{array}{ccc}
		\Phi \left( t\right) & \text{ if } & E\leq t<\infty \\ 
		\frac{\Phi \left( E\right) }{E}t & \text{ if } & 0\leq t\leq E%
	\end{array}%
	\right. .
\end{equation*}%
We claim that $\Phi _{\func{ext}}$ is submultiplicative on $\left( 0,\infty
\right) $, i.e. 
\begin{equation*}
	\Phi _{\func{ext}}\left( ab\right) \leq \Phi _{\func{ext}}\left( a\right)
	\Phi _{\func{ext}}\left( b\right) ,\ \ \ \ \ a,b>0.
\end{equation*}%
In fact, the identity$\frac{\Phi _{\func{ext}}(t)}{t}=\frac{\Phi _{\func{ext}%
	}(\max \{t,E\})}{\max \{t,E\}}$ and the monotonicity of $\frac{\Phi (t)}{t}$
imply 
\begin{eqnarray*}
	\frac{\Phi _{\func{ext}}(ab)}{ab} &\leq &\frac{\Phi _{\func{ext}}(\max
		\{a,E\}\max \{b,E\})}{\max \{a,E\}\max \{b,E\}} \\
	&\leq &\frac{\Phi _{\func{ext}}(\max \{a,E\})}{\max \{a,E\}}\cdot \frac{\Phi
		_{\func{ext}}(\max \{b,E\})}{\max \{b,E\}}=\frac{\Phi _{\func{ext}}(a)}{a}%
	\frac{\Phi _{\func{ext}}(b)}{b}.
\end{eqnarray*}

\begin{conclusion}
	\label{sub extensions}If $\Phi :[E,\infty )\rightarrow {\mathbb{R}}^{+}$ is
	a submultiplicative piecewise differentiable strictly convex function with
	the property that $\frac{\Phi (t)}{t}$ is nondecreasing on $[E,\infty )$,
	then we can extend $\Phi $ to a submultiplicative piecewise differentiable
	convex function $\Phi _{\func{ext}}$ on $\left[ 0,\infty \right) $ that
	vanishes at $0$ \emph{if and only if} 
	\begin{equation}
	\Phi ^{\prime }\left( E\right) \geq \frac{\Phi \left( E\right) }{E}.
	\label{extension}
	\end{equation}
\end{conclusion}

So now we suppose that $\Phi $ and $E\in \left( 0,\infty \right) $ satisfy (%
\ref{extension}) and that $\Phi _{\func{ext}}$ is a submultiplicative
piecewise differentiable convex function on $\left[ 0,\infty \right) $ that
vanishes at $0$, and moreover is strictly convex on $\left( E,\infty \right) 
$. Let 
\begin{equation*}
	\Psi \left( t\right) =\Phi _{\func{ext}}^{\prime }\left( t\right) =\left\{ 
	\begin{array}{ccc}
		\Phi ^{\prime }\left( t\right) & \text{ if } & E\leq t<\infty \\ 
		\frac{\Phi \left( E\right) }{E} & \text{ if } & 0\leq t<E%
	\end{array}%
	\right. .
\end{equation*}%
Now $\Psi $ is increasing on $\left( 0,\infty \right) $, but is constant on $%
\left( 0,E\right) $, and in addition has a jump discontinuity at $E$ if $%
\Phi ^{\prime }\left( E\right) >\frac{\Phi \left( E\right) }{E}$. Since $%
\Psi ^{\left( -1\right) }$ does not exist on all of $\left( 0,\infty \right) 
$, we instead \textbf{define} the function $\Psi ^{\left( -1\right) }$ by
reflecting the graph of $\Psi $ about the line $s=t$ in the $\left(
t,s\right) $-plane: 
\begin{equation*}
	\Psi ^{\left( -1\right) }\left( s\right) =\left\{ 
	\begin{array}{ccc}
		\left( \Phi ^{\prime }\right) ^{\left( -1\right) }\left( s\right) & \text{
			if } & \Phi ^{\prime }\left( E\right) \leq s<\infty \\ 
		E & \text{ if } & \frac{\Phi \left( E\right) }{E}\leq s<\Phi ^{\prime
		}\left( E\right) \\ 
		0 & \text{ if } & 0<s<\frac{\Phi \left( E\right) }{E}%
	\end{array}%
	\right. .
\end{equation*}%
Finally we \textbf{define} the conjugate function $\widetilde{\Phi }$ of $%
\Phi $ by the formula%
\begin{align*}
	\widetilde{\Phi }\left( s\right) &\equiv \int_{0}^{s}\Psi ^{\left( -1\right)
	}\left( x\right) dx\\
	&=\left\{ 
	\begin{array}{ccc}
		E\Phi ^{\prime }\left( E\right) -\Phi \left( E\right) +\int_{\Phi ^{\prime
			}\left( E\right) }^{s}\left( \Phi ^{\prime }\right) ^{\left( -1\right)
		}\left( x\right) dx & \text{ if } & \Phi ^{\prime }\left( E\right) \leq
		s<\infty \\ 
		Es-\Phi \left( E\right) & \text{ if } & \frac{\Phi \left( E\right) }{E}\leq
		s<\Phi ^{\prime }\left( E\right) \\ 
		0 & \text{ if } & 0<s<\frac{\Phi \left( E\right) }{E}%
	\end{array}%
	\right. .
\end{align*}%
One now has the standard Young's inequality for the pair of functions $%
\left( \Psi ,\Psi ^{\left( -1\right) }\right) $:%
\begin{equation*}
	ts\leq \Phi \left( t\right) +\widetilde{\Phi }\left( s\right) ,\ \ \ \ \
	0\leq t,s<\infty .
\end{equation*}%
Indeed, the left hand side is the area of the rectangle $\left[ 0,t\right]
\times \left[ 0,s\right] $ in the $\left( t,s\right) $-plane, which is at
most the area $\int_{0}^{t}\Psi \left( y\right) dy$ under the graph of $\Psi 
$ up to $t$ plus the area $\int_{0}^{s}\Psi ^{\left( -1\right) }\left(
y\right) dy$ under the graph of $\Psi ^{\left( -1\right) }$ up to $s$. As a
consequence we have the following generalization H\"{o}lder's inequality:%
\begin{eqnarray*}
	&&\frac{\int_{X}\left\vert f\left( x\right) g\left( x\right) \right\vert
		d\mu \left( x\right) }{\left\Vert f\right\Vert _{L^{\Phi }\left( \mu \right)
		}\left\Vert g\right\Vert _{L^{\widetilde{\Phi }}\left( \mu \right) }}%
	=\int_{X}\frac{\left\vert f\left( x\right) \right\vert }{\left\Vert
		f\right\Vert _{L^{\Phi }\left( \mu \right) }}\frac{\left\vert g\left(
		x\right) \right\vert }{\left\Vert g\right\Vert _{L^{\widetilde{\Phi }}\left(
			\mu \right) }}d\mu \left( x\right) \\
	&&\ \ \ \ \ \ \ \ \ \ \leq \int_{X}\left\{ \Phi \left( \frac{\left\vert
		f\left( x\right) \right\vert }{\left\Vert f\right\Vert _{L^{\Phi }\left( \mu
			\right) }}\right) +\widetilde{\Phi }\left( \frac{\left\vert g\left( x\right)
		\right\vert }{\left\Vert g\right\Vert _{L^{\widetilde{\Phi }}\left( \mu
			\right) }}\right) \right\} d\mu \left( x\right) \leq 1+1=2.
\end{eqnarray*}

We now restrict attention to a particular family of bump functions $\Phi
_{N} $, that we will use in our adaptation of the DeGiorgi iteration scheme,
namely 
\begin{equation}
\Phi _{N}\left( t\right) \equiv \left\{ 
\begin{array}{ccc}
t(\ln t)^{N} & \text{ if } & t\geq E=E_{N}=e^{2N} \\ 
\left( \ln E\right) ^{N}t & \text{ if } & 0\leq t\leq E=E_{N}=e^{2N}%
\end{array}%
\right. .  \label{def Phi N ext}
\end{equation}%
We have that $\Phi _{N}$ is submultiplicative for $s,t\geq e^{2N}$ and $%
N\geq 1$ upon using that 
\begin{eqnarray*}
	st\left[ \ln \left( st\right) \right] ^{N} &=&st\left[ \ln s+\ln t\right]
	^{N}\leq s\left[ \ln s\right] ^{N}t\left[ \ln t\right] ^{N} \\
	&\Longleftrightarrow &\ln s+\ln t\leq \left[ \ln s\right] \left[ \ln t\right]
	,
\end{eqnarray*}%
and $a+b\leq ab$ if $a,b\geq 2$. Thus the above considerations apply to show
that for $N\geq 1$, the Young function $\Phi =\Phi _{N}$ is convex and
submultiplicative, and that the H\"{o}lder inequality 
\begin{equation*}
	\int_{X}\left\vert f\left( x\right) g\left( x\right) \right\vert d\mu \left(
	x\right) \leq 2\left\Vert f\right\Vert _{L^{\Phi }\left( \mu \right)
	}\left\Vert g\right\Vert _{L^{\widetilde{\Phi }}\left( \mu \right) }
\end{equation*}%
holds with conjugate function 
\begin{equation*}
	\widetilde{\Phi }\left( s\right) =\left\{ 
	\begin{array}{ccc}
		\frac{1}{2}\left( 2Ne^{2}\right) ^{N}+\int_{\frac{3}{2}\left( 2N\right)
			^{N}}^{s}\left( \Phi ^{\prime }\right) ^{\left( -1\right) }\left( x\right) dx
		& \text{ if } & \frac{3}{2}\left( 2N\right) ^{N}\leq s<\infty \\ 
		e^{2N}\left\{ s-\left( 2N\right) ^{N}\right\} & \text{ if } & \left(
		2N\right) ^{N}\leq s<\frac{3}{2}\left( 2N\right) ^{N} \\ 
		0 & \text{ if } & 0<s<\left( 2N\right) ^{N}%
	\end{array}%
	\right. ,
\end{equation*}%
since $\frac{d}{dt}\left\{ t\left( \ln t\right) ^{N}\right\} =\left( \ln
t\right) ^{N}+N\left( \ln t\right) ^{N-1}=\left( \ln t\right) ^{N}\left( 1+%
\frac{N}{\ln t}\right) $ implies%
\begin{equation*}
	E=e^{2N},\ \ \ \Phi \left( E\right) =\left( 2Ne^{2}\right) ^{N},\ \ \ \frac{%
		\Phi \left( E\right) }{E}=\left( 2N\right) ^{N},\ \ \ \Phi ^{\prime }\left(
	E\right) =\frac{3}{2}\left( 2N\right) ^{N}.
\end{equation*}

Finally, we will use the estimate%
\begin{equation}
\frac{1}{\widetilde{\Phi _{N}}^{\left( -1\right) }(\frac{1}{x})}\leq \frac{%
	(2\ln(2N))^{N}}{\left( \ln \frac{1}{x}\right) ^{N}},\ \ \ \ \ 0<x\ll 1,  \label{use}
\end{equation}%
where $\widetilde{\Phi _{N}}^{\left( -1\right) }$ is the inverse of the
conjugate Young function $\widetilde{\Phi _{N}}$. To see this, first note
that we can write 
\begin{equation*}
	\Phi (t)=\Phi _{N}(t)=t\left( \ln t\right) ^{N},\ \ \ \ \ t\geq e^{2N},
\end{equation*}%
and therefore 
\begin{equation*}
	\Phi ^{\prime }(t)=\left( \ln t\right) ^{N}+tN\left( \ln t\right) ^{N-1}%
	\frac{1}{t}\geq \left( \ln t\right) ^{N},\ \ \ \ \ t\geq e^{2N}.
\end{equation*}%
With $s=\Phi ^{\prime }(t)$, we then have 
\begin{equation*}
	s\geq \left( \ln t\right) ^{N}\text{ for }t\geq e^{2N},\text{ i.e. }\left(
	\Phi ^{\prime }\right) ^{-1}\left( s\right) =t\leq e^{s^{\frac{1}{N}}}\text{
		for }s\geq \left( 2N\right) ^{N},
\end{equation*}%
and thus 
\begin{eqnarray*}
	\widetilde{\Phi }^{\prime }\left( s\right)  &=&\left( \Phi ^{\prime }\right)
	^{-1}\left( s\right) \leq e^{s^{\frac{1}{N}}}\text{ for }s\geq \left(
	2N\right) ^{N}\text{,} \\
	\text{ i.e. }\widetilde{\Phi }(s) &\leq &Ns^{1-\frac{1}{N}}e^{s^{\frac{1}{N}%
	}}\text{ for }s\geq \left( 2N\right) ^{N},
\end{eqnarray*}%
since%
\begin{eqnarray*}
	\frac{d}{ds}\left( Ns^{1-\frac{1}{N}}e^{s^{\frac{1}{N}}}\right)  &=&Ns^{1-%
		\frac{1}{N}}e^{s^{\frac{1}{N}}}\frac{1}{N}s^{\frac{1}{N}-1}+N\left( 1-\frac{1%
	}{N}\right) s^{-\frac{1}{N}}e^{s^{\frac{1}{N}}} \\
	&=&e^{s^{\frac{1}{N}}}+\left( N-1\right) s^{-\frac{1}{N}}e^{s^{\frac{1}{N}}}
	\\
	&=&e^{s^{\frac{1}{N}}}\left\{ 1+\frac{N-1}{s^{\frac{1}{N}}}\right\} \geq
	e^{s^{\frac{1}{N}}}\geq \frac{d}{ds}\widetilde{\Phi }\left( s\right) .
\end{eqnarray*}

In order to estimate $\widetilde{\Phi }^{-1}(t)$%
, we write 
\begin{align*}
	t& =\widetilde{\Phi }(s)\leq Ns^{1-\frac{1}{N}}e^{s^{\frac{1}{N}}}; \\
	\ln t& \leq \ln N+\left( 1-\frac{1}{N}\right) \ln s+s^{\frac{1}{N}}\leq (2\ln(2N))s^{%
		\frac{1}{N}},\ \ \ \ \ \text{for all}\ \ s\geq \left( 2N\right) ^{N}; \\
	\widetilde{\Phi }^{-1}(t)& =s\geq \left( \frac{1}{2\ln(2N)}\ln t\right) ^{N},\ \ \
	\ \ t\geq e^{2N}.
\end{align*}%
Then we obtain (\ref{use}) by noting that%
\begin{equation*}
	\frac{1}{\widetilde{\Phi }^{-1}(\frac{1}{x})}\leq \frac{1}{\left( \frac{1}{2\ln(2N)}%
		\ln \frac{1}{x}\right) ^{N}}=\frac{(2\ln(2N))^{N}}{\left( \ln \frac{1}{x}\right) ^{N}}%
	,\ \ \ \ \ 0<x<e^{-2N}.
\end{equation*}

\section{DeGiorgi iteration}

In the next proposition we apply DeGiorgi iteration to a sequence of Orlicz
Sobolev and Caccioppoli inequalities involving a family of bump functions
adapted to the strongly degenerate geometries $F_{\sigma }$. Recall that the
strongly degenerate geometries $F_{\alpha }$ have degeneracy function%
\begin{equation*}
	f(x)=e^{-\frac{1}{x^{\alpha }}},\quad \alpha >0,\quad x\geq 0,
\end{equation*}%
and that the family of bump functions $\left\{ \Phi _{N}\right\} _{N>1}$ is
given by 
\begin{equation}
\Phi _{N}\left( t\right) \equiv \left\{ 
\begin{array}{ccc}
t(\ln t)^{N} & \text{ if } & t\geq E=E_{N}=e^{2N} \\ 
\left( \ln E\right) ^{N}t & \text{ if } & 0\leq t\leq E=E_{N}=e^{2N}%
\end{array}%
\right. .  \label{bump_N}
\end{equation}

\begin{proposition}
	\label{DG}Assume that the Orlicz Sobolev norm inequality (\ref{OSN}) holds
	with $\Phi =\Phi _{N}$ for some $N>1$ and superradius $\varphi \left(
	r\right) $, and with a geometry $F$ satisfying Definition \ref{structure
		conditions}.
	
	\begin{enumerate}
		\item Then every weak subsolution $u$ to $\mathcal{L}u=\phi $ in $\Omega $,
		with $\mathcal{L}$ as in (\ref{eq_0}), (\ref{struc_0}), and $A$-admissible $%
		\phi $, satisfies the \emph{inner ball inequality}%
		\begin{equation}
		\left\Vert u_{+}\right\Vert _{L^{\infty }(\frac{1}{2}B)}\leq A_{N,\varepsilon}(r)\left( 
		\frac{1}{\left\vert B\right\vert }\int_{B}u_{+}^{2}\right) ^{\frac{1}{2}%
		}+\left\Vert \phi \right\Vert _{X},  \label{Inner ball inequ}
		\end{equation}%
		\begin{equation}
		\text{where }A_{N,\varepsilon}(r)\equiv C_{1}\exp \left\{ C_{2}\left( \frac{\varphi (r)}{%
			r}\right) ^{\frac{1}{N-1-\varepsilon}}\right\} ,  \label{def AN}
		\end{equation}%
		for every ball $B\subset \Omega $ with radius $r$ that is centered on the $x_{n}$-axis, and every $0<\varepsilon<N-1$. Here the constants $C_{1}$ and $C_{2}$ depend on $N$ and $\varepsilon$ but not on $r$.
		
		\item Thus $u$ is locally bounded above in $\Omega $, since $\mathcal{L}$ is
		elliptic away from the $x_{n}$-axis by the structure conditions in
		Definition \ref{structure conditions}.
		
		\item In particular, weak solutions are locally bounded in $\Omega $.
	\end{enumerate}
\end{proposition}

\begin{proof}
	Without loss of generality, we may assume that $B=B\left( 0,r\right) $ is a
	ball centered at the origin with radius $r>0$. Let $\left\{ \psi
	_{j}\right\} _{j=1}^{\infty }$ be a \emph{standard} sequence of Lipschitz
	cutoff functions at $\left( 0,r\right) $ as in Definition \ref{def_cutoff}
	with $\gamma =1+\frac{\varepsilon }{2}$, and associated with the balls $%
	B_{j}\equiv B(0,r_{j})\supset \func{supp}\psi _{j+1}$, $\psi _{j}=1$ on $%
	B_{j}$, with $r_{1}=r$, $r_{\infty }\equiv \lim_{j\rightarrow \infty }r_{j}=%
	\frac{r}{2}$, $r_{j}-r_{j+1}=\frac{c}{j^{1+\frac{\varepsilon }{2}}}r$ for a
	uniquely determined constant $c=c_{\varepsilon }$, and $\left\Vert \nabla
	_{A}\psi _{j}\right\Vert _{\infty }\lesssim \frac{j^{1+\frac{\varepsilon }{2}%
	}}{r}$ with $\nabla _{A}$ as in (\ref{def A grad}) above (see Proposition 68
	in \cite{SaWh4} for more detail). Following DeGiorgi (\cite{DeG}, see also 
	\cite{CaVa}), we consider the family of truncations 
	\begin{equation}
	u_{k}=(u-C_{k})_{+},\quad C_{k}=\tau \left\Vert \phi \right\Vert _{X}\left(
	1-c\left( k+1\right) ^{-\varepsilon /2}\right) ,  \label{truncations}
	\end{equation}%
	and denote the $L^{2}$ norm of the truncation $u_{k}$ by 
	\begin{equation}
	U_{k}\equiv \int_{B_{k}}|u_{k}|^{2}d\mu ,  \label{def Uk}
	\end{equation}%
	where $d\mu =\frac{dx}{|B(0,r)|}=\frac{dx}{\left\vert B\right\vert }$ is
	independent of $k$. Here we have introduced a parameter $\tau \geq 1$ that
	will be used later for rescaling. We will assume $||\phi||_X>0$, otherwise we replace it with parameter $m>0$ and take the limit $m\to 0$ at the end of the argument.
	
	Using H\"{o}lder's inequality for Young functions we can write 
	\begin{equation}
	\int \left( \psi _{k+1}u_{k+1}\right) ^{2}d\mu \leq C||\left( \psi
	_{k+1}u_{k+1}\right) ^{2}||_{L^{\Phi }\left( B_{k};\mu \right) }\cdot
	||1||_{L^{\tilde{\Phi}}\left( \{\psi _{k+1}u_{k+1}>0\};\mu \right) }\ ,
	\label{hol}
	\end{equation}%
	where the norms are taken with respect to the measure $\mu $. For the first
	factor on the right we have, using the Orlicz Sobolev inequality (\ref{OSN})
	and Cauchy-Schwartz inequality with $\delta $ to be determined later,%
	\begin{align*}
		||\left( \psi _{k+1}u_{k+1}\right) ^{2}&||_{L^{\Phi }\left( B_{k};\mu \right)
		} \\
		&\leq C\varphi (r)\int \left\vert \nabla _{A}\left( \psi
		_{k+1}^{2}u_{k+1}^{2}\right) \right\vert d\mu =C\varphi (r)\int \left\vert
		\nabla _{A}\left( \psi _{k+1}^{2}u_{k+1}^{2}\right) \right\vert d\mu \\
		&=C\varphi (r)\int \left\vert \psi _{k+1}u_{k+1}\right\vert \left\vert
		\nabla _{A}\left( \psi _{k+1}u_{k+1}\right) \right\vert d\mu \\
		&\leq \delta \int \left\vert \nabla _{A}\left( \psi _{k+1}u_{k+1}\right)
		\right\vert ^{2}d\mu +\frac{C}{\delta }\varphi (r)^{2}\int \left( \psi
		_{k+1}^{2}u_{k+1}^{2}\right) d\mu .
	\end{align*}
	
	We would now like to apply the Caccioppoli inequality (\ref{Cacc_inhomog})
	with an appropriate function $v$ to the first term on the right, and
	therefore we need to establish estimate (\ref{P_bound}). For that we observe
	that 
	\begin{eqnarray}
	u_{k+1} &>&0\Longrightarrow u>C_{k+1}=\tau \left\Vert \phi \right\Vert
	_{X}\left( 1-c\left( k+2\right) ^{-\varepsilon /2}\right)  \label{uk_support}
	\\
	&\Longrightarrow &u_{k}=\left( u-C_{k}\right) _{+}>c\tau \left\Vert \phi
	\right\Vert _{X}\left[ \left( k+1\right) ^{-\frac{\varepsilon }{2}}-\left(
	k+2\right) ^{-\frac{\varepsilon }{2}}\right] ,
	\end{eqnarray}%
	and%
	\begin{eqnarray*}
		&&c\tau \left\Vert \phi \right\Vert _{X}\left[ \left( k+1\right) ^{-\frac{%
				\varepsilon }{2}}-\left( k+2\right) ^{-\frac{\varepsilon }{2}}\right] =c\tau
		\left\Vert \phi \right\Vert _{X}\left( k+1\right) ^{-\frac{\varepsilon }{2}}%
		\left[ 1-\left( \frac{k+1}{k+2}\right) ^{\frac{\varepsilon }{2}}\right] \\
		&=&c\tau \left\Vert \phi \right\Vert _{X}\left( k+1\right) ^{-\frac{%
				\varepsilon }{2}}\left( 1-\frac{k+1}{k+2}\right) \frac{\varepsilon }{2}%
		\theta ^{\frac{\varepsilon }{2}-1}\ \ \ \ \ \text{where }\frac{k+1}{k+2}%
		<\theta <1 \\
		&\geq &\frac{\varepsilon }{2}c\tau \left\Vert \phi \right\Vert _{X}\left(
		k+1\right) ^{-\frac{\varepsilon }{2}}\frac{1}{k+2}\left( \frac{k+1}{k+2}%
		\right) ^{\frac{\varepsilon }{2}-1}\geq \frac{\varepsilon }{2}c\tau
		\left\Vert \phi \right\Vert _{X}\left( k+2\right) ^{-1-\frac{\varepsilon }{2}%
		}.
	\end{eqnarray*}%
	This implies 
	\begin{equation}\label{phi_x_bound}
	\left\Vert \phi \right\Vert _{X}\leq \frac{2}{c\tau \varepsilon }\left(
	k+2\right) ^{1+\frac{\varepsilon }{2}}u_{k}\leq \frac{2}{c\varepsilon }%
	\left( k+2\right) ^{1+\frac{\varepsilon }{2}}u_{k},
	\end{equation}%
	which is (\ref{P_bound}) with $P=\frac{2}{c\varepsilon }\left( k+2\right)
	^{1+\frac{\varepsilon }{2}}$ and $v=u_{k}$. Note that we have used our
	assumption that $\tau \geq 1$ in the display above.
	
	Thus by the Caccioppoli inequality (\ref{Cacc_inhomog}) we have, 
	\begin{align*}
		\int \left\vert \nabla _{A}\left( \psi _{k+1}u_{k+1}\right) \right\vert
		^{2}&d\mu \\
		&\leq C\left( \left\Vert \nabla _{A}\psi _{k+1}\right\Vert
		_{L^{\infty }}+\frac{2}{c\varepsilon }\left( k+2\right) ^{1+\frac{%
				\varepsilon }{2}}\right) ^{2}\int_{B_{k}}\left( u_{k+1}^{2}+u_{k}^{2}\right)
		d\mu \\
		&\leq C\frac{\left( k+1\right) ^{2+\varepsilon }}{r^{2}}%
		\int_{B_{k}}u_{k}^{2}d\mu \ .
	\end{align*}%
	Finally, choosing $\delta =\frac{r\varphi \left( r\right) }{\left(
		k+1\right) ^{1+\frac{\varepsilon }{2}}}$, we obtain%
	\begin{align}
		||\left( \psi _{k+1}u_{k+1}\right) ^{2}||_{L^{\Phi }\left( B_{k};\mu \right)
		} &\leq \delta C\frac{\left( k+1\right) ^{2+\varepsilon }}{r^{2}}%
		\int_{B_{k}}u_{k}^{2}d\mu +\frac{C}{\delta }\varphi (r)^{2}\int \left( \psi
		_{k+1}^{2}u_{k+1}^{2}\right) d\mu  \label{sob_cac} \\
		&\leq C\frac{\varphi (r)}{r}\left( k+1\right) ^{1+\frac{\varepsilon }{2}%
		}\int_{B_{k}}u_{k}^{2}d\mu \ . \notag 
	\end{align}
	
	For the second factor in (\ref{hol}) we claim 
	\begin{equation}
	||1||_{L^{\tilde{\Phi}}\left( \{\psi _{k+1}u_{k+1}>0\}d\mu \right) }\leq
	\Gamma \left( \left\vert \left\{ \psi _{k}u_{k}>\frac{\varepsilon }{2}%
	c\left( k+2\right) ^{-1-\frac{\varepsilon }{2}}\right\} \right\vert _{\mu
	}\right) ,  \label{super}
	\end{equation}%
	with the notation%
	\begin{equation}
	\Gamma \left( t\right) \equiv \frac{1}{\tilde{\Phi}^{-1}\left( \frac{1}{t}%
		\right) }.  \label{psi_inv}
	\end{equation}%
	First recall 
	\begin{equation*}
		||f||_{L^{\tilde{\Phi}}(X)}\equiv \inf \left\{ a:\int_{X}\tilde{\Phi}\left( 
		\frac{f}{a}\right) \leq 1\right\} ,
	\end{equation*}%
	and note 
	\begin{equation*}
		\int_{\{\psi _{k+1}u_{k+1}>0\}}\tilde{\Phi}\left( \frac{1}{a}\right) =\tilde{%
			\Phi}\left( \frac{1}{a}\right) \left\vert \left\{ \psi
		_{k+1}u_{k+1}>0\right\} \right\vert _{\mu }.
	\end{equation*}%
	Now take 
	\begin{equation*}
		a=\Gamma \left( \left\vert \left\{ \psi _{k+1}u_{k+1}>0\right\} \right\vert
		_{\mu }\right) \equiv \frac{1}{\tilde{\Phi}^{-1}\left( \frac{1}{\left\vert
				\left\{ \psi _{k+1}u_{k+1}>0\right\} \right\vert _{\mu }}\right) }
	\end{equation*}%
	which obviously satisfies 
	\begin{equation*}
		\int_{\left\{ \psi _{k+1}u_{k+1}>0\right\} }\tilde{\Phi}\left( \frac{1}{a}%
		\right) d\mu =1.
	\end{equation*}%
	This gives 
	\begin{equation*}
		||1||_{L^{\tilde{\Psi}}\left( \left\{ \psi _{k+1}u_{k+1}>0\right\} d\mu
			\right) }\leq a=\Gamma \left( \left\vert \left\{ \psi
		_{k+1}u_{k+1}>0\right\} \right\vert _{\mu }\right) ,
	\end{equation*}%
	and to conclude (\ref{super}) we only need to observe that 
	\begin{equation*}
		\left\{ \psi _{k+1}u_{k+1}>0\right\} \subset \left\{ \psi _{k+1}u_{k}>\frac{%
			\varepsilon }{2}c\tau \left\Vert \phi \right\Vert _{X}\left( k+2\right) ^{-1-%
			\frac{\varepsilon }{2}}\right\} ,
	\end{equation*}%
	which follows from (\ref{uk_support}).
	
	Next we use Chebyshev's inequality to obtain 
	\begin{equation}
	\left\vert \left\{ \psi _{k+1}u_{k}>\frac{\varepsilon }{2}c\left( k+2\right)
	^{-1-\frac{\varepsilon }{2}}\right\} \right\vert _{\mu }\leq \gamma \left(
	k+2\right) ^{2+\varepsilon }\int \left( \psi _{k}u_{k}\right) ^{2}d\mu ,
	\label{cheb}
	\end{equation}%
	where $\gamma =\frac{4}{c^{2}\tau ^{2}\left\Vert \phi \right\Vert
		_{X}^{2}\varepsilon ^{2}}$. Combining (\ref{hol})-(\ref{cheb}) we obtain 
	\begin{align*}
		\int &\left( \psi _{k+1}u_{k+1}\right) ^{2}d\mu  \\
		&\leq C||\left( \psi
		_{k+1}u_{k+1}\right) ^{2}||_{L^{\Phi }\left( B_{k};\mu \right) }\cdot
		||1||_{L^{\tilde{\Phi}}\left( \{\psi _{k+1}u_{k+1}>0\};\mu \right) } \\
		&\leq C\frac{\varphi (r)}{r}\frac{\left( k+1\right) ^{1+\frac{\varepsilon }{%
					2}}}{\left\vert B_{1}\right\vert }\int_{B_{k}}u_{k}^{2}\cdot \Gamma \left(
		\left\vert \left\{ \psi _{k}u_{k}>\frac{\varepsilon }{2}c\left( k+2\right)
		^{-1-\frac{\varepsilon }{2}}\right\} \right\vert _{\mu }\right)  \\
		&\leq C\frac{\varphi (r)}{r}\left( k+1\right) ^{1+\frac{\varepsilon }{2}%
		}\left( \int_{B_{k}}u_{k}^{2}d\mu \right) \Gamma \left( \gamma \left(
		k+2\right) ^{2+\varepsilon }\int \left( \psi _{k}u_{k}\right) ^{2}d\mu
		\right)  \\
		&=C\frac{\varphi (r)}{r}\left( k+1\right) ^{1+\frac{\varepsilon }{2}}\left(
		\int_{B_{k}}u_{k}^{2}d\mu \right) \Gamma \left( \gamma \left( k+2\right)
		^{2+\varepsilon }\int \left( \psi _{k}u_{k}\right) ^{2}d\mu \right) ,
	\end{align*}%
	or in terms of the quantities $U_{k}$,%
	\begin{eqnarray}
	U_{k+1} &\leq &C\frac{\varphi (r)}{r}\left( k+1\right) ^{1+\frac{\varepsilon 
		}{2}}U_{k}\Gamma \left( \gamma \left( k+2\right) ^{2+\varepsilon
	}U_{k}\right)   \label{iter} \\
	&=&C\frac{\varphi (r)}{r}\left( k+1\right) ^{1+\frac{\varepsilon }{2}}U_{k}%
	\frac{1}{\tilde{\Phi}^{-1}\left( \frac{1}{\gamma \left( k+2\right)
			^{2+\varepsilon }U_{k}}\right) }.  \notag
	\end{eqnarray}
	
	Now we use the estimate (\ref{use}) on $\frac{1}{\tilde{\Phi}^{-1}\left( 
		\frac{1}{x}\right) }$ to determine the values of $N$ and $\varepsilon $ for
	which DeGiorgi iteration provides local boundedness of weak subsolutions,
	i.e. for which $U_{k}\rightarrow 0$ as $k\rightarrow \infty $ provided $U_{0}
	$ is small enough. From (\ref{use}) and (\ref{iter}) we have 
	\begin{equation*}
		U_{k+1}\leq C\frac{\varphi (r)}{r}\frac{\left( k+1\right) ^{1+\frac{%
					\varepsilon }{2}}U_{k}}{\left( \ln \frac{1}{\gamma }\frac{1}{\left(
				k+2\right) ^{(2+\varepsilon )}U_{k}}\right) ^{N}},
	\end{equation*}%
	provided%
	\begin{equation*}
		\gamma \left( k+2\right) ^{(2+\varepsilon )}U_{k}<e^{-2N},
	\end{equation*}%
	and using the notation 
	\begin{equation*}
		b_{k}\equiv \ln \frac{1}{U_{k}},
	\end{equation*}%
	we can rewrite this as 
	\begin{eqnarray*}
		b_{k+1} &\geq &b_{k}-(1+\frac{\varepsilon }{2})\ln \left( k+1\right)  \\
		&&+N\ln \left( b_{k}-(2+\varepsilon )\ln \left( k+2\right) -\ln \gamma
		\right) -\ln \left( C\frac{\varphi (r)}{r}\right) ,
	\end{eqnarray*}%
	for $k\geq 0$, provided%
	\begin{equation*}
		\frac{1}{\gamma }\left( k+2\right) ^{-(2+\varepsilon )}\frac{1}{U_{k}}%
		>e^{2N};
	\end{equation*}%
	i.e.%
	\begin{equation}
	b_{k}>\left( 2+\varepsilon \right) \ln \left( k+2\right) +\ln \gamma +2N,\ \
	\ \ \ k\geq 0.  \label{prov}
	\end{equation}%
	We now use induction to show that
	
	\textbf{Claim}: Both (\ref{prov}) and 
	\begin{equation}
	b_{k}\geq b_{0}+k,\ \ \ \ \ k\geq 0,  \label{bk}
	\end{equation}
	
	hold for $b_{0}$ taken sufficiently large depending on $N>1$ and $%
	0<r<R $.
	
	Indeed, both (\ref{prov}) and (\ref{bk}) are trivial if $k=0$ and $b_{0}$ is
	large enough. Assume now that the claim is true for some $k\geq 0$. Then 
	\begin{eqnarray*}
		b_{k+1} &\geq &b_{0}+k+1+N\ln \left( b_{0}+k-(2+\varepsilon )\ln \left(
		k+2\right) -\ln \gamma \right)  \\
		&&-(1+\frac{\varepsilon }{2})\ln \left( k+1\right) -1-\ln \left( C\frac{%
			\varphi (r)}{r}\right) .
	\end{eqnarray*}%
	Now for $ N>1+\frac{\varepsilon }{2}$ we have 
	\begin{equation*}
		N\ln \left( b_{0}+k-(2+\varepsilon )\ln \left( k+2\right) -\ln \gamma
		\right) -(1+\frac{\varepsilon }{2})\ln \left( k+1\right) -1-\ln \left( C%
		\frac{\varphi (r)}{r}\right) \rightarrow \infty ,
	\end{equation*}%
	as $k\rightarrow \infty $, and therefore for $b_{0}$ sufficiently large depending on $N$ and $r$, we
	obtain 
	\begin{equation*}
		N\ln \left( b_{0}+k-(2+\varepsilon )\ln \left( k+2\right) -\ln \gamma
		\right) -(1+\frac{\varepsilon }{2})\ln \left( k+1\right) -1-\ln \left( C%
		\frac{\varphi (r)}{r}\right) \geq 0,
	\end{equation*}%
	for all $k\geq 1$, which gives (\ref{bk}) for $k+1$,%
	\begin{equation*}
		b_{k+1}\geq b_{0}+k+1,
	\end{equation*}%
	and also (\ref{prov}) for $k+1$, 
	\begin{eqnarray*}
		b_{k+1} &\geq &b_{0}+k+1>\left( 2+\varepsilon \right) \ln 2+\ln \gamma
		+2N+k+1 \\
		&\geq &\left( 2+\varepsilon \right) \ln \left( k+3\right) +\ln \gamma +2N.
	\end{eqnarray*}
	
	We note that it is sufficient to require 
	\[
	(N-1-\varepsilon)\ln( b_0-8-\ln\gamma)\geq \ln\left(C\frac{\varphi(r)}{r}\right),
	\]
	or
	\begin{equation}
	b_{0}\geq A_{N,\varepsilon}\left( \frac{\varphi (r)}{r}\right) ^{\frac{1}{N-1-\varepsilon}}+\ln (e^8\gamma)
	=A_{N,\varepsilon}\left( \frac{\varphi (r)}{r}\right) ^{\frac{1}{N-1-\varepsilon}}+\ln \frac{4e^8}{%
		c^{2}\tau ^{2}\left\Vert \phi \right\Vert _{X}^{2}\varepsilon ^{2}}
	\label{b_0_cond}
	\end{equation}%
	for $A_{N,\varepsilon}$ sufficiently large depending on $N$ and $\varepsilon$. This completes the proof of
	the induction step and therefore $b_{k}\rightarrow \infty $ as $k\rightarrow
	\infty $, or $U_{k}\rightarrow 0$ as $k\rightarrow \infty $, provided $%
	U_{0}=e^{-b_{0}}$ is sufficiently small.
	
	Altogether, we have shown that 
	\begin{equation*}
		u_{\infty }=(u-\tau \left\Vert \phi \right\Vert _{X})_{+}=0\quad \text{on}\
		B_{\infty }=B\left( 0,\frac{r}{2}\right) =\frac{1}{2}B\left( 0,r\right) ,
	\end{equation*}%
	and thus that 
	\begin{equation}
	u\leq \tau \left\Vert \phi \right\Vert _{X}\quad \text{on}\ B_{\infty }\ ,
	\label{bddness_temp}
	\end{equation}%
	provided $U_{0}=\int_{B}|u_{0}|^{2}d\mu =\frac{1}{\left\vert B\right\vert }%
	\int_{B}|u_{0}|^{2}$ is sufficiently small. From (\ref{b_0_cond}) it follows
	that it is sufficient to require 
	\begin{equation}
	\sqrt{\frac{1}{\left\vert B\right\vert }\int_{B}|u_{0}|^{2}}=e^{-b_{0}}\leq
	\tau \left\Vert \phi \right\Vert _{X}C\varepsilon\exp \left(
	-\frac{A_{N,\varepsilon}}{2}\left( \frac{\varphi (r)}{r}\right) ^{\frac{1}{N-1-\varepsilon}}\right) =\eta
	_{N,\varepsilon}(r)\tau \left\Vert \phi \right\Vert _{X}\ ,  \label{def delta}
	\end{equation}%
	where $A_{N,\varepsilon}$ is the constant in (\ref{b_0_cond}), $C=c/2e^4$ and%
	\begin{equation*}
		\eta _{N,\varepsilon}(r)\equiv C\varepsilon \exp \left\{ -\frac{A_{N,\varepsilon}}{2}%
		\left( \frac{\varphi (r)}{r}\right) ^{\frac{1}{N-1-\varepsilon}}\right\} .
	\end{equation*}
	
	To recover the general case we now consider two cases, $\sqrt{\frac{1}{%
			\left\vert B\right\vert }\int_{B}|u_{0}|^{2}}\leq \eta _{N,\varepsilon}(r)\left\Vert
	\phi \right\Vert _{X}$ and $\sqrt{\frac{1}{\left\vert B\right\vert }%
		\int_{B}|u_{0}|^{2}}>\eta _{N,\varepsilon}(r)\left\Vert \phi \right\Vert _{X}$. In the
	first case we obtain that (\ref{def delta}) holds with $\tau =1$ and thus 
	\begin{equation*}
		||u_{+}||_{L^{\infty }(\frac{1}{2}B)}\leq \left\Vert \phi \right\Vert _{X}\ .
	\end{equation*}%
	In the second case, when $\sqrt{\frac{1}{\left\vert B\right\vert }%
		\int_{B}|u_{0}|^{2}}>\eta _{N,\varepsilon}(r)\left\Vert \phi \right\Vert _{X}$, we let $%
	\tau =\frac{\sqrt{\frac{1}{\left\vert B\right\vert }\int_{B}|u_{0}|^{2}}}{%
		\eta _{N,\varepsilon}(r)\left\Vert \phi \right\Vert _{X}}>1$ so that (\ref{def delta})
	holds, and then from (\ref{bddness_temp}) we get 
	\begin{equation*}
		||u_{+}||_{L^{\infty }(\frac{1}{2}B)}\leq \frac{\sqrt{\frac{1}{\left\vert
					B\right\vert }\int_{B}|u_{0}|^{2}}}{\eta _{N,\varepsilon}(r)}.
	\end{equation*}%
	Altogether this gives 
	\begin{equation*}
		||u_{+}||_{L^{\infty }(\frac{1}{2}B)}\leq \left\Vert \phi \right\Vert
		_{X}+A_{N,\varepsilon}(r)\sqrt{\frac{1}{\left\vert B\right\vert }\int_{B}|u_{0}|^{2}},
	\end{equation*}%
	where 
	\begin{equation*}
		A_{N,\varepsilon}\left( r\right) =\frac{1}{\eta _{N,\varepsilon}\left( r\right) }=C_{1}\exp \left(
		C_{2}\left( \frac{\varphi \left( r\right) }{r}\right) ^{\frac{1}{N-1-\varepsilon}}\right) ,
	\end{equation*}%
	and where the constants $C_{1}$ and $C_{2}$ depend on $N$ and $\varepsilon$, but do not depend
	on $r$.
\end{proof}

The following corollary makes the somewhat trivial observation that we may
replace the factor $\frac{1}{\left\vert B\right\vert }$ with the much
smaller factor $\frac{1}{\left\vert 3B\right\vert }$ at the expense of
replacing the constant $A_{N,\varepsilon}\left( r\right) $ with the much larger constant 
$A_{N,\varepsilon}\left( 3r\right) $. This turns out to be a convenient renormalization
of the local boundedness inequality to prove a continuity theorem in the subsequent paper.

\begin{corollary}[of the proof]
	\label{renorm}Suppose all the assumptions of Proposition \ref{DG} are
	satisfied. Then 
	\begin{equation}
	\left\Vert u_{+}\right\Vert _{L^{\infty }(\frac{1}{2}B)}\leq A_{N,\varepsilon}\left(
	3r\right) \left( \frac{1}{\left\vert 3B\right\vert }\int_{B}u_{+}^{2}\right)
	^{\frac{1}{2}}+\left\Vert \phi \right\Vert _{X}\ ,  \label{Inner ball inequ'}
	\end{equation}%
	with $A_{N,\varepsilon}$ as in (\ref{def AN}).
\end{corollary}

\begin{proof}
	The standard cutoff functions $\{\psi _{k}\}$ defined in the proof of
	Proposition \ref{Inner ball inequ} can also be considered as cutoff
	functions supported on $3B=B(0,3r)$. Thus we can repeat the proof of the
	proposition but applying the Orlicz Sobolev inequality relative to $3B$
	instead of $B$. This results in changing the measure $d\mu =\frac{dx}{%
		\left\vert B\right\vert }$ to the measure $d\mu =\frac{dx}{\left\vert
		3B\right\vert }$, and in changing the superradius ratio $\frac{\varphi
		\left( r\right) }{r}$ to $\frac{\varphi \left( 3r\right) }{3r}$. The
	remaining estimates are unchanged since the cutoff functions are still only
	supported inside $B$, so the regions of integration remain unchanged.
\end{proof}

\chapter{Maximum Principle}

Now we turn to the abstract maximum principle for weak subsolutions. We will assume that $f(x)\neq 0$ if $x\neq 0$, and that $F$ satisfies the five geometric
structure conditions in Definition \ref{structure conditions}. We also assume the
following global Orlicz Sobolev inequality 
\begin{equation}
\left\Vert w\right\Vert _{L^{\Phi }\left( \Omega \right) }\leq C
\left( \Omega \right) \left\Vert \nabla _{A}w\right\Vert _{L^{1}\left(
	\Omega \right) },\ \ \ \ \ w\in \left( W_{A}^{1,1}\right) _{0}\left( \Omega
\right)  \label{Sob_gen}
\end{equation}%
where $\Phi =\Phi_N $ with $N>1$ as defined in (\ref{bump_N}).

The second ingredient of the proof is the following Caccioppoli inequality, which we show follows from the proof of Proposition \ref{Cacc_prop} similar to Corollary \ref{Cacc_cor}.

\begin{proposition}
	Let $u\in \left(W^{1,2}_{A}\right)_{0}(B)$ be a weak subsolution to $\mathcal{L}u=\phi $ with $\mathcal{L}$ as
	in (\ref{eq_0}), (\ref{struc_0}), and $A$-admissible $\phi $, and suppose
	for some ball $B$, a constant $P>0$, and a nonnegative function $v\in W_{A}^{1,2}(B)$
	there holds 
	\begin{equation}
	\left\Vert \phi \right\Vert _{X}^{2}\leq Pv(x),\quad \text{a.e.}\ x\in
	\{u>0\}\cap B.  \label{P_bound'}
	\end{equation}%
	Then 
	\begin{equation}
	\int_{B} |\nabla _{A} u_{+}|^{2}d\mu \leq CP^{2}\int_{B} v^{2} d\mu ,  \label{Cacc_max}
	\end{equation}%
	where $d\mu =\frac{dx}{\left\vert B\right\vert }$, and $C=1/c$ with $c$ from (\ref{struc_0}). The same inequality holds with $u$ and $u_{-}$ in place of $u_{+}$.
\end{proposition}
\begin{proof}
	Since $w\equiv u_{+}\in \left(W^{1,2}_{A}\right)_{0}(B)$ and $u$ is a weak subsolution to $\mathcal{L}u=\phi $ we have
	\begin{eqnarray*}
		\int \nabla \left( u_{+}\right) \mathcal{A}\nabla u_{+}d\mu &=&\int
		\nabla u_{+} \mathcal{A}\nabla ud\mu 
		\leq -\int
		u_{+}\phi d\mu \notag\\
		&\leq& \left\Vert \phi \right\Vert
		_{X}\int \left\vert \nabla _{\mathcal{A}}u_{+}
		\right\vert d\mu \leq P\int v \left\vert \nabla _{\mathcal{A}}u_{+}
		\right\vert d\mu \notag \ ,  
	\end{eqnarray*}%
	where for the last inequality we used condition (\ref{P_bound'}). Using H\"{o}lder's inequality and (\ref{struc_0}) this gives
	\[
	\int \left\vert \nabla _{A}u_{+}
	\right\vert^{2} d\mu \leq \frac{1}{c} \int \left\vert \nabla _{\mathcal{A}}u_{+}
	\right\vert^{2} d\mu \leq CP^{2}\int v^2 d\mu,
	\]
	which is (\ref{Cacc_max}). Using $w=u$ or $w=u_{-}$ as a test function we obtain the last statement of the Proposition.
\end{proof}

We are now ready to prove the maximum principle.
\begin{theorem}
	\label{2D max}Let $\Omega$ be a bounded open subset of $\mathbb{R}^n$, and assume the global
	Orlicz Sobolev inequality (\ref{Sob_gen}) holds for some $N>1$. Assume that $u$ is a weak
	subsolution to $\mathcal{L}u=\phi $ in $\Omega $ with $A$-admissible $\phi $%
	, and that $u$ is bounded on the boundary $\partial \Omega $. Then the
	following maximum principle holds,%
	\begin{equation*}
		\sup_{\Omega }u\leq \sup_{\partial \Omega }u+C\left\Vert \phi \right\Vert
		_{X(\Omega )}\ ,
	\end{equation*}%
	and in particular $u$ is globally bounded.
\end{theorem}

\begin{proof}
	First, we proceed similarly to the proof of Proposition \ref{DG}. Define
	\[
	u_k\equiv (u-\tau\sup_{\partial \Omega}u -C_k)_{+},\quad C_{k}=\tau \left\Vert \phi \right\Vert _{X}\left(
	1-c\left( k+1\right) ^{-\varepsilon /2}\right) ,\quad \tau\geq 1,
	\]
	and note that $u_{k}\equiv 0$ for all $k$ on $\partial \Omega $, so we can formally take $%
	\psi _{k}\equiv 1$ for all $k$.
	Let $B=B(0,r_0)$ be a ball containing $\Omega$ and extend $u_k$ to be zero outside $\Omega$. Then we can assume $u_k\in \left(W^{1,2}_{A}\right)_{0}(B)$ and therefore Caccioppoli inequality (\ref{Cacc_max}) holds.
	Using (\ref{Cacc_max}) and the bound (\ref{phi_x_bound}) on $\{x\in \Omega: u_{k+1}>0\}$ we obtain the following version of inequality (\ref{iter})
	\begin{equation*}
		U_{k+1}\leq C(k+2)^{1+\varepsilon/2}U_{k}\frac{1}{\tilde{\Phi}^{-1}\left( \frac{1}{\gamma \left( k+2\right)
				^{2+\varepsilon }U_{k}}\right) },
	\end{equation*}%
	where $U_{k}=\int_{\Omega}|u_{k}|^{2}$. Proceeding exactly as in the proof of Proposition \ref{DG} we thus obtain
	\begin{equation*}
		u_{\infty }=(u-\tau \sup_{\partial \Omega}u-\tau \left\Vert \phi \right\Vert _{X})_{+}=0 ,
	\end{equation*}
	provided $U_0$ is sufficiently small, and by the same argument as in the proof of Proposition \ref{DG}
	\[
	||u||_{L^{\infty}(\Omega)}\leq ||\phi||_{X}+\sup_{\partial \Omega}u+ C||u||_{L^2(\Omega)}.
	\]
	Finally, we use Caccioppoli inequality together with Orlicz Sobolev inequality to bound the last term
	\[
	||u||_{L^2}^{2}\leq ||u^2||_{L^{\Phi}}\leq C||\nabla_{A}u||_{L^2}^{2}\leq C||\phi||_{X},
	\]
	where we used (\ref{Cacc_max}) with $u$ in place of $u_{+}$, $v\equiv ||\varphi||_{X}$, and $P=1$.
\end{proof}

\part{Geometric theory}

In this third part of the paper, we turn to the problem of finding specific
geometric conditions on the structure of our equations that permit us to
prove the Orlicz Sobolev inequality
needed to apply the abstract theory in Part 2 above. The first chapter here
deals with basic geometric estimates for a specific family of geometries,
which are then applied in the next chapter to obtain the needed Orlicz
Sobolev inequality. Finally, in the
third chapter in this part we prove our geometric theorems on local
boundedness and the maximum principle for weak solutions.

\chapter{Infinitely Degenerate Geometries}

In this first chapter of the third part of the paper, we begin with
degenerate geometries in the plane, the properties of their geodesics and
balls, and the associated subrepresentation inequalities. The final chapter
will treat higher dimensional geometries. Recall from (\ref{form bound'})
that we are considering the inverse metric tensor given by the $2\times 2$
diagonal matrix 
\begin{equation*}
	A=%
	\begin{bmatrix}
		1 & 0 \\ 
		0 & f\left( x\right) ^{2}%
	\end{bmatrix}%
	.
\end{equation*}%
Here the function $f\left( x\right) $ is an even twice continuously
differentiable function on the real line $\mathbb{R}$ with $f(0)=0$ and $%
f^{\prime }(x)>0$ for all $x>0$. The $A$-distance $dt$ is given by 
\begin{equation*}
	dt^{2}=dx^{2}+\frac{1}{f\left( x\right) ^{2}}dy^{2}.
\end{equation*}%
This distance coincides with the control distance as in \cite{SaWh4}, etc.
since a vector $v$ is subunit for an invertible symmetric matrix $A$, i.e. $%
\left( v\cdot \xi \right) ^{2}\leq \xi ^{\func{tr}}A\xi $ for all $\xi $,
if and only if $v^{\func{tr}}A^{-1}v\leq 1$. Indeed, if $v$ is subunit
for $A$, then 
\begin{equation*}
	\left( v^{\func{tr}}A^{-1}v\right) ^{2}=\left( v\cdot A^{-1}v\right)
	^{2}\leq \left( A^{-1}v\right) ^{\func{tr}}AA^{-1}v=v^{\func{tr}%
	}A^{-1}v,
\end{equation*}%
and for the converse, Cauchy-Schwartz gives 
\begin{equation*}
	\left( v\cdot \xi \right) ^{2}=\left( v^{\func{tr}}\xi \right)
	^{2}=\left( v^{\func{tr}}A^{-1}A\xi \right) ^{2}\leq \left( v^{\func{tr%
	}}A^{-1}AA^{-1}v\right) \left( \xi ^{\func{tr}}A\xi \right) =\left( v^{%
		\func{tr}}A^{-1}v\right) \left( \xi ^{\func{tr}}A\xi \right) .
\end{equation*}

\section{Calculation of the $A$-geodesics}

We now compute the equation satisfied by an $A$-geodesic $\gamma $ passing
through the origin. A geodesic minimizes the distance 
\begin{equation*}
	\int_{0}^{x_{0}}\sqrt{1+\frac{\left( \frac{dy}{dx}\right) ^{2}}{f^{2}}}dx,\
	\ \ \ \ \text{where\ }\left( x,y\right) \text{ is on }\gamma \text{,}
\end{equation*}%
and so the calculus of variations gives the equation 
\begin{equation*}
	\frac{d}{dx}\left[ \frac{\frac{dy}{dx}}{f^{2}\sqrt{1+\frac{\left( \frac{dy}{%
					dx}\right) ^{2}}{f^{2}}}}\right] =0.
\end{equation*}%
Consequently, the function 
\begin{equation*}
	\lambda =\frac{f^{2}\sqrt{1+\frac{\left( \frac{dy}{dx}\right) ^{2}}{f^{2}}}}{%
		\frac{dy}{dx}}
\end{equation*}%
is actually a positive constant conserved along the geodesic $y=y\left(
x\right) $ that satisfies%
\begin{equation*}
	\lambda ^{2}=\frac{f^{2}\left[ f^{2}+\left( \frac{dy}{dx}\right) ^{2}\right] 
	}{\left( \frac{dy}{dx}\right) ^{2}}\Longrightarrow \left( \lambda
	^{2}-f^{2}\right) \left( \frac{dy}{dx}\right) ^{2}=f^{4}.
\end{equation*}%
Thus if $\gamma _{0,\lambda }$ denotes the geodesic starting at the origin
going in the vertical direction for $x>0$, and parameterized by the constant 
$\lambda $, we have $\lambda =f\left( x\right) $ if and only if $\frac{dy}{dx%
}=\infty $. For convenience we temporarily assume that $f$ is defined on $%
\left( 0,\infty \right) $. Thus the geodesic $\gamma _{0,\lambda }$ turns
back toward the $y$-axis at the unique point $\left( X\left( \lambda \right)
,Y\left( \lambda \right) \right) $ on the geodesic where $\lambda =f\left(
X\left( \lambda \right) \right) $, provided of course that $\lambda <f\left(
\infty \right) \equiv \sup_{x>0}f\left( x\right) $. On the other hand, if $%
\lambda >f\left( \infty \right) $, then $\frac{dy}{dx}$ is essentially
constant for $x$ large and the geodesics $\gamma _{0,\lambda }$ for $\lambda
>f\left( \infty \right) $ look like straight lines with slope $\frac{f\left(
	\infty \right) ^{2}}{\sqrt{\lambda ^{2}-f\left( \infty \right) ^{2}}}$ for $%
x $ large. Finally, if $\lambda =f\left( \infty \right) $, then the geodesic 
$\gamma _{0,\lambda }$ has slope that blows up at infinity.

\begin{definition}
	We refer to the parameter $\lambda $ as the \emph{turning parameter} of the
	geodesic $\gamma _{0,\lambda }$, and to the point $T\left( \lambda \right)
	=\left( X\left( \lambda \right) ,Y\left( \lambda \right) \right) $ as the 
	\emph{turning point} on the geodesic $\gamma _{0,\lambda }$.
\end{definition}

\begin{summary}
	We summarize the turning behaviour of the geodesic $\gamma _{0,\lambda }$ as
	the turning parameter $\lambda $ decreases from $\infty $ to $0$:
	
	\begin{enumerate}
		\item When $\lambda =\infty $ the geodesic $\gamma _{0,\infty }$ is
		horizontal,
		
		\item As $\lambda $ decreases from $\infty $ to $f\left( \infty \right) $,
		the geodesics $\gamma _{0,\lambda }$ are asymptotically lines whose slopes
		increase to infinity,
		
		\item At $\lambda =f\left( \infty \right) $ the geodesic $\gamma _{0,f\left(
			\infty \right) }$ has slope that increases to infinity as $x$ increases,
		
		\item As $\lambda $ decreases from $f\left( \infty \right) $ to $0$, the
		geodesics $\gamma _{0,\lambda }$ are turn back at $X\left( \lambda \right)
		=f^{-1}\left( \lambda \right) $, and return to the $y$-axis in a path
		symmetric about the line $y=Y\left( \lambda \right) $.
	\end{enumerate}
\end{summary}

Solving for $\frac{dy}{dx}$ we obtain the equation 
\begin{equation*}
	\frac{dy}{dx}=\frac{\pm f^{2}\left( x\right) }{\sqrt{\lambda ^{2}-f\left(
			x\right) ^{2}}}.
\end{equation*}%
Thus the geodesic $\gamma _{0,\lambda }$ that starts from the origin going
in the vertical direction for $x>0$, and with turning parameter $\lambda $,
is given by 
\begin{equation*}
	y=\int_{0}^{x}\frac{f\left( u\right) ^{2}}{\sqrt{\lambda ^{2}-f\left(
			u\right) ^{2}}}du,\qquad x>0.
\end{equation*}%
Since the metric is invariant under vertical translations, we see that the
geodesic $\gamma _{\eta ,\lambda }\left( t\right) $ whose lower point of
intersection with the $y$-axis has coordinates $\left( 0,\eta \right) $, and
whose positive turning parameter is $\lambda $, is given by the equation%
\begin{equation*}
	y=\eta +\int_{0}^{x}\frac{f\left( u\right) ^{2}}{\sqrt{\lambda ^{2}-f\left(
			u\right) ^{2}}}du,\qquad x>0.
\end{equation*}%
Thus the entire family of $A$-geodesics in the right half plane is $\left\{
\gamma _{\eta ,\lambda }\right\} $ parameterized by $\left( \eta ,\lambda
\right) \in \left( -\infty ,\infty \right) \times \left( 0,\infty \right] $,
where when $\lambda =\infty $, the geodesic $\gamma _{\eta ,\infty }\left(
t\right) $ is the horizontal line through the point $\left( 0,\eta \right) $.

\section{Calculation of $A$-arc length}

Let $dt$ denote $A$-arc length along the geodesic $\gamma _{0,\lambda }$ and
let $ds$ denote Euclidean arc length along $\gamma _{0,\lambda }$.

\begin{lemma}
	For $0<x<X\left( \lambda \right) $ and $\left( x,y\right) $ on the lower
	half of the geodesic $\gamma _{0,\lambda }$ we have%
	\begin{eqnarray*}
		\frac{dy}{dx} &=&\frac{f\left( x\right) ^{2}}{\sqrt{\lambda ^{2}-f\left(
				x\right) ^{2}}}, \\
		\frac{dt}{dx}\left( x,y\right) &=&\frac{\lambda }{\sqrt{\lambda ^{2}-f\left(
				x\right) ^{2}}}, \\
		\frac{dt}{ds}\left( x,y\right) &=&\frac{\lambda }{\sqrt{\lambda ^{2}-f(x)^{2}%
				\left[ 1-f(x)^{2}\right] }}.
	\end{eqnarray*}
\end{lemma}

\begin{proof}
	First we note that $y=\int_{0}^{x}\frac{f\left( u\right) ^{2}}{\sqrt{\lambda
			^{2}-f^{2}\left( u\right) }}du$ implies $\frac{dy}{dx}=\frac{f\left(
		x\right) ^{2}}{\sqrt{\lambda ^{2}-f\left( x\right) ^{2}}}$. Thus from $%
	dt^{2}=dx^{2}+\frac{1}{f\left( x\right) ^{2}}dy^{2}$ we have%
	\begin{eqnarray*}
		\left( \frac{dt}{dx}\right) ^{2} &=&1+\frac{1}{f\left( x\right) ^{2}}\left( 
		\frac{dy}{dx}\right) ^{2}=1+\frac{1}{f\left( x\right) ^{2}}\left( \frac{dy}{%
			dx}\right) ^{2} \\
		&=&1+\frac{1}{f\left( x\right) ^{2}}\frac{f\left( x\right) ^{4}}{\lambda
			^{2}-f\left( x\right) ^{2}}=\frac{\lambda ^{2}}{\lambda ^{2}-f\left(
			x\right) ^{2}}.
	\end{eqnarray*}%
	Then the density of $t$ with respect to $s$ at the point $\left( x,y\right) $
	on the lower half of the geodesic $\gamma _{0,\lambda }$ is given by 
	\begin{eqnarray*}
		\frac{dt}{ds} &=&\frac{\frac{dt}{dx}}{\frac{ds}{dx}}=\frac{\frac{\lambda }{%
				\sqrt{\lambda ^{2}-f\left( x\right) ^{2}}}}{\sqrt{1+\left( \frac{dy}{dx}%
				\right) ^{2}}}=\frac{\frac{\lambda }{\sqrt{\lambda ^{2}-f\left( x\right) ^{2}%
		}}}{\sqrt{1+\frac{f\left( x\right) ^{4}}{\lambda ^{2}-f\left( x\right) ^{2}}}%
		} \\
		&=&\frac{\lambda }{\sqrt{\left( \lambda ^{2}-f\left( x\right) ^{2}\right)
				\left( 1+\frac{f\left( x\right) ^{4}}{\lambda ^{2}-f\left( x\right) ^{2}}%
				\right) }}=\frac{\lambda }{\sqrt{\lambda ^{2}-f\left( x\right) ^{2}+f\left(
				x\right) ^{4}}} \\
		&=&\frac{\lambda }{\sqrt{\lambda ^{2}-f\left( x\right) ^{2}\left[ 1-f\left(
				x\right) ^{2}\right] }}.
	\end{eqnarray*}
\end{proof}

Thus at the $y$-axis when $x=0$, we have $\frac{dt}{ds}=1$, and at the
turning point $T\left( \lambda \right) =\left( X\left( \lambda \right)
,Y\left( \lambda \right) \right) $ of the geodesic, when $\lambda
^{2}=f(x)^{2}$, we have $\frac{dt}{ds}=\frac{1}{\lambda }=\frac{1}{f(x)}$.
This reflects the fact that near the $y$ axis, the geodesic is nearly
horizontal and so the metric arc length is close to Euclidean arc length;
while at the turning point for $\lambda $ small, the density of metric arc
length is large compared to Euclidean arc length since movement in the
vertical direction meets with much resistance when $x$ is small.

In order to make precise estimates of arc length, we will need to assume
some additional properties on the function $f\left( x\right) $ when $%
\left\vert x\right\vert $ is small.

\begin{description}
	\item[Assumptions] Fix $R>0$ and let $F\left( x\right) =-\ln f\left(
	x\right) $ for $0<x<R$, so that%
	\begin{equation*}
		f\left( x\right) =e^{-F\left( \left\vert x\right\vert \right) },\ \ \ \ \
		0<\left\vert x\right\vert <R.
	\end{equation*}%
	We assume the following for some constants $C\geq 1$ and $0<\varepsilon <1$:
	
	\begin{enumerate}
		\item $\lim_{x\rightarrow 0^{+}}F\left( x\right) =+\infty $;
		
		\item $F^{\prime }\left( x\right) <0$ and $F^{\prime \prime }\left( x\right)
		>0$ for all $x\in (0,R)$;
		
		\item $\frac{1}{C}\left\vert F^{\prime }\left( r\right) \right\vert \leq
		\left\vert F^{\prime }\left( x\right) \right\vert \leq C\left\vert F^{\prime
		}\left( r\right) \right\vert ,\ \ \ \ \ \frac{1}{2}r<x<2r<R$;
		
		\item $\frac{1}{-xF^{\prime }\left( x\right) }$ is increasing in the
		interval $\left( 0,R\right) $ and satisfies $\frac{1}{-xF^{\prime }\left(
			x\right) }\leq \frac{1}{\varepsilon }\,$for $x\in (0,R)$;
		
		\item $\frac{F^{\prime \prime }\left( x\right) }{-F^{\prime }\left( x\right) 
		}\approx \frac{1}{x}$ for $x\in (0,R)$.
	\end{enumerate}
\end{description}

These assumptions have the following consequences.

\begin{lemma}
	\label{consequences}Suppose that $R$, $f$ and $F$ are as above.
	
	\begin{enumerate}
		\item If $0<x_{1}<x_{2}<R$, then we have%
		\begin{equation*}
			F\left( x_{1}\right) >F\left( x_{2}\right) +\varepsilon \ln \frac{x_{2}}{%
				x_{1}},\text{ equivalently }f\left( x_{1}\right) <\left( \frac{x_{1}}{x_{2}}%
			\right) ^{\varepsilon }f\left( x_{2}\right) .
		\end{equation*}
		
		\item If $x_{1},x_{2}\in (0,R)$ and $\max \left\{ \varepsilon x_{1},x_{1}-%
		\frac{1}{\left\vert F^{\prime }\left( x_{1}\right) \right\vert }\right\}
		\leq x_{2}\leq x_{1}+\frac{1}{\left\vert F^{\prime }\left( x_{1}\right)
			\right\vert }$, then we have%
		\begin{eqnarray*}
			\left\vert F^{\prime }\left( x_{1}\right) \right\vert &\approx &\left\vert
			F^{\prime }\left( x_{2}\right) \right\vert , \\
			f\left( x_{1}\right) &\approx &f\left( x_{2}\right) .
		\end{eqnarray*}
		
		\item If $x\in (0,R)$, then we have 
		\begin{equation*}
			\frac{F^{\prime \prime }\left( x\right) }{\left\vert F^{\prime }\left(
				x\right) \right\vert ^{2}}\approx \frac{1}{-xF^{\prime }\left( x\right) }%
			\lesssim 1.
		\end{equation*}
	\end{enumerate}
\end{lemma}

\begin{proof}
	Assumptions (2) and (4) give $\left\vert F^{\prime }\left( x_{1}\right)
	\right\vert >\frac{\varepsilon }{x}$, and so we have%
	\begin{equation*}
		F\left( x_{1}\right) -F\left( x_{2}\right) >\int_{x_{1}}^{x_{2}}\frac{%
			\varepsilon }{x}dx=\varepsilon \ln \frac{x_{2}}{x_{1}},
	\end{equation*}%
	which proves Part (1) of the lemma. Without loss of generality, assume now
	that $x_{1}\leq x_{2}\leq x_{1}+\frac{1}{\left\vert F^{\prime }\left(
		x_{1}\right) \right\vert }$. Then by Assumption (4) we also have $x_{1}\leq
	x_{2}\leq \left( 1+\frac{1}{\varepsilon }\right) x_{1}$, and then by
	Assumption (3), the first assertion in Part (2) of the lemma holds, and with
	the bound, 
	\begin{eqnarray*}
		F\left( x_{1}\right) -F\left( x_{2}\right) &=&\int_{x_{1}}^{x_{2}}-F^{\prime
		}\left( x\right) ~dx \\
		&\approx &\left\vert F^{\prime }\left( x_{1}\right) \right\vert \left(
		x_{2}-x_{1}\right) \leq 1.
	\end{eqnarray*}%
	From this we get%
	\begin{equation*}
		1\leq \frac{f\left( x_{2}\right) }{f\left( x_{1}\right) }=e^{F\left(
			x_{1}\right) -F\left( x_{2}\right) }\lesssim 1,
	\end{equation*}%
	which proves the second assertion in Part (2) of the lemma. Finally,
	Assumptions (4) and (5) give%
	\begin{equation*}
		\frac{F^{\prime \prime }\left( x\right) }{\left\vert F^{\prime }\left(
			x\right) \right\vert ^{2}}=\frac{F^{\prime \prime }\left( x\right) }{%
			-F^{\prime }\left( x\right) }\frac{1}{-F^{\prime }\left( x\right) }\approx 
		\frac{1}{x}\frac{1}{-F^{\prime }\left( x\right) }\lesssim 1,
	\end{equation*}%
	which proves Part (3) of the lemma.
\end{proof}

\begin{lemma}
	Suppose $\lambda >0$, $0<x<X\left( \lambda \right) $ and 
	\begin{equation*}
		y=\int_{0}^{x}\frac{f\left( u\right) ^{2}}{\sqrt{\lambda ^{2}-f^{2}\left(
				u\right) }}du.
	\end{equation*}%
	Then $\left( x,y\right) $ lies on the lower half of the geodesic $\gamma
	_{0,\lambda }$ and%
	\begin{equation*}
		y\approx \frac{f\left( x\right) ^{2}}{\lambda \left\vert F^{\prime }\left(
			x\right) \right\vert }.
	\end{equation*}
\end{lemma}

\begin{proof}
	Using first that $\frac{f\left( u\right) ^{2}}{\sqrt{\lambda
			^{2}-f^{2}\left( u\right) }}$ is increasing in $u$, and then that $F\left(
	u\right) =-\ln f\left( u\right) $, we have%
	\begin{equation*}
		y=\int_{0}^{x}\frac{f\left( u\right) ^{2}}{\sqrt{\lambda ^{2}-f\left(
				u\right) ^{2}}}du\approx \int_{\frac{x}{2}}^{x}\frac{f\left( u\right) ^{2}}{%
			\sqrt{\lambda ^{2}-f\left( u\right) ^{2}}}du=\int_{\frac{x}{2}}^{x}\frac{1}{%
			-2F^{\prime }\left( u\right) }\frac{\left[ f\left( u\right) ^{2}\right]
			^{\prime }}{\sqrt{\lambda ^{2}-f\left( u\right) ^{2}}}du,
	\end{equation*}%
	and then using Assumption (3) we get%
	\begin{equation*}
		y\approx \frac{1}{-F^{\prime }\left( x\right) }\int_{\frac{x}{2}}^{x}\frac{%
			\left[ f\left( u\right) ^{2}\right] ^{\prime }du}{2\sqrt{\lambda
				^{2}-f\left( u\right) ^{2}}}=\frac{1}{-F^{\prime }\left( x\right) }%
		\int_{f\left( \frac{x}{2}\right) ^{2}}^{f\left( x\right) ^{2}}\frac{dv}{2%
			\sqrt{\lambda ^{2}-v}}.
	\end{equation*}%
	Now from Part (1) of Lemma \ref{consequences} we obtain $f\left( \frac{x}{2}%
	\right) ^{2}<\left( \frac{1}{2}\right) ^{2\varepsilon }f\left( x\right) ^{2}$
	and so%
	\begin{equation*}
		y\approx \frac{1}{-F^{\prime }\left( x\right) }\int_{0}^{f\left( x\right)
			^{2}}\frac{dv}{2\sqrt{\lambda ^{2}-v}}=\frac{\lambda -\sqrt{\lambda
				^{2}-f\left( x\right) ^{2}}}{-F^{\prime }\left( x\right) }\approx \frac{%
			f\left( x\right) ^{2}}{\lambda \left\vert F^{\prime }\left( x\right)
			\right\vert },
	\end{equation*}%
	where the final estimate follows from $1-\sqrt{1-t}=\frac{t}{1+\sqrt{1-t}}%
	\approx t$, $0<t<1$, with $t=\frac{f\left( x\right) ^{2}}{\lambda ^{2}}$.
\end{proof}

\begin{remark}
	\label{upper bound}We actually have the upper bound $y\leq \frac{f\left(
		x\right) ^{2}}{\lambda \left\vert F^{\prime }\left( x\right) \right\vert }$
	since $F^{\prime \prime }\left( x\right) >0$. Indeed, then $\frac{1}{%
		-F^{\prime }\left( x\right) }$ is increasing and for $f\left( x\right)
	<\lambda $ we have%
	\begin{eqnarray*}
		y &=&\int_{0}^{x}\frac{f\left( u\right) ^{2}}{\sqrt{\lambda ^{2}-f\left(
				u\right) ^{2}}}du=\int_{0}^{x}\frac{1}{-2F^{\prime }\left( u\right) }\frac{%
			\left[ f\left( u\right) ^{2}\right] ^{\prime }}{\sqrt{\lambda ^{2}-f\left(
				u\right) ^{2}}}du \\
		&\leq &\frac{1}{-F^{\prime }\left( x\right) }\int_{0}^{x}\frac{\left[
			f\left( u\right) ^{2}\right] ^{\prime }}{2\sqrt{\lambda ^{2}-f\left(
				u\right) ^{2}}}du=\frac{f\left( x\right) ^{2}}{-\lambda F^{\prime }\left(
			x\right) }.
	\end{eqnarray*}
\end{remark}

Now we can estimate the $A$-arc length of the geodesic $\gamma _{0,\lambda }$
between the two points $P_{0}=\left( 0,0\right) $ and $P_{1}=\left(
x_{1},y_{1}\right) $ where $0<x_{1}<X\left( \lambda \right) $ and 
\begin{equation*}
	y_{1}=\int_{0}^{x_{1}}\frac{f\left( u\right) ^{2}}{\sqrt{\lambda
			^{2}-f^{2}\left( u\right) }}du.
\end{equation*}%
We have the formula%
\begin{equation*}
	d\left( P_{0},P_{1}\right) =\int_{P_{0}}^{P_{1}}dt=\int_{P_{0}}^{P_{1}}\frac{%
		dt}{dx}dx=\int_{0}^{x_{1}}\frac{\lambda }{\sqrt{\lambda ^{2}-f\left(
			x\right) ^{2}}}dx,
\end{equation*}%
from which we obtain $x_{1}<d\left( P_{0},P_{1}\right) $.

\begin{lemma}
	\label{arc length}With notation as above we have%
	\begin{eqnarray*}
		&&x_{1}<d\left( P_{0},P_{1}\right) \leq d\left( \left( 0,0\right) ,\left(
		x_{1},0\right) \right) +d\left( \left( x_{1},0\right) ,\left(
		x_{1},y_{1}\right) \right) \ ; \\
		&&d\left( \left( 0,0\right) ,\left( x_{1},0\right) \right) =x_{1}\ , \\
		&&d\left( \left( x_{1},0\right) ,\left( x_{1},y_{1}\right) \right) \leq 
		\frac{f\left( x_{1}\right) }{-\lambda F^{\prime }\left( x_{1}\right) }\leq 
		\frac{1}{-F^{\prime }\left( x_{1}\right) }<\frac{1}{\varepsilon }x_{1}\ .
	\end{eqnarray*}%
	In particular we have $d\left( P_{0},P_{1}\right) \approx x_{1}$.
\end{lemma}

\begin{proof}
	From Remark \ref{upper bound} we have%
	\begin{equation*}
		d\left( \left( x_{1},0\right) ,\left( x_{1},y_{1}\right) \right) \leq \frac{%
			y_{1}}{f\left( x_{1}\right) }\leq \frac{f\left( x_{1}\right) }{-\lambda
			F^{\prime }\left( x_{1}\right) },
	\end{equation*}%
	and then we use $f\left( x_{1}\right) \leq \lambda $ and Assumption (4).
\end{proof}

\begin{corollary}
	\label{similar r and x}$\left\vert F^{\prime }\left( d\left(
	P_{0},P_{1}\right) \right) \right\vert \approx \left\vert F^{\prime }\left(
	x_{1}\right) \right\vert $ and $f\left( d\left( P_{0},P_{1}\right) \right)
	\approx f\left( x_{1}\right) $.
\end{corollary}

\begin{proof}
	Combine Part (2) of Lemma \ref{consequences} with Lemma \ref{arc length}.
\end{proof}

\section{Integration over $A$-balls and Area\label{Regions}}

Here we investigate properties of the $A$-ball $B\left( 0,r_{0}\right) $
centered at the origin $0$ with radius $r_{0}>0$:%
\begin{equation*}
	B\left( 0,r_{0}\right) \equiv \left\{ x\in \mathbb{R}^{2}:d\left( 0,x\right)
	<r_{0}\right\} ,\ \ \ \ \ r_{0}>0.
\end{equation*}%
For this we will use `$A$-polar coordinates' where $d\left( 0,\left(
x,y\right) \right) $ plays the role of the radial variable, and the turning
parameter $\lambda $ plays the role of the angular coordinate. More
precisely, given Cartesian coordinates $\left( x,y\right) $, the $A$-polar
coordinates $\left( r,\lambda \right) $ are given implicitly by the pair of
equations%
\begin{eqnarray}
r &=&\int_{0}^{x}\frac{\lambda }{\sqrt{\lambda ^{2}-f\left( u\right) ^{2}}}%
du\ ,  \label{r and y} \\
y &=&\int_{0}^{x}\frac{f\left( u\right) ^{2}}{\sqrt{\lambda ^{2}-f\left(
		u\right) ^{2}}}du\ .  \notag
\end{eqnarray}%
In this section we will work out the change of variable formula for the
quarter $A$-ball $QB(0,r_{0})$ lying in the first quadrant. See Figure \ref{ball_origin}.

\begin{figure}%[ht]
	\includegraphics{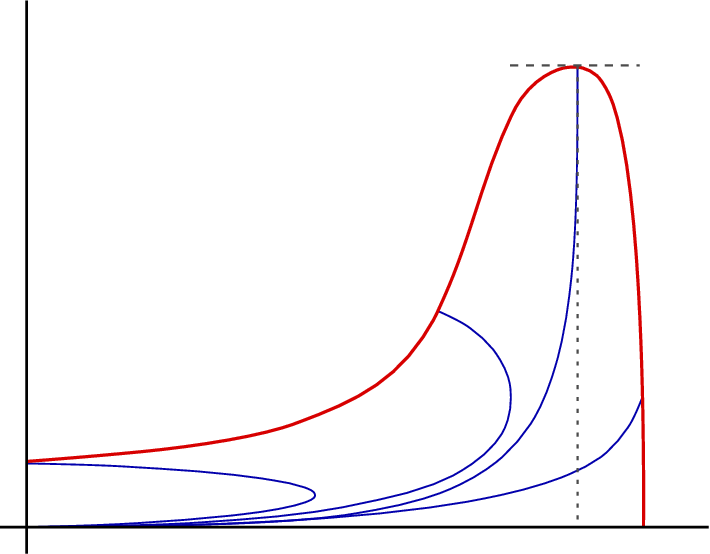}
	\caption{A first quadrant view of a control ball centered at the origin.}
	\label{ball_origin}
\end{figure}

\begin{definition}
	\label{definition of L and Y} Let $\lambda \in (0,\infty )$. The geodesic
	with turning parameter $\lambda $ first moves to the right and then curls
	back at the turning point $T\left( \lambda \right) =\left( X\left( \lambda
	\right) ,Y\left( \lambda \right) \right) $ when $x=X\left( \lambda \right)
	\equiv f^{-1}\left( \lambda \right) $. If $R\left( \lambda \right) $ denotes
	the $A$-arc length from the origin to the turning point $T\left( \lambda
	\right) $, we have%
	\begin{eqnarray*}
		R\left( \lambda \right) &\equiv &d\left( 0,T\left( \lambda \right) \right)
		=\int_{0}^{X\left( \lambda \right) }\frac{\lambda }{\sqrt{\lambda
				^{2}-f\left( u\right) ^{2}}}du, \\
		Y\left( \lambda \right) &=&\int_{0}^{X\left( \lambda \right) }\frac{f\left(
			u\right) ^{2}}{\sqrt{\lambda ^{2}-f\left( u\right) ^{2}}}du.
	\end{eqnarray*}
\end{definition}

The two parts of the geodesic $\gamma _{0,\lambda }$,cut at the point $%
T\left( \lambda \right) $, have different equations:%
\begin{equation}
y=\left\{ 
\begin{array}{ll}
\int_{0}^{x}\frac{f\left( u\right) ^{2}}{\sqrt{\lambda ^{2}-f\left( u\right)
		^{2}}}du & \text{when}\;y\in \left[ 0,Y\left( \lambda \right) \right] \\ 
2Y\left( \lambda \right) -\int_{0}^{x}\frac{f\left( u\right) ^{2}}{\sqrt{%
		\lambda ^{2}-f\left( u\right) ^{2}}}du & \text{when}\;y\in \left[ Y\left(
\lambda \right) ,2Y\left( \lambda \right) \right]%
\end{array}%
\right. .  \label{two-piece geodesic}
\end{equation}%
We define the region covered by the first equation for the geodesics to be
Region 1, and the region covered by the second equation for the geodesics to
be Region 2. They are separated by the curve $y=Y(f(x))$. We now calculate
the first derivative matrix $%
\begin{bmatrix}
\frac{\partial x}{\partial r} & \frac{\partial x}{\partial \lambda } \\ 
\frac{\partial y}{\partial r} & \frac{\partial y}{\partial \lambda }%
\end{bmatrix}%
$ and the Jacobian $\frac{\partial \left( x,y\right) }{\partial \left(
	r,\lambda \right) }$ in Regions 1 and 2 separately.

\subsection{Region 1}

Applying implicit differentiation to the first equation in (\ref{r and y}),
we have 
\begin{align*}
	1=\frac{\partial r}{\partial r}& =\frac{\partial x}{\partial r}\cdot \frac{%
		\lambda }{\sqrt{\lambda ^{2}-f\left( x\right) ^{2}}}, \\
	0=\frac{\partial r}{\partial \lambda }& =\frac{\partial x}{\partial \lambda }%
	\cdot \frac{\lambda }{\sqrt{\lambda ^{2}-f\left( x\right) ^{2}}}+\int_{0}^{x}%
	\frac{\partial }{\partial \lambda }\left[ \frac{\lambda }{\sqrt{\lambda
			^{2}-f\left( u\right) ^{2}}}\right] du,
\end{align*}%
where%
\begin{equation*}
	\frac{\partial }{\partial \lambda }\left[ \frac{\lambda }{\sqrt{\lambda
			^{2}-f\left( u\right) ^{2}}}\right] =\frac{1\cdot \sqrt{\lambda ^{2}-f\left(
			u\right) ^{2}}-\lambda \cdot \frac{2\lambda }{2\sqrt{\lambda ^{2}-f\left(
				u\right) ^{2}}}}{\lambda ^{2}-f\left( u\right) ^{2}}=\frac{-f\left( u\right)
		^{2}}{\left( \lambda ^{2}-f\left( u\right) ^{2}\right) ^{\frac{3}{2}}}.
\end{equation*}
Thus we have%
\begin{eqnarray*}
	\frac{\partial x}{\partial r} &=&\frac{\sqrt{\lambda ^{2}-f\left( x\right)
			^{2}}}{\lambda }, \\
	\frac{\partial x}{\partial \lambda } &=&\frac{\sqrt{\lambda ^{2}-f\left(
			x\right) ^{2}}}{\lambda }\cdot \int_{0}^{x}\frac{f\left( u\right) ^{2}}{%
		\left( \lambda ^{2}-f\left( u\right) ^{2}\right) ^{\frac{3}{2}}}du.
\end{eqnarray*}
Applying implicit differentiation to the second equation in (\ref{r and y}),
we have 
\begin{align*}
	\frac{\partial y}{\partial r}& =\frac{\partial x}{\partial r}\cdot \frac{%
		f\left( x\right) ^{2}}{\sqrt{\lambda ^{2}-f\left( x\right) ^{2}}}; \\
	\frac{\partial y}{\partial \lambda }& =\frac{\partial x}{\partial \lambda }%
	\cdot \frac{f\left( x\right) ^{2}}{\sqrt{\lambda ^{2}-f\left( x\right) ^{2}}}%
	+\int_{0}^{x}\frac{\partial }{\partial \lambda }\left[ \frac{f\left(
		u\right) ^{2}}{\sqrt{\lambda ^{2}-f\left( u\right) ^{2}}}\right] du \\
	& =\frac{\partial x}{\partial \lambda }\cdot \frac{f\left( x\right) ^{2}}{%
		\sqrt{\lambda ^{2}-f\left( x\right) ^{2}}}+\int_{0}^{x}\frac{-\lambda
		f\left( u\right) ^{2}}{\left( \lambda ^{2}-f\left( u\right) ^{2}\right) ^{%
			\frac{3}{2}}}du
\end{align*}%
Plugging the equation for $\frac{\partial x}{\partial \lambda }$ into these
equations, we obtain%
\begin{eqnarray*}
	\frac{\partial y}{\partial r} &=&\frac{f\left( x\right) ^{2}}{\lambda }, \\
	\frac{\partial y}{\partial \lambda } &=&\frac{f\left( x\right) ^{2}-\lambda
		^{2}}{\lambda }\cdot \int_{0}^{x}\frac{f\left( u\right) ^{2}}{\left( \lambda
		^{2}-f\left( u\right) ^{2}\right) ^{\frac{3}{2}}}du,
\end{eqnarray*}%
and this completes the calculation of the first derivative matrix $%
\begin{bmatrix}
\frac{\partial x}{\partial r} & \frac{\partial x}{\partial \lambda } \\ 
\frac{\partial y}{\partial r} & \frac{\partial y}{\partial \lambda }%
\end{bmatrix}%
$.

Now we can calculate the Jacobian 
\begin{eqnarray*}
	\frac{\partial \left( x,y\right) }{\partial \left( r,\lambda \right) }
	&=&\det 
	\begin{bmatrix}
		\frac{\sqrt{\lambda ^{2}-f\left( x\right) ^{2}}}{\lambda } & \frac{\sqrt{%
				\lambda ^{2}-f\left( x\right) ^{2}}}{\lambda }\cdot \int_{0}^{x}\frac{%
			f\left( u\right) ^{2}}{\left( \lambda ^{2}-f\left( u\right) ^{2}\right) ^{%
				\frac{3}{2}}}du \\ 
		\frac{f\left( x\right) ^{2}}{\lambda } & \frac{f\left( x\right) ^{2}-\lambda
			^{2}}{\lambda }\cdot \int_{0}^{x}\frac{f\left( u\right) ^{2}}{\left( \lambda
			^{2}-f\left( u\right) ^{2}\right) ^{\frac{3}{2}}}du%
	\end{bmatrix}
	\\
	&=&-\sqrt{\lambda ^{2}-f\left( x\right) ^{2}}\int_{0}^{x}\frac{f\left(
		u\right) ^{2}}{\left( \lambda ^{2}-f\left( u\right) ^{2}\right) ^{\frac{3}{2}%
	}}du.
\end{eqnarray*}%
In addition we have 
\begin{equation*}
	\int_{0}^{x}\frac{f\left( u\right) ^{2}}{\left( \lambda ^{2}-f\left(
		u\right) ^{2}\right) ^{\frac{3}{2}}}du\approx \int_{x/2}^{x}\frac{f\left(
		u\right) ^{2}}{\left( \lambda ^{2}-f\left( u\right) ^{2}\right) ^{\frac{3}{2}%
	}}du=\int_{x/2}^{x}\frac{\frac{d}{du}\left[ f\left( u\right) ^{2}\right]
		\cdot \frac{f\left( u\right) }{2f^{\prime }\left( u\right) }}{\left( \lambda
		^{2}-f\left( u\right) ^{2}\right) ^{\frac{3}{2}}}du,
\end{equation*}%
where 
\begin{equation*}
	\frac{f\left( u\right) }{2f^{\prime }\left( u\right) }=\frac{1}{-2F^{\prime
		}\left( u\right) }\approx \frac{1}{-F^{\prime }\left( x\right) },
\end{equation*}%
and so we have 
\begin{equation*}
	\int_{0}^{x}\frac{f\left( u\right) ^{2}}{\left( \lambda ^{2}-f\left(
		u\right) ^{2}\right) ^{\frac{3}{2}}}du\approx \frac{1}{-F^{\prime }\left(
		x\right) }\int_{f\left( \frac{x}{2}\right) ^{2}}^{f\left( x\right) ^{2}}%
	\frac{1}{\left( \lambda ^{2}-v\right) ^{\frac{3}{2}}}dv.
\end{equation*}%
By Part (1) of Lemma \ref{consequences}, we have $f\left( \frac{x}{2}\right)
<\left( \frac{1}{2}\right) ^{\varepsilon }f\left( x\right) $, and as a
result, we obtain 
\begin{align*}
	\int_{0}^{x}\frac{f\left( u\right) ^{2}}{\left( \lambda ^{2}-f\left(
		u\right) ^{2}\right) ^{\frac{3}{2}}}du&\approx \frac{1}{-F^{\prime }\left(
		x\right) }\int_{0}^{f\left( x\right) ^{2}}\frac{1}{\left( \lambda
		^{2}-v\right) ^{\frac{3}{2}}}dv\\
	&\approx \frac{1}{-F^{\prime }\left( x\right) }%
	\left( \frac{1}{\sqrt{\lambda ^{2}-f\left( x\right) ^{2}}}-\frac{1}{\lambda }%
	\right) .
\end{align*}%
Altogether we have the estimate 
\begin{equation}
\left\vert \frac{\partial \left( x,y\right) }{\partial \left( r,\lambda
	\right) }\right\vert \approx \frac{1}{-F^{\prime }\left( x\right) }\cdot 
\frac{\lambda -\sqrt{\lambda ^{2}-f\left( x\right) ^{2}}}{\lambda }\approx 
\frac{f\left( x\right) ^{2}}{\lambda ^{2}\left\vert F^{\prime }\left(
	x\right) \right\vert }.  \label{Jacobian2}
\end{equation}%
From Corollary \ref{similar r and x}, we also have 
\begin{equation}
\left\vert \frac{\partial \left( x,y\right) }{\partial \left( r,\lambda
	\right) }\right\vert \approx \frac{f\left( r\right) ^{2}}{\lambda
	^{2}\left\vert F^{\prime }\left( r\right) \right\vert }.  \label{Jacobian3}
\end{equation}

\subsection{Region 2}

In Region 2 we have the following pair of formulas:%
\begin{eqnarray}
r &=&2R\left( \lambda \right) -\int_{0}^{x}\frac{\lambda }{\sqrt{\lambda
		^{2}-f\left( u\right) ^{2}}}du\ ,  \label{r and y Region 2} \\
y &=&2Y\left( \lambda \right) -\int_{0}^{x}\frac{f\left( u\right) ^{2}}{%
	\sqrt{\lambda ^{2}-f\left( u\right) ^{2}}}du\ .  \notag
\end{eqnarray}%
where we recall that $R\left( \lambda \right) =\int_{0}^{X\left( \lambda
	\right) }\frac{\lambda }{\sqrt{\lambda ^{2}-f\left( u\right) ^{2}}}du$ is
the arc length of the geodesic $\gamma _{0,\lambda }$ from the origin $0$ to
the turning point $T\left( \lambda \right) $. Before proceeding, we
calculate the derivative of $Y\left( \lambda \right) $. We note that due to
cancellation, the derivative $R^{\prime }\left( \lambda \right) $ does not
explicitly enter into the formula for the Jacobian $\frac{\partial \left(
	x,y\right) }{\partial \left( r,\lambda \right) }$ below, so we defer its
calculation for now.

\begin{lemma}
	The derivative of $Y\left( \lambda \right) $ is given by%
	\begin{equation*}
		Y^{\prime }\left( \lambda \right) =\int_{0}^{f^{-1}\left( \lambda \right) }%
		\frac{F^{\prime \prime }\left( u\right) }{\left\vert F^{\prime }\left(
			u\right) \right\vert ^{2}}\frac{\lambda }{\sqrt{\lambda ^{2}-f\left(
				u\right) ^{2}}}du.
	\end{equation*}
\end{lemma}

\begin{proof}
	Integrating by parts we obtain%
	\begin{eqnarray*}
		Y\left( \lambda \right) &=&\int_{0}^{f^{-1}\left( \lambda \right) }\frac{%
			-f\left( u\right) }{f^{\prime }\left( u\right) }\cdot \frac{d}{du}\sqrt{%
			\lambda ^{2}-f\left( u\right) ^{2}}du \\
		&=&\int_{0}^{f^{-1}\left( \lambda \right) }\frac{1}{F^{\prime }\left(
			u\right) }\cdot \frac{d}{du}\sqrt{\lambda ^{2}-f\left( u\right) ^{2}}du \\
		&=&-\int_{0}^{f^{-1}\left( \lambda \right) }\sqrt{\lambda ^{2}-f\left(
			u\right) ^{2}}\cdot \frac{d}{du}\frac{1}{F^{\prime }\left( u\right) }du \\
		&=&\int_{0}^{f^{-1}\left( \lambda \right) }\frac{F^{\prime \prime }\left(
			u\right) }{\left\vert F^{\prime }\left( u\right) \right\vert ^{2}}\sqrt{%
			\lambda ^{2}-f\left( u\right) ^{2}}du,
	\end{eqnarray*}%
	and so from $\lambda ^{2}-f\left( f^{-1}\left( \lambda \right) \right)
	^{2}=0 $, we have 
	\begin{equation*}
		Y^{\prime }\left( \lambda \right) =0+\int_{0}^{f^{-1}\left( \lambda \right) }%
		\frac{F^{\prime \prime }\left( u\right) }{\left\vert F^{\prime }\left(
			u\right) \right\vert ^{2}}\frac{\lambda }{\sqrt{\lambda ^{2}-f\left(
				u\right) ^{2}}}du.
	\end{equation*}
\end{proof}

Applying implicit differentiation to the first equation in (\ref{r and y
	Region 2}), we have 
\begin{align*}
	1=\frac{\partial r}{\partial r}& =-\frac{\partial x}{\partial r}\cdot \frac{%
		\lambda }{\sqrt{\lambda ^{2}-f\left( x\right) ^{2}}}, \\
	0=\frac{\partial r}{\partial \lambda }& =2R^{\prime }\left( \lambda \right) -%
	\frac{\partial x}{\partial \lambda }\cdot \frac{\lambda }{\sqrt{\lambda
			^{2}-f\left( x\right) ^{2}}}-\int_{0}^{x}\frac{\partial }{\partial \lambda }%
	\left[ \frac{\lambda }{\sqrt{\lambda ^{2}-f\left( u\right) ^{2}}}\right] du,
\end{align*}%
where 
\begin{equation*}
	\frac{\partial }{\partial \lambda }\left[ \frac{\lambda }{\sqrt{\lambda
			^{2}-f\left( u\right) ^{2}}}\right] =\frac{1\cdot \sqrt{\lambda ^{2}-f\left(
			u\right) ^{2}}-\lambda \cdot \frac{2\lambda }{2\sqrt{\lambda ^{2}-f\left(
				u\right) ^{2}}}}{\lambda ^{2}-f\left( u\right) ^{2}}=\frac{-f\left( u\right)
		^{2}}{\left( \lambda ^{2}-f\left( u\right) ^{2}\right) ^{\frac{3}{2}}}.
\end{equation*}%
Thus we have%
\begin{eqnarray*}
	\frac{\partial x}{\partial r} &=&-\frac{\sqrt{\lambda ^{2}-f\left( x\right)
			^{2}}}{\lambda }, \\
	\frac{\partial x}{\partial \lambda } &=&\frac{2\sqrt{\lambda ^{2}-f\left(
			x\right) ^{2}}}{\lambda }L^{\prime }\left( \lambda \right) +\frac{\sqrt{%
			\lambda ^{2}-f\left( x\right) ^{2}}}{\lambda }\cdot \int_{0}^{x}\frac{%
		f\left( u\right) ^{2}}{\left( \lambda ^{2}-f\left( u\right) ^{2}\right) ^{%
			\frac{3}{2}}}du.
\end{eqnarray*}%
Applying implicit differentiation to the second equation in (\ref{r and y
	Region 2}), we have 
\begin{align*}
	\frac{\partial y}{\partial r}& =-\frac{\partial x}{\partial r}\cdot \frac{%
		f\left( x\right) ^{2}}{\sqrt{\lambda ^{2}-f\left( x\right) ^{2}}}; \\
	\frac{\partial y}{\partial \lambda }& =2Y^{\prime }\left( \lambda \right) -%
	\frac{\partial x}{\partial \lambda }\cdot \frac{f\left( x\right) ^{2}}{\sqrt{%
			\lambda ^{2}-f\left( x\right) ^{2}}}-\int_{0}^{x}\frac{\partial }{\partial
		\lambda }\left[ \frac{f\left( u\right) ^{2}}{\sqrt{\lambda ^{2}-f\left(
			u\right) ^{2}}}\right] du \\
	& =2Y^{\prime }\left( \lambda \right) -\frac{\partial x}{\partial \lambda }%
	\cdot \frac{f\left( x\right) ^{2}}{\sqrt{\lambda ^{2}-f\left( x\right) ^{2}}}%
	-\int_{0}^{x}\frac{-\lambda f\left( u\right) ^{2}}{\left( \lambda
		^{2}-f\left( u\right) ^{2}\right) ^{\frac{3}{2}}}du
\end{align*}%
Plugging the equation for $\frac{\partial x}{\partial \lambda }$ above into
these equations, we have%
\begin{eqnarray*}
	\frac{\partial y}{\partial r} &=&\frac{f\left( x\right) ^{2}}{\lambda }, \\
	\frac{\partial y}{\partial \lambda } &=&2Y^{\prime }\left( \lambda \right) -%
	\frac{2f\left( x\right) ^{2}}{\lambda }R^{\prime }\left( \lambda \right) +%
	\frac{\lambda ^{2}-f\left( x\right) ^{2}}{\lambda }\cdot \int_{0}^{x}\frac{%
		f\left( u\right) ^{2}}{\left( \lambda ^{2}-f\left( u\right) ^{2}\right) ^{%
			\frac{3}{2}}}du.
\end{eqnarray*}%
Thus the Jacobian is given by 
\begin{align*}
	&\frac{\partial \left( x,y\right) }{\partial \left( r,\lambda \right) }\\
	&=\det 
	\begin{bmatrix}
		-\frac{\sqrt{\lambda ^{2}-f\left( x\right) ^{2}}}{\lambda } & \frac{2\sqrt{%
				\lambda ^{2}-f\left( x\right) ^{2}}}{\lambda }R^{\prime }\left( \lambda
		\right) +\frac{\sqrt{\lambda ^{2}-f\left( x\right) ^{2}}}{\lambda }\cdot
		\int_{0}^{x}\frac{f\left( u\right) ^{2}}{\left( \lambda ^{2}-f\left(
			u\right) ^{2}\right) ^{\frac{3}{2}}}du \\ 
		\frac{f\left( x\right) ^{2}}{\lambda } & 2Y^{\prime }\left( \lambda \right) -%
		\frac{2f\left( x\right) ^{2}}{\lambda }R^{\prime }\left( \lambda \right) +%
		\frac{\lambda ^{2}-f\left( x\right) ^{2}}{\lambda }\cdot \int_{0}^{x}\frac{%
			f\left( u\right) ^{2}}{\left( \lambda ^{2}-f\left( u\right) ^{2}\right) ^{%
				\frac{3}{2}}}du%
	\end{bmatrix}
	\\
	& =\det 
	\begin{bmatrix}
		-\frac{\sqrt{\lambda ^{2}-f\left( x\right) ^{2}}}{\lambda } & \frac{\sqrt{%
				\lambda ^{2}-f\left( x\right) ^{2}}}{\lambda }\cdot \int_{0}^{x}\frac{%
			f\left( u\right) ^{2}}{\left( \lambda ^{2}-f\left( u\right) ^{2}\right) ^{%
				\frac{3}{2}}}du \\ 
		\frac{f\left( x\right) ^{2}}{\lambda } & 2Y^{\prime }\left( \lambda \right) +%
		\frac{\lambda ^{2}-f\left( x\right) ^{2}}{\lambda }\cdot \int_{0}^{x}\frac{%
			f\left( u\right) ^{2}}{\left( \lambda ^{2}-f\left( u\right) ^{2}\right) ^{%
				\frac{3}{2}}}du%
	\end{bmatrix}
	\\
	&= -\sqrt{\lambda ^{2}-f\left( x\right) ^{2}}\left\{ \int_{0}^{x}\frac{%
		f\left( u\right) ^{2}}{\left( \lambda ^{2}-f\left( u\right) ^{2}\right) ^{%
			\frac{3}{2}}}du+\frac{2}{\lambda }Y^{\prime }\left( \lambda \right) \right\}
	\\
	&= -\sqrt{\lambda ^{2}-f\left( x\right) ^{2}}\left\{ \int_{0}^{x}\frac{%
		f\left( u\right) ^{2}}{\left( \lambda ^{2}-f\left( u\right) ^{2}\right) ^{%
			\frac{3}{2}}}du+\frac{2}{\lambda }\int\limits_{0}^{f^{-1}\left( \lambda \right) }%
	\frac{F^{\prime \prime }\left( u\right) }{\left\vert F^{\prime }\left(
		u\right) \right\vert ^{2}}\frac{\lambda }{\sqrt{\lambda ^{2}-f\left(
			u\right) ^{2}}}du\right\} \\
	&= -\sqrt{\lambda ^{2}-f\left( x\right) ^{2}}\left\{ \int_{0}^{x}\frac{%
		f\left( u\right) ^{2}}{\left( \lambda ^{2}-f\left( u\right) ^{2}\right) ^{%
			\frac{3}{2}}}du+\int\limits_{0}^{f^{-1}(\lambda )}\frac{F^{\prime \prime }\left(
		u\right) }{\left\vert F^{\prime }\left( u\right) \right\vert ^{2}}\cdot 
	\frac{2}{\sqrt{\lambda ^{2}-f\left( u\right) ^{2}}}du\right\} .
\end{align*}%
In fact, we have 
\begin{align*}
	\int_{0}^{x}&\frac{f\left( u\right) ^{2}}{\left( \lambda ^{2}-f\left(
		u\right) ^{2}\right) ^{\frac{3}{2}}}du =\int_{0}^{x}\frac{f\left( u\right) 
	}{f^{\prime }\left( u\right) }\cdot \frac{d}{du}\left[ \frac{1}{\sqrt{%
			\lambda ^{2}-f\left( u\right) ^{2}}}\right] du \\
	& =\int_{0}^{x}\frac{1}{-F^{\prime }\left( u\right) }\cdot \frac{d}{du}\left[
	\frac{1}{\sqrt{\lambda ^{2}-f\left( u\right) ^{2}}}\right] du \\
	& =\frac{1}{-F^{\prime }\left( x\right) }\cdot \frac{1}{\sqrt{\lambda
			^{2}-f\left( x\right) ^{2}}}-\int_{0}^{x}\frac{1}{\sqrt{\lambda ^{2}-f\left(
			u\right) ^{2}}}\cdot \frac{d}{du}\left[ \frac{1}{-F^{\prime }\left( u\right) 
	}\right] du \\
	& =\frac{1}{-F^{\prime }\left( x\right) }\cdot \frac{1}{\sqrt{\lambda
			^{2}-f\left( x\right) ^{2}}}-\int_{0}^{x}\frac{1}{\sqrt{\lambda ^{2}-f\left(
			u\right) ^{2}}}\cdot \frac{F^{\prime \prime }\left( u\right) }{\left\vert
		F^{\prime }\left( u\right) \right\vert ^{2}}du
\end{align*}%
As a result, we have within a factor of $2$, 
\begin{align}\label{jacobian est1}
	\left\vert \frac{\partial \left( x,y\right) }{\partial \left( r,\lambda
		\right) }\right\vert  \approx &\sqrt{\lambda ^{2}-f\left( x\right) ^{2}}%
	\left\{ \frac{1}{-F^{\prime }\left( x\right) }\cdot \frac{1}{\sqrt{\lambda
			^{2}-f\left( x\right) ^{2}}}\right. \\
	&\qquad\qquad\qquad\qquad+\left.\int_{0}^{f^{-1}\left( \lambda \right) }\frac{%
		F^{\prime \prime }\left( u\right) }{\left\vert F^{\prime }\left( u\right)
		\right\vert ^{2}}\cdot \frac{1}{\sqrt{\lambda ^{2}-f\left( u\right) ^{2}}}%
	du\right\} \notag \\
	=&\frac{1}{-F^{\prime }\left( x\right) }+\sqrt{\lambda ^{2}-f\left(
		x\right) ^{2}}\int_{0}^{f^{-1}\left( \lambda \right) }\frac{F^{\prime \prime
		}\left( u\right) }{\left\vert F^{\prime }\left( u\right) \right\vert ^{2}}%
	\cdot \frac{1}{\sqrt{\lambda ^{2}-f\left( u\right) ^{2}}}du.\notag 
\end{align}%
By Assumption (5), we have 
\begin{equation*}
	\int_{0}^{f^{-1}\left( \lambda \right) }\frac{F^{\prime \prime }\left(
		u\right) }{\left\vert F^{\prime }\left( u\right) \right\vert ^{2}}\cdot 
	\frac{1}{\sqrt{\lambda ^{2}-f\left( u\right) ^{2}}}du\approx
	\int_{0}^{f^{-1}\left( \lambda \right) }\frac{1}{-uF^{\prime }\left(
		u\right) }\cdot \frac{1}{\sqrt{\lambda ^{2}-f\left( u\right) ^{2}}}du.
\end{equation*}%
By Assumptions (3) and (4), the function $\frac{1}{-uF^{\prime }\left(
	u\right) }$ increases and satisfies the doubling property, and so 
\begin{align}\label{jacobian est2}
	\int_{0}^{f^{-1}\left( \lambda \right) }\frac{F^{\prime \prime }\left(
		u\right) }{\left\vert F^{\prime }\left( u\right) \right\vert ^{2}}&\cdot 
	\frac{1}{\sqrt{\lambda ^{2}-f\left( u\right) ^{2}}}du\\
	& \approx \frac{1}{%
		-f^{-1}\left( \lambda \right) F^{\prime }\left( f^{-1}\left( \lambda \right)
		\right) }\int_{0}^{f^{-1}\left( \lambda \right) }\frac{1}{\sqrt{\lambda
			^{2}-f\left( u\right) ^{2}}}du  \notag \\
	& =\frac{1}{-f^{-1}\left( \lambda \right) F^{\prime }\left( f^{-1}\left(
		\lambda \right) \right) }\frac{R\left( \lambda \right) }{\lambda }  \notag \\
	& \simeq \frac{1}{-\lambda F^{\prime }\left( f^{-1}\left( \lambda \right)
		\right) }  \notag
\end{align}%
since $R\left( \lambda \right) \approx f^{-1}\left( \lambda \right) $ by
Lemma \ref{arc length}. Finally we can combine (\ref{jacobian est1}) and (%
\ref{jacobian est2}) to obtain 
\begin{equation*}
	\left\vert \frac{\partial \left( x,y\right) }{\partial \left( r,\lambda
		\right) }\right\vert \approx \frac{1}{-F^{\prime }\left( x\right) }+\frac{%
		\sqrt{\lambda ^{2}-f\left( x\right) ^{2}}}{\lambda }\cdot \frac{1}{%
		-F^{\prime }\left( f^{-1}\left( \lambda \right) \right) }\approx \frac{1}{%
		-F^{\prime }\left( f^{-1}\left( \lambda \right) \right) }.
\end{equation*}%
According to Corollary \ref{similar r and x}, we also have%
\begin{equation*}
	\left\vert \frac{\partial \left( x,y\right) }{\partial \left( r,\lambda
		\right) }\right\vert \approx \frac{1}{-F^{\prime }(R\left( \lambda \right) )}%
	.
\end{equation*}

\subsection{Integral of Radial Functions}

Summarizing our estimates on the Jacobian we have 
\begin{equation*}
	\left\vert \frac{\partial \left( x,y\right) }{\partial \left( r,\lambda
		\right) }\right\vert \approx \left\{ 
	\begin{array}{ll}
		\frac{f\left( r\right) ^{2}}{\lambda ^{2}\left\vert F^{\prime }\left(
			r\right) \right\vert }\simeq \frac{f\left( x\right) ^{2}}{\lambda
			^{2}\left\vert F^{\prime }\left( x\right) \right\vert } & \text{when}%
		\;r<R\left( \lambda \right) \\ 
		&  \\ 
		\frac{1}{\left\vert F^{\prime }\left( f^{-1}\left( \lambda \right) \right)
			\right\vert }\simeq \frac{1}{\left\vert F^{\prime }\left( R\left( \lambda
			\right) \right) \right\vert } & \text{when}\;R\left( \lambda \right)
		<r<2R\left( \lambda \right)%
	\end{array}%
	\right. .
\end{equation*}%
Therefore we have the following change of variable formula for nonnegative
functions $w$: 
\begin{align*}
	&\iint_{QB\left( 0,r_{0}\right) }wdxdy =\int_{0}^{r_{0}}\left[
	\int_{R^{-1}\left( \frac{r}{2}\right) }^{\infty }w\left\vert \frac{\partial
		(x,y)}{\partial (r,\lambda )}\right\vert d\lambda \right] dr \\
	& \approx \int_{0}^{r_{0}}\left[ \int_{R^{-1}\left( \frac{r}{2}\right)
	}^{R^{-1}\left( r\right) }w\left( r,\lambda \right) \frac{1}{\left\vert
		F^{\prime }\left( R\left( \lambda \right) \right) \right\vert }d\lambda
	+\int_{R^{-1}\left( r\right) }^{\infty }w\left( r,\lambda \right) \frac{%
		f\left( r\right) ^{2}}{\lambda ^{2}\left\vert F^{\prime }\left( r\right)
		\right\vert }d\lambda \right] dr \\
	& \approx \int_{0}^{r_{0}}\left[ \int_{R^{-1}\left( \frac{r}{2}\right)
	}^{R^{-1}\left( r\right) }w(r,\lambda )\frac{1}{|F^{\prime }(r)|}d\lambda
	+\int_{R^{-1}(r)}^{\infty }w(r,\lambda )\frac{f^{2}(r)}{\lambda
		^{2}|F^{\prime }(r)|}d\lambda \right] dr
\end{align*}%
If $w$ is a radial function, then we have%
\begin{align*}
	\iint_{QB\left( 0,r_{0}\right) }wdxdy& \approx \int_{0}^{r_{0}}w\left(
	r\right) \left[ \int_{R^{-1}\left( \frac{r}{2}\right) }^{R^{-1}\left(
		r\right) }\frac{1}{\left\vert F^{\prime }\left( r\right) \right\vert }%
	d\lambda +\int_{R^{-1}\left( r\right) }^{\infty }\frac{f\left( r\right) ^{2}%
	}{\lambda ^{2}\left\vert F^{\prime }\left( r\right) \right\vert }d\lambda %
	\right] dr \\
	& \approx \int_{0}^{r_{0}}w\left( r\right) \left[ \frac{R^{-1}\left(
		r\right) -R^{-1}\left( \frac{r}{2}\right) }{\left\vert F^{\prime }\left(
		r\right) \right\vert }+\frac{f\left( r\right) ^{2}}{R^{-1}\left( r\right)
		\left\vert F^{\prime }\left( r\right) \right\vert }\right] dr.
\end{align*}%
From Corollary \ref{similar r and x}, we have $R^{-1}\left( r\right) \simeq
f\left( r\right) $, and so we have 
\begin{equation}
\iint_{B\left( 0,r_{0}\right) }w\left( r\right) dxdy\approx
\int_{0}^{r_{0}}w\left( r\right) \frac{f\left( r\right) }{\left\vert
	F^{\prime }\left( r\right) \right\vert }dr.  \label{radial integration}
\end{equation}

\begin{conclusion}
	The area of the $A$-ball $B\left( 0,r_{0}\right) $ satisfies%
	\begin{equation}
	\left\vert B\left( 0,r_{0}\right) \right\vert =\iint_{B\left( 0,r_{0}\right)
	}dxdy\approx \int_{0}^{r_{0}}\frac{f\left( r\right) }{\left\vert F^{\prime
		}\left( r\right) \right\vert }dr\approx \frac{f\left( r_{0}\right) }{%
		\left\vert F^{\prime }\left( r_{0}\right) \right\vert ^{2}}.
	\label{ball-origin}
	\end{equation}
\end{conclusion}

\begin{proof}
	Since $F\left( r\right) =-\ln f\left( r\right) $, we have $F^{\prime }\left(
	r\right) =-\frac{f^{\prime }\left( r\right) }{f\left( r\right) }$ and $\frac{%
		f\left( r\right) }{-F^{\prime }\left( r\right) }=\frac{f\left( r\right) ^{2}%
	}{f^{\prime }\left( r\right) }=\frac{f\left( r\right) ^{2}}{f^{\prime
		}\left( r\right) ^{2}}f^{\prime }\left( r\right) =\frac{f^{\prime }\left(
		r\right) }{\left\vert F^{\prime }\left( r\right) \right\vert ^{2}}$, and so%
	\begin{eqnarray*}
		\iint_{B\left( 0,r_{0}\right) }dxdy &\approx &\int_{0}^{r_{0}}\frac{f\left(
			r\right) }{\left\vert F^{\prime }\left( r\right) \right\vert }dr\approx
		\int_{\frac{r_{0}}{2}}^{r_{0}}\frac{f\left( r\right) }{\left\vert F^{\prime
			}\left( r\right) \right\vert }dr=\int_{\frac{r_{0}}{2}}^{r_{0}}\frac{%
			f^{\prime }\left( r\right) }{\left\vert F^{\prime }\left( r\right)
			\right\vert ^{2}}dr \\
		&\approx &\frac{1}{\left\vert F^{\prime }\left( r_{0}\right) \right\vert ^{2}%
		}\int_{\frac{r_{0}}{2}}^{r_{0}}f^{\prime }\left( r\right) dr=\frac{f\left(
			r_{0}\right) -f\left( \frac{r_{0}}{2}\right) }{\left\vert F^{\prime }\left(
			r_{0}\right) \right\vert ^{2}}\approx \frac{f\left( r_{0}\right) }{%
			\left\vert F^{\prime }\left( r_{0}\right) \right\vert ^{2}}.
	\end{eqnarray*}
\end{proof}

\subsection{Balls centered at an arbitrary point\label{arbitrary balls}}

In this section we consider the \textquotedblleft height\textquotedblright\
of an arbitrary $A$-ball and the relative position at which it is achieved
in the ball.

\begin{definition}
	\label{def r*}Let $X=(x_{1},0)$ be a point on the positive $x$-axis and let $%
	r$ be a positive real number. Let the upper half of the boundary of the ball 
	$B\left( X,r\right) $ be given as a graph of the function $\varphi \left(
	x\right) $, $x_{1}-r<x<x_{1}+r$. Denote by $\beta _{X,Q}$ the geodesic
	connecting the center $X$ of the ball $B\left( X,r\right) $ with a point $Q$
	on the boundary $\partial B\left( X,r\right) $ of the ball $B\left(
	X,r\right) $.\\[0.15in]
	Denote by $P=P_{x_{1},r}=\left( x_{1}+r^{\ast },h\right) $ the unique point
	on the boundary $\partial B\left( X,r\right) $ of the ball $B\left(
	X,r\right) $ with $r^{\ast }>0$ and $h>0$ at which the geodesic $\beta
	_{X,P} $ connecting $X$ and $P$ has a \emph{vertical tangent} at $P$. This
	defines 
	\begin{equation*}
		r^{\ast }=r^{\ast }\left( x_{1},r\right) \text{ and }h=h\left(
		x_{1},r\right) =\varphi \left( x_{1}+r^{\ast }\right)
	\end{equation*}%
	implicitly as functions of the two independent variables $x_{1}$ and $r$.
\end{definition}

We will often write simply $r^{\ast }$ and $h$ in place of $r^{\ast }\left(
x_{1},r\right) $ and $h\left( x_{1},r\right) $ respectively when $x_{1}$ and 
$r$ are understood. See figure \ref{ball_right}.

\begin{figure}%[ht]
	\includegraphics{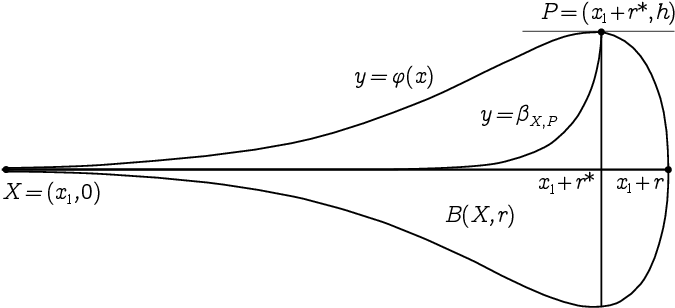}
	\caption{The right `half' of a control ball centered at $\left( x_{1},0\right) $.}
	\label{ball_right}
\end{figure}

\begin{proposition}
	\label{height}Let $\beta _{X,P}$, $r^{\ast }$ and $h$ be defined as above.
	Define $\lambda \left( x\right) $ implicitly by%
	\begin{equation*}
		r=\int_{x_{1}}^{x}\frac{\lambda \left( x\right) }{\sqrt{\lambda \left(
				x\right) ^{2}-f\left( u\right) ^{2}}}du.
	\end{equation*}%
	Then
	
	\begin{enumerate}
		\item For $x_{1}-r<x<x_{1}+r$ we have $\varphi \left( x\right) \leq \varphi
		\left( x_{1}+r^{\ast }\right) =h$.
		
		\item If $r\geq \frac{1}{\left\vert F^{\prime }\left( x_{1}\right)
			\right\vert }$, then 
		\begin{equation*}
			h\approx \frac{f\left( x_{1}+r\right) }{\left\vert F^{\prime }\left(
				x_{1}+r\right) \right\vert }\text{ and }r-r^{\ast }\approx \frac{1}{%
				\left\vert F^{\prime }\left( x_{1}+r\right) \right\vert }.
		\end{equation*}
		
		\item If $r\leq \frac{1}{\left\vert F^{\prime }\left( x_{1}\right)
			\right\vert }$, then%
		\begin{equation*}
			h\approx rf\left( x_{1}\right) \text{ and }r-r^{\ast }\approx r.
		\end{equation*}
	\end{enumerate}
\end{proposition}

We begin by proving part (1) of Proposition \ref{height}. First consider the
case $x\geq x_{1}+r^{\ast }$. Then we are in Region 1 and so $\lambda \left(
x\right) \geq f\left( x\right) $ and we have 
\begin{equation*}
	\varphi \left( x\right) =\int_{x_{1}}^{x}\frac{f\left( u\right) ^{2}}{\sqrt{%
			\lambda \left( x\right) ^{2}-f\left( u\right) ^{2}}}du.
\end{equation*}%
Differentiating $\varphi \left( x\right) $ we get%
\begin{equation*}
	\varphi ^{\prime }\left( x\right) =\frac{f\left( x\right) ^{2}}{\sqrt{%
			\lambda \left( x\right) ^{2}-f\left( x\right) ^{2}}}-\left( \int_{x_{1}}^{x}%
	\frac{f\left( u\right) ^{2}}{\left( \lambda \left( x\right) ^{2}-f\left(
		u\right) ^{2}\right) ^{\frac{3}{2}}}du\right) \lambda \left( x\right)
	\lambda ^{\prime }\left( x\right) ,
\end{equation*}%
and differentiating the definition of $\lambda \left( x\right) $ implicitly
gives%
\begin{equation*}
	0=\frac{\lambda \left( x\right) }{\sqrt{\lambda \left( x\right) ^{2}-f\left(
			x\right) ^{2}}}-\left( \int_{x_{1}}^{x}\frac{f\left( u\right) ^{2}}{\left(
		\lambda \left( x\right) ^{2}-f\left( u\right) ^{2}\right) ^{\frac{3}{2}}}%
	du\right) \lambda ^{\prime }\left( x\right) .
\end{equation*}%
Combining equalities yields%
\begin{equation*}
	\varphi ^{\prime }\left( x\right) =\frac{f\left( x\right) ^{2}}{\sqrt{%
			\lambda \left( x\right) ^{2}-f\left( x\right) ^{2}}}-\lambda \left( x\right) 
	\frac{\lambda \left( x\right) }{\sqrt{\lambda \left( x\right) ^{2}-f\left(
			x\right) ^{2}}}=-\sqrt{\lambda \left( x\right) ^{2}-f\left( x\right) ^{2}}.
\end{equation*}%
When $x=x_{1}+r^{\ast }$ we have $\infty =\frac{dy}{dx}=\frac{f\left(
	x\right) ^{2}}{\sqrt{\lambda \left( x\right) ^{2}-f\left( x\right) ^{2}}}$,
which implies $\lambda \left( x\right) =f\left( x\right) $, and hence $%
\varphi ^{\prime }\left( x_{1}+r^{\ast }\right) =0$. Thus we have $\varphi
\left( x\right) \leq \varphi \left( x_{1}+r^{\ast }\right) =h$ for $x\geq
x_{1}+r^{\ast }$. Similar arguments show that $\varphi \left( x\right) \leq
\varphi \left( x_{1}+r^{\ast }\right) =h$ for $x_{1}-r\leq x<x_{1}+r^{\ast }$%
, and this completes the proof of part (1).

Now we turn to the proofs of parts (2) and (3) of Proposition \ref{height}.
The locus $\left( x,y\right) $ of the geodesic $\beta _{X,P}$ satisfies 
\begin{equation}
y=\int_{x_{1}}^{x}\frac{f\left( u\right) ^{2}}{\sqrt{\left( \lambda ^{\ast
		}\right) ^{2}-f\left( u\right) ^{2}}}\,du,  \label{eqn for geo1}
\end{equation}%
where $\lambda ^{\ast }=f\left( x_{1}+r^{\ast }\right) $. We will use the
following two lemmas in the proofs of parts (2) and (3) of Proposition \ref%
{height}.

\begin{lemma}
	\label{height 1} The height $h=h\left( x_{1},r\right) $ and the horizontal
	displacement $r-r^{\ast }=r-r^{\ast }\left( x_{1},r\right) $ satisfy 
	\begin{equation*}
		f\left( x_{1}+r^{\ast }\right) \cdot \left( r-r^{\ast }\right) \leq h\leq
		2f\left( x_{1}+r^{\ast }\right) \cdot \left( r-r^{\ast }\right) .
	\end{equation*}
\end{lemma}

\begin{proof}
	The $A$-arc length $r$ of $\beta _{X,P}$ is given by%
	\begin{equation*}
		r=\int_{x_{1}}^{x_{1}+r^{\ast }}\frac{\lambda ^{\ast }}{\sqrt{\left( \lambda
				^{\ast }\right) ^{2}-f\left( u\right) ^{2}}}\,du.
	\end{equation*}%
	Thus 
	\begin{align*}
		r-r^{\ast }&=\int_{x_{1}}^{x_{1}+r^{\ast }}\left( \frac{\lambda ^{\ast }}{%
			\sqrt{\left( \lambda ^{\ast }\right) ^{2}-f\left( u\right) ^{2}}}-1\right)
		\,du\\
		&=\int_{x_{1}}^{x_{1}+r^{\ast }}\frac{f\left( u\right) ^{2}}{\sqrt{\left(
				\lambda ^{\ast }\right) ^{2}-f\left( u\right) ^{2}}}\cdot \frac{1}{\lambda
			^{\ast }+\sqrt{\left( \lambda ^{\ast }\right) ^{2}-f\left( u\right) ^{2}}}%
		\,du
	\end{align*}%
	Comparing this with the height $h=\int_{x_{1}}^{x_{1}+r^{\ast }}\frac{%
		f^{2}(u)}{\sqrt{\left( \lambda ^{\ast }\right) ^{2}-f\left( u\right) ^{2}}}%
	\,du$, we have 
	\begin{equation*}
		\frac{h}{2\lambda ^{\ast }}\leq r-r^{\ast }\leq \frac{h}{\lambda ^{\ast }}.
	\end{equation*}%
	This completes the proof since $\lambda ^{\ast }=f\left( x_{1}+r^{\ast
	}\right) $.
\end{proof}

\begin{lemma}
	\label{height 2} The height $h$ satisfies the estimate 
	\begin{equation*}
		h\approx \frac{1}{\left\vert F^{\prime }\left( x_{1}+r^{\ast }\right)
			\right\vert }\sqrt{f\left( x_{1}+r^{\ast }\right) ^{2}-f\left( x_{1}\right)
			^{2}}.
	\end{equation*}%
	In fact the right hand side is an exact upper bound: 
	\begin{equation*}
		h\leq \frac{1}{\left\vert F^{\prime }\left( x_{1}+r^{\ast }\right)
			\right\vert }\sqrt{f\left( x_{1}+r^{\ast }\right) ^{2}-f\left( x_{1}\right)
			^{2}}.
	\end{equation*}
\end{lemma}

\begin{proof}
	Using the fact that $\frac{1}{-F^{\prime }\left( u\right) }=\frac{1}{%
		\left\vert F^{\prime }\left( u\right) \right\vert }$ is increasing, together
	with the equation \eqref{eqn for geo1} for the geodesic $\beta _{X,P}$, we
	have 
	\begin{align*}
		h\left( x_{1},r\right) =&\int_{x_{1}}^{x_{1}+r^{\ast }\left( x_{1},r\right) }%
		\frac{f\left( u\right) ^{2}}{\sqrt{\left( \lambda ^{\ast }\right)
				^{2}-f\left( u\right) ^{2}}}\,du\\
		=& \int_{x_{1}}^{x_{1}+r^{\ast }}\frac{\frac{%
				d}{du}\left[ f\left( u\right) ^{2}\right] }{\sqrt{\left( \lambda ^{\ast
				}\right) ^{2}-f\left( u\right) ^{2}}}\cdot \frac{1}{-2F^{\prime }\left(
			u\right) }\,du \\
		\leq & \frac{1}{\left\vert F^{\prime }\left( x_{1}+r^{\ast }\right)
			\right\vert }\int_{x_{1}}^{x_{1}+r^{\ast }}\frac{\frac{d}{du}\left[ f\left(
			u\right) ^{2}\right] }{2\sqrt{\left( \lambda ^{\ast }\right) ^{2}-f\left(
				u\right) ^{2}}}\,du \\
		=& \frac{1}{\left\vert F^{\prime }\left( x_{1}+r^{\ast }\right) \right\vert }%
		\sqrt{f\left( x_{1}+r^{\ast }\right) ^{2}-f\left( x_{1}\right) ^{2}},
	\end{align*}%
	where in the last line we used $\lambda ^{\ast }=f\left( x_{1}+r^{\ast
	}\right) $. To prove the reverse estimate, we consider two cases:
	
	\textbf{Case 1}: If $r^{\ast }<x_{1}$, then we use our assumption that $%
	\frac{1}{-F^{\prime }\left( u\right) }=\frac{1}{\left\vert F^{\prime }\left(
		u\right) \right\vert }$ has the doubling property to obtain 
	\begin{align*}
		h=\int_{x_{1}}^{x_{1}+r^{\ast }}\frac{f\left( u\right) ^{2}}{\sqrt{\left(
				\lambda ^{\ast }\right) ^{2}-f\left( u\right) ^{2}}}\,du=&
		\int_{x_{1}}^{x_{1}+r^{\ast }}\frac{\frac{d}{du}\left[ f\left( u\right) ^{2}%
			\right] }{\sqrt{\left( \lambda ^{\ast }\right) ^{2}-f\left( u\right) ^{2}}}%
		\cdot \frac{1}{-2F^{\prime }\left( u\right) }\,du \\
		\simeq & \frac{1}{\left\vert F^{\prime }\left( x_{1}+r^{\ast }\right)
			\right\vert }\int_{x_{1}}^{x_{1}+r^{\ast }}\frac{\frac{d}{du}\left[ f\left(
			u\right) ^{2}\right] }{2\sqrt{\left( \lambda ^{\ast }\right) ^{2}-f\left(
				u\right) ^{2}}}\,du \\
		=& \frac{1}{\left\vert F^{\prime }\left( x_{1}+r^{\ast }\right) \right\vert }%
		\sqrt{f\left( x_{1}+r^{\ast }\right) ^{2}-f\left( x_{1}\right) ^{2}}.
	\end{align*}
	
	\textbf{Case 2}: If $r^{\ast }\geq x_{1}$, we make a similar estimate by
	modifying the lower limit of integral, and using the fact that $f\left(
	u\right) $ increases: 
	\begin{align*}
		h\approx \int_{x_{1}+\frac{r^{\ast }}{2}}^{x_{1}+r^{\ast }}\frac{f\left(
			u\right) ^{2}}{\sqrt{\left( \lambda ^{\ast }\right) ^{2}-f\left( u\right)
				^{2}}}\,du=& \int_{x_{1}+\frac{r^{\ast }}{2}}^{x_{1}+r^{\ast }}\frac{\frac{d%
			}{du}\left[ f\left( u\right) ^{2}\right] }{\sqrt{\left( \lambda ^{\ast
				}\right) ^{2}-f\left( u\right) ^{2}}}\cdot \frac{1}{-2F^{\prime }\left(
			u\right) }\,du \\
		\approx & \frac{1}{\left\vert F^{\prime }\left( x_{1}+r^{\ast }\right)
			\right\vert }\int_{x_{1}+\frac{r^{\ast }}{2}}^{x_{1}+r^{\ast }}\frac{\frac{d%
			}{du}\left[ f\left( u\right) ^{2}\right] }{2\sqrt{\left( \lambda ^{\ast
				}\right) ^{2}-f\left( u\right) ^{2}}}\,du \\
		=& \frac{1}{\left\vert F^{\prime }\left( x_{1}+r^{\ast }\right) \right\vert }%
		\sqrt{f\left( x_{1}+r^{\ast }\right) ^{2}-f\left( x_{1}+\frac{r^{\ast }}{2}%
			\right) ^{2}}.
	\end{align*}%
	Finally we have 
	\begin{equation*}
		\sqrt{f\left( x_{1}+r^{\ast }\right) ^{2}-f\left( x_{1}+\frac{r^{\ast }}{2}%
			\right) ^{2}}\approx f\left( x_{1}+r^{\ast }\right) \approx \sqrt{f\left(
			x_{1}+r^{\ast }\right) ^{2}-f\left( x_{1}\right) ^{2}}
	\end{equation*}%
	by the assumption $r^{\ast }\geq x_{1}$ together with Part 1 of Lemma \ref%
	{consequences}.
	
	This completes the proof of Lemma \ref{height 2}.
\end{proof}

\begin{corollary}
	Combining Lemmas \ref{height 1} and \ref{height 2}, for $h=h\left(
	x_{1},r\right) $ and $r^{\ast }=r^{\ast }\left( x_{1},r\right) $, we have 
	\begin{equation*}
		f\left( x_{1}+r^{\ast }\right) \cdot \left( r-r^{\ast }\right) \leq h\leq 
		\frac{1}{\left\vert F^{\prime }\left( x_{1}+r^{\ast }\right) \right\vert }%
		\sqrt{f\left( x_{1}+r^{\ast }\right) ^{2}-f\left( x_{1}\right) ^{2}},
	\end{equation*}%
	and as a result, 
	\begin{equation}
	r-r^{\ast }\leq \frac{1}{\left\vert F^{\prime }\left( x_{1}+r^{\ast }\right)
		\right\vert }\cdot \frac{\sqrt{f\left( x_{1}+r^{\ast }\right) ^{2}-f\left(
			x_{1}\right) ^{2}}}{f\left( x_{1}+r^{\ast }\right) }\leq \frac{1}{\left\vert
		F^{\prime }\left( x_{1}+r^{\ast }\right) \right\vert }.
	\label{upper bound of r star}
	\end{equation}%
	From part (2) of Lemma \ref{consequences} we obtain%
	\begin{eqnarray}
	\left\vert F^{\prime }\left( x_{1}+r\right) \right\vert &\approx &\left\vert
	F^{\prime }\left( x_{1}+r^{\ast }\right) \right\vert ,  \label{new f equ} \\
	f\left( x_{1}+r\right) &\simeq &f\left( x_{1}+r^{\ast }\right) .  \notag
	\end{eqnarray}
\end{corollary}

We now split the proof of Proposition \ref{height} into two cases.

\subsubsection{Proof of part (2) of Proposition \protect\ref{height} for $%
	r\geq \frac{1}{\left\vert F^{\prime }\left( x_{1}\right) \right\vert }$}

By Lemmas \ref{height 1} and \ref{height 2}, we have 
\begin{equation}
r-r^{\ast }\left( x_{1},r\right) =r-r^{\ast }\approx \frac{1}{\left\vert
	F^{\prime }\left( x_{1}+r^{\ast }\right) \right\vert }\cdot \frac{\sqrt{%
		f\left( x_{1}+r^{\ast }\right) ^{2}-f\left( x_{1}\right) ^{2}}}{f\left(
	x_{1}+r^{\ast }\right) }.  \label{approx of r minus r star}
\end{equation}%
We consider two cases.

\textbf{Case A}: If $r^{\ast }>r_{1}\equiv \frac{1}{2\left\vert F^{\prime
	}\left( x_{1}\right) \right\vert }$, then we have 
\begin{align*}
	F\left( x_{1}\right) -F\left( x_{1}+r^{\ast }\right)
	=\int_{x_{1}}^{x_{1}+r^{\ast }}\left\vert F^{\prime }\left( x_{1}\right)
	\right\vert dx&\geq \int_{x_{1}}^{x_{1}+r_{1}}\left\vert F^{\prime }\left(
	x_{1}\right) \right\vert dx\\
	&\geq \left\vert F^{\prime }\left(
	x_{1}+r_{1}\right) \right\vert \cdot r_{1}\gtrsim 1.
\end{align*}%
Here we used the estimate $\left\vert F^{\prime }\left( x_{1}+r_{1}\right)
\right\vert \approx \left\vert F^{\prime }\left( x_{1}\right) \right\vert $
given by Part 2 of Lemma \ref{consequences}. This implies 
\begin{equation*}
	\ln \frac{f\left( x_{1}+r^{\ast }\right) }{f\left( x_{1}\right) }\gtrsim 1,
\end{equation*}%
and we have 
\begin{equation*}
	\frac{\sqrt{f\left( x_{1}+r^{\ast }\right) ^{2}-f\left( x_{1}\right) ^{2}}}{%
		f\left( x_{1}+r^{\ast }\right) }\approx 1.
\end{equation*}%
Plugging this into \eqref{approx of r minus r star}, we have $r-r^{\ast
}\approx \frac{1}{\left\vert F^{\prime }\left( x_{1}+r^{\ast }\right)
	\right\vert }$. The proof is completed using \eqref{new f equ} and Lemma \ref%
{height 1}.

\textbf{Case B}: If $r^{\ast }\leq r_{1}$, then we have $\left\vert
F^{\prime }\left( x_{1}+r^{\ast }\right) \right\vert \approx \left\vert
F^{\prime }\left( x_{1}\right) \right\vert $ and $r-r^{\ast }\geq \frac{1}{%
	2\left\vert F^{\prime }\left( x_{1}\right) \right\vert }$. Therefore we have 
\begin{equation*}
	r-r^{\ast }\gtrsim \frac{1}{\left\vert F^{\prime }\left( x_{1}+r^{\ast
		}\right) \right\vert }.
\end{equation*}%
Combining this with \eqref{upper bound of r star}, we obtain $r-r^{\ast
}\approx \frac{1}{\left\vert F^{\prime }\left( x_{1}+r^{\ast }\right)
	\right\vert }$ again, and the proof is completed as in the first case.

\subsubsection{Proof of part (3) of Proposition \protect\ref{height} for $%
	r\leq \frac{1}{\left\vert F^{\prime }\left( x_{1}\right) \right\vert }$}

In this case Lemma \ref{height 1} and Lemma \ref{height 2} give with $%
r^{\ast }=r^{\ast }\left( x_{1},r\right) $, 
\begin{align*}
	f\left( x_{1}+r^{\ast }\right) \cdot \left( r-r^{\ast }\right) \approx
	h\left( x_{1},r\right) & \approx \frac{1}{\left\vert F^{\prime }\left(
		x_{1}+r^{\ast }\right) \right\vert }\sqrt{f\left( x_{1}+r^{\ast }\right)
		^{2}-f\left( x_{1}\right) ^{2}} \\
	& \approx \frac{1}{\left\vert F^{\prime }\left( x_{1}\right) \right\vert }%
	\left( \int_{x_{1}}^{x_{1}+r^{\ast }}2f\left( u\right) ^{2}\left\vert
	F^{\prime }\left( u\right) \right\vert \,du\right) ^{\frac{1}{2}} \\
	& \approx \frac{1}{\left\vert F^{\prime }\left( x_{1}\right) \right\vert }%
	\left[ 2f\left( x_{1}+r^{\ast }\right) ^{2}\left\vert F^{\prime }\left(
	x_{1}\right) \right\vert \cdot r^{\ast }\right] ^{\frac{1}{2}} \\
	& \approx \frac{\sqrt{r^{\ast }}f\left( x_{1}+r^{\ast }\right) }{\sqrt{%
			\left\vert F^{\prime }\left( x_{1}\right) \right\vert }},
\end{align*}%
where we have used Part 2 of Lemma \ref{consequences} and the fact $r^{\ast
}=r^{\ast }\left( x_{1},r\right) <r\leq \frac{1}{\left\vert F^{\prime
	}\left( x_{1}\right) \right\vert }$. This implies 
\begin{equation*}
	\left[ \left\vert F^{\prime }\left( x_{1}\right) \right\vert \left(
	r-r^{\ast }\right) \right] ^{2}\approx \left\vert F^{\prime }\left(
	x_{1}\right) \right\vert r^{\ast }.
\end{equation*}%
Thus 
\begin{equation*}
	\left[ \left\vert F^{\prime }\left( x_{1}\right) \right\vert \left(
	r-r^{\ast }\right) \right] ^{2}+\left\vert F^{\prime }\left( x_{1}\right)
	\right\vert \left( r-r^{\ast }\right) |\approx \left\vert F^{\prime }\left(
	x_{1}\right) \right\vert r\leq 1.
\end{equation*}%
As a result, we have $\left\vert F^{\prime }\left( x_{1}\right) \right\vert
\left( r-r^{\ast }\right) \approx \left\vert F^{\prime }\left( x_{1}\right)
\right\vert r\;\Longrightarrow \;r-r^{\ast }\approx r$. This also gives the
estimate for $h$ by Lemma \ref{height 1} since we already have $f\left(
x_{1}+r^{\ast }\right) \approx f\left( x_{1}\right) $.

\subsection{Area of balls centered at an arbitrary point}

In the following proposition we obtain an estimate, similar to (\ref%
{ball-origin}), for areas of balls centered at arbitrary points.

\begin{proposition}
	\label{general-area} Let $P=\left( x_{1},x_{2}\right) \in \mathbb{R}^{2}$
	and $r>0$. Set 
	\begin{equation*}
		B_{+}\left( P,r\right) \equiv \left\{ \left( y_{1},y_{2}\right) \in B\left(
		P,r\right) :y_{1}>x_{1}+r^{\ast }\right\} .
	\end{equation*}%
	If $r\geq \frac{1}{\left\vert F^{\prime }\left( x_{1}\right) \right\vert }$
	then we recover (\ref{ball-origin})%
	\begin{equation*}
		\left\vert B\left( P,r\right) \right\vert \approx \frac{f\left(
			x_{1}+r\right) }{\left\vert F^{\prime }\left( x_{1}+r\right) \right\vert ^{2}%
		}\approx \left\vert B_{+}\left( P,r\right) \right\vert .
	\end{equation*}%
	On the other hand, if $r\leq \frac{1}{\left\vert F^{\prime }\left(
		x_{1}\right) \right\vert }$ we have 
	\begin{equation*}
		\left\vert B\left( P,r\right) \right\vert \approx r^{2}f(x_{1})\approx
		\left\vert B_{+}\left( P,r\right) \right\vert
	\end{equation*}
\end{proposition}

\begin{proof}
	Because of symmetry, it is enough to consider $x_{1}>0$ and $y_{1}=0$. So
	let $P_{1}=\left( x_{1},0\right) $ with $x_{1}>0$.
	
	\textbf{Case }$r\geq \frac{1}{\left\vert F^{\prime }\left( x_{1}\right)
		\right\vert }$. In this case we will compare the ball $B\left( P,r\right) $
	to the ball $B\left( 0,R\right) $ centered at the origin with radius $%
	R=x_{1}+r$. First we note that $B\left( P,r\right) \subset B\left(
	0,R\right) $ since if $\left( x,y\right) \in B\left( P,r\right) $, then 
	\begin{equation*}
		\func{dist}\left( \left( 0,0\right) ,\left( x,y\right) \right) \leq 
		\func{dist}\left( \left( 0,0\right) ,P\right) +\func{dist}\left(
		P,\left( x,y\right) \right) <x_{1}+r=R.
	\end{equation*}%
	Thus from (\ref{ball-origin}) we have%
	\begin{equation*}
		\left\vert B\left( P,r\right) \right\vert \leq \left\vert B\left( 0,R\right)
		\right\vert \approx \frac{f\left( R\right) }{\left\vert F^{\prime }\left(
			R\right) \right\vert ^{2}}=\frac{f\left( x_{1}+r\right) }{\left\vert
			F^{\prime }\left( x_{1}+r\right) \right\vert ^{2}},
	\end{equation*}%
	By parts (2) and (3) of Proposition \ref{height}, $h\approx \frac{f\left(
		x_{1}+r\right) }{\left\vert F^{\prime }\left( x_{1}+r\right) \right\vert }$
	and $r-r^{\ast }\approx \frac{1}{\left\vert F^{\prime }\left( x_{1}+r\right)
		\right\vert }$ when $r\geq \frac{1}{\left\vert F^{\prime }\left(
		x_{1}\right) \right\vert }$, and so we have%
	\begin{equation*}
		\left\vert B\left( P,r\right) \right\vert \lesssim h\left( x_{1},r\right) \
		\left( r-r^{\ast }\left( x_{1},r\right) \right) .
	\end{equation*}%
	Finally, we claim that%
	\begin{equation*}
		h\left( x_{1},r\right) \ \left( r-r^{\ast }\left( x_{1},r\right) \right)
		\lesssim \left\vert B\left( P,r\right) \right\vert .
	\end{equation*}%
	To see this we consider $x$ satisfying $x_{1}+r^{\ast }\leq x\leq x_{1}+%
	\frac{r+r^{\ast }}{2}$, where $x_{1}+\frac{r+r^{\ast }}{2}$ is the midpoint
	of the interval $\left[ x_{1}+r^{\ast },x_{1}+r\right] $ corresponding to
	the "thick" part of the ball $B\left( P,r\right) $. For such $x$ we let $y>0$
	be defined so that $\left( x,y\right) \in \partial B\left( P,r\right) $.
	Then using the taxicab path $\left( x_{1},0\right) \rightarrow \left(
	x,0\right) \rightarrow \left( x,y\right) $, we see that%
	\begin{equation}
	x-x_{1}+\frac{y}{f\left( x\right) }\geq \func{dist}\left( \left(
	x_{1},0\right) ,\left( x,y\right) \right) =r,  \label{distance}
	\end{equation}%
	implies%
	\begin{equation*}
		y\geq f\left( x\right) \left( r-x+x_{1}\right) \approx f\left(
		x_{1}+r\right) \left( r-r^{\ast }\right) ,
	\end{equation*}%
	where the final approximation follows from $r-r^{\ast }\approx \frac{1}{%
		\left\vert F^{\prime }\left( x_{1}+r\right) \right\vert }$ and part (2) of
	Lemma \ref{consequences} upon using $x_{1}+r^{\ast }\leq x\leq x_{1}+\frac{%
		r+r^{\ast }}{2}$. Thus, using parts (2) and (3) of Proposition \ref{height}
	again, we obtain%
	\begin{eqnarray*}
		\left\vert B\left( P,r\right) \right\vert &\gtrsim &\left[ f\left(
		x_{1}+r\right) \left( r-r^{\ast }\right) \right] \ \left( r-r^{\ast }\right)
		\\
		&\approx &\frac{f\left( x_{1}+r\right) }{\left\vert F^{\prime }\left(
			x_{1}+r\right) \right\vert }\left( r-r^{\ast }\right) \approx h\left(
		x_{1},r\right) \ \left( r-r^{\ast }\left( x_{1},r\right) \right) ,
	\end{eqnarray*}
	
	which proves our claim and concludes the proof that $\left\vert B\left(
	P,r\right) \right\vert \approx \frac{f\left( x_{1}+r\right) }{\left\vert
		F^{\prime }\left( x_{1}+r\right) \right\vert ^{2}}\approx \left\vert
	B_{+}\left( P,r\right) \right\vert $ when $r\geq \frac{1}{\left\vert
		F^{\prime }\left( x_{1}\right) \right\vert }$.
	
	\textbf{Case }$r<\frac{1}{\left\vert F^{\prime }\left( x_{1}\right)
		\right\vert }$. In this case parts (2) and (3) of Proposition \ref{height}
	show that $h\approx rf\left( x_{1}\right) $ and $r-r^{\ast }\approx r$, and
	part (1) shows that $h$ maximizes the `height' of the ball. Thus we
	immediately obtain the upper bound%
	\begin{equation*}
		\left\vert B\left( P,r\right) \right\vert \lesssim hr\lesssim f\left(
		x_{1}\right) r^{2}.
	\end{equation*}%
	To obtain the corresponding lower bound, we use notation as in the first
	case and note that (\ref{distance}) now implies%
	\begin{equation}
	y\geq f\left( x\right) \left( r-x+x_{1}\right) \approx f\left( x_{1}\right)
	r,  \label{low bound}
	\end{equation}%
	where the final approximation follows from part (2) of Proposition \ref%
	{height} and part (2) of Lemma \ref{consequences} upon using $x_{1}+r^{\ast
	}\leq x\leq x_{1}+\frac{r+r^{\ast }}{2}$. Thus%
	\begin{equation*}
		\left\vert B\left( P,r\right) \right\vert \gtrsim \left[ f\left(
		x_{1}\right) r\right] \ \left( r-r^{\ast }\right) \approx f\left(
		x_{1}\right) r^{2},
	\end{equation*}
	
	which concludes the proof that $\left\vert B\left( P,r\right) \right\vert
	\approx f\left( x_{1}\right) r^{2}\approx \left\vert B_{+}\left( P,r\right)
	\right\vert $ when $r<\frac{1}{\left\vert F^{\prime }\left( x_{1}\right)
		\right\vert }$.
\end{proof}

Using Proposition \ref{height} we obtain a useful corollary for the measure
of the \textquotedblleft thick\textquotedblright\ part of a ball. But first
we need to establish that $r^{\ast }\left( x_{1},r\right) $ is increasing in 
$r$ where 
\begin{equation*}
	T\left( x_{1},r\right) \equiv \left( x_{1}+r^{\ast }\left( x_{1},r\right)
	,h\left( x_{1},r\right) \right)
\end{equation*}%
is the turning point for the geodesic $\gamma _{r}$ that passes through $%
P=\left( x_{1},0\right) $ in the upward direction and has vertical slope at
the boundary of the ball $B\left( P,r\right) $.

\begin{lemma}
	Let $x_{1}>0$. Then $r^{\ast }\left( x_{1},r^{\prime }\right) <r^{\ast
	}\left( x_{1},r\right) $ if $0<r^{\prime }<r$.
\end{lemma}

\begin{proof}
	Let $T\left( x_{1},r\right) \equiv \left( x_{1}+r^{\ast }\left(
	x_{1},r\right) ,h\left( x_{1},r\right) \right) $ be the turning point for
	the geodesic $\gamma _{r}$ that passes through $P=\left( x_{1},0\right) $
	and has vertical slope at the boundary of the ball $B\left( P,r\right) $. A
	key property of this geodesic is that it continues beyond the point $T\left(
	x_{1},r\right) $ by vertical reflection. Now we claim that this key property
	implies that when $0<r^{\prime }<r$, the geodesic $\gamma _{r^{\prime }}$
	cannot lie below $\gamma _{r}$ just to the right of $P$. Indeed, if it did,
	then since $B\left( P,r^{\prime }\right) \subset B\left( P,r\right) $
	implies $h\left( x_{1},r^{\prime }\right) <h\left( x_{1},r\right) $, the
	geodesic $\gamma _{r^{\prime }}$ would turn back and intersect $\gamma _{r}$
	in the first quadrant, contradicting the fact that geodesics cannot
	intersect twice in the first quadrant. Thus the geodesic $\gamma _{r^{\prime
	}}$ lies above $\gamma _{r}$ just to the right of $P$, and it is now evident
	that $\gamma _{r^{\prime }}$ must turn back `before' $\gamma _{r}$, i.e.
	that $r^{\ast }\left( x_{1},r^{\prime }\right) <r^{\ast }\left(
	x_{1},r\right) $.
\end{proof}

\begin{corollary}
	\label{thick_part}Let $x_{1}>0$. Denote 
	\begin{eqnarray*}
		B_{+}\left( P,r\right) &\equiv &\left\{ \left( y_{1},y_{2}\right) \in
		B\left( P,r\right) :y_{1}>x_{1}+r^{\ast }\right\} , \\
		B_{-}\left( P,r\right) &\equiv &\left\{ \left( y_{1},y_{2}\right) \in
		B\left( P,r\right) :y_{1}\leq x_{1}+r^{\ast }\right\} .
	\end{eqnarray*}%
	Then 
	\begin{equation*}
		\left\vert B_{+}\left( P,r\right) \right\vert \approx \left\vert B_{-}\left(
		P,r\right) \right\vert \approx \left\vert B\left( P,r\right) \right\vert .
	\end{equation*}
\end{corollary}

\begin{proof}
	\textbf{Case} $r<\frac{1}{\left\vert F^{\prime }\left( x_{1}\right)
		\right\vert }$. Recall from Assumption (4) that $\frac{1}{\left\vert
		F^{\prime }\left( x_{1}\right) \right\vert }\leq \frac{1}{\varepsilon }x_{1}$%
	, so that in this case we have $r<\frac{1}{\varepsilon }x_{1}$, and hence
	also that $x_{1}-\max \left\{ \varepsilon x_{1},x_{1}-\frac{r}{2}\right\}
	\approx r$. From Proposition \ref{general-area} we have%
	\begin{equation*}
		\left\vert B\left( P,r\right) \right\vert \approx r^{2}f\left( x_{1}\right) .
	\end{equation*}%
	From part (2) of Lemma \ref{consequences}, there is a positive constant $c$
	such that $f\left( x\right) \geq cf\left( x_{1}\right) $ for $\max \left\{
	\varepsilon x_{1},x_{1}-\frac{r}{2}\right\} \leq x\leq x_{1}$. It follows
	that $B_{-}\left( P,r\right) \supset \left( \max \left\{ \varepsilon
	x_{1},x_{1}-\frac{r}{2}\right\} ,x_{1}\right) \times \left( -\frac{c}{2}%
	f\left( x_{1}\right) r,\frac{c}{2}f\left( x_{1}\right) r\right) $ since 
	\begin{eqnarray*}
		d\left( \left( x_{1},0\right) ,\left( x,y\right) \right) &\leq &d\left(
		\left( x_{1},0\right) ,\left( x,0\right) \right) +d\left( \left( x,0\right)
		,\left( x,y\right) \right) \\
		&=&\left\vert x_{1}-x\right\vert +\frac{\left\vert y\right\vert }{f\left(
			x\right) }<\frac{r}{2}+\frac{r}{2}=r,
	\end{eqnarray*}%
	provided $\max \left\{ \varepsilon x_{1},x_{1}-\frac{r}{2}\right\} <x<x_{1}$
	and $-\frac{c}{2}f\left( x_{1}\right) r<y<\frac{c}{2}f\left( x_{1}\right) r$%
	. Thus we have%
	\begin{equation*}
		\left\vert B_{-}\left( P,r\right) \right\vert \geq cr^{2}f\left(
		x_{1}\right) .
	\end{equation*}
	
	\textbf{Case} $r\geq \frac{1}{\left\vert F^{\prime }\left( x_{1}\right)
		\right\vert }$. The bound $\left\vert B_{-}\left( P,r\right) \right\vert
	\leq \left\vert B\left( P,r\right) \right\vert \approx \left\vert
	B_{+}\left( P,r\right) \right\vert $ follows from Proposition \ref%
	{general-area}. We now consider two subcases in order to obtain the lower
	bound $\left\vert B_{-}\left( P,r\right) \right\vert \gtrsim \left\vert
	B\left( P,r\right) \right\vert $.
	
	\textbf{Subcase} $r\geq \frac{1}{\left\vert F^{\prime }\left( x_{1}\right)
		\right\vert }\geq r^{\ast }$. By (\ref{upper bound of r star}) and part (2)
	of Lemma \ref{consequences} we have%
	\begin{equation*}
		\left\vert F^{\prime }\left( x_{1}+r\right) \right\vert \approx \left\vert
		F^{\prime }\left( x_{1}+r^{\ast }\right) \right\vert \text{ and }f\left(
		x_{1}+r\right) \approx f\left( x_{1}+r^{\ast }\right) .
	\end{equation*}%
	Then by Proposition \ref{general-area}, followed by the above inequalities,
	and then another application of part (2) of Lemma \ref{consequences}, we
	obtain%
	\begin{equation*}
		\left\vert B\left( P,r\right) \right\vert \approx \frac{f\left(
			x_{1}+r\right) }{\left\vert F^{\prime }\left( x_{1}+r\right) \right\vert ^{2}%
		}\approx \frac{f\left( x_{1}+r^{\ast }\right) }{\left\vert F^{\prime }\left(
			x_{1}+r^{\ast }\right) \right\vert ^{2}}\approx \frac{f\left( x_{1}\right) }{%
			\left\vert F^{\prime }\left( x_{1}\right) \right\vert ^{2}}.
	\end{equation*}%
	On the other hand, with $r_{0}=\frac{1}{\left\vert F^{\prime }\left(
		x_{1}\right) \right\vert }$, we can apply the case already proved above,
	together with the fact that $m\left( x_{1},r\right) $ is increasing in $r$,
	to obtain that%
	\begin{eqnarray*}
		\left\vert B_{-}\left( P,r\right) \right\vert &\geq &\left\vert \left\{
		\left( y_{1},y_{2}\right) \in B\left( P,r_{0}\right) :y_{1}\leq
		x_{1}+r^{\ast }\right\} \right\vert \\
		&\geq &\left\vert \left\{ \left( y_{1},y_{2}\right) \in B\left(
		P,r_{0}\right) :y_{1}\leq x_{1}+\left( r_{0}\right) ^{\ast }\right\}
		\right\vert \approx \frac{f\left( x_{1}\right) }{\left\vert F^{\prime
			}\left( x_{1}\right) \right\vert ^{2}}.
	\end{eqnarray*}
	
	\textbf{Subcase} $r\geq r^{\ast }\geq \frac{1}{\left\vert F^{\prime }\left(
		x_{1}\right) \right\vert }$. Since $B\left( P,r^{\ast }\right) \subset
	B_{-}\left( P,r\right) $ we can apply Proposition \ref{general-area} to $%
	B\left( P,r^{\ast }\right) $ to obtain%
	\begin{equation*}
		\left\vert B_{-}\left( P,r\right) \right\vert \geq \left\vert B\left(
		P,r^{\ast }\right) \right\vert \approx \frac{f\left( x_{1}+r^{\ast }\right) 
		}{\left\vert F^{\prime }\left( x_{1}+r^{\ast }\right) \right\vert ^{2}}.
	\end{equation*}%
	Now we apply (\ref{new f equ}) and Proposition \ref{general-area} again to
	conclude that%
	\begin{equation*}
		\frac{f\left( x_{1}+r^{\ast }\right) }{\left\vert F^{\prime }\left(
			x_{1}+r^{\ast }\right) \right\vert ^{2}}\approx \frac{f\left( x_{1}+r\right) 
		}{\left\vert F^{\prime }\left( x_{1}+r\right) \right\vert ^{2}}\approx
		\left\vert B\left( P,r\right) \right\vert .
	\end{equation*}
\end{proof}

\section{Higher dimensional geometries}

First we consider the $3$-dimensional case.

\subsection{Geodesics and metric balls}

Let $\gamma (t)=(x_{1}(t),x_{2}(t),x_{3}(t))$ be a path. Then the arc length
element is given by 
\begin{equation*}
	ds=\sqrt{[x_{1}^{\prime }(t)]^{2}+[x_{2}^{\prime }(t)]^{2}+\frac{1}{%
			[f(x_{1})]^{2}}[x_{3}^{\prime }(t)]^{2}}dt.
\end{equation*}%
Thus we can factor the associated control space by 
\begin{equation*}
	\left( {\mathbb{R}}^{3},%
	\begin{bmatrix}
		1 & 0 & 0 \\ 
		0 & 1 & 0 \\ 
		0 & 0 & [f(x_{1})]^{-2}%
	\end{bmatrix}%
	\right) =\left( {\mathbb{R}}_{x_{1},x_{3}}^{2},%
	\begin{bmatrix}
		1 & 0 \\ 
		0 & [f(x_{1})]^{-2}%
	\end{bmatrix}%
	\right) \times {\mathbb{R}}_{x_{2}}\ .
\end{equation*}%
We begin with a lemma regarding paths in product spaces.

\begin{lemma}
	Let $(M_{1},g^{M_{1}})$ and $(M_{2},g^{M_{2}})$ be two Riemannian manifolds.
	Consider the Cartesian product $M_{1}\times M_{2}$ whose Riemannian metric
	is defined by 
	\begin{equation*}
		g_{(p,q)}((u_{1},u_{2}),(v_{1},v_{2}))=g_{p}^{M_{1}}(u_{1},v_{1})+g_{q}^{M_{2}}(u_{2},v_{2}).
	\end{equation*}%
	Here we have%
	\begin{equation*}
		\left( p,q\right) \in M_{1}\times M_{2}\text{ and }\left( u_{1},u_{2}\right)
		,\left( v_{1},v_{2}\right) \in T_{p}\left( M_{1}\right) \oplus T_{p}\left(
		M_{2}\right) \approx T_{\left( p,q\right) }\left( M_{1}\times M_{2}\right) .
	\end{equation*}%
	Given any $C^{1}$ path $\gamma :[a,b]\rightarrow M_{1}\times M_{2}$, we can
	write it in the form $(\gamma _{1}(t),\gamma _{2}(t))$, where $\gamma
	_{1}:[a,b]\rightarrow M_{1}$ and $\gamma _{2}:[a,b]\rightarrow M_{2}$ are $%
	C^{1}$ paths on $M_{1}$ and $M_{2}$, respectively. Then we have 
	\begin{equation*}
		\Vert \gamma \Vert \geq \sqrt{\Vert \gamma _{1}\Vert ^{2}+\Vert \gamma
			_{2}\Vert ^{2}}
	\end{equation*}%
	where $\Vert \gamma \Vert $, $\Vert \gamma _{1}\Vert $ and $\Vert \gamma
	_{2}\Vert $ represent the arc length of each path. In addition, equality
	occurs if and only if 
	\begin{equation}
	\frac{\Vert \gamma _{1}^{\prime }(t)\Vert _{g^{M_{1}}}}{\Vert \gamma
		_{1}\Vert }=\frac{\Vert \gamma _{2}^{\prime }(t)\Vert _{g^{M_{2}}}}{\Vert
		\gamma _{2}\Vert },\ \ \ \ \ a\leq t\leq b.  \label{proportional1}
	\end{equation}
\end{lemma}

\begin{proof}
	For simplicity we omit the subscripts of the norms $\Vert \gamma
	_{1}^{\prime }(t)\Vert _{g^{M_{1}}}$ and $\Vert \gamma _{2}^{\prime
	}(t)\Vert _{g^{M_{2}}}$ so that $\left\Vert \gamma _{j}\right\Vert
	=\int_{a}^{b}\sqrt{\left\Vert \gamma _{j}^{\prime }\left( t\right)
		\right\Vert ^{2}}dt$. Using that%
	\begin{eqnarray*}
		&&\frac{\left\Vert \gamma _{1}\right\Vert }{\sqrt{\left\Vert \gamma
				_{1}\right\Vert ^{2}+\left\Vert \gamma _{2}\right\Vert ^{2}}}\left\Vert
		\gamma _{1}^{\prime }\left( t\right) \right\Vert +\frac{\left\Vert \gamma
			_{2}\right\Vert }{\sqrt{\left\Vert \gamma _{1}\right\Vert ^{2}+\left\Vert
				\gamma _{2}\right\Vert ^{2}}}\left\Vert \gamma _{2}^{\prime }\left( t\right)
		\right\Vert \\
		&&\ \ \ \ \ \ \ \ \ \ \ \ \ \ \ \ \ \ \ \ \ \ \ \ \ \ \ \ \ \ \leq \sqrt{%
			\left\Vert \gamma _{1}^{\prime }\left( t\right) \right\Vert ^{2}+\left\Vert
			\gamma _{2}^{\prime }\left( t\right) \right\Vert ^{2}},
	\end{eqnarray*}%
	with equality if and only if 
	\begin{equation*}
		\left( 
		\begin{array}{c}
			\left\Vert \gamma _{1}^{\prime }\left( t\right) \right\Vert \\ 
			\left\Vert \gamma _{2}^{\prime }\left( t\right) \right\Vert%
		\end{array}%
		\right) \text{ is parallel to }\left( 
		\begin{array}{c}
			\left\Vert \gamma _{1}\right\Vert \\ 
			\left\Vert \gamma _{2}\right\Vert%
		\end{array}%
		\right) ,
	\end{equation*}%
	we obtain that 
	\begin{eqnarray*}
		\left\Vert \gamma \right\Vert &=&\int_{a}^{b}\sqrt{\left\Vert \gamma
			_{1}^{\prime }\left( t\right) \right\Vert ^{2}+\left\Vert \gamma
			_{2}^{\prime }\left( t\right) \right\Vert ^{2}}dt \\
		&\geq &\int_{a}^{b}\left( \frac{\left\Vert \gamma _{1}\right\Vert }{\sqrt{%
				\left\Vert \gamma _{1}\right\Vert ^{2}+\left\Vert \gamma _{2}\right\Vert ^{2}%
		}}\left\Vert \gamma _{1}^{\prime }\left( t\right) \right\Vert +\frac{%
			\left\Vert \gamma _{2}\right\Vert }{\sqrt{\left\Vert \gamma _{1}\right\Vert
				^{2}+\left\Vert \gamma _{2}\right\Vert ^{2}}}\left\Vert \gamma _{2}^{\prime
		}\left( t\right) \right\Vert \right) dt \\
		&=&\frac{\left\Vert \gamma _{1}\right\Vert ^{2}}{\sqrt{\left\Vert \gamma
				_{1}\right\Vert ^{2}+\left\Vert \gamma _{2}\right\Vert ^{2}}}+\frac{%
			\left\Vert \gamma _{2}\right\Vert ^{2}}{\sqrt{\left\Vert \gamma
				_{1}\right\Vert ^{2}+\left\Vert \gamma _{2}\right\Vert ^{2}}}=\sqrt{%
			\left\Vert \gamma _{1}\right\Vert ^{2}+\left\Vert \gamma _{2}\right\Vert ^{2}%
		},
	\end{eqnarray*}%
	with equality if and only if (\ref{proportional1}) holds.
\end{proof}

\begin{corollary}
	A $C^{1}$ path $\gamma =(\gamma _{1},\gamma _{2})$ is a geodesic of $%
	M_{1}\times M_{2}$ if and only if
	
	\begin{enumerate}
		\item $\gamma _{1}$ is a geodesic of $M_{1}$,
		
		\item $\gamma _{2}$ is a geodesic of $M_{2}$,
		
		\item and the speeds of $\gamma _{1}$ and $\gamma _{2}$ match, i.e. the
		identity $\frac{\Vert \gamma _{1}^{\prime }(t)\Vert _{g^{M_{1}}}}{\Vert
			\gamma _{1}\Vert }=\frac{\Vert \gamma _{2}^{\prime }(t)\Vert _{g^{M_{2}}}}{%
			\Vert \gamma _{2}\Vert }$ holds for all $t$.
	\end{enumerate}
\end{corollary}

\begin{corollary}
	The distance between two points $(p_{1},q_{1}),(p_{2},q_{2})\in M_{1}\times
	M_{2}$ is given by 
	\begin{equation*}
		d_{g}((p_{1},q_{1}),(p_{2},q_{2}))=\sqrt{\left[ d_{g^{M_{1}}}(p_{1},p_{2})%
			\right] ^{2}+\left[ d_{g^{M_{2}}}(q_{1},q_{2})\right] ^{2}}.
	\end{equation*}
\end{corollary}

Thus we can write a typical geodesic in the form 
\begin{equation*}
	\left\{ 
	\begin{array}{l}
		x_{2}=C_{2}\pm k\int_{0}^{x_{1}}\frac{\lambda }{\sqrt{\lambda ^{2}-[f(u)]^{2}%
		}}\,du \\ 
		x_{3}=C_{3}\pm \int_{0}^{x_{1}}\frac{[f(u)]^{2}}{\sqrt{\lambda
				^{2}-[f(u)]^{2}}}\,du%
	\end{array}%
	\right. ,
\end{equation*}%
and a metric ball centered at $y=\left( y_{1},y_{2},y_{3}\right) $ with
radius $r>0$ is given by 
\begin{equation*}
	B\left( y,r\right) \equiv \left\{ \left( x_{1},x_{2},x_{3}\right) :\left(
	x_{1},x_{3}\right) \in B_{2D}\left( \left( y_{1},y_{3}\right) ,\sqrt{%
		r^{2}-\left\vert x_{2}-y_{2}\right\vert ^{2}}\right) \right\} ,
\end{equation*}%
where $B_{2D}\left( a,s\right) $ denotes the $2$-dimensional control ball
centered at $a$ in the plane with radius $s$ that was associated with $f$
above.

\subsection{Volumes of $n$-dimensional balls}

Recall that the Lebesgue measure of the \emph{two} dimensional ball $%
B_{2D}\left( x,r\right) $ satisfies%
\begin{equation*}
	\left\vert B_{2D}\left( x,r\right) \right\vert \approx \left\{ 
	\begin{array}{ccc}
		r^{2}f(x_{1}) & \text{ if } & r\leq \frac{1}{\left\vert F^{\prime }\left(
			x_{1}\right) \right\vert } \\ 
		\frac{f\left( x_{1}+r\right) }{\left\vert F^{\prime }\left( x_{1}+r\right)
			\right\vert ^{2}} & \text{ if } & r\geq \frac{1}{\left\vert F^{\prime
			}\left( x_{1}\right) \right\vert }%
	\end{array}%
	\right. .
\end{equation*}%
Recall also that in the two dimensional case, we had%
\begin{equation*}
	\left\vert B_{2D}\left( x,d\left( x,y\right) \right) \right\vert \approx
	h_{x,y}\widehat{d}\left( x,y\right) \approx h_{x,y}\min \left\{ d\left(
	x,y\right) ,\frac{1}{\left\vert F^{\prime }\left( x_{1}+d\left( x,y\right)
		\right) \right\vert }\right\} .
\end{equation*}%
In the three dimensional case, the quantities $h_{x,y}$ and $\widehat{d}%
\left( x,y\right) $ remain formally the same and as was done above, we can
write a typical geodesic in the form 
\begin{equation*}
	\left\{ 
	\begin{array}{l}
		x_{2}=C_{2}\pm k\int_{0}^{x_{1}}\frac{\lambda }{\sqrt{\lambda ^{2}-[f(u)]^{2}%
		}}\,du \\ 
		x_{3}=C_{3}\pm \int_{0}^{x_{1}}\frac{[f(u)]^{2}}{\sqrt{\lambda
				^{2}-[f(u)]^{2}}}\,du%
	\end{array}%
	\right. ,
\end{equation*}%
so that a metric ball centered at $y=\left( y_{1},y_{2},y_{3}\right) $ with
radius $r>0$ is given by 
\begin{equation*}
	B\left( y,r\right) \equiv \left\{ \left( x_{1},x_{2},x_{3}\right) :\left(
	x_{1},x_{3}\right) \in B_{2D}\left( \left( y_{1},y_{3}\right) ,\sqrt{%
		r^{2}-\left\vert x_{2}-y_{2}\right\vert ^{2}}\right) \right\} ,
\end{equation*}%
where $B_{2D}\left( a,s\right) $ denotes the $2$-dimensional control ball
centered at $a$ in the plane parallel to the $x_{1},x_{3}$-plane with radius 
$s$ that was associated with $f$ above.

In dimension $n\geq 4$, the same arguments show that a typical geodesic has
the form 
\begin{equation*}
	\left\{ 
	\begin{array}{l}
		\mathbf{x}_{2}=\mathbf{C}_{2}\pm \mathbf{k}\int_{0}^{x_{1}}\frac{\lambda }{%
			\sqrt{\lambda ^{2}-[f(u)]^{2}}}\,du \\ 
		x_{3}=C_{3}\pm \int_{0}^{x_{1}}\frac{[f(u)]^{2}}{\sqrt{\lambda
				^{2}-[f(u)]^{2}}}\,du%
	\end{array}%
	\right. ,
\end{equation*}%
where $\mathbf{x}_{2},\mathbf{C}_{2},\mathbf{k}\in \mathbb{R}^{n-2}$ are now 
$\left( n-2\right) $-dimensional vectors, so that a metric ball centered at 
\begin{equation*}
	y=\left( y_{1},\mathbf{y}_{2},y_{3}\right) \in \mathbb{R}\times \mathbb{R}%
	^{n-2}\times \mathbb{R}=\mathbb{R}^{n},
\end{equation*}%
with radius $r>0$ is given by 
\begin{equation*}
	B\left( y,r\right) \equiv \left\{ \left( x_{1},\mathbf{x}_{2},x_{3}\right)
	:\left( x_{1},x_{3}\right) \in B_{2D}\left( \left( y_{1},y_{3}\right) ,\sqrt{%
		r^{2}-\left\vert \mathbf{x}_{2}-\mathbf{y}_{2}\right\vert ^{2}}\right)
	\right\} ,
\end{equation*}%
where $B_{2D}\left( a,s\right) $ denotes the $2$-dimensional control ball
centered at $a$ in the plane parallel to the $x_{1},x_{3}$-plane with radius 
$s$ that was associated with $f$ above.

\begin{lemma}
	\label{new}The Lebesgue measure of the \emph{three} dimensional ball $%
	B_{3D}\left( x,r\right) $ satisfies%
	\begin{equation*}
		\left\vert B_{3D}\left( x,r\right) \right\vert \approx \left\{ 
		\begin{array}{ccc}
			r^{3}f(x_{1}) & \text{ if } & r\leq \frac{2}{\left\vert F^{\prime }\left(
				x_{1}\right) \right\vert } \\ 
			\frac{f\left( x_{1}+r\right) }{\left\vert F^{\prime }\left( x_{1}+r\right)
				\right\vert ^{3}}\sqrt{r\left\vert F^{\prime }\left( x_{1}+r\right)
				\right\vert } & \text{ if } & r\geq \frac{2}{\left\vert F^{\prime }\left(
				x_{1}\right) \right\vert }%
		\end{array}%
		\right. ,
	\end{equation*}%
	and that of the $n$-dimensional ball $B_{nD}\left( x,r\right) $ satisfies%
	\begin{equation*}
		\left\vert B_{nD}\left( x,r\right) \right\vert \approx \left\{ 
		\begin{array}{ccc}
			r^{n}f(x_{1}) & \text{ if } & r\leq \frac{2}{\left\vert F^{\prime }\left(
				x_{1}\right) \right\vert } \\ 
			\frac{f\left( x_{1}+r\right) }{\left\vert F^{\prime }\left( x_{1}+r\right)
				\right\vert ^{n}}\left( r\left\vert F^{\prime }\left( x_{1}+r\right)
			\right\vert \right) ^{\frac{n}{2}-1} & \text{ if } & r\geq \frac{2}{%
				\left\vert F^{\prime }\left( x_{1}\right) \right\vert }%
		\end{array}%
		\right. ,
	\end{equation*}
\end{lemma}

\begin{proof}
	We estimate the measure $\left\vert B\left( x,r\right) \right\vert $ of an $%
	\emph{n}$-dimensional ball $B\left( x,r\right) =B_{nD}\left( x,r\right) $,
	where $x=\left( x_{1},\mathbf{x}_{2},x_{3}\right) \in \mathbb{R}\times 
	\mathbb{R}^{n-2}\times \mathbb{R}=\mathbb{R}^{n}$, and where we use boldface
	font for $\mathbf{x}_{2}$ to emphasize that it belongs to $\mathbb{R}^{n-2}$
	as opposed to $\mathbb{R}$. We consider two cases, where we may assume by
	symmetry that $\mathbf{x}_{2}=\mathbf{0}$ and $x_{3}=0$.
	
	\textbf{Case }$r<\frac{2}{\left\vert F^{\prime }\left( x_{1}\right)
		\right\vert }$: In this case we have $\sqrt{r^{2}-\left\vert \mathbf{y}%
		_{2}\right\vert ^{2}}\leq r<\frac{2}{\left\vert F^{\prime }\left(
		x_{1}\right) \right\vert }$ and 
	\begin{equation*}
		\left\vert B_{2D}\left( \left( x_{1},\mathbf{0},0\right) ,\sqrt{%
			r^{2}-\left\vert \mathbf{y}_{2}\right\vert ^{2}}\right) \right\vert \approx
		(r^{2}-\left\vert \mathbf{y}_{2}\right\vert ^{2})f(x_{1}),
	\end{equation*}%
	where for $r<\frac{1}{\left\vert F^{\prime }\left( x_{1}\right) \right\vert }
	$ we appeal to the second assertion in Proposition \ref{general-area}, while
	for $\frac{1}{\left\vert F^{\prime }\left( x_{1}\right) \right\vert }\leq r<%
	\frac{2}{\left\vert F^{\prime }\left( x_{1}\right) \right\vert }$ we appeal
	to the first assertion in Proposition \ref{general-area} and use the
	estimates for $f$ and $\left\vert F^{\prime }\right\vert $ in (2) of Lemma %
	\ref{consequences}. With $\mathrm{A}\left( a,b\right) \equiv \left\{ \mathbf{%
		y}_{2}\in \mathbb{R}^{n-2}:a\leq \left\vert \mathbf{y}_{2}\right\vert \leq
	b\right\} $ denoting the annulus centered at the origin in $\mathbb{R}^{n-2}$
	with radii $a<b$, the above gives 
	\begin{align*}
		\left\vert B\left( x,r\right) \right\vert =\int\limits_{\mathrm{A}\left(
			0,r\right) }\left\vert B_{2D}\left( \left( x_{1},\mathbf{0},0\right) ,\sqrt{%
			r^{2}-\left\vert \mathbf{y}_{2}\right\vert ^{2}}\right) \right\vert d\mathbf{%
			y}_{2}&\approx \int\limits_{\mathrm{A}\left( 0,r\right) }(r^{2}-\left\vert 
		\mathbf{y}_{2}\right\vert ^{2})f(x_{1})d\mathbf{y}_{2}\\
		&\approx r^{n}f(x_{1}).
	\end{align*}%
	\textbf{Case }$r\geq \frac{2}{\left\vert F^{\prime }\left( x_{1}\right)
		\right\vert }$: Again we have%
	\begin{equation*}
		\left\vert B\left( x,r\right) \right\vert =\int\limits_{B\left( 0,r\right)
		}\left\vert B_{2D}\left( \left( x_{1},\mathbf{0},0\right) ,\sqrt{%
			r^{2}-\left\vert \mathbf{y}_{2}\right\vert ^{2}}\right) \right\vert d\mathbf{%
			y}_{2}\ .
	\end{equation*}%
	Since the measure of the ball $B_{2D}\left( \left( x_{1},\mathbf{0},0\right)
	,R\right) $ is nondecreasing as a function of the radius $R$, we have%
	\begin{align*}
		\left\vert B\left( x,r\right) \right\vert &\approx \int\limits_{B\left( 0,%
			\frac{r}{2}\right) }\left\vert B_{2D}\left( \left( x_{1},\mathbf{0},0\right)
		,\sqrt{r^{2}-\left\vert \mathbf{y}_{2}\right\vert ^{2}}\right) \right\vert d%
		\mathbf{y}_{2}\\
		&=\int\limits_{B\left( 0,\frac{r}{2}\right) }\frac{f\left(
			x_{1}+\sqrt{r^{2}-\left\vert \mathbf{y}_{2}\right\vert ^{2}}\right) }{%
			\left\vert F^{\prime }\left( x_{1}+\sqrt{r^{2}-\left\vert \mathbf{y}%
				_{2}\right\vert ^{2}}\right) \right\vert ^{2}}d\mathbf{y}_{2}\ .
	\end{align*}%
	Using polar coordinates and our assumptions on $F^{\prime }$ we continue with%
	\begin{align*}
		\left\vert B\left( x,r\right) \right\vert &\approx \int\limits_{0}^{\frac{r}{2%
		}}\frac{f\left( x_{1}+\sqrt{r^{2}-R^{2}}\right) }{\left\vert F^{\prime
			}\left( x_{1}+\sqrt{r^{2}-R^{2}}\right) \right\vert ^{2}}R^{n-3}dR\\
		&\approx 
		\frac{1}{\left\vert F^{\prime }\left( x_{1}+r\right) \right\vert ^{2}}%
		\int\limits_{0}^{\frac{r}{2}}f\left( x_{1}+\sqrt{r^{2}-R^{2}}\right)
		R^{n-3}dR.
	\end{align*}%
	Now use the change of variable $w=r-\sqrt{r^{2}-R^{2}}\in \left( 0,\frac{2-%
		\sqrt{3}}{2}r\right) $ to write the last integral as%
	\begin{equation*}
		\int\limits_{0}^{\frac{r}{2}}f\left( x_{1}+\sqrt{r^{2}-R^{2}}\right)
		R^{n-3}dR\approx \int\limits_{0}^{\frac{2-\sqrt{3}}{2}r}f\left(
		x_{1}+r-w\right) \left( rw\right) ^{\frac{n}{2}-2}rdw,
	\end{equation*}%
	so that we obtain%
	\begin{equation*}
		\left\vert B\left( x,r\right) \right\vert \approx \frac{r^{\frac{n}{2}-1}}{%
			\left\vert F^{\prime }\left( x_{1}+r\right) \right\vert ^{2}}%
		\int\limits_{0}^{\frac{2-\sqrt{3}}{2}r}f\left( x_{1}+r-w\right) w^{\frac{n}{2%
			}-2}dw.
	\end{equation*}%
	Now we observe that the upper limit of the integral above satisfies%
	\begin{equation*}
		\frac{2-\sqrt{3}}{2}r\gtrsim \frac{1}{\left\vert F^{\prime }\left(
			x_{1}+r\right) \right\vert }.
	\end{equation*}%
	Indeed, if $r<x_{1}$, then our assumption on $r$ gives $\frac{2-\sqrt{3}}{2}%
	r\gtrsim \frac{1}{\left\vert F^{\prime }\left( x_{1}\right) \right\vert }%
	\approx \frac{1}{\left\vert F^{\prime }\left( x_{1}+r\right) \right\vert }$,
	while if $r\geq x_{1}$, then our assumption on $F^{\prime }$ gives $\frac{2-%
		\sqrt{3}}{2}r\gtrsim x_{1}+r\gtrsim \frac{1}{\left\vert F^{\prime }\left(
		x_{1}+r\right) \right\vert }$. Lemma \ref{new} now follows immediately from
	this equivalence:
	
	\qquad For $\beta >-1$, $0<\varepsilon <1$ and $\frac{\varepsilon }{%
		\left\vert F^{\prime }\left( z_{1}\right) \right\vert }<r<z_{1}$, we have%
	\begin{equation}
	\int\limits_{0}^{r}f\left( z_{1}-w\right) w^{\beta }dw\approx \frac{f\left(
		z_{1}\right) }{\left\vert F^{\prime }\left( z_{1}\right) \right\vert ^{\beta
			+1}}.  \label{this equiv}
	\end{equation}
	
	To see (\ref{this equiv}), we note that on the one hand,%
	\begin{equation*}
		\int\limits_{0}^{r}f\left( z_{1}-w\right) w^{\beta }dw\geq \int\limits_{0}^{%
			\frac{\varepsilon }{\left\vert F^{\prime }\left( z_{1}\right) \right\vert }%
		}f\left( z_{1}-w\right) w^{\beta }dw\approx \frac{f\left( z_{1}\right) }{%
			\left\vert F^{\prime }\left( z_{1}\right) \right\vert ^{\beta +1}}.
	\end{equation*}%
	On the other hand,%
	\begin{eqnarray*}
		\ln \frac{f\left( z_{1}-w\right) }{f\left( z_{1}\right) } &=&\int%
		\limits_{z_{1}-w}^{z_{1}}F^{\prime }\left( t\right) dt\leq -\left\vert
		F^{\prime }\left( z_{1}\right) \right\vert w \\
		&\Longrightarrow &f\left( z_{1}-w\right) \leq f\left( z_{1}\right)
		e^{-\left\vert F^{\prime }\left( z_{1}\right) \right\vert w},
	\end{eqnarray*}%
	which gives%
	\begin{equation*}
		\int\limits_{0}^{z_{1}}f\left( z_{1}-w\right) w^{\beta }dw\lesssim \frac{%
			f\left( z_{1}\right) }{\left\vert F^{\prime }\left( z_{1}\right) \right\vert
			^{\beta +1}}.
	\end{equation*}
\end{proof}

\chapter{Orlicz Norm Sobolev Inequalities}

In this second chapter of Part 3, we prove Orlicz Sobolev inequalities for infinitely degenerate geometries.
The key to these inequalities is a subrepresentation formula for a Lipschitz
function $w$ in terms of its control gradient that vanishes at the `end' of
a ball. The kernel $K\left( x,y\right) $ of this subrepresentation in the
infinitely degenerate setting is in general much smaller that the familiar $%
\frac{\limfunc{distance}}{\limfunc{volume}}=\frac{d\left( x,y\right) }{%
	\left\vert B\left( x,d\left( x,y\right) \right) \right\vert }$ kernel that
arises in the finite type case - see Remark \ref{dhat} below for more on
this. With this we then establish Orlicz Sobolev bump inequalities and the
more familiar $1-1$ Poincar\'{e} inequality.

\section{Subrepresentation inequalities}

We first consider the two dimensional case, and then generalize to higher
dimensions in the subsequent subsection.

\subsection{The $2$-dimensional case}

We will obtain a$\ $subrepresentation formula for the degenerate geometry by
applying the method of Lemma 79 in \cite{SaWh4}. For simplicity, we will
only consider $x$ with $x_{1}>0$; since our metric is symmetric about the $y$
axis it suffices to consider this case. For the general case, all objects
defined on the right half plane must be defined on the left half plane by
reflection about the $y$-axis.

Consider a sequence of control balls $\left\{ B\left( x,r_{k}\right)
\right\} _{k=1}^{\infty }$ centered at $x$ with radii $r_{k}\searrow 0$ such
that $r_{0}=r$ and%
\begin{equation*}
	\left\vert B\left( x,r_{k}\right) \setminus B\left( x,r_{k+1}\right)
	\right\vert \approx \left\vert B\left( x,r_{k+1}\right) \right\vert ,\ \ \ \
	\ k\geq 1,
\end{equation*}%
so that $B\left( x,r_{k}\right) $ is divided into two parts having
comparable area. We may in fact assume that%
\begin{equation}
r_{k+1}=\left\{ 
\begin{array}{lll}
r^{\ast }\left( x_{1},r_{k}\right) &  & \text{if }r_{k}\geq \frac{1}{%
	\left\vert F^{\prime }\left( x_{1}\right) \right\vert } \\ 
\frac{1}{2}r_{k} &  & \text{if }r_{k}<\frac{1}{\left\vert F^{\prime }\left(
	x_{1}\right) \right\vert }%
\end{array}%
\right.  \label{rkp1}
\end{equation}%
where $r^{\ast }$ is defined in Proposition \ref{height}. Indeed, if $%
r_{k}\geq \frac{1}{\left\vert F^{\prime }\left( x_{1}\right) \right\vert }$,
then by (1) in Proposition \ref{height} we have that%
\begin{equation*}
	r_{k}-r_{k+1}\approx \frac{1}{\left\vert F^{\prime }\left(
		x_{1}+r_{k}\right) \right\vert }
\end{equation*}%
and then by (2) in Lemma \ref{consequences} it follows that $f\left(
x_{1}+r_{k}\right) \approx f\left( x_{1}+r_{k+1}\right) $ and $\left\vert
F^{\prime }\left( x_{1}+r_{k}\right) \right\vert \approx \left\vert
F^{\prime }\left( x_{1}+r_{k+1}\right) \right\vert $, so by Corollary \ref%
{thick_part} and (1) in Proposition \ref{height} it follows that 
\begin{eqnarray*}
	\left\vert B\left( x,r_{k}\right) \right\vert &\approx &\left\vert \left\{
	B\left( x,r_{k}\right) \bigcap y_{1}>x_{1}+r_{k+1}\right\} \right\vert
	\approx \left( r_{k}-r_{k+1}\right) h\left( x_{1},x_{1}+r_{k}\right) \\
	&\approx &\frac{1}{\left\vert F^{\prime }\left( x_{1}+r_{k}\right)
		\right\vert }\frac{f\left( x_{1}+r_{k}\right) }{\left\vert F^{\prime }\left(
		x_{1}+r_{k}\right) \right\vert }\approx \frac{1}{\left\vert F^{\prime
		}\left( x_{1}+r_{k+1}\right) \right\vert }\frac{f\left( x_{1}+r_{k+1}\right) 
	}{\left\vert F^{\prime }\left( x_{1}+r_{k+1}\right) \right\vert } \\
	&\approx &\left( r_{k+1}-r_{k+2}\right) h\left( x_{1},x_{1}+r_{k+1}\right)\\
	&\approx& \left\vert \left\{ B\left( x,r_{k+1}\right) \bigcap
	y_{1}>x_{1}+r_{k+2}\right\} \right\vert 
	\approx \left\vert B\left( x,r_{k+1}\right) \right\vert .
\end{eqnarray*}%
On the other hand, if $r_{k}\leq \frac{1}{\left\vert F^{\prime }\left(
	x_{1}\right) \right\vert }$ then by (2) in Proposition \ref{height} $%
r_{k}-r_{k+1}\approx r_{k}$ and $h\left( x_{1},x_{1}+r_{k}\right) \approx
r_{k}f\left( x_{1}\right) $, hence by Corollary \ref{thick_part} 
\begin{equation*}
	\left\vert B\left( x,r_{k}\right) \right\vert \approx \left(
	r_{k}-r_{k+1}\right) h\left( x_{1},x_{1}+r_{k+1}\right) \approx
	r_{k}^{2}f\left( x_{1}\right) \approx r_{k+1}^{2}f\left( x_{1}\right)
	\approx \left\vert B\left( x,r_{k+1}\right) \right\vert .
\end{equation*}%
As a consequence we also have that%
\begin{align*}
	\left( r_{k}-r_{k+1}\right) h\left( x_{1},x_{1}+r_{k}\right) &\approx \left(
	r_{k+1}-r_{k+2}\right) h\left( x_{1},x_{1}+r_{k+1}\right)\\ 
	&\lesssim \left(
	r_{k+1}-r_{k+2}\right) h\left( x_{1},x_{1}+r_{k}\right)
\end{align*}%
so $r_{k}-r_{k+1}\leq C\left( r_{k+1}-r_{k+2}\right) \leq Cr_{k+1}$, which
yields%
\begin{equation}
\frac{1}{C+1}r_{k}\leq r_{k+1}.  \label{rkp12}
\end{equation}

Now for $x_{1},t>0$ define 
\begin{equation}
h^{\ast }\left( x_{1},t\right) =\int_{x_{1}}^{x_{1}+t}\frac{f^{2}\left(
	u\right) }{\sqrt{f^{2}\left( x_{1}+t\right) -f^{2}\left( u\right) }}du,
\label{def h*}
\end{equation}%
so that $h^{\ast }\left( x_{1},t\right) $ describes the `height' above $%
x_{2} $ at which the geodesic through $x=\left( x_{1},x_{2}\right) $ curls
back toward the $y$-axis at the point $\left( x_{1}+t,x_{2}+h^{\ast }\left(
x_{1},t\right) \right) $. Thus the graph of $y=h^{\ast }\left(
x_{1},t\right) $ is the curve separating the analogues of Region 1 and
Region 2 relative to the ball $B\left( x,r\right) $. See Figure \ref{graph-h}.

\begin{figure}%[ht]
	\includegraphics{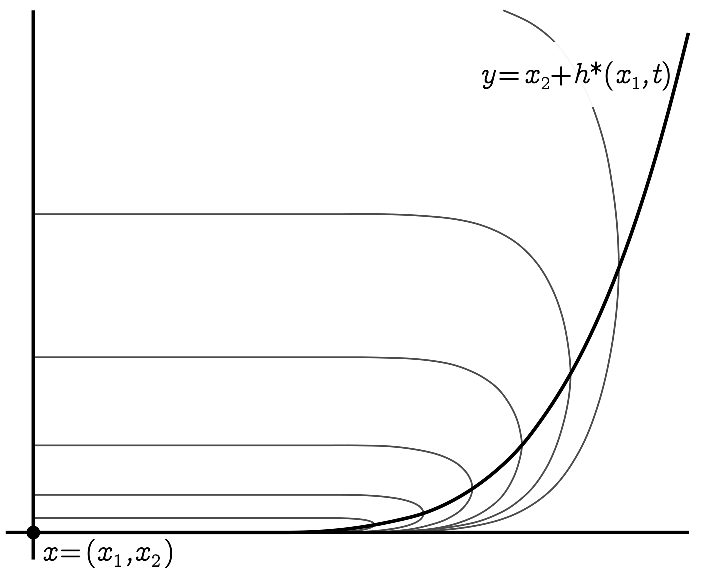}
	\caption{The graph of $h^{\ast }$ for $x_{1}>0$ small.}
	\label{graph-h}
\end{figure}

Then in the case $r_{k}\geq \frac{1}{\left\vert F^{\prime }\left(
	x_{1}\right) \right\vert }$, we have $h^{\ast }\left( x_{1},r_{k+1}\right)
=h\left( x_{1},r_{k}\right) $, $k\geq 0$, where $h\left( x_{1},r_{k}\right) $
is the height of $B\left( x,r_{k}\right) $ as defined in Proposition \ref%
{height}. In the opposite case $r_{k}<\frac{1}{\left\vert F^{\prime }\left(
	x_{1}\right) \right\vert }$, we have $r_{k+1}=\frac{1}{2}r_{k}$ instead, and
we will estimate differently.

For $k\geq 0$ define
\begin{align*}
	&E\left( x,r_{k}\right) \\
	&\equiv \left\{ 
	\begin{array}{l}
		\left\{ y:x_{1}+r_{k+1}\leq y_{1}<x_{1}+r_{k},~\left\vert y_{2}\right\vert
		<h^{\ast }\left( x_{1},y_{1}-x_{1}\right) \right\},  \text{ if }  r_{k}\geq 
		\frac{1}{\left\vert F^{\prime }\left( x_{1}\right) \right\vert } \\ 
		\left\{ y:x_{1}+r_{k+1}\leq y_{1}<x_{1}+r_{k},~\left\vert y_{2}\right\vert
		<h^{\ast }\left( x_{1},r_{k}^{\ast }\right) =h\left( x_{1},r_{k}\right)
		\right\},  \text{ if }  r_{k}<\frac{1}{\left\vert F^{\prime }\left(
			x_{1}\right) \right\vert }%
	\end{array}%
	\right. 
\end{align*}%
where we have written $r_{k}^{\ast }=r^{\ast }\left( x_{1},r_{k}\right) $
for convenience. We claim that 
\begin{equation}
\left\vert E\left( x,r_{k}\right) \right\vert \approx \left\vert E\left(
x,r_{k}\right) \bigcap B\left( x,r_{k}\right) \right\vert \approx \left\vert
B\left( x,r_{k}\right) \right\vert \text{ for all }k\geq 1.
\label{claim that}
\end{equation}%
Indeed, in the first case $r_{k}\geq \frac{1}{\left\vert F^{\prime }\left(
	x_{1}\right) \right\vert }$, the second set of inequalities follows
immediately by Corollary \ref{thick_part}, and since $E\left( x,r_{k}\right)
\subset B\left( x,r_{k-1}\right) $ we have that%
\begin{eqnarray*}
	\left\vert E\left( x,r_{k}\right) \bigcap B\left( x,r_{k}\right) \right\vert
	&\leq &\left\vert E\left( x,r_{k}\right) \right\vert \leq \left\vert B\left(
	x,r_{k-1}\right) \right\vert \\
	&\lesssim &\left\vert B\left( x,r_{k}\right) \right\vert \lesssim \left\vert
	E\left( x,r_{k}\right) \bigcap B\left( x,r_{k}\right) \right\vert ,
\end{eqnarray*}%
which establishes the first set of inequalities in (\ref{claim that}). In
the second case $r_{k}<\frac{1}{\left\vert F^{\prime }\left( x_{1}\right)
	\right\vert }$, we have 
\begin{equation*}
	\left\vert E\left( x,r_{k}\right) \right\vert =\frac{1}{2}r_{k}h^{\ast
	}\left( x_{1},r_{k}^{\ast }\right) \approx \left( r_{k}-r_{k}^{\ast }\right)
	h\left( x_{1},r_{k}^{\ast }\right) \approx \left\vert B\left( x,r_{k}\right)
	\right\vert ,
\end{equation*}%
and from (\ref{low bound}) with $\left( x,y\right) \in \partial B\left(
x,r_{k}\right) $, we have%
\begin{equation*}
	y\geq f\left( x\right) \left( r_{k}-x+x_{1}\right) \approx f\left(
	x_{1}\right) r_{k},
\end{equation*}%
for \emph{all} $x\in \left[ x_{1},x_{1}+r\right] $ since we are in the case $%
r_{k}<\frac{1}{\left\vert F^{\prime }\left( x_{1}\right) \right\vert }$. It
follows that%
\begin{equation*}
	E\left( x,r_{k}\right) \cap B\left( x,r_{k}\right) \supset \left[ x_{1}+%
	\frac{r_{k}}{2},x_{1}+\frac{3r_{k}}{4}\right] \times \left[ -cf\left(
	x_{1}\right) r_{k},cf\left( x_{1}\right) r_{k}\right]
\end{equation*}%
and hence that 
\begin{equation*}
	\left\vert E\left( x,r_{k}\right) \cap B\left( x,r_{k}\right) \right\vert
	\geq \frac{1}{2}cr_{k}f\left( x_{1}\right) r_{k}\approx \left\vert B\left(
	x,r_{k}\right) \right\vert \geq \left\vert E\left( x,r_{k}\right) \cap
	B\left( x,r_{k}\right) \right\vert .
\end{equation*}%
This completes the proof of (\ref{claim that}).

Now define $\Gamma \left( x,r\right) $ to be the set 
\begin{equation*}
	\Gamma \left( x,r\right) =\bigcup\limits_{k=1}^{\infty }\func{co}\left[
	E\left( x,r_{k}\right) \cup E\left( x,r_{k+1}\right) \right] ,
\end{equation*}%
where $\func{co}E$ denotes the convex hull of the set $E$. Set%
\begin{equation*}
	\mathbb{E}_{x,r_{1}}w\equiv \frac{1}{\left\vert E(x,r_{1})\right\vert }\int
	\int_{E(x,r_{1})}w.
\end{equation*}

\begin{lemma}
	\label{lemma-subrepresentation}With $\Gamma \left( x,r\right) $, $E\left(
	x,r_{1}\right) $ and $\mathbb{E}_{x,r_{1}}$ as above, and in particular with 
	$r_{0}=r$ and $r_{1}$ given by (\ref{rkp1}), we have the subrepresentation
	formula%
	\begin{equation}
	\left\vert w\left( x\right) -\mathbb{E}_{x,r_{1}}w\right\vert \leq
	C\int_{\Gamma \left( x,r\right) }\left\vert \nabla _{A}w\left( y\right)
	\right\vert \frac{\widehat{d}\left( x,y\right) }{\left\vert B\left(
		x,d\left( x,y\right) \right) \right\vert }dy,  \label{subrepresentation}
	\end{equation}%
	where $\nabla _{A}$ is as in (\ref{def A grad}) and 
	\begin{equation}
	\widehat{d}\left( x,y\right) \equiv \min \left\{ d\left( x,y\right) ,\frac{1%
	}{\left\vert F^{\prime }\left( x_{1}+d\left( x,y\right) \right) \right\vert }%
	\right\} .  \label{def d hat'}
	\end{equation}
\end{lemma}

Note that when $f\left( r\right) =r^{N}$ is finite type, then $\widehat{d}%
\left( x,y\right) \approx d\left( x,y\right) $.

\begin{proof}
	Recall the sequence $\left\{ r_{k}\right\} _{k=1}^{\infty }$ of decreasing
	radii above. Then since $w$ is \textit{a priori} Lipschitz continuous, we
	can write, 
	\begin{align*}
		w\left( x\right) -&\mathbb{E}_{x,r_{1}}w \\
		&=\lim_{k\rightarrow \infty }\frac{1%
		}{\left\vert E\left( x,r_{k}\right) \right\vert }\int_{E\left(
			x,r_{k}\right) }w\left( y\right) dy-\frac{1}{\left\vert
			E(x,r_{1})\right\vert }\int_{E(x,r_{1})}w \\
		&=\sum_{k=1}^{\infty }\left\{ \frac{1}{\left\vert E\left( x,r_{k+1}\right)
			\right\vert }\int_{E\left( x,r_{k+1}\right) }w\left( y\right) dy-\frac{1}{%
			\left\vert E\left( x,r_{k}\right) \right\vert }\int_{E\left( x,r_{k}\right)
		}w\left( z\right) dz\right\} ,
	\end{align*}%
	and so we have%
	\begin{align*}
		\left\vert w\left( x\right) -\mathbb{E}_{x,r_{1}}w\right\vert &\\
		\lesssim&
		\sum_{k=1}^{\infty }\frac{1}{\left\vert B\left( x,r_{k}\right) \right\vert
			^{2}}\int_{E\left( x,r_{k+1}\right) \times E\left( x,r_{k}\right)
		}\left\vert w\left( y\right) -w\left( z\right) \right\vert dydz \\
		\lesssim &\sum_{k=1}^{\infty }\frac{1}{\left\vert B\left( x,r_{k}\right)
			\right\vert ^{2}}\int_{E\left( x,r_{k+1}\right) \times E\left(
			x,r_{k}\right) } \\
		&\ \ \ \ \ \ \ \ \ \ \times \left\{ \left\vert w\left( y_{1},y_{2}\right)
		-w\left( z_{1},y_{2}\right) \right\vert +\left\vert w\left(
		z_{1},y_{2}\right) -w\left( z_{1},z_{2}\right) \right\vert \right\} dydz \\
		\lesssim &\sum_{k=1}^{\infty }\frac{1}{\left\vert B\left( x,r_{k}\right)
			\right\vert ^{2}}\int_{E\left( x,r_{k+1}\right) \times E\left(
			x,r_{k}\right) }\int_{y_{1}}^{z_{1}}\left\vert w_{x}\left( s,y_{2}\right)
		\right\vert dsdy_{1}dy_{2}dz_{1}dz_{2} \\
		+&\sum_{k=1}^{\infty }\frac{1}{\left\vert B\left( x,r_{k}\right)
			\right\vert ^{2}}\int_{E\left( x,r_{k+1}\right) \times E\left(
			x,r_{k}\right) }\int_{y_{2}}^{z_{2}}\left\vert w_{y}\left( z_{1},t\right)
		\right\vert dtdy_{1}dy_{2}dz_{1}dz_{2}\ ,
	\end{align*}%
	which, with $H_{k}\left( x\right) \equiv E\left( x,r_{k+1}\right) \bigcup
	E\left( x,r_{k}\right) $, is dominated by%
	\begin{eqnarray*}
		&&\sum_{k=1}^{\infty }\frac{1}{\left\vert B\left( x,r_{k}\right) \right\vert
			^{2}}\left( \int_{H_{k}\left( x\right) }\left\vert \nabla _{A}w\left(
		s,y_{2}\right) \right\vert dsdy_{2}\right) \left\{ r_{k}-r_{k+1}\right\}
		\int_{H_{k}\left( x\right) }dz_{1}dz_{2} \\
		&&+\sum_{k=1}^{\infty }\frac{1}{\left\vert B\left( x,r_{k}\right)
			\right\vert ^{2}}\left( \int_{H_{k}\left( x\right) }\left\vert \nabla
		_{A}w\left( z_{1},t\right) \right\vert dz_{1}dt\right) \ \frac{h_{k}}{%
			f\left( x_{1}+r_{k+1}\right) }\ \int_{H_{k}\left( x\right) }dy_{1}dy_{2}\ ,
	\end{eqnarray*}%
	where for the last term we used that 
	\begin{eqnarray*}
		\left\vert w_{y}\left( z_{1},t\right) \right\vert &=&\frac{f(z_{1})}{f(z_{1})%
		}\left\vert w_{y}\left( z_{1},t\right) \right\vert \leq \frac{1}{f(z_{1})}%
		\left\vert \nabla _{A}w\left( z_{1},t\right) \right\vert \\
		&\leq &\frac{1}{f(x_{1}+r_{k+1})}\left\vert \nabla _{A}w\left(
		z_{1},t\right) \right\vert \quad \forall (z_{1},z_{2})\in E(x,r_{k}).
	\end{eqnarray*}%
	Next, recall from Lemma \ref{height 1} that $h_{k}\approx
	(r_{k}-r_{k+1})\cdot f(x_{1}+r_{k+1})$ by our choice of $r_{k+1}$ in (\ref%
	{rkp1}). Moreover, by the estimates above we have that $|H_{k}(x)|\approx
	|B(x,r_{k})|,$ and 
	\begin{eqnarray}
	\left\vert w\left( x\right) -\mathbb{E}_{x,r_{1}}w\right\vert &\lesssim
	&\sum_{k=1}^{\infty }\frac{r_{k}-r_{k+1}}{\left\vert B\left( x,r_{k}\right)
		\right\vert }\left( \int_{H_{k}\left( x\right) }\left\vert \nabla
	_{A}w\left( s,y_{2}\right) \right\vert dsdy_{2}\right)  \notag \\
	&\lesssim &\int_{\Gamma (x,r)}\left\vert \nabla _{A}w\left( y\right)
	\right\vert \left( \sum_{k=1}^{\infty }\frac{r_{k}-r_{k+1}}{\left\vert
		B\left( x,r_{k}\right) \right\vert }\mathbf{1}_{E\left( x,r_{k}\right)
	}\left( y\right) \right) dy.  \label{pre-subr}
	\end{eqnarray}%
	To make further estimates we need to consider two regions separately, namely;
	
	\begin{enumerate}
		\item[\textbf{case 1}] $d(x,y)\geq \frac{1}{|F^{\prime }(x_{1})|}$. In this
		case we have 
		\begin{equation*}
			r_{k}>d(x,y)\geq \frac{1}{|F^{\prime }(x_{1})|},
		\end{equation*}%
		which implies by Proposition \ref{height} and (\ref{new f equ}) 
		\begin{equation*}
			r_{k}-r_{k+1}\approx \frac{1}{|F^{\prime }(x_{1}+r_{k})|}\approx \frac{1}{%
				|F^{\prime }(x_{1}+r_{k+2})|}<\frac{1}{|F^{\prime }(x_{1}+d(x,y))|}.
		\end{equation*}%
		Therefore, we are left with%
		\begin{align}\label{case 1}
			&\left\vert w\left( x\right) -\mathbb{E}_{x,r_{1}}w\right\vert \\
			\lesssim
			&\int_{\Gamma (x,r)}\left\vert \nabla _{A}w\left( y\right) \right\vert \frac{%
				1}{|F^{\prime }(x_{1}+d(x,y))|}\sum_{k:r_{k+1}<d(x,y)<r_{k}}\frac{1}{%
				\left\vert B\left( x,r_{k}\right) \right\vert }dy   \notag\\
			\approx &\int_{\Gamma (x,r)}\left\vert \nabla _{A}w\left( y\right)
			\right\vert \frac{1}{|F^{\prime }(x_{1}+d(x,y))|}\frac{1}{|B(x,d(x,y))|}dy. 
			\notag
		\end{align}
		
		\item[\textbf{case 2}] $d(x,y)<\frac{1}{|F^{\prime }(x_{1})|}$. We can write 
		\begin{equation*}
			\sum_{k:r_{k+1}<d(x,y)<r_{k}}\frac{r_{k}-r_{k+1}}{\left\vert B\left(
				x,r_{k}\right) \right\vert }\leq \sum_{k:r_{k+1}<d(x,y)<r_{k}}\frac{r_{k}}{%
				\left\vert B\left( x,r_{k}\right) \right\vert }\lesssim \frac{d(x,y)}{%
				|B(x,d(x,y))|},
		\end{equation*}%
		which gives 
		\begin{equation}
		\left\vert w\left( x\right) -\mathbb{E}_{x,r_{1}}w\right\vert \lesssim
		\int_{\Gamma (x,r)}\left\vert \nabla _{A}w\left( y\right) \right\vert \frac{%
			d(x,y)}{|B(x,d(x,y))|}.  \label{case 2}
		\end{equation}
	\end{enumerate}
	
	To finish the proof we need to dominate the right hand sides of (\ref{case 1}%
	) and (\ref{case 2}) with%
	\begin{equation}
	\int_{\Gamma (x,r)}\left\vert \nabla _{A}w\left( y\right) \right\vert \frac{%
		\widehat{d}(x,y)}{|B(x,d(x,y))|}  \label{dom}
	\end{equation}
	where $\widehat{d}\left( x,y\right) =\min \left\{ d\left( x,y\right) ,\frac{1%
	}{\left\vert F^{\prime }\left( x_{1}+d\left( x,y\right) \right) \right\vert }%
	\right\} $ in cases $1$ and $2$ respectively.
	
	Suppose first that $d(x,y)\geq \frac{1}{|F^{\prime }(x_{1}+d(x,y))|}$. Since 
	$|F^{\prime }(x_{1})|$ is a decreasing function of $x_{1}$ we have $%
	d(x,y)\geq \frac{1}{|F^{\prime }(x_{1})|}$ and therefore we are in \textbf{%
		case 1} and (\ref{dom}) then follows from $\frac{1}{|F^{\prime
		}(x_{1}+d(x,y))|}=\widehat{d}(x,y)$.
	
	If the reverse inequality holds, namely $d(x,y)<\frac{1}{|F^{\prime
		}(x_{1}+d(x,y))|}$, we have to consider two subcases. First, if $d(x,y)\leq 
	\frac{1}{|F^{\prime }(x_{1})|}$, then we are in \textbf{case 2} and (\ref%
	{dom}) then follows from $d(x,y)=\widehat{d}(x,y)$. Finally, if 
	\begin{equation*}
		\frac{1}{|F^{\prime }(x_{1})|}\leq d(x,y)<\frac{1}{|F^{\prime
			}(x_{1}+d(x,y))|},
	\end{equation*}%
	we are back in \textbf{case 1,} but by Proposition \ref{height} we have 
	\begin{equation*}
		\frac{1}{|F^{\prime }(x_{1}+d(x,y))|}\approx d(x,y)-d(x,y)^{\ast }<d(x,y),
	\end{equation*}%
	and again (\ref{dom}) holds since $d(x,y)=\widehat{d}(x,y)$.
\end{proof}

As a simple corollary we obtain a connection between $\widehat{d}(x,y)$ and
the `width' of the thickest part of a ball of radius $d(x,y)$, namely $%
d(x,y)-d^{\ast }(x,y)$, where if $r=d(x,y)$ and $r^{\ast }$ is as defined at
the beginning of Subsection \ref{arbitrary balls} of Chapter 7, then we
define $d^{\ast }(x,y)$ by%
\begin{equation}
d^{\ast }(x,y)\equiv r^{\ast }.  \label{def d*}
\end{equation}%
Note that if $x$ and $r$ are fixed, then for every $y\in \partial B\left(
x,r\right) $ we have $d(x,y)-d^{\ast }(x,y)=r-r^{\ast }$.

\begin{corollary}
	\label{d_hat_geom} Let $d(x,y)>0$ be the distance between any two points $%
	x,y\in \Omega $ and let $d^{\ast }(x,y)$ be as in (\ref{def d*}), and $%
	\widehat{d}(x,y)$ be as defined in (\ref{def d hat'}) of Lemma \ref%
	{lemma-subrepresentation}. Then 
	\begin{equation*}
		\widehat{d}(x,y)\approx d(x,y)-d^{\ast }(x,y)
	\end{equation*}
\end{corollary}

\begin{proof}
	As before, we consider two cases
	
	\begin{enumerate}
		\item[\textbf{case 1}] $d(x,y)\geq \frac{1}{|F^{\prime }(x_{1})|}$. In this
		case we have from Proposition \ref{height} 
		\begin{equation*}
			d(x,y)-d^{\ast }(x,y)\approx \frac{1}{|F^{\prime }(x_{1}+d(x,y))|}.
		\end{equation*}%
		If $d(x,y)\geq \frac{1}{|F^{\prime }(x_{1}+d(x,y))|}$, then $\widehat{d}%
		(x,y)=\frac{1}{|F^{\prime }(x_{1}+d(x,y))|}$ and the claim is proved. If, on
		the other hand, 
		\begin{equation*}
			d(x,y)\leq \frac{1}{|F^{\prime }(x_{1}+d(x,y))|},
		\end{equation*}%
		then $\widehat{d}(x,y)=d(x,y)$ and 
		\begin{equation*}
			d(x,y)>d(x,y)-d^{\ast }(x,y)\approx \frac{1}{|F^{\prime }(x_{1}+d(x,y))|}%
			\geq d(x,y),
		\end{equation*}%
		and the claim follows.
		
		\item[\textbf{case 2}] $d(x,y)<\frac{1}{|F^{\prime }(x_{1})|}$. From
		Proposition \ref{height} we have in this case 
		\begin{equation*}
			d(x,y)-d^{\ast }(x,y)\approx d(x,y).
		\end{equation*}%
		From the monotonicity of the function $F^{\prime }(x)$ we have 
		\begin{equation*}
			d(x,y)<\frac{1}{|F^{\prime }(x_{1})|}\leq \frac{1}{|F^{\prime
				}(x_{1}+d(x,y))|},
		\end{equation*}%
		and therefore $\widehat{d}(x,y)=d(x,y)\approx d(x,y)-d^{\ast }(x,y)$.
	\end{enumerate}
\end{proof}

\subsection{The higher dimensional case}

The subrepresentation inequality here is similar to Lemma \ref%
{lemma-subrepresentation} in two dimensions, with the main differences being
in the definition of the cusp-like region $\Gamma \left( x,r\right) $ in
higher dimensions. On the one hand, the shape of the higher dimensional
balls dictates the rough form of $\Gamma \left( x,r\right) $, but we will
also need to redefine the sequence of radii $\left\{ r_{k}\right\}
_{k=1}^{\infty }$ used in the definition of $\Gamma \left( x,r\right) $. We
begin by addressing the higher dimensional form, and later will turn to the
new sequences $\left\{ r_{k}\right\} _{k=1}^{\infty }$.

Recall that we denote points $x\in \mathbb{R}^{n}$ as 
\begin{equation*}
	x=\left( x_{1},\mathbf{x}_{2},x_{3}\right) \in \mathbb{R}\times \mathbb{R}%
	^{n-2}\times \mathbb{R\,}.
\end{equation*}%
Let $\left\vert B\left( x,d\left( x,y\right) \right) \right\vert $ denote
the $n$-dimensional Lebesgue measure of $B\left( x,d\left( x,y\right)
\right) $ where $d\left( x,y\right) $ is now the $n$-dimensional control
distance. We define the cusp-like region $\Gamma \left( x,r\right) $ and the
`ends' $E\left( x,r_{k}\right) $ of the balls $B\left( x,r_{k}\right) $ by%
\begin{eqnarray}
\Gamma \left( x,r\right) &=&\bigcup\limits_{k=1}^{\infty }E\left(
x,r_{k}\right) ;  \label{def Gamma and E high dim} \\
E\left( x,r_{k}\right) &\equiv &\left\{ y=\left( y_{1},\mathbf{y}%
_{2},y_{3}\right) :%
\begin{array}{c}
r_{k+1}\leq y_{1}-x_{1}<r_{k} \\ 
\left\vert \mathbf{y}_{2}-\mathbf{x}_{2}\right\vert <\sqrt{r_{k}^{2}-\left(
	y_{1}-x_{1}\right) ^{2}} \\ 
\left\vert y_{3}-x_{3}\right\vert <h^{\ast }\left( x_{1},y_{1}-x_{1}\right)%
\end{array}%
\right\} ,  \notag
\end{eqnarray}%
where we recall 
\begin{equation*}
	r_{k+1}=\left\{ 
	\begin{array}{lll}
		r^{\ast }\left( x_{1},r_{k}\right) &  & \text{if }r_{k}\geq \frac{1}{%
			\left\vert F^{\prime }\left( x_{1}\right) \right\vert } \\ 
		\frac{1}{2}r_{k} &  & \text{if }r_{k}<\frac{1}{\left\vert F^{\prime }\left(
			x_{1}\right) \right\vert }%
	\end{array}%
	\right. ,
\end{equation*}%
and where $r^{\ast }$ is defined in Definition \ref{def r*} right before
Proposition \ref{height}. We also define the modified `end' $\widetilde{E}$
by 
\begin{equation}
\widetilde{E}\left( x,r_{k}\right) \equiv \left\{ y=\left( y_{1},\mathbf{y}%
_{2},y_{3}\right) :%
\begin{array}{c}
r_{k+1}\leq y_{1}-x_{1}<r_{k} \\ 
\left\vert \mathbf{y}_{2}-\mathbf{x}_{2}\right\vert <\sqrt{%
	r_{k}^{2}-r_{k+1}^{2}} \\ 
\left\vert y_{3}-x_{3}\right\vert <h^{\ast }\left( x_{1},r_{k}\right)%
\end{array}%
\right. .  \label{def E tilda}
\end{equation}%
Note the estimate%
\begin{eqnarray}
\left\vert \mathbf{y}_{2}-\mathbf{x}_{2}\right\vert &<&\sqrt{%
	r_{k}^{2}-r_{k+1}^{2}}\approx \sqrt{r_{k}-r_{k+1}}\sqrt{r_{k}}
\label{the est} \\
\text{for }y &\in &\widetilde{E}\left( x,r_{k}\right) .  \notag
\end{eqnarray}%
We claim the following lemma.

\begin{lemma}
	\label{B and E}With notation as above we have%
	\begin{equation*}
		\left\vert \widetilde{E}\left( x,r_{k}\right) \right\vert \approx \left\vert
		E\left( x,r_{k}\right) \right\vert \approx \left\vert E\left( x,r_{k}\right)
		\cap B\left( x,r_{k}\right) \right\vert \approx \left\vert B\left(
		x,r_{k}\right) \right\vert .
	\end{equation*}
\end{lemma}

\begin{proof}
	Recall that by Lemma \ref{new} we have 
	\begin{equation}
	\left\vert B\left( x,r_{k}\right) \right\vert \approx \left\{ 
	\begin{array}{ccc}
	r_{k}^{n}f(x_{1}) & \text{ if } & r_{k}\leq \frac{2}{\left\vert F^{\prime
		}\left( x_{1}\right) \right\vert } \\ 
	\frac{f\left( x_{1}+r_{k}\right) }{\left\vert F^{\prime }\left(
		x_{1}+r_{k}\right) \right\vert ^{n}}\left( r_{k}\left\vert F^{\prime }\left(
	x_{1}+r_{k}\right) \right\vert \right) ^{\frac{n}{2}-1} & \text{ if } & 
	r_{k}\geq \frac{2}{\left\vert F^{\prime }\left( x_{1}\right) \right\vert }%
	\end{array}%
	\right. .  \label{measure_B_k}
	\end{equation}%
	To show 
	\begin{equation}
	\left\vert E\left( x,r_{k}\right) \right\vert \approx \left\vert B\left(
	x,r_{k}\right) \right\vert  \label{E_B_1}
	\end{equation}%
	we first note that $\left\vert E\left( x,r_{k}\right) \right\vert \approx
	\left\vert \widetilde{E}\left( x,r_{k}\right) \right\vert $. Integrating, we
	easily obtain 
	\begin{equation*}
		|\widetilde{E}\left( x,r_{k}\right) |\approx h^{\ast }\left(
		x_{1},r_{k}\right) r_{k}^{\frac{n-2}{2}}(r_{k}-r_{k+1})^{\frac{n}{2}}.
	\end{equation*}
	
	Now, in the first case $r_{k}\geq \frac{1}{\left\vert F^{\prime }\left(
		x_{1}\right) \right\vert }$ we have by Proposition \ref{height} $%
	r_{k}-r_{k+1}=r_{k}-r_{k}^{\ast }\approx \frac{1}{\left\vert F^{\prime
		}\left( x_{1}+r_{k}\right) \right\vert }$ and 
	\begin{equation*}
		h^{\ast }\left( x_{1},r_{k}\right) =h^{\ast }\left( x_{1},r_{k-1}^{\ast
		}\right) =h(x_{1},r_{k-1})\approx \frac{f\left( x_{1}+r_{k-1}\right) }{%
			\left\vert F^{\prime }\left( x_{1}+r_{k-1}\right) \right\vert }\approx \frac{%
			f\left( x_{1}+r_{k}\right) }{\left\vert F^{\prime }\left( x_{1}+r_{k}\right)
			\right\vert },
	\end{equation*}%
	where for the last set of inequalities we used $r_{k}=r_{k-1}^{\ast }$ and
	the estimate (\ref{new f equ}). This gives 
	\begin{equation*}
		|\widetilde{E}\left( x,r_{k}\right) |\approx r_{k}^{\frac{n}{2}-1}\frac{%
			f\left( x_{1}+r_{k}\right) }{\left\vert F^{\prime }\left( x_{1}+r_{k}\right)
			\right\vert ^{\frac{n}{2}+1}},
	\end{equation*}%
	which is the second estimate in (\ref{measure_B_k}) provided we also have $%
	r_{k}\geq \frac{2}{\left\vert F^{\prime }\left( x_{1}\right) \right\vert }$.
	Moreover, when $\frac{1}{\left\vert F^{\prime }\left( x_{1}\right)
		\right\vert }\leq r_{k}\leq \frac{2}{\left\vert F^{\prime }\left(
		x_{1}\right) \right\vert }$ the two estimates in (\ref{measure_B_k})
	coincide, so we conclude (\ref{E_B_1}) for $r_{k}\geq \frac{1}{\left\vert
		F^{\prime }\left( x_{1}\right) \right\vert }$.
	
	In the second case $r_{k}<\frac{1}{\left\vert F^{\prime }\left( x_{1}\right)
		\right\vert }$ we have $r_{k}-r_{k+1}=\frac{r_{k}}{2}$ and using part (3) of
	Proposition \ref{height} 
	\begin{equation*}
		h^{\ast }\left( x_{1},r_{k}\right) \approx h^{\ast }\left(
		x_{1},r_{k+1}\right) =h(x_{1},r_{k})\approx r_{k}f(x_{1}),
	\end{equation*}%
	which gives 
	\begin{equation*}
		|\widetilde{E}\left( x,r_{k}\right) |\approx r_{k}^{n}f(x_{1}).
	\end{equation*}%
	This concludes the proof of (\ref{E_B_1}).
	
	We are thus left to show 
	\begin{equation}
	\left\vert E\left( x,r_{k}\right) \cap B\left( x,r_{k}\right) \right\vert
	\approx \left\vert B\left( x,r_{k}\right) \right\vert .  \label{E_B_2}
	\end{equation}%
	In the case $r_{k}\geq \frac{1}{\left\vert F^{\prime }\left( x_{1}\right)
		\right\vert }$ we have 
	\begin{equation*}
		E\left( x,r_{k}\right) \cap B\left( x,r_{k}\right) \supset B_{+}\left(
		x,r_{k}\right) \equiv \left\{ \left( y_{1},\mathbf{y}_{2},y_{3}\right) \in
		B\left( x,r_{k}\right) :y_{1}>x_{1}+r_{k}^{\ast }\right\}
	\end{equation*}%
	and therefore 
	\begin{align*}
		&\left\vert E\left( x,r_{k}\right) \cap B\left( x,r_{k}\right) \right\vert
		\geq \left\vert B_{+}\left( x,r_{k}\right) \right\vert\\ 
		&\quad=\int\limits_{|%
			\mathbf{y}_{2}|\leq r_{k}}\left\vert {B_{2D}}\left( \left( x_{1},\mathbf{0}%
		,0\right) ,\sqrt{r_{k}^{2}-\left\vert \mathbf{y}_{2}\right\vert ^{2}}%
		\right)\cap \{y_{1}-x_{1}>r_{k}^{*}\} \right\vert d\mathbf{y}_{2} \\
		&\quad\approx \int\limits_{|\mathbf{y}_{2}|\leq r_{k}} \left(\sqrt{%
			r_{k}^{2}-\left\vert \mathbf{y}_{2}\right\vert ^{2}}-r_{k}^{\ast
		}\right)_{+}\cdot h\left(x_{1},\sqrt{r_{k}^{2}-\left\vert \mathbf{y}%
			_{2}\right\vert ^{2}}\right)d\mathbf{y}_{2} \\
		&\quad\approx \int\limits_{|\mathbf{y}_{2}|\leq r_{k}} \left(\sqrt{%
			r^{2}-\left\vert \mathbf{y}_{2}\right\vert ^{2}}-r_{k}^{\ast
		}\right)_{+}\cdot \frac{f\left(x_{1}+\sqrt{r_{k}^{2}-\left\vert \mathbf{y}%
				_{2}\right\vert ^{2}}\right)}{\left\vert F^{\prime }\left(x_{1}+\sqrt{%
				r^{2}-\left\vert \mathbf{y}_{2}\right\vert ^{2}}\right)\right\vert}d\mathbf{y%
		}_{2} \\
		&\quad=\int\limits_{|\mathbf{y}_{2}|^{2}\leq r_{k}^{2}-{r_{k}^{\ast }}^{2}} \left(%
		\sqrt{r_{k}^{2}-\left\vert \mathbf{y}_{2}\right\vert ^{2}}-r_{k}^{\ast
		}\right)\cdot \frac{f\left(x_{1}+\sqrt{r^{2}-\left\vert \mathbf{y}%
				_{2}\right\vert ^{2}}\right)}{\left\vert F^{\prime }\left(x_{1}+\sqrt{%
				r_{k}^{2}-\left\vert \mathbf{y}_{2}\right\vert ^{2}}\right)\right\vert}d%
		\mathbf{y}_{2} \\
		&\quad\approx \frac{f(x_{1}+r_{k})}{|F^{\prime }(x_{1}+r_{k})|}\int\limits_{|%
			\mathbf{y}_{2}|^{2}\leq r_{k}^{2}-{r_{k}^{\ast }}^{2}} \left(\sqrt{%
			r_{k}^{2}-\left\vert \mathbf{y}_{2}\right\vert ^{2}}-r_{k}^{\ast }\right)d%
		\mathbf{y}_{2} ,
	\end{align*}%
	where for the last equality we used (\ref{new f equ}). Passing to the polar
	coordinates, $\rho= \left\vert \mathbf{y}_{2}\right\vert$, we have 
	\begin{align*}
		\int\limits_{|\mathbf{y}_{2}|^{2}\leq r_{k}^{2}-{r_{k}^{\ast }}^{2}} \left(%
		\sqrt{r_{k}^{2}-\left\vert \mathbf{y}_{2}\right\vert ^{2}}-r_{k}^{\ast
		}\right)d\mathbf{y}_{2} &\approx \int\limits_{0}^{\sqrt{r_{k}^{2}-{%
					r_{k}^{\ast}}^{2}}} \left(\sqrt{r_{k}^{2}-\rho^{2}}-r_{k}^{\ast
		}\right)\rho^{n-3}d\rho \\
		&\geq \int\limits_{0}^{\sqrt{r_{k}^{2}-\frac{1}{4}(r_{k}+r_{k}^{\ast})^{2}}}
		\left(\sqrt{r_{k}^{2}-\rho^{2}}-r_{k}^{\ast }\right)\rho^{n-3}d\rho \\
		&\geq \frac{1}{2}(r_{k}-r_{k}^{\ast})\int\limits_{0}^{\sqrt{r_{k}^{2}-\frac{1%
				}{4}(r_{k}+r_{k}^{\ast})^{2}}} \rho^{n-3}d\rho\\
		&\approx (r_{k}-r_{k}^{\ast})^{%
			\frac{n}{2}}r^{\frac{n}{2}-1}\\
		&\approx (r_{k}-r_{k}^{\ast})^{\frac{n}{2}%
		}r_{k}^{\frac{n}{2}-1}.
	\end{align*}
	Using part (2) of Proposition \ref{height} we have $r_{k}-r_{k}^{\ast}%
	\approx 1/|F^{\prime }(x_{1}+r_{k})|$ and therefore 
	\begin{equation*}
		\left\vert E\left( x,r_{k}\right) \cap B\left( x,r_{k}\right)
		\right\vert\gtrsim \frac{f(x_{1}+r_{k})}{|F^{\prime }(x_{1}+r_{k})|^{\frac{n%
				}{2}+1}}r^{\frac{n}{2}-1}\approx \left\vert B\left( x,r_{k}\right)
		\right\vert.
	\end{equation*}
	
	Finally, in the case $r_{k}<\frac{1}{\left\vert F^{\prime }\left(
		x_{1}\right) \right\vert }$, proceeding exactly as in the proof of (\ref%
	{claim that}) in the $2$-dimensional case, we can show 
	\begin{equation*}
		E\left( x,r_{k}\right) \cap B\left( x,r_{k}\right) \supset \left[ x_{1}+%
		\frac{r_{k}}{2},x_{1}+\frac{3r_{k}}{4}\right] \times \left\{ |\mathbf{y}%
		_{2}|\leq \frac{r_{k}}{2}\right\} \times \left[ -cf\left( x_{1}\right)
		r_{k},cf\left( x_{1}\right) r_{k}\right]
	\end{equation*}%
	and thus 
	\begin{equation*}
		\left\vert E\left( x,r_{k}\right) \cap B\left( x,r_{k}\right) \right\vert
		\gtrsim r_{k}^{n}f(x_{1})\approx \left\vert B\left( x,r_{k}\right)
		\right\vert .
	\end{equation*}%
	This concludes the proof of (\ref{E_B_2}), and therefore the proof of Lemma %
	\ref{B and E}.
\end{proof}

\subsubsection{The difficulty with the standard sequence of radii}

Recall that we began the proof of Lemma \ref{lemma-subrepresentation} in two
dimensions by subtracting consecutive averages of $w$ over the ends $E\left(
x,r_{k}\right) $ and $E\left( x,r_{k+1}\right) $ to obtain%
\begin{align*}
	w\left( x\right) -&\mathbb{E}_{x,r_{1}}w\\ 
	&=\lim_{k\rightarrow \infty }\frac{1%
	}{\left\vert E\left( x,r_{k}\right) \right\vert }\int_{E\left(
		x,r_{k}\right) }w\left( y\right) dy-\mathbb{E}_{x,r_{1}}w \\
	&=\sum_{k=1}^{\infty }\left\{ \frac{1}{\left\vert E\left( x,r_{k+1}\right)
		\right\vert }\int_{E\left( x,r_{k+1}\right) }w\left( y\right) dy-\frac{1}{%
		\left\vert E\left( x,r_{k}\right) \right\vert }\int_{E\left( x,r_{k}\right)
	}w\left( z\right) dz\right\} .
\end{align*}%
If we simply proceed in this way in higher dimensions we will obtain, just
as in the two dimensional proof, that%
\begin{align*}
	&\left\vert w\left( x\right) -\mathbb{E}_{x,r_{1}}w\right\vert \\
	&\quad\lesssim
	\sum_{k=1}^{\infty }\frac{1}{\left\vert B\left( x,r_{k}\right) \right\vert
		^{2}}\int\limits_{E\left( x,r_{k+1}\right) \times E\left( x,r_{k}\right)
	}\left\vert w\left( y\right) -w\left( z\right) \right\vert dydz \\
	&\quad\lesssim \sum_{k=1}^{\infty }\frac{1}{\left\vert B\left( x,r_{k}\right)
		\right\vert ^{2}}\int\limits_{E\left( x,r_{k+1}\right) \times E\left(
		x,r_{k}\right) }\left\vert w\left( y_{1},\mathbf{y}_{2},y_{3}\right)
	-w\left( z_{1},\mathbf{y}_{2},y_{3}\right) \right\vert dydz \\
	&\quad\quad+\sum_{k=1}^{\infty }\frac{1}{\left\vert B\left( x,r_{k}\right)
		\right\vert ^{2}}\int\limits_{E\left( x,r_{k+1}\right) \times E\left(
		x,r_{k}\right) }\left\vert w\left( z_{1},\mathbf{y}_{2},y_{3}\right)
	-w\left( z_{1},\mathbf{z}_{2},y_{3}\right) \right\vert dydz \\
	&\quad\quad+\sum_{k=1}^{\infty }\frac{1}{\left\vert B\left( x,r_{k}\right)
		\right\vert ^{2}}\int\limits_{E\left( x,r_{k+1}\right) \times E\left(
		x,r_{k}\right) }\left\vert w\left( z_{1},\mathbf{z}_{2},y_{3}\right)
	-w\left( z_{1},\mathbf{z}_{2},z_{3}\right) \right\vert dydz \\
	&\quad\equiv I+II+III,
\end{align*}%
but where now 
\begin{equation*}
	I\lesssim \sum_{k=1}^{\infty }\frac{1}{\left\vert B\left( x,r_{k}\right)
		\right\vert ^{2}}\int\limits_{E\left( x,r_{k+1}\right) \times E\left(
		x,r_{k}\right) }\int_{y_{1}}^{z_{1}}\left\vert w_{x_{1}}\left( s,\mathbf{y}%
	_{2},y_{3}\right) \right\vert dsdy_{1}d\mathbf{y}_{2}dy_{3}dz_{1}d\mathbf{z}%
	_{2}dz_{3}\ ,
\end{equation*}%
and 
\begin{align*}
	II \lesssim &\sum_{k=1}^{\infty }\frac{1}{\left\vert B\left( x,r_{k}\right)
		\right\vert ^{2}}\int\limits_{E\left( x,r_{k+1}\right) \times E\left(
		x,r_{k}\right) }\int_{0}^{1}\left\vert \left( \mathbf{z}_{2}-\mathbf{y}%
	_{2}\right) \cdot \nabla _{\mathbf{x}_{2}}w\left( z_{1},t\mathbf{z}%
	_{2}+\left( 1-t\right) \mathbf{y}_{2},y_{3}\right) \right\vert \\
	&\ \ \ \ \ \ \ \ \ \ \ \ \ \ \ \ \ \ \ \ \ \ \ \ \ \ \ \ \ \ \ \ \ \ \ \ \
	\ \ \ \ \ \ \ \ \ \ \ \ \ \ \ \ \ \ \ \ \times dtdy_{1}d\mathbf{y}%
	_{2}dy_{3}dz_{1}d\mathbf{z}_{2}dz_{3}\ ,
\end{align*}%
and 
\begin{equation*}
	III\lesssim \sum_{k=1}^{\infty }\frac{1}{\left\vert B\left( x,r_{k}\right)
		\right\vert ^{2}}\int\limits_{E\left( x,r_{k+1}\right) \times E\left(
		x,r_{k}\right) }\int_{y_{3}}^{z_{3}}\left\vert w_{x_{3}}\left( z_{1},\mathbf{%
		z}_{2},u\right) \right\vert dudy_{1}d\mathbf{y}_{2}dy_{3}dz_{1}d\mathbf{z}%
	_{2}dz_{3}\ .
\end{equation*}%
Thus with $\bigtriangleup r_{k}\equiv r_{k}-r_{k+1}$, we have%
\begin{align*}
	I&\lesssim \sum_{k=1}^{\infty }\frac{\bigtriangleup r_{k}}{\left\vert B\left(
		x,r_{k}\right) \right\vert }\int_{H\left( x,r_{k}\right) }\left\vert
	w_{x_{1}}\left( s,\mathbf{y}_{2},y_{3}\right) \right\vert dsd\mathbf{y}%
	_{2}dy_{3}\\
	&\lesssim \sum_{k=1}^{\infty }\frac{\bigtriangleup r_{k}}{%
		\left\vert B\left( x,r_{k}\right) \right\vert }\int_{H\left( x,r_{k}\right)
	}\left\vert \nabla _{A}w\right\vert ,
\end{align*}%
and an easy computation also shows that%
\begin{equation*}
	III\lesssim \sum_{k=1}^{\infty }\frac{\bigtriangleup r_{k}}{\left\vert
		B\left( x,r_{k}\right) \right\vert }\int_{H\left( x,r_{k}\right) }\left\vert
	\nabla _{A}w\right\vert ,
\end{equation*}%
which for terms $I$ and $III$ delivers the good estimate (\ref{pre-subr}) in
the $2$-dimensional proof above. But upon using the inequality $\left\vert 
\mathbf{y}_{2}-\mathbf{x}_{2}\right\vert \lesssim \sqrt{r_{k}-r_{k+1}}\sqrt{%
	r_{k}}$ from (\ref{the est}), the corresponding estimate for $II$ is%
\begin{equation*}
	II\lesssim \sum_{k=1}^{\infty }\frac{\sqrt{r_{k}\bigtriangleup r_{k}}}{%
		\left\vert B\left( x,r_{k}\right) \right\vert }\int_{H\left( x,r_{k}\right)
	}\left\vert \nabla _{A}w\right\vert ,
\end{equation*}%
which is much too large as $\sqrt{r_{k}\bigtriangleup r_{k}}\gg
\bigtriangleup r_{k}$ when $\bigtriangleup r_{k}\ll r_{k}$.

This suggests that we hold the variable $\mathbf{y}_{2}$ fixed, and take the
difference of lower dimensional averages, and then average over $\mathbf{y}%
_{2}$. But this will require additional information on the regularity of the
sequence of radii $\left\{ r_{k}\right\} _{k=1}^{\infty }$, something we
cannot easily derive from the current definition of $r_{k}$. So we now turn
to redefining the sequence of radii to be used in our subrepresentation
inequalities.

\subsubsection{Geometric estimates}

We will estimate the differences of the quantities 
\begin{equation}
q_{k}\equiv \sqrt{r_{k}^{2}-r_{k+1}^{2}}  \label{def q_k}
\end{equation}%
appearing as the widths of the modified ends $\widetilde{E}\left(
x_{1},r_{k}\right) $ defined in (\ref{def E tilda}). First recall that there
are positive constants $c,C$ such that%
\begin{equation*}
	c\leq \left( r_{k}-r_{k+1}\right) \left\vert F^{\prime }\left(
	x_{1}+r_{k}\right) \right\vert \leq C,\ \ \ \ \ \text{for }r_{k}\geq \frac{2%
	}{\left\vert F^{\prime }\left( x_{1}\right) \right\vert }.
\end{equation*}%
In view of this, let us redefine, for each $\gamma >0$, the sequence $%
\left\{ r_{k}\right\} _{k=0}^{\infty }$ recursively by demanding that the
first inequality above be an equality.

\begin{definition}
	For $\gamma >0$ set 
	\begin{equation}
	r_{k+1}^{\gamma }\equiv \left\{ 
	\begin{array}{ccc}
	r_{k}^{\gamma }-\frac{\gamma }{\left\vert F^{\prime }\left(
		x_{1}+r_{k}\right) \right\vert } & \text{ if } & r_{k}^{\gamma }\geq \frac{%
		\gamma }{\left\vert F^{\prime }\left( x_{1}\right) \right\vert } \\ 
	\frac{1}{2}r_{k}^{\gamma } & \text{ if } & r_{k}^{\gamma }<\frac{\gamma }{%
		\left\vert F^{\prime }\left( x_{1}\right) \right\vert }%
	\end{array}%
	\right. ,\ \ \ \ \ k\geq 0.  \label{redef r_k}
	\end{equation}
\end{definition}

We will typically suppress the superscript $\gamma $ and continue to write $%
r_{k}$ in place of $r_{k}^{\gamma }$ when $\gamma $ is understood. With this
revised definition of the sequence of radii, and the corresponding balls and
ends, we retain the two-dimensional volume estimates and the
subrepresentation inequality in Lemma \ref{lemma-subrepresentation}, with
perhaps larger constants of comparability. These details are easily verified
and left for the reader.

Now, continuing to suppress $\gamma $, let 
\begin{equation*}
	\bigtriangleup r_{k}\equiv r_{k}-r_{k+1}\text{ and }\bigtriangleup
	^{2}r_{k}\equiv \bigtriangleup r_{k}-\bigtriangleup r_{k+1}\text{ and }%
	\bigtriangleup q_{k}\equiv q_{k}-q_{k+1}
\end{equation*}%
denote the first and second order differences of the sequences $\left\{
r_{k}\right\} _{k=0}^{\infty }$ and $\left\{ q_{k}\right\} _{k=0}^{\infty }$%
. The point of the new definitions of the sequence $\left\{ r_{k}^{\gamma
}\right\} _{k=0}^{\infty }$ is to obtain a good estimate on its second order
differences $\bigtriangleup ^{2}r_{k}^{\gamma }$, and hence also on the
first order differences of $\left\{ q_{k}^{\gamma }\right\} _{k=0}^{\infty }$%
.

\begin{lemma}
	\label{control q}With $\gamma >0$ and the sequence $\left\{ r_{k}^{\gamma
	}\right\} _{k=0}^{\infty }$ defined as in (\ref{redef r_k}), and $\left\{
	q_{k}^{\gamma }\right\} _{k=0}^{\infty }$ defined as in (\ref{def q_k}), we
	have the following estimates:%
	\begin{eqnarray*}
		\bigtriangleup r_{k}^{\gamma } &=&\left\{ 
		\begin{array}{ccc}
			\frac{\gamma }{\left\vert F^{\prime }\left( x_{1}+r_{k}\right) \right\vert }
			& \text{ if } & r_{k}\geq \frac{\gamma }{\left\vert F^{\prime }\left(
				x_{1}\right) \right\vert } \\ 
			\frac{1}{2}r_{k} & \text{ if } & r_{k}<\frac{\gamma }{\left\vert F^{\prime
				}\left( x_{1}\right) \right\vert }%
		\end{array}%
		\right. , \\
		\bigtriangleup ^{2}r_{k}^{\gamma } &\approx &\frac{\left[ \bigtriangleup
			r_{k}\right] ^{2}}{r_{k}}, \\
		q_{k}^{\gamma } &\approx &\sqrt{r_{k}\bigtriangleup r_{k}}, \\
		\bigtriangleup q_{k}^{\gamma } &\lesssim &\bigtriangleup r_{k},
	\end{eqnarray*}%
	where the implied constants of comparability depend on $\gamma >0$.
\end{lemma}

\begin{proof}
	We suppress the superscript $\gamma $ and prove only the case where $%
	r_{k}\geq \frac{\gamma }{\left\vert F^{\prime }\left( x_{1}\right)
		\right\vert }$, since in the opposite case where $r_{k+1}=\frac{1}{2}r_{k}$
	and $\bigtriangleup r_{k}=\frac{1}{2}r_{k}$, the estimates then follow
	immediately. We begin with%
	\begin{eqnarray*}
		\bigtriangleup ^{2}r_{k} &=&\bigtriangleup r_{k}-\bigtriangleup r_{k+1}=%
		\frac{\gamma }{\left\vert F^{\prime }\left( x_{1}+r_{k}\right) \right\vert }-%
		\frac{\gamma }{\left\vert F^{\prime }\left( x_{1}+r_{k+1}\right) \right\vert 
		} \\
		&=&\gamma \frac{F^{\prime }\left( x_{1}+r_{k}\right) -F^{\prime }\left(
			x_{1}+r_{k+1}\right) }{\left\vert F^{\prime }\left( x_{1}+r_{k}\right)
			\right\vert \left\vert F^{\prime }\left( x_{1}+r_{k+1}\right) \right\vert }
		\\
		&=&\gamma \frac{F^{\prime \prime }\left( x_{1}+\left( 1-\theta \right)
			r_{k}+\theta r_{k+1}\right) \ \left( r_{k}-r_{k+1}\right) }{\left\vert
			F^{\prime }\left( x_{1}+r_{k}\right) \right\vert \left\vert F^{\prime
			}\left( x_{1}+r_{k+1}\right) \right\vert } \\
		&\approx &\frac{\frac{\left\vert F^{\prime }\left( x_{1}+\left( 1-\theta
				\right) r_{k}+\theta r_{k+1}\right) \right\vert }{x_{1}+\left( 1-\theta
				\right) r_{k}+\theta r_{k+1}}\left( r_{k}-r_{k+1}\right) }{\left\vert
			F^{\prime }\left( x_{1}+r_{k}\right) \right\vert \left\vert F^{\prime
			}\left( x_{1}+r_{k+1}\right) \right\vert }\lesssim \frac{\left[
			\bigtriangleup r_{k}\right] ^{2}}{r_{k+1}},
	\end{eqnarray*}%
	where the last two line follows from our assumptions on $F$ and (\ref{redef
		r_k}). Next we have%
	\begin{equation*}
		q_{k}=\sqrt{r_{k}-r_{k+1}}\sqrt{r_{k}+r_{k+1}}\approx \sqrt{%
			r_{k}\bigtriangleup r_{k}},
	\end{equation*}%
	and finally we have%
	\begin{eqnarray*}
		\bigtriangleup q_{k} &=&\sqrt{\bigtriangleup r_{k}}\sqrt{r_{k}+r_{k+1}}-%
		\sqrt{\bigtriangleup r_{k+1}}\sqrt{r_{k+1}+r_{k+2}} \\
		&=&\left( \sqrt{\bigtriangleup r_{k}}-\sqrt{\bigtriangleup r_{k+1}}\right) 
		\sqrt{r_{k}+r_{k+1}} \\
		&&+\sqrt{\bigtriangleup r_{k+1}}\left( \sqrt{r_{k}+r_{k+1}}-\sqrt{%
			r_{k+1}+r_{k+2}}\right) \\
		&\equiv &I+II.
	\end{eqnarray*}%
	Now%
	\begin{eqnarray*}
		I &\approx &\left( \sqrt{\bigtriangleup r_{k}}-\sqrt{\bigtriangleup r_{k+1}}%
		\right) \sqrt{r_{k}}\approx \left( \frac{\bigtriangleup r_{k}-\bigtriangleup
			r_{k+1}}{\sqrt{\bigtriangleup r_{k}}+\sqrt{\bigtriangleup r_{k+1}}}\right) 
		\sqrt{r_{k}} \\
		&\approx &\left( \bigtriangleup ^{2}r_{k}\right) \sqrt{\frac{r_{k}}{%
				\bigtriangleup r_{k}}}\lesssim \frac{\left[ \bigtriangleup r_{k}\right] ^{2}%
		}{r_{k+1}}\sqrt{\frac{r_{k}}{\bigtriangleup r_{k}}}=\frac{r_{k}}{r_{k+1}}%
		\sqrt{\frac{\bigtriangleup r_{k}}{r_{k}}}\bigtriangleup r_{k}\lesssim
		\bigtriangleup r_{k}
	\end{eqnarray*}%
	and 
	\begin{eqnarray*}
		II &=&\sqrt{\bigtriangleup r_{k+1}}\frac{\left( r_{k}+r_{k+1}\right) -\left(
			r_{k+1}+r_{k+2}\right) }{\sqrt{r_{k}+r_{k+1}}+\sqrt{r_{k+1}+r_{k+2}}} \\
		&\approx &\sqrt{\frac{\bigtriangleup r_{k+1}}{r_{k}}}\bigtriangleup
		r_{k}\lesssim \bigtriangleup r_{k}\ .
	\end{eqnarray*}
\end{proof}

\subsubsection{Statements of subrepresentation inequalities}

Set%
\begin{equation*}
	\mathbb{E}_{x,r_{1}}w\equiv \frac{1}{\left\vert E(x,r_{1})\right\vert }%
	\int_{E(x,r_{1})}w.
\end{equation*}

\begin{lemma}[$n$D subrepresentation]
	\label{lemma-subrepresentation'}With $\Gamma \left( x,r\right) $ as above,
	and $r_{0}=r$ and $r_{1}$ given by (\ref{rkp1}), we have the
	subrepresentation formula%
	\begin{equation*}
		\left\vert w\left( x\right) -\mathbb{E}_{x,r_{1}}w\right\vert \leq
		C\int_{\Gamma \left( x,r\right) }\left\vert \nabla _{A}w\left( y\right)
		\right\vert \frac{\widehat{d}\left( x,y\right) }{\left\vert B\left(
			x,d\left( x,y\right) \right) \right\vert }dy,
	\end{equation*}%
	where $\nabla _{A}$ is as in (\ref{def A grad}) and 
	\begin{equation*}
		\widehat{d}\left( x,y\right) \equiv \min \left\{ d\left( x,y\right) ,\frac{1%
		}{\left\vert F^{\prime }\left( x_{1}+d\left( x,y\right) \right) \right\vert }%
		\right\} .
	\end{equation*}
\end{lemma}

We next claim that the subrepresentation inequality continues to hold if we
use the \emph{modified} `end' $\widetilde{E}\left( x,r_{k}\right) $ in (\ref%
{def E tilda}) to define a \emph{modified} cusp-like region 
\begin{equation}
\widetilde{\Gamma }\left( x,r\right) =\bigcup\limits_{k=1}^{\infty }%
\widetilde{E}\left( x,r_{k}\right) .  \label{def Gamma tilda}
\end{equation}%
Indeed, Lemma \ref{lemma-subrepresentation'} extends to higher dimensions
with $\widetilde{\Gamma }\left( x,r\right) $ in place of $\Gamma \left(
x,r\right) $ in the subrepresentation formula, and with the average%
\begin{equation*}
	\mathbb{\tilde{E}}_{x,r_{1}}w\equiv \frac{1}{\left\vert \widetilde{E}%
		(x,r_{1})\right\vert }\int_{\widetilde{E}(x,r_{1})}w
\end{equation*}%
in place of $\mathbb{E}_{x,r_{1}}w$.

\begin{lemma}[$n$D subrepresentation]
	\label{lemma-subrepresentation''}With notation as above, and $r_{0}=r$, $%
	r_{1}$ given by (\ref{rkp1}), and $\widetilde{\Gamma }$ by (\ref{def Gamma
		tilda}), we have the subrepresentation formula%
	\begin{equation*}
		\left\vert w\left( x\right) -\mathbb{\tilde{E}}_{x,r_{1}}w\right\vert \leq
		C\int_{\widetilde{\Gamma }\left( x,r\right) }\left\vert \nabla _{A}w\left(
		y\right) \right\vert \frac{\widehat{d}\left( x,y\right) }{\left\vert B\left(
			x,d\left( x,y\right) \right) \right\vert }dy,
	\end{equation*}%
	where $\nabla _{A}$ is as in (\ref{def A grad}) and 
	\begin{equation*}
		\widehat{d}\left( x,y\right) \equiv \min \left\{ d\left( x,y\right) ,\frac{1%
		}{\left\vert F^{\prime }\left( x_{1}+d\left( x,y\right) \right) \right\vert }%
		\right\} .
	\end{equation*}
\end{lemma}

Combining Lemma \ref{lemma-subrepresentation''} with Lemma \ref{new}, we
obtain that in dimension $n\geq 3$ we have the estimate 
\begin{align}
	& \label{K est} \\
	&K_{B\left( 0,r_{0}\right) }(x,y) =\frac{\widehat{d}\left( x,y\right) }{%
		\left\vert B\left( x,d\left( x,y\right) \right) \right\vert }\mathbf{1}_{%
		\widetilde{\Gamma }\left( x,r_{0}\right) }\left( y\right)  \notag \\
	&\approx %
	\begin{cases}
		\begin{split}
			\frac{1}{r^{n-1}f(x_{1})}\mathbf{1}_{\widetilde{\Gamma }\left(
				x,r_{0}\right) }\left( y\right) ,\quad 0& <r=y_{1}-x_{1}<\frac{2}{|F^{\prime
				}(x_{1})|} \\
			\frac{\left\vert F^{\prime }\left( x_{1}+r\right) \right\vert ^{n-1}}{%
				f(x_{1}+r)\lambda \left( x_{1},r\right) ^{n-2}}\mathbf{1}_{\widetilde{\Gamma 
				}\left( x,r_{0}\right) }\left( y\right) ,\quad R& \geq r=y_{1}-x_{1}\geq 
			\frac{2}{|F^{\prime }(x_{1})|}
		\end{split}%
	\end{cases}%
	,  \notag
\end{align}%
where we have defined%
\begin{equation}
\lambda \left( x_{1},r\right) \equiv \sqrt{r\left\vert F^{\prime }\left(
	x_{1}+r\right) \right\vert }.  \label{def lambda}
\end{equation}

The proofs of Lemmas \ref{lemma-subrepresentation'} and \ref%
{lemma-subrepresentation''} are both similar to that of the two dimensional
analogue, Lemma \ref{lemma-subrepresentation} above. The main difference
lies in the fact, already noted above, that we can no longer simply subtract
the averages of $w$ over the ends $\widetilde{E}\left( x,r_{k}\right) $ and $%
\widetilde{E}\left( x,r_{k+1}\right) $ since the diameter of these ends in
the $\mathbf{x}_{2}$ direction is comparable to $\sqrt{r_{k}\left(
	r_{k}-r_{k+1}\right) }$, a quantity much larger than $r_{k}-r_{k+1}$ when $%
r_{k}-r_{k+1}\ll r_{k}$.

There is one more estimate we give. Define the half metric ball 
\begin{equation*}
	HB(0,r)=B(0,r)\cap \{(x,y)\in \mathbb{R}^{2}:x>0\}.
\end{equation*}
We show that for $x\in HB\left( 0,r\right) $ and $0<x_{1}\leq r^{\ast
}=r^{\ast }\left( 0,r\right) $, the rectangle $\widetilde{E}\left(
x,r-r^{\ast }\right) $ has volume comparable to that of $\widetilde{E}\left(
0,r\right) $, and hence the averages of $w$ over the modified ends $%
\widetilde{E}\left( x,r-r^{\ast }\right) $ and $\widetilde{E}\left(
0,r\right) $ have controlled difference. We will not use this estimate in
this paper, but it is natural and may prove useful elsewhere. More precisely
we have the following lemma.

\begin{lemma}
	Suppose that $0<r<R$, $x\in HB\left( 0,r\right) $, and $0<x_{1}\leq r^{\ast
	} $. Then we have
	
	\begin{enumerate}
		\item $\left\vert \widetilde{E}\left( x,r-r^{\ast }\right) \right\vert
		\approx \left\vert \widetilde{E}\left( x,r-r^{\ast }\right) \cap HB\left(
		0,r\right) \right\vert \approx \left\vert \widetilde{E}\left( 0,r\right)
		\right\vert $,
		
		\item $\left\vert \frac{1}{\left\vert \widetilde{E}\left( x,r-r^{\ast
			}\right) \right\vert }\int\limits_{\widetilde{E}\left( x,r-r^{\ast }\right) \cap
			HB\left( 0,r\right) }w-\frac{1}{\left\vert \widetilde{E}\left( 0,r\right)
			\right\vert }\int_{\widetilde{E}\left( 0,r\right) }w\right\vert \lesssim
		r\int_{HB(0,r)}|\nabla _{A}w|dx$.
	\end{enumerate}
\end{lemma}

\begin{proof}
	This is a straightforward exercise.
\end{proof}

As a consequence of this lemma, we obtain the average control%
\begin{equation}
\left\vert \frac{1}{\left\vert \widetilde{E}\left( x,r-r^{\ast }\right)
	\right\vert }\int_{\widetilde{E}\left( x,r-r^{\ast }\right) \cap B_{+}\left(
	0,r\right) }w-\frac{1}{\left\vert \widetilde{E}\left( 0,r\right) \right\vert 
}\int_{\widetilde{E}\left( 0,r\right) }w\right\vert \lesssim
r\int_{B(0,r)}|\nabla _{A}w|dx.  \label{avg control}
\end{equation}%
For the case when $r^{\ast }<x_{1}\leq r$, we simply use instead of the end $%
\widetilde{E}\left( 0,r\right) $, the substitute%
\begin{equation*}
	\overleftrightarrow{E}\left( 0,r\right) \equiv \left\{ y=\left( y_{1},%
	\mathbf{y}_{2},y_{3}\right) :%
	\begin{array}{c}
		r-3\left( r-r^{\ast }\right) \leq y_{1}-x_{1}<r-2\left( r-r^{\ast }\right)
		\\ 
		\left\vert \mathbf{y}_{2}-\mathbf{x}_{2}\right\vert <\sqrt{r_{k}^{2}-\left(
			y_{1}-x_{1}\right) ^{2}} \\ 
		\left\vert y_{3}-x_{3}\right\vert <h^{\ast }\left( x_{1},y_{1}-x_{1}\right)%
	\end{array}%
	\right\} ,
\end{equation*}%
which looks like $\widetilde{E}\left( 0,r\right) $ translated a distance $%
2\left( r-r^{\ast }\right) $ to the left. This gives the average control (%
\ref{avg control}) in the case $r^{\ast }<x_{1}\leq r$ as well. With these
considerations we have obtained the following corollary.

\begin{corollary}
	With notation as above, and $r_{0}=r$, $r_{1}$ given by (\ref{rkp1}),
	suppose that $\mathbb{\tilde{E}}_{0,r}w=0$. Then we have the
	subrepresentation formula%
	\begin{equation*}
		\left\vert w\left( x\right) \right\vert \leq C\int_{\widetilde{\Gamma }%
			\left( x,r\right) }\left\vert \nabla _{A}w\left( y\right) \right\vert \frac{%
			\widehat{d}\left( x,y\right) }{\left\vert B\left( x,d\left( x,y\right)
			\right) \right\vert }dy,
	\end{equation*}%
	where $\nabla _{A}$ is as in (\ref{def A grad}) and 
	\begin{equation*}
		\widehat{d}\left( x,y\right) \equiv \min \left\{ d\left( x,y\right) ,\frac{1%
		}{\left\vert F^{\prime }\left( x_{1}+d\left( x,y\right) \right) \right\vert }%
		\right\} .
	\end{equation*}
\end{corollary}

\subsubsection{Proofs of the subrepresentation inequalities}

We begin with a preliminary estimate on difference of averages that will set
the stage for the proofs of the subrepresentation inequalities. Recall that
the modified end $\widetilde{E}\left( x,r_{k}\right) $ defined in (\ref{def
	E tilda}) is a product set consisting of those $y=\left( y_{1},\mathbf{y}%
_{2},y_{3}\right) \in \mathbb{R}\times \mathbb{R}^{n-2}\times \mathbb{R}=%
\mathbb{R}^{n}$ belonging to%
\begin{equation*}
	\widetilde{E}\left( x,r_{k}\right) =\left[ x_{1}+r_{k+1},x_{1}+r_{k}\right)
	\times Q\left( \mathbf{x}_{2},q_{k}\right) \times \left( x_{3}-h^{\ast
	}\left( x_{1},r_{k}\right) ,x_{3}+h^{\ast }\left( x_{1},r_{k}\right) \right)
	,
\end{equation*}%
where $Q\left( \mathbf{z},q\right) $ denotes the $\left( n-2\right) $%
-dimensional Euclidean ball centered at $\mathbf{z}\in \mathbb{R}^{n-2}$
with radius $q$. With Lemma \ref{control q} in hand, we now dilate the
modified end $\widetilde{E}\left( x,r_{k}\right) $ in the $\mathbf{x}_{2}$
variable so that it has the same `thickness' as $\widetilde{E}\left(
x,r_{k+1}\right) $, namely $q_{k+1}$. More precisely, we define%
\begin{equation}
\widehat{E}\left( x,r_{k}\right) \equiv \left[ x_{1}+r_{k+1},x_{1}+r_{k}%
\right) \times Q\left( \mathbf{x}_{2},q_{k+1}\right) \times \left(
x_{3}-h^{\ast }\left( x_{1},r_{k}\right) ,x_{3}+h^{\ast }\left(
x_{1},r_{k}\right) \right) .  \label{def E hat}
\end{equation}%
Note that the only difference between $\widehat{E}\left( x,r_{k}\right) $
and $\widetilde{E}\left( x,r_{k}\right) $ is the change of width from $q_{k}$
to $q_{k+1}$. Then we observe that the dilation 
\begin{equation*}
	\mathbf{y}_{2}\rightarrow \mathbf{y}_{2}^{\prime }=\mathbf{x}_{2}+\frac{%
		q_{k+1}}{q_{k}}\left( \mathbf{y}_{2}-\mathbf{x}_{2}\right) ,
\end{equation*}%
with $y_{1}$ and $y_{3}$ kept fixed, maps $\widetilde{E}\left(
x,r_{k}\right) $ one-to-one onto $\widehat{E}\left( x,r_{k}\right) $, and
satisfies%
\begin{align*}
	&\frac{1}{\left\vert \widetilde{E}\left( x,r_{k}\right) \right\vert }\int\limits_{%
		\widetilde{E}\left( x,r_{k}\right) }w\left( y\right) dy-\frac{1}{\left\vert 
		\widehat{E}\left( x,r_{k}\right) \right\vert }\int\limits_{\widehat{E}\left(
		x,r_{k}\right) }w\left( y\right) dy \\
	&=\frac{1}{\bigtriangleup r_{k}}\int\limits_{r_{k+1}}^{r_{k}}\frac{1}{2h^{\ast
		}\left( x_{1},r_{k}\right) }\\
	&\qquad\qquad\qquad\times\int\limits_{-h^{\ast }\left( x_{1},r_{k}\right)
	}^{h^{\ast }\left( x_{1},r_{k}\right) }\left[ \frac{1}{\left\vert Q\left( 
		\mathbf{x}_{2},q_{k}\right) \right\vert }\int\limits_{\mathbf{y}_{2}\in Q\left( 
		\mathbf{x}_{2},q_{k}\right) }w\left( y_{1},\mathbf{y}_{2},y_{3}\right) d%
	\mathbf{y}_{2}\right] dy_{3}dy_{1} \\
	&-\frac{1}{\bigtriangleup r_{k}}\int\limits_{r_{k+1}}^{r_{k}}\frac{1}{2h^{\ast
		}\left( x_{1},r_{k}\right) }\\
	&\qquad\qquad\qquad\times\int\limits_{-h^{\ast }\left( x_{1},r_{k}\right)
	}^{h^{\ast }\left( x_{1},r_{k}\right) }\left[ \frac{1}{\left\vert Q\left( 
		\mathbf{x}_{2},q_{k+1}\right) \right\vert }\int\limits_{\mathbf{y}_{2}^{\prime }\in
		Q\left( \mathbf{x}_{2},q_{k+1}\right) }w\left( y_{1},\mathbf{y}_{2}^{\prime
	},y_{3}\right) d\mathbf{y}_{2}^{\prime }\right] dy_{3}dy_{1},
\end{align*}%
where the difference of averages in square brackets over $Q\left( \mathbf{x}%
_{2},q_{k}\right) $ and $Q\left( \mathbf{x}_{2},q_{k+1}\right) $ satisfies%
\begin{align*}
	&\frac{1}{\left\vert Q\left( \mathbf{x}_{2},q_{k}\right) \right\vert }\int_{%
		\mathbf{y}_{2}\in B\left( \mathbf{x}_{2},q_{k}\right) }w\left( y_{1},\mathbf{%
		y}_{2},y_{3}\right) d\mathbf{y}_{2}\\
	&\ \ \ \ \ \ \ \ \ \ \ \ \ \ \ \ \ \ \ \ \ \ \ \ \ -\frac{1}{\left\vert Q\left( \mathbf{x}%
		_{2},q_{k+1}\right) \right\vert }\int_{\mathbf{y}_{2}^{\prime }\in Q\left( 
		\mathbf{x}_{2},q_{k+1}\right) }w\left( y_{1},\mathbf{y}_{2}^{\prime
	},y_{3}\right) d\mathbf{y}_{2}^{\prime } \\
	&=\frac{1}{\left\vert Q\left( \mathbf{x}_{2},q_{k}\right) \right\vert }%
	\int_{\mathbf{y}_{2}^{\prime }\in Q\left( \mathbf{x}_{2},q_{k+1}\right)
	}w\left( y_{1},\mathbf{x}_{2}+\frac{q_{k}}{q_{k+1}}\left( \mathbf{y}%
	_{2}^{\prime }-\mathbf{x}_{2}\right) ,y_{3}\right) \frac{\partial \mathbf{y}%
		_{2}}{\partial \mathbf{y}_{2}^{\prime }}d\mathbf{y}_{2}^{\prime } \\
	&\ \ \ \ \ \ \ \ \ \ \ \ \ \ \ \ \ \ \ \ \ \ \ \ \ -\frac{1}{\left\vert
		Q\left( \mathbf{x}_{2},q_{k+1}\right) \right\vert }\int_{\mathbf{y}%
		_{2}^{\prime }\in Q\left( \mathbf{x}_{2},q_{k+1}\right) }w\left( y_{1},%
	\mathbf{y}_{2}^{\prime },y_{3}\right) d\mathbf{y}_{2}^{\prime } \\
	&=\frac{1}{\left\vert Q\left( \mathbf{x}_{2},q_{k+1}\right) \right\vert }%
	\int\limits_{Q\left( \mathbf{x}_{2},q_{k+1}\right) }\left[w\left( y_{1},\mathbf{x}%
	_{2}+\frac{q_{k+1}}{q_{k}}\left( \mathbf{y}_{2}^{\prime }-\mathbf{x}%
	_{2}\right) ,y_{3}\right) -w\left( y_{1},\mathbf{y}_{2}^{\prime
	},y_{3}\right) \right] d\mathbf{y}_{2}^{\prime }.
\end{align*}%
Thus we have the following estimate for the difference of averages%
\begin{eqnarray*}
	\widetilde{\mathbb{E}}_{x,r_{k}} &\equiv &\frac{1}{\left\vert \widetilde{E}%
		\left( x,r_{k}\right) \right\vert }\int_{\widetilde{E}\left( x,r_{k}\right)
	}w\left( y\right) dy, \\
	\widehat{\mathbb{E}}_{x,r_{k}} &\equiv &\frac{1}{\left\vert \widehat{E}%
		\left( x,r_{k}\right) \right\vert }\int_{\widehat{E}\left( x,r_{k}\right)
	}w\left( y\right) dy
\end{eqnarray*}%
over the modified ends $\widehat{E}\left( x,r_{k}\right) $ and $E\left(
x,r_{k}\right) $:%
\begin{align*}
	&\left\vert \widetilde{\mathbb{E}}_{x,r_{k}}-\widehat{\mathbb{E}}%
	_{x,r_{k}}\right\vert \leq \frac{1}{\bigtriangleup r_{k}}%
	\int_{r_{k+1}}^{r_{k}}\frac{1}{2h^{\ast }\left( x_{1},r_{k}\right) }%
	\int_{-h^{\ast }\left( x_{1},r_{k}\right) }^{h^{\ast }\left(
		x_{1},r_{k}\right) }\frac{1}{\left\vert Q\left( \mathbf{x}%
		_{2},q_{k+1}\right) \right\vert } \\
	&\ \ \ \ \ \times \int_{Q\left( \mathbf{x}_{2},q_{k+1}\right)}\left\vert w\left( y_{1},\mathbf{x}_{2}+\frac{q_{k+1}}{%
		q_{k}}\left( \mathbf{y}_{2}^{\prime }-\mathbf{x}_{2}\right) ,y_{3}\right)
	-w\left( y_{1},\mathbf{y}_{2}^{\prime },y_{3}\right) \right\vert d\mathbf{y}%
	_{2}^{\prime }dy_{1}dy_{3} \\
	&\leq \frac{1}{\bigtriangleup r_{k}}\int_{r_{k+1}}^{r_{k}}\frac{1}{2h^{\ast
		}\left( x_{1},r_{k}\right) }\int_{-h^{\ast }\left( x_{1},r_{k}\right)
	}^{h^{\ast }\left( x_{1},r_{k}\right) }\frac{1}{\left\vert Q\left( \mathbf{x}%
		_{2},q_{k+1}\right) \right\vert }\\
	&\ \ \ \ \ \times \int_{Q\left( \mathbf{x}_{2},q_{k+1}\right)} \left\vert \nabla _{\mathbf{x}_{2}}w\left( y_{1},\mathbf{z}%
	_{2},y_{3}\right) \right\vert d\mathbf{z}_{2}dy_{1}dy_{3}\left( 1-\frac{q_{k+1}}{%
		q_{k}}\right) \left\vert \mathbf{y}_{2}^{\prime }-\mathbf{x}_{2}\right\vert
	\\
	&\leq \bigtriangleup q_{k}\frac{1}{\left\vert \widetilde{E}\left(
		x,r_{k}\right) \right\vert }\int_{\widetilde{E}\left( x,r_{k}\right)
	}\left\vert \nabla _{A}w\left( y\right) \right\vert dy\lesssim \frac{%
		\bigtriangleup r_{k}}{\left\vert \widetilde{E}\left( x,r_{k}\right)
		\right\vert }\int_{\widetilde{E}\left( x,r_{k}\right) }\left\vert \nabla
	_{A}w\left( y\right) \right\vert dy.
\end{align*}%
This estimate,%
\begin{equation}
\left\vert \widetilde{\mathbb{E}}_{x,r_{k}}-\widehat{\mathbb{E}}%
_{x,r_{k}}\right\vert \lesssim \frac{\bigtriangleup r_{k}}{\left\vert 
	\widetilde{E}\left( x,r_{k}\right) \right\vert }\int_{\widetilde{E}\left(
	x,r_{k}\right) }\left\vert \nabla _{A}w\left( y\right) \right\vert dy,
\label{this est}
\end{equation}%
for the difference of averages has the same form as the corresponding
estimates for the summands in terms $I$ and $III$ in the previous
subsubsection, and thus we can replace the average of $w$ over $\widetilde{E}%
\left( x,r_{k}\right) $ by its average over $\widehat{E}\left(
x,r_{k}\right) $ whenever we wish.

We can now prove the subrepresentation formula in Lemma \ref%
{lemma-subrepresentation''}. The proof of Lemma \ref%
{lemma-subrepresentation'} is similar and left for the reader.

\begin{proof}[Proof of Lemma \protect\ref{lemma-subrepresentation''}]
	We have%
	\begin{align*}
		&w\left( x\right) -\widetilde{\mathbb{E}}_{x,r_{1}}w =\lim_{k\rightarrow
			\infty }\frac{1}{\left\vert \widetilde{E}\left( x,r_{k}\right) \right\vert }%
		\int_{\widetilde{E}\left( x,r_{k}\right) }w\left( y\right) dy-\widetilde{%
			\mathbb{E}}_{x,r_{1}}w \\
		&=\sum_{k=1}^{\infty }\left\{ \frac{1}{\left\vert \widetilde{E}\left(
			x,r_{k+1}\right) \right\vert }\int_{\widetilde{E}\left( x,r_{k+1}\right)
		}w\left( y\right) dy-\frac{1}{\left\vert \widetilde{E}\left(
			x,r_{k+1}\right) \right\vert }\int_{\widetilde{E}\left( x,r_{k+1}\right)
		}w\left( z\right) dz\right\} ,
	\end{align*}%
	and so%
	\begin{align*}
		w\left( x\right) -\widetilde{\mathbb{E}}_{x,r_{1}}w&\\
		=&\sum_{k=1}^{\infty }%
		\frac{1}{\left\vert \widetilde{E}\left( x,r_{k}\right) \right\vert }\int_{%
			\widetilde{E}\left( x,r_{k}\right) }w\left( y\right) dy-\frac{1}{\left\vert 
			\widehat{E}\left( x,r_{k}\right) \right\vert }\int_{\widehat{E}\left(
			x,r_{k}\right) }w\left( z\right) dz \\
		&+\sum_{k=1}^{\infty }\frac{1}{\left\vert \widehat{E}\left( x,r_{k}\right)
			\right\vert }\int_{\widehat{E}\left( x,r_{k}\right) }w\left( y\right) dy-%
		\frac{1}{\left\vert \widetilde{E}\left( x,r_{k+1}\right) \right\vert }%
		\int_{E\left( x,r_{k+1}\right) }w\left( z\right) dz \\
		\equiv &\ I+\sum_{k=1}^{\infty }II_{k},
	\end{align*}%
	where 
	\begin{equation*}
		\left\vert I\right\vert \lesssim \sum_{k=1}^{\infty }\frac{\bigtriangleup
			r_{k}}{\left\vert \widetilde{E}\left( x,r_{k}\right) \right\vert }\int_{%
			\widetilde{E}\left( x,r_{k}\right) }\left\vert \nabla _{A}w\left( y\right)
		\right\vert dy,
	\end{equation*}%
	and%
	\begin{align*}
		II_{k} =&\frac{1}{\bigtriangleup r_{k}}\int_{r_{k+1}}^{r_{k}}\frac{1}{%
			2h^{\ast }\left( x_{1},r_{k}\right) }\int_{-h^{\ast }\left(
			x_{1},r_{k}\right) }^{h^{\ast }\left( x_{1},r_{k}\right) }\frac{1}{%
			\left\vert Q\left( \mathbf{x}_{2},q_{k+1}\right) \right\vert }\int_{Q\left( 
			\mathbf{x}_{2},q_{k+1}\right) }w \\
		&-\frac{1}{\bigtriangleup r_{k+1}}\int_{r_{k+2}}^{r_{k+1}}\frac{1}{2h^{\ast
			}\left( x_{1},r_{k+1}\right) }\int_{-h^{\ast }\left( x_{1},r_{k+1}\right)
		}^{h^{\ast }\left( x_{1},r_{k+1}\right) }\frac{1}{\left\vert Q\left( \mathbf{%
				x}_{2},q_{k+1}\right) \right\vert }\int_{Q\left( \mathbf{x}%
			_{2},q_{k+1}\right) }w \\
		=&\frac{1}{\left\vert Q\left( \mathbf{x}_{2},q_{k+1}\right) \right\vert }%
		\int_{Q\left( \mathbf{x}_{2},q_{k+1}\right) } \left\{ \frac{1}{\bigtriangleup r_{k}}\int_{r_{k+1}}^{r_{k}}\frac{1%
		}{2h^{\ast }\left( x_{1},r_{k}\right) }\int_{-h^{\ast }\left(
			x_{1},r_{k}\right) }^{h^{\ast }\left( x_{1},r_{k}\right) }w\right.\\
		&-\left.\frac{1}{%
			\bigtriangleup r_{k+1}}\int_{r_{k+2}}^{r_{k+1}}\frac{1}{2h^{\ast }\left(
			x_{1},r_{k+1}\right) }\int_{-h^{\ast }\left( x_{1},r_{k+1}\right) }^{h^{\ast
			}\left( x_{1},r_{k+1}\right) }w\right\} .
	\end{align*}%
	Now the difference of the $2$-dimensional integrals in braces has the
	variable $\mathbf{y}_{2}$ fixed, and using the two dimensional proof above,
	the modulus of this difference is easily seen to be controlled by%
	\begin{equation*}
		\left[ \bigtriangleup r_{k}\right] \frac{1}{\bigtriangleup r_{k}}%
		\int_{r_{k+2}}^{r_{k}}\frac{1}{2h^{\ast }\left( x_{1},r_{k}\right) }%
		\int_{-h^{\ast }\left( x_{1},r_{k}\right) }^{h^{\ast }\left(
			x_{1},r_{k}\right) }\left\vert \nabla _{A}w\right\vert .
	\end{equation*}%
	Then averaging over $\mathbf{y}_{2}\in B\left( \mathbf{x}_{2},q_{k+1}\right) 
	$, we obtain the bound%
	\begin{eqnarray*}
		\left\vert II_{k}\right\vert &\lesssim &\left[ \bigtriangleup r_{k}\right] 
		\frac{1}{\left\vert B\left( \mathbf{x}_{2},q_{k+1}\right) \right\vert }%
		\int_{B\left( \mathbf{x}_{2},q_{k+1}\right) }\frac{1}{\bigtriangleup r_{k}}%
		\int_{r_{k+2}}^{r_{k}}\frac{1}{2h^{\ast }\left( x_{1},r_{k}\right) }%
		\int_{-h^{\ast }\left( x_{1},r_{k}\right) }^{h^{\ast }\left(
			x_{1},r_{k}\right) }\left\vert \nabla _{A}w\right\vert \\
		&\lesssim &\left[ \bigtriangleup r_{k}\right] \frac{1}{\left\vert \widetilde{%
				H}\left( x,r_{k}\right) \right\vert }\int_{\widetilde{H}\left(
			x,r_{k}\right) }\left\vert \nabla _{A}w\right\vert ,
	\end{eqnarray*}%
	where $\widetilde{H}_{k}\left( x,r_{k}\right) \equiv \widetilde{E}\left(
	x,r_{k+1}\right) \bigcup \widetilde{E}\left( x,r_{k}\right) $. Altogether
	then we have%
	\begin{align*}
		\left\vert w\left( x\right) -\widetilde{\mathbb{E}}_{x,r_{1}}w\right\vert&\\
		\lesssim& \sum_{k=1}^{\infty }\left\{ \frac{\bigtriangleup r_{k}}{\left\vert 
			\widetilde{E}\left( x,r_{k}\right) \right\vert }\int_{\widetilde{E}\left(
			x,r_{k}\right) }\left\vert \nabla _{A}w\right\vert +\frac{\bigtriangleup
			r_{k}}{\left\vert \widetilde{H}\left( x_{1},r_{k}\right) \right\vert }\int_{%
			\widetilde{H}\left( x_{1},r_{k}\right) }\left\vert \nabla _{A}w\right\vert
		\right\} .
	\end{align*}%
	Moreover, by estimates above we have that $|\widetilde{H}(x,r_{k})|\approx
	|B(x,r_{k})|,$ and%
	\begin{eqnarray*}
		\left\vert w\left( x\right) -\widetilde{\mathbb{E}}_{x,r_{1}}w\right\vert
		&\lesssim &\sum_{k=1}^{\infty }\frac{r_{k}-r_{k+1}}{\left\vert B\left(
			x,r_{k}\right) \right\vert }\left( \int_{\widetilde{H}\left(
			x_{1},r_{k}\right) }\left\vert \nabla _{A}w\left( y\right) \right\vert
		dy\right) \\
		&\lesssim &\int_{\widetilde{\Gamma }(x,r)}\left\vert \nabla _{A}w\left(
		y\right) \right\vert \left( \sum_{k=1}^{\infty }\frac{r_{k}-r_{k+1}}{%
			\left\vert B\left( x,r_{k}\right) \right\vert }\mathbf{1}_{E\left(
			x,r_{k}\right) }\left( y\right) \right) dy.
	\end{eqnarray*}%
	At this point, we have obtained the higher dimensional analogue of the
	two-dimensional inequality (\ref{pre-subr}), just as in Case $1$ and Case $2$
	of the $2$-dimensional proof of Lemma \ref{lemma-subrepresentation}, and the
	proof now proceeds exactly as in the $2$-dimensional case there.
\end{proof}

\section{$\left( 1,1\right) $-Sobolev and $\left( 1,1\right) $-Poincar\'{e}
	inequalities}

We will give statements and proofs only in dimension $n=2$, since these results are not
actually used in this paper, but might be interesting on their own. First we establish a simple ``straight-across'' estimate. Define 
\begin{equation}
K_{r}\left( x,y\right) \equiv \frac{\widehat{d}\left( x,y\right) }{%
	\left\vert B\left( x,d\left( x,y\right) \right) \right\vert }\mathbf{1}%
_{\Gamma \left( x,r\right) }\left( y\right) ,  \label{kernel_est}
\end{equation}%
and 
\begin{equation*}
	\Gamma \left( x,r\right) =\left\{ y\in B\left( x,r\right) :x_{1}\leq
	y_{1}\leq x_{1}+r,\ \left\vert y_{2}-x_{2}\right\vert <h^{\ast }\left(
	x_{1},y_{1}-x_{1}\right) \right\} ,
\end{equation*}
and for $y\in \Gamma \left( x,r\right) $ let $h_{x,y}=h^{\ast }\left(
x_{1},y_{1}-x_{1}\right) $. First, recall from Proposition \ref{general-area}
that we have an estimate 
\begin{equation*}
	\left\vert B\left( x,d\left( x,y\right) \right) \right\vert \approx
	h_{x,y}\left( d(x,y)-d^{\ast }(x,y)\right)
\end{equation*}%
and by Corollary \ref{d_hat_geom} we have $\left\vert B\left( x,d\left(
x,y\right) \right) \right\vert \approx h_{x,y}\widehat{d}(x,y)$. Thus, 
\begin{equation*}
	K_{r}\left( x,y\right) \approx \frac{1}{h_{x,y}}\mathbf{1}_{\left\{ \left(
		x,y\right) :x_{1}\leq y_{1}\leq x_{1}+r,\ \left\vert y_{2}-x_{2}\right\vert
		<h_{x,y}\right\} }\left( x,y\right) .
\end{equation*}

Now denote the dual cone $\Gamma ^{\ast }\left( y,r\right) $ by%
\begin{equation*}
	\Gamma ^{\ast }\left( y,r\right) \equiv \left\{ x\in B\left( y,r\right)
	:y\in \Gamma \left( x,r\right) \right\} .
\end{equation*}%
Then we have%
\begin{eqnarray}
\Gamma ^{\ast }\left( y,r\right) &=&\left\{ x\in B\left( y,r\right)
:x_{1}\leq y_{1}\leq x_{1}+r,\ \left\vert y_{2}-x_{2}\right\vert
<h_{x,y}\right\}  \label{Gammastar} \\
&=&\left\{ x\in B\left( y,r\right) :y_{1}-r\leq x_{1}\leq y_{1},\ \left\vert
x_{2}-y_{2}\right\vert <h_{x,y}\right\} ,  \notag
\end{eqnarray}%
and consequently we get the `straight across' estimate in $n=2$ dimensions,%
\begin{equation}
\int K_{r}\left( x,y\right) ~dx\approx \int_{y_{1}-r}^{y_{1}}\left\{
\int_{y_{2}-h_{x,y}}^{y_{2}+h_{x,y}}\frac{1}{h_{x,y}}dx_{2}\right\}
dx_{1}\approx \int_{x_{1}}^{x_{1}+r}dy_{1}=r\ .  \label{straight}
\end{equation}

\begin{lemma}
	\label{1-1 Sob}For $w\in Lip_{0}\left( B\left( x_{0},r\right) \right) $ and $%
	\nabla _{A}$ a degenerate gradient as above, we have%
	\begin{equation*}
		\int_{B\left( x_{0},r\right) }\left\vert w\left( x\right) \right\vert dx\leq
		Cr\int_{B\left( x_{0},r\right) }\left\vert \nabla _{A}w\left( y\right)
		\right\vert dy.
	\end{equation*}
\end{lemma}

\begin{proof}
	If $x\in B\left( x_{0},r\right) $, then $w$ satisfies the hypothesis of
	Lemma \ref{lemma-subrepresentation} in $B\left( x,2\left( C+1\right)
	^{2}r\right) $ for the constant $C$ as in (\ref{rkp12}). Indeed, let $r_{k}$
	be defined by (\ref{rkp1}) for $x=y$ and $r_{0}=2\left( C+1\right) ^{2}r$,
	then 
	\begin{equation*}
		r_{2}\geq \frac{1}{\left( C+1\right) ^{2}}r_{0}=2r.
	\end{equation*}%
	Hence, since 
	\begin{equation*}
		E\left( x,r_{1}\right) \equiv \left\{ y:x_{1}+r_{2}\leq
		y_{1}<x_{1}+r_{1},~\left\vert y_{2}-x_{2}\right\vert <h^{\ast }\left(
		x_{1},z_{1}-x_{1}\right) \right\} ,
	\end{equation*}%
	we have that $E\left( x,r_{1}\right) \bigcap B\left( x_{0},r\right)
	=\emptyset $ so $\int_{E(x,r_{1})}w=0$ so we may apply Lemma \ref%
	{lemma-subrepresentation} in $B\left( x_{0},2\left( \left( C+1\right)
	^{2}+1\right) r\right) $ for all $x\in B\left( x_{0},r\right) $.
	
	Let $R=2\left( \left( C+1\right) ^{2}+1\right) r$. Using the
	subrepresentation inequality and (\ref{straight}) we have%
	\begin{eqnarray*}
		\int \left\vert w\left( x\right) \right\vert ~dx &\leq &\int \int_{\Gamma
			\left( x,R\right) }\frac{\widehat{d}\left( x,y\right) }{\left\vert B\left(
			x,d\left( x,y\right) \right) \right\vert }\left\vert \nabla _{A}w\left(
		y\right) \right\vert ~dydx \\
		&=&\int \int K_{R}\left( x,y\right) \ \left\vert \nabla _{A}w\left( y\right)
		\right\vert ~dydx \\
		&=&\int \left\{ \int K_{R}\left( x,y\right) dx\right\} \ \left\vert \nabla
		_{A}w\left( y\right) \right\vert ~dy \\
		&\approx &\int R\left\vert \nabla _{A}w\left( y\right) \right\vert
		~dy\approx r\int \left\vert \nabla _{A}w\left( y\right) \right\vert ~dy.
	\end{eqnarray*}
\end{proof}

\begin{remark}
	\label{dhat}The larger kernel $\widetilde{K}_{r}\left( x,y\right) \equiv 
	\mathbf{1}_{\Gamma \left( x,r\right) }\left( y\right) \frac{d\left(
		x,y\right) }{\left\vert B\left( x,d\left( x,y\right) \right) \right\vert },$
	with $\widehat{d}$ replaced by $d$, does \textbf{not} in general yield the $%
	\left( 1,1\right) $ Sobolev inequality. More precisely, the inequality%
	\begin{equation}
	\int \int \widetilde{K}_{r}\left( x,y\right) \ \left\vert \nabla _{A}w\left(
	y\right) \right\vert ~dydx\lesssim r\int \left\vert \nabla _{A}w\left(
	y\right) \right\vert ~dy,\ \ \ \ \ 0<r\ll 1,  \label{ineq fails}
	\end{equation}%
	fails in the case 
	\begin{equation*}
		F(x)=\frac{1}{x},\;x>0.
	\end{equation*}%
	To see this take $y_{2}=0$. We now make estimates on the integral%
	\begin{equation}
	\int \widetilde{K}_{r}\left( x,y\right) ~dx\approx
	\int_{y_{1}-r}^{y_{1}}\left\{ \int_{y_{2}-h_{x,y}}^{y_{2}+h_{x,y}}\frac{1}{%
		h_{x,y}}\frac{d\left( x,y\right) }{\widehat{d}\left( x,y\right) }%
	dx_{2}\right\} dx_{1},  \label{kernel_int}
	\end{equation}%
	where $\widehat{d}\left( x,y\right) =\min \left\{ d\left( x,y\right) ,\frac{1%
	}{\left\vert F^{\prime }\left( x_{1}+d\left( x,y\right) \right) \right\vert }%
	\right\} $. Consider the region where 
	\begin{equation}
	d(x,y)\geq \frac{1}{|F^{\prime }(x_{1}+d(x,y))|}=\left( x_{1}+d(x,y)\right)
	^{2}.  \label{region_tmp}
	\end{equation}%
	In this region we have 
	\begin{equation*}
		\frac{d(x,y)}{\widehat{d}\left( x,y\right) }=d(x,y)|F^{\prime
		}(x_{1}+d(x,y))|=\frac{d(x,y)}{\left( x_{1}+d(x,y)\right) ^{2}}.
	\end{equation*}%
	Moreover, since $d(x,y)\leq r$, we have 
	\begin{equation*}
		\frac{d(x,y)}{\widehat{d}(x,y)}\geq \frac{r}{\left( x_{1}+r\right) ^{2}}.
	\end{equation*}%
	On the other hand, we have $d(x,y)\geq y_{1}-x_{1}$ and $d(x,y)\ll 1$, so
	the condition in (\ref{region_tmp}) is guaranteed by $y_{1}-x_{1}\geq \left(
	x_{1}+y_{1}-x_{1}\right) ^{2}$, i.e. $x_{1}\leq y_{1}-y_{1}^{2}$. We then
	have the following estimate for (\ref{kernel_int}): 
	\begin{equation*}
		\int \widetilde{K}_{r}(x,y)dx\gtrsim \int_{y_{1}-r}^{y_{1}-y_{1}^{2}}\frac{r%
		}{\left( x_{1}+r\right) ^{2}}dx_{1}=\frac{r(r-y_{1}^{2})}{%
			y_{1}(y_{1}-y_{1}^{2}+r)}.
	\end{equation*}%
	Therefore, if $y_{1}\leq r$, we have 
	\begin{equation*}
		\int \widetilde{K}_{r}(x,y)dx\gtrsim 1,
	\end{equation*}%
	and (\ref{ineq fails}) fails for small $r>0$.
\end{remark}

Now we turn to establishing the $\left( 1,1\right) $-\emph{Poincar\'{e}}
inequality. For this we will need the following extension of Lemma 79 in 
\cite{RSaW}. Recall the half metric ball 
\begin{equation*}
	HB(0,r)=B(0,r)\cap \{(x,y)\in \mathbb{R}^{2}:x>0\}.
\end{equation*}

\begin{proposition}
	\label{1 1 Poin} Let the balls $B(0,r)$ and the degenerate gradient $\nabla
	_{A}$ be as above. There exists a constant $C$ such that the Poincar\'{e}
	Inequality 
	\begin{equation*}
		\iint_{HB(0,r)}\left\vert w-\bar{w}\right\vert dxdy\leq
		Cr\iint_{HB(0,r)}|\nabla _{A}w|dxdy
	\end{equation*}%
	holds for any Lipschitz function $w$ and sufficiently small $r>0$. Here $%
	\bar{w}$ is the average defined by 
	\begin{equation*}
		\bar{w}=\frac{1}{|HB(0,r)|}\iint_{HB(0,r)}wdxdy.
	\end{equation*}
\end{proposition}

\subsection{Proof of Poincar\'{e}}

The left hand side can be estimated by 
\begin{align*}
	\iint_{HB(0,r)}&\left\vert w-\bar{w}\right\vert dxdy\\
	&=\iint_{HB(0,r)}\left\vert w(x_{1},y_{1})-\frac{1}{|HB(0,r)|}%
	\iint_{HB(0,r)}w(x_{2},y_{2})dx_{2}dy_{2}\right\vert dx_{1}dy_{1} \\
	& \leq \frac{1}{|HB(0,r)|}\int_{HB(0,r)\times HB(0,r)}\left\vert
	w(x_{1},y_{1})-w(x_{2},y_{2})\right\vert dx_{1}dy_{1}dx_{2}dy_{2}
\end{align*}%
The idea now is to estimate the difference $\left\vert
w(x_{1},y_{1})-w(x_{2},y_{2})\right\vert $ by the integral of $\nabla w$
along some path. Because the half metric ball is somewhat complicated
geometrically, we can simplify the argument by applying the following lemma,
sacrificing only the best constant $C$ in the Poincar\'{e} inequality.

\begin{lemma}
	\label{division of regions} Let $(X,\mu )$ be a measure space. If $\Omega
	\subset X$ is the disjoint union of $2$ measurable subsets $\Omega =\Omega
	_{1}\cup \Omega _{2}$ so that the measure of the subsets are comparable 
	\begin{equation*}
		\frac{1}{C_{1}}\leq \frac{\mu (\Omega _{1})}{\mu (\Omega _{2})}\leq C_{1}
	\end{equation*}%
	Then there exists a constant $C=C(C_{1})$, such that 
	\begin{equation}
	\iint_{\Omega \times \Omega }|w(x)-w(y)|d\mu (x)d\mu (y)\leq C\iint_{\Omega
		_{1}\times \Omega _{2}}|w(x)-w(y)|d\mu (x)d\mu (y).  \label{seperate}
	\end{equation}%
	for any measurable function $w$ defined on $\Omega $.
\end{lemma}

\begin{proof}
	Define 
	\begin{equation*}
		S_{i,j}=\iint_{\Omega _{i}\times \Omega _{j}}|w(x)-w(y)|d\mu (x)d\mu
		(y),\quad i,j=1,2.
	\end{equation*}%
	Since $\Omega =\Omega _{1}\cup \Omega _{2}$, we can rewrite inequality %
	\eqref{seperate} as 
	\begin{equation*}
		S_{1,1}+2S_{1,2}+S_{2,2}\leq CS_{1,2}.
	\end{equation*}%
	Now, we compute 
	\begin{align*}
		S_{1,1}& =\frac{1}{\mu (\Omega _{2})}\iiint_{\Omega _{1}\times \Omega
			_{1}\times \Omega _{2}}\left\vert [w(x)-w(z)]+[w(z)-w(y)]\right\vert d\mu
		(x)d\mu (y)d\mu (z) \\
		& \leq \frac{1}{\mu (\Omega _{2})}\iiint_{\Omega _{1}\times \Omega
			_{1}\times \Omega _{2}}\left\vert w(x)-w(z)\right\vert d\mu (x)d\mu (y)d\mu
		(z) \\
		& \qquad +\frac{1}{\mu (\Omega _{2})}\iiint_{\Omega _{1}\times \Omega
			_{1}\times \Omega _{2}}\left\vert w(y)-w(z)\right\vert d\mu (x)d\mu (y)d\mu
		(z) \\
		& =\frac{2\mu (\Omega _{1})}{\mu (\Omega _{2})}\iint_{\Omega _{1}\times
			\Omega _{2}}\left\vert w(x)-w(z)\right\vert d\mu (x)d\mu (z)=\frac{2\mu
			(\Omega _{1})}{\mu (\Omega _{2})}S_{1,2}\ ,
	\end{align*}%
	and similarly $S_{2,2}\leq \frac{2\mu (\Omega _{2})}{\mu (\Omega _{1})}%
	S_{1,2}$.
\end{proof}

We will apply this lemma with%
\begin{eqnarray*}
	\Omega _{1} &=&B_{+}=\left\{ \left( x,y\right) \in HB\left( 0,r_{0}\right)
	:r^{\ast }\leq x\leq r\right\} , \\
	\Omega _{2} &=&B_{-}=\left\{ \left( x,y\right) \in HB\left( 0,r_{0}\right)
	:0\leq x\leq r^{\ast }\right\} ,
\end{eqnarray*}%
where $r^{\ast }$, $B_{+}$ and $B_{-}$ are as in Lemma \ref{thick_part}
above. Then from Lemma \ref{thick_part} we have%
\begin{equation*}
	\left\vert \Omega _{1}\right\vert \approx \left\vert \Omega _{2}\right\vert
	\approx \left\vert B\left( 0,r_{0}\right) \right\vert .
\end{equation*}%
By Lemma \ref{division of regions}, the proof of Proposition \ref{1 1 Poin}
reduces to the following inequality: 
\begin{align}\label{toprove12}
	I&=\iint_{\Omega _{1}\times \Omega
		_{2}}|w(x_{1},y_{1})-w(x_{2},y_{2})|dx_{1}dy_{1}dx_{2}dy_{2}\\
	&\leq
	C|HB(0,r_{0})|r_{0}\iint_{HB(0,r_{0})}|\nabla _{A}w(x,y)|dxdy.\notag
\end{align}

Let $P_{1}=(x_{1},y_{1})\in \Omega _{1}$ and $P_{2}=(x_{2},y_{2})\in \Omega
_{2}$. We can connect $P_{1}$ and $P_{2}$ by first traveling vertically and
then horizontally. This integral path is completely contained in the half
metric ball. This immediately gives an inequality 
\begin{equation*}
	|w(x_{1},y_{1})-w(x_{2},y_{2})|\leq \left\vert
	\int_{y_{1}}^{y_{2}}w_{y}(x_{1},y)dy\right\vert +\left\vert
	\int_{x_{1}}^{x_{2}}w_{x}(x,y_{2})dx\right\vert .
\end{equation*}%
As a result, we have 
\begin{align*}
	I&=\iint_{\Omega _{1}\times \Omega
		_{2}}|w(x_{1},y_{1})-w(x_{2},y_{2})|dx_{1}dy_{1}dx_{2}dy_{2}\\
	& \leq
	\iint_{\Omega _{1}\times \Omega _{2}}\left\vert
	\int_{y_{1}}^{y_{2}}w_{y}(x_{1},y)dy\right\vert dx_{1}dy_{1}dx_{2}dy_{2} \\
	& \quad +\iint_{\Omega _{1}\times \Omega _{2}}\left\vert
	\int_{x_{1}}^{x_{2}}w_{x}(x,y_{2})dx\right\vert dx_{1}dy_{1}dx_{2}dy_{2} \\
	& =I_{1}+I_{2}
\end{align*}

We first estimate the integral 
\begin{align*}
	I_{1}&=\iint_{\Omega _{1}\times \Omega
		_{2}}|w(x_{1},y_{1})-w(x_{2},y_{2})|dx_{1}dy_{1}dx_{2}dy_{2}\\
	&\leq
	\iint_{\Omega _{1}\times \Omega _{2}}\left\vert
	\int_{y_{1}}^{y_{2}}w_{y}(x_{1},y)dy\right\vert dx_{1}dy_{1}dx_{2}dy_{2}
\end{align*}%
where $\Omega _{1}=B_{+}$ and $\Omega _{2}=B_{-}$. We have 
\begin{align*}
	I_{1}& \leq
	\int_{B_{-}}\int_{B_{+}}\!\!%
	\int_{y_{1}}^{y_{2}}|w_{y}(x_{1},y)|dydx_{1}dy_{1}dx_{2}dy_{2}\\
	&\leq
	\int_{B_{-}}\int_{B_{+}}\!\!\int_{y_{1}}^{y_{2}}\frac{1}{f(x_{1})}|\nabla
	_{A}w(x_{1},y)|dydx_{1}dy_{1}dx_{2}dy_{2} \\
	& \leq \int_{B_{-}}\int_{B_{+}}\frac{h(r)}{f(x_{1})}|\nabla
	_{A}w(x_{1},y)|dydx_{1}dx_{2}dy_{2},
\end{align*}%
where $h(r)\lesssim rf(r)$ is the \textquotedblleft maximal
height\textquotedblright\ given in Proposition \ref{height}. Moreover, for $%
x_{1}\in B_{+}$ we have $|r-x_{1}|\leq 1/|F^{\prime }(r)|$ and therefore $%
f(x_{1})\approx f(r)$. This gives 
\begin{equation*}
	\frac{h(r)}{f(x_{1})}\leq r,
\end{equation*}%
and substituting this into the above we get 
\begin{equation*}
	I_{1}\leq Cr|B_{-}|\int_{B_{+}}|\nabla _{A}w(x,y)|dxdy\leq
	Cr|B|\int_{B}|\nabla _{A}w(x,y)|dxdy.
\end{equation*}%
To estimate 
\begin{equation*}
	I_{2}=\iint_{\Omega _{1}\times \Omega _{2}}\left\vert
	\int_{x_{1}}^{x_{2}}w_{x}(x,y_{2})dx\right\vert dx_{1}dy_{1}dx_{2}dy_{2},
\end{equation*}%
we note that $|w_{x}(x,y_{2})|\leq |\nabla _{A}w(x,y_{2})|$, and therefore 
\begin{align*}
	I_{2}& \leq \iint_{B_{+}\times B_{-}}\left[ \int_{(x,y_{2})\in
		HB(0,r)}|\nabla _{A}w(x,y_{2})|dx\right] dx_{1}dy_{1}dx_{2}dy_{2} \\
	& \leq Cr|B_{+}|\int_{B}|\nabla _{A}w(x,y_{2})|dxdy_{2}\leq
	Cr|B|\int_{B}|\nabla _{A}w(x,y)|dxdy.
\end{align*}

This finishes the proof of inequality \eqref{toprove12}, and hence finishes
the proof of the Poincar\'{e} inequality in Proposition \ref{1 1 Poin}.

We now wish to extend this Poincar\'{e} inequality to hold for the full ball 
$B\left( 0,r\right) $ in Proposition \ref{1 1 Poin}. We cannot simply use
geodesics that connect the left end $E_{\func{left}}\left( 0,r\right) $
of the ball to the right end $E_{\func{right}}\left( 0,r\right) $ of the
ball and that also stay entirely within the ball $B\left( 0,r\right) $. The
problem is that the thin `neck' of the ball near the origin is too thin to
support such geodesics without compensating with a huge Jacobian. Instead we
will enlarge the ball $B\left( 0,r\right) $ enough so that the enlarged ball
contains the \emph{rectangle} $\left( -r,r\right) \times \left( -h\left(
r\right) ,h\left( r\right) \right) $. This can be achieved with the ball of
doubled radius as we now show.

\begin{lemma}
	\label{inclusion} For $0<r<\frac{R}{2}$, we have the inclusions, 
	\begin{equation*}
		B\left( 0,r\right) \subset \left( -r,r\right) \times \left( -h\left(
		r\right) ,h\left( r\right) \right) \subset \left( -r,r\right) \times \left(
		-2h\left( r\right) ,2h\left( r\right) \right) \subset B\left( 0,2r\right) .
	\end{equation*}
\end{lemma}

\begin{proof}
	The inclusion $B\left( 0,r\right) \subset \left( -r,r\right) \times \left(
	-h\left( r\right) ,h\left( r\right) \right) $ is immediate. Now consider the
	geodesic $\gamma \left( t\right) $ from $\gamma \left( 0\right) =\left(
	0,0\right) $ to the point $\gamma \left( r\right) =\left( r^{\ast },h\right)
	=\left( r^{\ast }\left( r\right) ,h\left( r\right) \right) $ on the boundary
	of the ball $B\left( 0,r\right) $ where $\gamma $ has a vertical tangent. If
	we continue this geodesic for a further time $r$, then by symmetry we curl
	back and return to the $y$-axis at the point $\gamma \left( 2r\right)
	=\left( 0,2h\left( r\right) \right) $. It is now clear by a further symmetry
	that $\left( -r,r\right) \times \left( -h\left( r\right) ,h\left( r\right)
	\right) \subset B\left( 0,2r\right) $.
\end{proof}

Now we can extend Proposition \ref{1 1 Poin} to the full ball.

\begin{proposition}
	\label{1 1 Poin'} Let the balls $B(0,r)$ and the degenerate gradient $\nabla
	_{A}$ be as above. There exists a constant $C$ such that the Poincar\'{e}
	Inequality 
	\begin{equation*}
		\int_{B(0,r)}\left\vert w(x)-\bar{w}\right\vert dx\leq
		Cr\int_{B(0,2r)}|\nabla _{A}w|dx
	\end{equation*}%
	holds for any Lipschitz function $w$ and sufficiently small $r>0$. Here $%
	\bar{w}$ is the average defined by 
	\begin{equation*}
		\bar{w}=\frac{1}{|B(0,r)|}\int_{B(0,r)}wdx.
	\end{equation*}
\end{proposition}

\begin{proof}
	Following Proposition \ref{1 1 Poin} we will denote the right half of the
	metric ball $B(0,r)$, by $HB(0,r)$. Recall that we have $HB(0,r)=B_{-}\cup
	B_{+}$ where 
	\begin{eqnarray*}
		B_{+} &=&\left\{ \left( x_{1},x_{2}\right) \in HB\left( 0,r_{0}\right)
		:r^{\ast }\leq x_{1}\leq r\right\} , \\
		B_{-} &=&\left\{ \left( x_{1},x_{2}\right) \in HB\left( 0,r_{0}\right)
		:0\leq x_{1}\leq r^{\ast }\right\} .
	\end{eqnarray*}%
	Similarly, we will denote the left half by $BH(0,r)$ and write $%
	BH(0,r)=B^{-}\cup B^{+}$ where 
	\begin{eqnarray*}
		B^{+} &=&\left\{ \left( x_{1},x_{2}\right) \in BH\left( 0,r_{0}\right)
		:-r\leq x_{1}\leq -r^{\ast }\right\} , \\
		B^{-} &=&\left\{ \left( x_{1},x_{2}\right) \in BH\left( 0,r_{0}\right)
		:-r^{\ast }\leq x_{1}\leq 0\right\} .
	\end{eqnarray*}%
	Now using Lemma \ref{division of regions} with $\Omega _{1}=BH(0,r)$ and $%
	\Omega _{2}=HB(0,r)$ we have 
	\begin{align*}
		\int_{B(0,r)}\left\vert w(x)-\bar{w}\right\vert dx&\leq \frac{1}{|B|}%
		\iint_{B\times B}\left\vert w(x)-w(y)\right\vert dxdy\\
		&\leq \frac{C}{|B|}%
		\iint_{BH\times HB}\left\vert w(x)-w(y)\right\vert dxdy.
	\end{align*}%
	Moreover, since we have $|B_{-}|=|B^{-}|\approx |B_{+}|=|B^{+}|\approx |B|$,
	proceeding the same way as in the proof of Lemma \ref{division of regions}
	we can show 
	\begin{align*}
		\iint_{BH\times HB}&\left\vert w(x)-w(y)\right\vert dxdy
		\leq C\left( \iint_{B^{-}\times B^{+}}\left\vert w(x)-w(y)\right\vert
		dxdy\right.\\
		&+\left.\iint_{B_{-}\times B_{+}}\left\vert w(x)-w(y)\right\vert
		dxdy+\iint_{B^{+}\times B_{+}}\left\vert w(x)-w(y)\right\vert dxdy\right) \\
		& \equiv I_{1}+I_{2}+I_{3}.
	\end{align*}%
	The estimate for $I_{2}$ follows from the proof of Proposition \ref{1 1 Poin}%
	, and the estimate for $I_{1}$ is shown in exactly the same way. We thus
	have 
	\begin{equation}
	I_{1}+I_{2}\leq Cr\int_{B(0,r)}|\nabla _{A}w|dx.  \label{I1_I2}
	\end{equation}%
	To estimate $I_{3}$ we connect two points $(x_{1},x_{2})\in B^{+}$ and $%
	(y_{1},y_{2})\in B_{+}$ by a curve consisting of one vertical and one
	horizontal segment, which according to Lemma \ref{inclusion} lies entirely
	in the ball $B(0,2r)$ 
	\begin{equation*}
		\left\vert w(x_{1},x_{2})-w(y_{1},y_{2})\right\vert \leq \left\vert
		\int_{x_{2}}^{y_{2}}w_{t}(x_{1},t)dt\right\vert +\left\vert
		\int_{x_{1}}^{y_{1}}w_{s}(s,y_{2})ds\right\vert .
	\end{equation*}%
	Next, proceeding as in the proof of Proposition \ref{1 1 Poin} we obtain 
	\begin{align*}
		& \iint_{B^{+}\times B_{+}}\left\vert
		\int_{x_{2}}^{y_{2}}w_{t}(x_{1},t)dt\right\vert dx_{1}dx_{2}dy_{1}dy_{2} \\
		& \leq |B_{+}|\int_{B^{+}}\frac{h(r)}{f(x_{1})}|\nabla
		_{A}w(x_{1},t)|dx_{1}dt\leq Cr|B|\int_{B}|\nabla _{A}w|dx,
	\end{align*}%
	where in the last inequality we used the fact that $x_{1}\in B^{+}$ and $%
	y_{1}\in B_{+}$, and therefore 
	\begin{equation*}
		\frac{h(r)}{f(x_{1})}\approx \frac{h(r)}{f(y_{1})}\approx \frac{rf(r)}{f(r)}%
		=Cr.
	\end{equation*}%
	Finally, for the integral along the horizontal segment we have 
	\begin{align*}
		\iint_{B^{+}\times B_{+}}&\left\vert
		\int_{x_{1}}^{y_{1}}w_{s}(s,y_{2})ds\right\vert dx_{1}dx_{2}dy_{1}dy_{2}\\
		&\leq
		Cr\int_{B^{+}}\int_{-h(r)}^{h(r)}\int_{-r}^{r}\left\vert
		w_{s}(s,y_{2})\right\vert dsdy_{2}dx_{1}dy_{2}\leq Cr|B|\int_{2B}|\nabla
		_{A}w|dx,
	\end{align*}%
	which gives 
	\begin{equation*}
		I_{3}\leq Cr\int_{B(0,2r)}|\nabla _{A}w|dx.
	\end{equation*}%
	Combining with (\ref{I1_I2}) finishes the proof.
\end{proof}

\subsection{Higher dimensional inequalities}

First, we state the following n-dimensional analog of Lemma \ref{inclusion}

\begin{lemma}
	\label{inclusion'} Define the set $Q(r)$ as follows 
	\begin{equation*}
		Q(r)\equiv\{(y_{1},\mathbf{y}_{2},y_{3}):\ |y_{1}|< \sqrt{r^{2}-|\mathbf{y}%
			_{2}|^{2}}, |y_{3}|<2h(\sqrt{r^{2}-|\mathbf{y}_{2}|^{2}}), |\mathbf{y}%
		_{2}|<r\}.
	\end{equation*}
	Then we have the following inclusion 
	\begin{equation*}
		B(0,r)\subset Q(r)\subset B(0,2r).
	\end{equation*}
\end{lemma}

The cross section $\mathbf{y}_{2}=const$ of the set $Q(r)$ is a rectangle 
\begin{equation*}
	(-\sqrt{r^{2}-|\mathbf{y}_{2}|^{2}},\sqrt{r^{2}-|\mathbf{y}_{2}|^{2}})\times
	(-2h(\sqrt{r^{2}-|\mathbf{y}_{2}|^{2}}),2h(\sqrt{r^{2}-|\mathbf{y}_{2}|^{2}}%
	))
\end{equation*}
in the $(x_{1},x_{3})$ plane.

Now, we first define the ``ends'' of n-dimensional balls, similar to the
sets $B_{+}$ and $B^{+}$ in 2 dimensions. As before we let $HB(0,r)=\{x\in
B(0,r):\ x_{1}>0\}$, $BH(0,r)=\{x\in B(0,r):\ x_{1}<0\}$ and

\begin{eqnarray*}
	B_{+}&=&\left\{ \left( x_{1},\mathbf{x}_{2},x_{3}\right) \in HB\left(
	0,r_{0}\right) :r^{\ast }\leq x_{1}\leq r\right\} , \\
	B_{-}&=&\left\{ \left( x_{1},\mathbf{x}_{2},x_{3}\right)\in HB\left(
	0,r_{0}\right) :0\leq x_{1}\leq r^{\ast }\right\}.
\end{eqnarray*}
Similarly, we will denote the left half by $BH(0,r)$ and write $%
BH(0,r)=B^{-}\cup B^{+}$ where 
\begin{eqnarray*}
	B^{+}&=&\left\{ \left( x_{1},\mathbf{x}_{2},x_{3}\right)\in BH\left(
	0,r_{0}\right) :-r\leq x_{1}\leq -r^{\ast }\right\} , \\
	B^{-}&=&\left\{ \left( x_{1},\mathbf{x}_{2},x_{3}\right)\in BH\left(
	0,r_{0}\right). :-r^{\ast }\leq x_{1}\leq 0\right\}.
\end{eqnarray*}

We need the following result analogous to Corollary \ref{thick_part} in two
dimensions

\begin{lemma}
	For the sets $B_{+}$, $B^{+}$, $B_{-}$, and $B^{-}$ defined above, we have 
	\begin{equation*}
		|B^{-}|=|B_{-}|\approx |B_{+}|=|B^{+}|\approx |B|.
	\end{equation*}
\end{lemma}

\begin{proof}
	The equalities are obvious from the symmetry, so we only need to show the
	approximate equalities. For convenience we will work with the right half of
	the ball and first prove that $|B_{+}|\approx |B|$. We have 
	\begin{align*}
		|B_{+}|& =\int\limits_{|\mathbf{y}_{2}|\leq r}\left\vert {B_{2D}}\left(
		\left( 0,\mathbf{0},0\right) ,\sqrt{r^{2}-\left\vert \mathbf{y}%
			_{2}\right\vert ^{2}}\right) \cap \{y_{1}>r^{\ast }\}\right\vert d\mathbf{y}%
		_{2}\\
		&\approx \int\limits_{|\mathbf{y}_{2}|\leq r}\left( \sqrt{%
			r^{2}-\left\vert \mathbf{y}_{2}\right\vert ^{2}}-r^{\ast }\right) _{+}\cdot
		h\left( \sqrt{r^{2}-\left\vert \mathbf{y}_{2}\right\vert ^{2}}\right) d%
		\mathbf{y}_{2} \\
		& \approx \int\limits_{|\mathbf{y}_{2}|\leq r}\left( \sqrt{r^{2}-\left\vert 
			\mathbf{y}_{2}\right\vert ^{2}}-r^{\ast }\right) _{+}\cdot \frac{f\left( 
			\sqrt{r^{2}-\left\vert \mathbf{y}_{2}\right\vert ^{2}}\right) }{\left\vert
			F^{\prime }\left( \sqrt{r^{2}-\left\vert \mathbf{y}_{2}\right\vert ^{2}}%
			\right) \right\vert }d\mathbf{y}_{2}\\
		&=\int\limits_{|\mathbf{y}_{2}|^{2}\leq
			r^{2}-{r^{\ast }}^{2}}\left( \sqrt{r^{2}-\left\vert \mathbf{y}%
			_{2}\right\vert ^{2}}-r^{\ast }\right) \cdot \frac{f\left( \sqrt{%
				r^{2}-\left\vert \mathbf{y}_{2}\right\vert ^{2}}\right) }{\left\vert
			F^{\prime }\left( \sqrt{r^{2}-\left\vert \mathbf{y}_{2}\right\vert ^{2}}%
			\right) \right\vert }d\mathbf{y}_{2} \\
		& \approx \frac{f(r)}{|F^{\prime }(r)|}\int\limits_{|\mathbf{y}_{2}|^{2}\leq
			r^{2}-{r^{\ast }}^{2}}\left( \sqrt{r^{2}-\left\vert \mathbf{y}%
			_{2}\right\vert ^{2}}-r^{\ast }\right) d\mathbf{y}_{2}
	\end{align*}%
	where for the last equality we used (\ref{new f equ}). Since we trivially
	have an upper bound $|B_{+}|\leq |B|$, we only need to obtain the lower
	bound. Passing to the polar coordinates, $\rho =\left\vert \mathbf{y}%
	_{2}\right\vert $, we have 
	\begin{align*}
		\int\limits_{|\mathbf{y}_{2}|^{2}\leq r^{2}-{r^{\ast }}^{2}}\left( \sqrt{%
			r^{2}-\left\vert \mathbf{y}_{2}\right\vert ^{2}}-r^{\ast }\right) d\mathbf{y}%
		_{2}& \approx \int\limits_{0}^{\sqrt{r^{2}-{r^{\ast }}^{2}}}\left( \sqrt{%
			r^{2}-\rho ^{2}}-r^{\ast }\right) \rho ^{n-3}d\rho \\
		& \geq \int\limits_{0}^{\sqrt{r^{2}-\frac{1}{4}(r+r^{\ast })^{2}}}\left( 
		\sqrt{r^{2}-\rho ^{2}}-r^{\ast }\right) \rho ^{n-3}d\rho \\
		& \geq \frac{1}{2}(r-r^{\ast })\int\limits_{0}^{\sqrt{r^{2}-\frac{1}{4}%
				(r+r^{\ast })^{2}}}\rho ^{n-3}d\rho \\
		&\approx (r-r^{\ast })^{\frac{n}{2}}r^{\frac{n}{2}-1}.
	\end{align*}%
	Using part (2) of Proposition \ref{height} we have $r-r^{\ast }\approx
	1/|F^{\prime }(r)|$ and therefore 
	\begin{equation*}
		|B_{+}|\gtrsim \frac{f(r)}{\left\vert F^{\prime }\left( r\right) \right\vert
			^{\frac{n}{2}+1}}r^{\frac{n}{2}-1}.
	\end{equation*}%
	Combining with the upper bound and Lemma \ref{new} we conclude 
	\begin{equation*}
		|B_{+}|\approx |B|.
	\end{equation*}%
	Now, for the lower bound $|B_{-}|\gtrsim |B|$ we note that $B(0,r^{\ast
	})\subset B_{-}(0,r)\cup B^{-}(0,r)$ and by Lemma \ref{new} 
	\begin{equation*}
		|B(0,r^{\ast })|\approx \frac{f(r^{\ast })}{\left\vert F^{\prime }\left(
			r^{\ast }\right) \right\vert ^{\frac{n}{2}+1}}{r^{\ast }}^{\frac{n}{2}-1}.
	\end{equation*}%
	Thus the only thing left to show is that $r^{\ast }\approx r$ for $x_{1}=0$.
	Indeed, combining the estimates from lemmas \ref{height 1} and \ref{height 2}
	with $x_{1}=0$ we get 
	\begin{equation*}
		f(r^{\ast })(r-r^{\ast })\leq h(r)\leq \frac{f(r^{\ast })}{\left\vert
			F^{\prime }\left( r^{\ast }\right) \right\vert },
	\end{equation*}%
	and thus 
	\begin{equation*}
		r^{\ast }+\frac{1}{\left\vert F^{\prime }\left( r^{\ast }\right) \right\vert 
		}\geq r.
	\end{equation*}%
	Using assumption (4) on the geometry $F$ this gives 
	\begin{equation*}
		r^{\ast }\left( 1+\frac{1}{\varepsilon }\right) \geq r,
	\end{equation*}%
	which together with the trivial bound $r^{\ast }\leq r$ concludes the proof.
\end{proof}

We are now ready to prove the n-dimensional $(1,1)$-Poincar\'{e} inequality.

\begin{proposition}
	\label{1 1 Poin''} Let the balls $B(0,r)$ and the degenerate gradient $%
	\nabla _{A}$ be as above. There exists a constant $C$ such that the Poincar\'{e} Inequality 
	\begin{equation*}
		\int_{B(0,r)}\left\vert w(x)-\bar{w}\right\vert dx\leq
		Cr\int_{B(0,2r)}|\nabla _{A}w|dx
	\end{equation*}%
	holds for any Lipschitz function $w$ and sufficiently small $r>0$. Here $%
	\bar{w}$ is the average defined by 
	\begin{equation*}
		\bar{w}=\frac{1}{|B(0,r)|}\int_{B(0,r)}wdx.
	\end{equation*}
\end{proposition}

\begin{proof}
	First note that using Lemma \ref{division of regions} and the symmetry of
	the problem it is sufficient to show 
	\begin{align*}
		I_{1}+I_{2}&\equiv \frac{1}{|B|}\left( \iint_{B_{-}\times B_{+}}\left\vert
		w(x)-w(y)\right\vert dxdy+\iint_{B^{+}\times B_{+}}\left\vert
		w(x)-w(y)\right\vert dxdy\right) \\
		&\leq Cr\int_{2B}|\nabla _{A}w|dx.
	\end{align*}%
	To estimate $I_{2}$ we connect two points $(x_{1},\mathbf{x}_{2},x_{3})\in
	B^{+}$ and $(y_{1},\mathbf{y}_{2},y_{3})\in B_{+}$, by straight segments
	connecting the following pairs of points 
	\begin{align*}
		(x_{1},\mathbf{x}_{2},x_{3})\ \ & \text{to}\ \ (x_{1},\mathbf{x}_{2},y_{3})
		\\
		(x_{1},\mathbf{x}_{2},y_{3})\ \ & \text{to}\ \ (x_{1},\mathbf{y}_{2},y_{3})
		\\
		(x_{1},\mathbf{y}_{2},y_{3})\ \ & \text{to}\ \ (y_{1},\mathbf{y}_{2},y_{3}).
	\end{align*}%
	Note that the curve described above lies entirely in in $Q(r)$ defined in
	Lemma \ref{inclusion'}. Moreover, since the first part of the path lies
	entirely in $B^{+}$ we have 
	\begin{align*}
		& \iint_{B^{+}\times B_{+}}\left\vert w(x_{1},\mathbf{x}_{2},x_{3})-w(x_{1},%
		\mathbf{x}_{2},y_{3})\right\vert dx_{1}d\mathbf{x}_{2}dx_{3}dy_{1}d\mathbf{y}%
		_{2}dy_{3} \\
		& =\iint_{B^{+}\times B_{+}}\left\vert \int_{x_{3}}^{y_{3}}w_{t}(x_{1},%
		\mathbf{x}_{2},t)dt\right\vert dx_{1}d\mathbf{x}_{2}dx_{3}dy_{1}d\mathbf{y}%
		_{2}dy_{3} \\
		& \leq |B_{+}|\int_{B^{+}}\frac{h(r)}{f(x_{1})}|\nabla _{A}w(x_{1},\mathbf{x}%
		_{2},t)|dx_{1}d\mathbf{x}_{2}dt\leq Cr|B|\int_{B}|\nabla _{A}w|dx,
	\end{align*}%
	where just as in the 2-dimensional case we used 
	\begin{equation*}
		\frac{h(r)}{f(x_{1})}\approx \frac{rf(r)}{f(r)}=Cr,
	\end{equation*}%
	since $-r\leq x_{1}\leq -r^{\ast }$ in $B_{+}$. For the other two parts of
	the path, the estimates are the same as in the proof of the classical $(1,1)$
	Poincar\'{e} in dimensions $n-2$ and $1$ respectively, so we obtain 
	\begin{align*}
		\iint_{B^{+}\times B_{+}}&\left\vert w(x_{1},\mathbf{x}_{2},y_{3})-w(x_{1},%
		\mathbf{y}_{2},y_{3})\right\vert dx_{1}d\mathbf{x}_{2}dx_{3}dy_{1}d\mathbf{y}%
		_{2}dy_{3}\\
		& \leq Cr|B|\int_{2B}|\nabla _{\mathbf{x}_{2}}w|dx\leq
		Cr|B|\int_{2B}|\nabla _{A}w|dx \\
		\iint_{B^{+}\times B_{+}}&\left\vert w(x_{1},\mathbf{y}_{2},y_{3})-w(y_{1},%
		\mathbf{y}_{2},y_{3})\right\vert dx_{1}d\mathbf{x}_{2}dx_{3}dy_{1}d\mathbf{y}%
		_{2}dy_{3}\\
		& \leq Cr|B|\int_{2B}|\partial _{x_{1}}w|dx\leq
		Cr|B|\int_{2B}|\nabla _{A}w|dx.
	\end{align*}%
	This concludes the estimate for $I_{2}$. To estimate $I_{1}$ we similarly
	connect points in $B_{-}$ and $B_{+}$ by first moving in the first and
	second variables to reach the set $B_{+}$ and then going \textquotedblleft
	vertically\textquotedblright\ to connect $x_{3}$ and $y_{3}$. The proof is
	similar to the 2-dimensional case and is left to the reader.
\end{proof}

\section{Orlicz Sobolev inequalities for bump functions $\Phi _{N}$\label%
	{Orlicz bump}}

Recall that the relevant bump functions $\Phi _{N}$ are given in (\ref{def
	Phi N ext}) by%
\begin{equation*}
	\Phi _{N}\left( t\right) \equiv \left\{ 
	\begin{array}{ccc}
		t(\ln t)^{N} & \text{ if } & t\geq E=E_{N}=e^{2N} \\ 
		\left( \ln E\right) ^{N}t & \text{ if } & 0\leq t\leq E=E_{N}=e^{2N}%
	\end{array}%
	\right. .
\end{equation*}
Next, define the positive operator $T_{B\left( 0,r_{0}\right) }:L^{1}\left(
\mu _{r_{0}}\right) \rightarrow L^{\Phi }\left( \mu _{r_{0}}\right) $ by 
\begin{equation*}
	T_{B\left( 0,r_{0}\right) }g(x)\equiv \int_{B(0,r_{0})}K_{B\left(
		0,r_{0}\right) }\left( x,y\right) g(y)dy
\end{equation*}%
where $d\mu _{r_{0}}\equiv \frac{dxdy}{|B(0,r_{0})|}$, with kernel $%
K_{B\left( 0,r_{0}\right) }$ defined by%
\begin{equation*}
	K_{B\left( 0,r_{0}\right) }\left( x,y\right) =\frac{\widehat{d}\left(
		x,y\right) }{\left\vert B\left( x,d\left( x,y\right) \right) \right\vert }%
	\mathbf{1}_{\widetilde{\Gamma }\left( r_{0}\right) }\left( x,y\right) ,
\end{equation*}%
where $\widetilde{\Gamma }$ is given by (\ref{def Gamma tilda}), and 
\begin{equation}
\widehat{d}\left( x,y\right) \equiv \min \left\{ d\left( x,y\right) ,\frac{1%
}{\left\vert F^{\prime }\left( x_{1}+d\left( x,y\right) \right) \right\vert }%
\right\} .  \label{def d hat}
\end{equation}

We will prove the strong form of the norm inequality%
\begin{equation}
\left\Vert T_{B\left( 0,r_{0}\right) }g\right\Vert _{L^{\Phi }\left( \mu
	_{r_{0}}\right) }\leq C\varphi \left( r_{0}\right) \ \left\Vert g\right\Vert
_{L^{1}\left( \mu _{r_{0}}\right) },  \label{Phi bump norm}
\end{equation}%
which in turn implies the norm inequality\ 
\begin{equation}
\left\Vert w\right\Vert _{L^{\Phi }\left( \mu _{r_{0}}\right) }\leq C\varphi
\left( r_{0}\right) \ \left\Vert \nabla _{A}w\right\Vert _{L^{1}\left( \mu
	_{r_{0}}\right) },\ \ \ \ \ w\in \left( W_{A}^{1,1}\right) _{0}\left(
B\left( 0,r_{0}\right) \right)  \label{Phi N norm}
\end{equation}%
by the subrepresentation inequality from Lemma \ref%
{lemma-subrepresentation''} with $g=\nabla _{A}w$. Indeed, we even have a
version of (\ref{Phi N norm}) for each of the half balls $HB\left(
0,r_{0}\right) =H_{\func{right}}B\left( 0,r_{0}\right) $ and $H_{\func{%
		left}}B\left( 0,r_{0}\right) $.

\begin{definition}
	Define $\left( W_{A}^{1,1}\right) _{0}\left( HB\left( 0,r_{0}\right) \right) 
	$ to be the $W_{A}^{1,1}$-closure of those Lipschitz functions $w$ in $%
	HB\left( 0,r_{0}\right) $ that vanish in a Euclidean neighbourhood of the
	compact set $\left\{ x\in \partial B\left( 0,r_{0}\right) :x_{1}\geq
	0\right\} $, and similarly for $\left( W_{A}^{1,1}\right) _{0}\left( H_{%
		\func{left}}B\left( 0,r_{0}\right) \right) $.
\end{definition}

\begin{lemma}
	\label{norm half}Assume that the strong form of the norm inequality (\ref%
	{Phi bump norm}) holds. Then the standard form of the norm inequality (\ref%
	{Phi N norm}) holds, and moreover, we also have the halfball inequalities%
	\begin{equation}
	\left\Vert w\right\Vert _{L^{\Phi }\left( HB_{\func{right}/\func{left}%
		}\left( 0,r_{0}\right) ,\mu _{r_{0}}\right) }\leq C\varphi \left(
	r_{0}\right) \ \left\Vert \nabla _{A}w\right\Vert _{L^{1}\left( HB_{\func{%
				right}/\func{left}}\left( 0,r_{0}\right) ,\mu _{r_{0}}\right) },\label{Phi N norm half}
	\end{equation}%
	$w\in \left( W_{A}^{1,1}\right) _{0}\left( HB_{\func{right}/\func{left}%
	}\left( 0,r_{0}\right) \right)$, where we take the same choice of $H_{\func{right}}B\left( 0,r_{0}\right) $
	or $H_{\func{left}}B\left( 0,r_{0}\right) $ on both sides of the
	inequality (\ref{Phi N norm half}).
\end{lemma}

\begin{proof}
	Given a radius $r_{0}>0$, choose $r_{-1}>r_{0}$ so that $r_{0}=\left(
	r_{-1}\right) ^{\ast }$. Then extend $w\in W_{0}^{1,1}\left( B\left(
	0,r_{0}\right) \right) $ to $w\in W_{0}^{1,1}\left( B\left( 0,r_{-1}\right)
	\right) $ by defining $w$ to vanish outside $B\left( 0,r_{0}\right) $. Now
	for each $x$ in the smaller halfball $HB\left( 0,r_{0}\right) $, we have 
	\begin{equation*}
		E\left( x,s\right) \cap HB\left( 0,r_{-1}\right) \subset E\left(
		0,r_{-1}\right) \text{ and }\left\vert E\left( x,s\right) \cap HB\left(
		0,r_{-1}\right) \right\vert \approx \left\vert E\left( 0,r_{-1}\right)
		\right\vert
	\end{equation*}%
	for an appropriate $s>0$. Now we apply Lemma \ref{lemma-subrepresentation''}
	in the larger halfball $HB\left( 0,r_{-1}\right) $, together with the fact
	that $w$ vanishes on the end $E\left( x,s\right) $, to conclude that 
	\begin{equation*}
		\left\vert w\left( x\right) \right\vert =\left\vert w\left( x\right)
		-E\left( x,s\right) \right\vert \lesssim T_{B\left( 0,r_{-1}\right) }\left( 
		\mathbf{1}_{HB\left( 0,r_{-1}\right) }\left\vert \nabla _{A}w\right\vert
		\right) \left( x\right) ,
	\end{equation*}%
	for $x\in HB\left( 0,r_{-1}\right) $. Thus from (\ref{Phi bump norm}) we
	obtain%
	\begin{eqnarray*}
		\left\Vert w\right\Vert _{L^{\Phi }\left( HB\left( 0,r_{-1}\right) ,\mu
			_{r_{-1}}\right) } &\leq &\left\Vert T_{B\left( 0,r_{-1}\right) }\left( 
		\mathbf{1}_{HB\left( 0,r_{-1}\right) }\left\vert \nabla _{A}w\right\vert
		\right) \right\Vert _{L^{\Phi }\left( HB\left( 0,r_{-1}\right) ,\mu
			_{r_{-1}}\right) } \\
		&\leq &C\varphi \left( r_{-1}\right) \ \left\Vert \left\vert \nabla
		_{A}w\right\vert \right\Vert _{L^{1}\left( HB\left( 0,r_{-1}\right) ,\mu
			_{r_{-1}}\right) }.
	\end{eqnarray*}%
	This gives the case of \ref{Phi N norm} for the halfball $HB\left(
	0,r_{0}\right) $ upon noting that both $w$ and $\left\vert \nabla
	_{A}w\right\vert $ vanish outside $B\left( 0,r_{0}\right) $, and $%
	r_{0}=\left( r_{-1}\right) ^{\ast }$, and so easy estimates show that we
	actually have $\varphi \left( r_{-1}\right) \approx \varphi \left(
	r_{0}\right) $ and both%
	\begin{eqnarray*}
		\left\Vert w\right\Vert _{L^{\Phi }\left( HB\left( 0,r_{-1}\right) ,\mu
			_{r_{-1}}\right) } &\approx &\left\Vert w\right\Vert _{L^{\Phi }\left(
			HB\left( 0,r_{0}\right) ,\mu _{r_{0}}\right) }\ , \\
		\left\Vert \left\vert \nabla _{A}w\right\vert \right\Vert _{L^{1}\left(
			HB\left( 0,r_{-1}\right) ,\mu _{r_{-1}}\right) } &\approx &\left\Vert
		\left\vert \nabla _{A}w\right\vert \right\Vert _{L^{1}\left( HB\left(
			0,r_{0}\right) ,\mu _{r_{0}}\right) }\ .
	\end{eqnarray*}
\end{proof}

We begin by proving that the bound (\ref{Phi bump norm}) holds if the
following endpoint inequality holds:%
\begin{equation}
\Phi ^{-1}\left( \sup_{y\in B}\int_{B}\Phi \left( K(x,y)|B|\alpha \right)
d\mu (x)\right) \leq C\alpha \varphi \left( r\right) \ .
\label{endpoint' inhom}
\end{equation}%
for all $\alpha >0$.

\begin{lemma}
	The endpoint inequality (\ref{endpoint' inhom}) implies the norm inequality (%
	\ref{Phi N norm}).
\end{lemma}

\begin{proof}
	If (\ref{endpoint' inhom}) holds, then Jensen's inequality applied to the
	convex function $\Phi $ gives 
	\begin{align*}
		\int_{B}&\Phi \left( \frac{T_{B\left( 0,r_{0}\right) }g}{C_{1}\varphi \left(
			r_{0}\right) \ \left\Vert g\right\Vert _{L^{1}\left( \mu \right) }}\right)
		d\mu (x) \\
		&\quad\lesssim \int_{B}\Phi \left( \int_{B}K(x,y)\ |B|\ \frac{1}{%
			C_{1}\varphi \left( r_{0}\right) }\frac{g\left( y\right) d\mu \left(
			y\right) }{||g||_{L^{1}(\mu )}}\right) d\mu (x) \\
		&\quad\leq \int_{B}\int_{B}\Phi \left( K(x,y)\ |B|\ \frac{1}{C_{1}\varphi \left(
			r_{0}\right) }\right) \frac{g\left( y\right) d\mu \left( y\right) }{%
			||g||_{L^{1}(\mu )}}d\mu (x) \\
		&\quad\leq \int_{B}\left\{ \sup_{y\in B}\int_{B}\Phi \left( K(x,y)\ |B|\ \frac{1%
		}{C_{1}\varphi \left( r_{0}\right) }\right) d\mu (x)\right\} \frac{g\left(
			y\right) d\mu \left( y\right) }{||g||_{L^{1}(\mu )}} \\
		&\quad\leq \Phi \left( \frac{C}{C_{1}}\right) \int_{B}\frac{g\left( y\right)
			d\mu \left( y\right) }{||g||_{L^{1}(\mu )}}\leq 1,
	\end{align*}%
	for $C_{1}$ sufficiently large, and where we used (\ref{endpoint' inhom})
	with $\alpha =\frac{1}{C_{1}\varphi (r_{0})}\equiv \frac{\Phi ^{-1}(1)}{%
		C\varphi (r_{0})}$. We conclude from the definition of $L^{\Phi }\left( \mu
	_{r_{0}}\right) $ that%
	\begin{equation*}
		\left\Vert T_{B\left( 0,r_{0}\right) }g\right\Vert _{L^{\Phi }\left( \mu
			_{r_{0}}\right) }\leq C_{1}\varphi \left( r_{0}\right) \ \left\Vert
		g\right\Vert _{L^{1}\left( \mu _{r_{0}}\right) }.
	\end{equation*}
\end{proof}

\begin{proposition}
	\label{sob_nd} Let $n\geq 2$. Assume that for some $C>0$ the function 
	\begin{equation}
	\varphi (r)\equiv C|F^{\prime }\left( r\right) |^{N}r^{N+1}
	\label{phi_incr N}
	\end{equation}%
	satisfies $\lim_{r\rightarrow 0}\varphi (r)=0$. Assume in addition that
	geometry $F$ satisfies 
	\begin{equation}
	F^{\prime \prime }(r)\leq \left( 1+\frac{1-\varepsilon }{N}\right) \frac{%
		|F^{\prime }(r)|}{r},\quad r\in (0,r_{0}),\quad \varepsilon >0.
	\label{geom_extra_cond}
	\end{equation}%
	Then:
	
	\begin{enumerate}
		\item the $\left( \Phi ,\varphi \right) $-Sobolev inequality (\ref{Phi bump
			norm}) holds with geometry $F$, with $\varphi $ as in (\ref{phi_incr N}),
		and with $\Phi $ as in (\ref{def Phi N ext}), $N>1$,
		
		\item and if $\varphi _{\max }\left( r\right) \equiv \sup_{0<s<r_{0}}\varphi
		(s)<\infty $ is a finite constant function, then the $\left( \Phi ,\varphi
		_{\max }\right) $-Sobolev inequality (\ref{Phi bump norm}) holds with
		geometry $F$, with $\varphi $ as in (\ref{phi_incr N}), and with $\Phi $ as
		in (\ref{def Phi N ext}), $N>1$,
		
		\item in particular, if for some $\varepsilon >0$ we have 
		\begin{equation}
		\left\vert F^{\prime }\left( r\right) \right\vert \leq C\left( \frac{1}{r}%
		\right) ^{1+\frac{1-\varepsilon }{N}},  \label{F_prime_bound N}
		\end{equation}%
		then the $\left( \Phi ,\varphi _{\max }\right) $-Sobolev inequality (\ref%
		{Phi bump norm}) holds with geometry $F$ and $\varphi _{\max }(r)\equiv C$.
	\end{enumerate}
\end{proposition}

\begin{proof}
	For Part (1) it suffices to prove the endpoint inequality 
	\begin{equation}
	\Phi ^{-1}\left( \sup_{y\in B}\int_{B}\Phi \left( K(x,y)|B|\alpha \right)
	d\mu (x)\right) \leq C\alpha \varphi (r(B))\ ,\ \ \ \ \ \alpha >0.
	\label{endpoint'}
	\end{equation}%
	for the balls and kernel associated with our geometry $F$, the Orlicz bump $%
	\Phi $, and the function $\varphi \left( r\right) $ satisfying (\ref%
	{phi_incr N}). Fix parameters $N>1$ and $t_{N}>1$. Now we consider the
	specific function $\omega \left( r\left( B\right) \right) $ given by%
	\begin{equation*}
		\omega \left( r\left( B\right) \right) =\frac{1}{t_{N}\left\vert F^{\prime
			}\left( r\left( B\right) \right) \right\vert }.
	\end{equation*}%
	Using the submultiplicativity of $\Phi $ we have%
	\begin{eqnarray*}
		\int_{B}\Phi \left( K(x,y)|B|\alpha \right) d\mu (x) &=&\int_{B}\Phi \left( 
		\frac{K(x,y)|B|}{\omega \left( r\left( B\right) \right) }\alpha \omega
		\left( r\left( B\right) \right) \right) d\mu (x) \\
		&\leq &\Phi \left( \alpha \omega \left( r\left( B\right) \right) \right)
		\int_{B}\Phi \left( \frac{K(x,y)|B|}{\omega \left( r\left( B\right) \right) }%
		\right) d\mu (x)
	\end{eqnarray*}%
	and we will now prove%
	\begin{equation}
	\int_{B}\Phi \left( \frac{K(x,y)|B|}{\omega \left( r\left( B\right) \right) }%
	\right) d\mu (x)\leq C_{N}\varphi (r(B))\left\vert F^{\prime }\left( r\left(
	B\right) \right) \right\vert ,  \label{will prove}
	\end{equation}%
	for all small balls $B$ of radius $r\left( B\right) $ centered at the
	origin. Altogether this will give us%
	\begin{equation*}
		\int_{B}\Phi \left( K(x,y)|B|\alpha \right) d\mu (x)\leq C_{N}\varphi
		(r(B))\left\vert F^{\prime }\left( r\left( B\right) \right) \right\vert \Phi
		\left( \frac{\alpha }{t_{N}\left\vert F^{\prime }\left( r\left( B\right)
			\right) \right\vert }\right) .
	\end{equation*}
	
	Now we note that $x\Phi \left( y\right) =xy\frac{\Phi \left( y\right) }{y}%
	\leq xy\frac{\Phi \left( xy\right) }{xy}=\Phi \left( xy\right) $ for $x\geq
	1 $ since $\frac{\Phi \left( t\right) }{t}$ is monotone increasing. But from
	(\ref{phi_incr N}) and assumptions $(1)$ and $(4)$ on the geometry $F$ we
	have $\varphi (r)\left\vert F^{\prime }\left( r\right) \right\vert \gg 1$
	and so%
	\begin{align*}
		\int_{B}\Phi \left( K(x,y)|B|\alpha \right) d\mu (x)&\leq \Phi \left(
		C_{N}\varphi (r(B))\left\vert F^{\prime }\left( r\left( B\right) \right)
		\right\vert \alpha \frac{1}{t_{N}\left\vert F^{\prime }\left( r\left(
			B\right) \right) \right\vert }\right) \\
		&=\Phi \left( \frac{C_{N}}{t_{N}}\alpha
		\varphi (r(B))\right) ,
	\end{align*}%
	which is (\ref{endpoint'}) with $C=\frac{C_{N}}{t_{N}}$. Thus it remains to
	prove (\ref{will prove}).
	
	So we now take $B=B\left( 0,r_{0}\right) $ with $r_{0}\ll 1$ so that $\omega
	\left( r\left( B\right) \right) =\omega \left( r_{0}\right) $. First, recall 
	\begin{equation*}
		\left\vert B\left( 0,r_{0}\right) \right\vert \approx \frac{f(r_{0})}{%
			|F^{\prime }(r_{0})|^{n}}\lambda \left( 0,r_{0}\right) ^{n-2},
	\end{equation*}%
	where we recall from (\ref{def lambda}) that%
	\begin{equation*}
		\lambda \left( x_{1},r\right) \equiv \sqrt{r\left\vert F^{\prime }\left(
			x_{1}+r\right) \right\vert },
	\end{equation*}%
	and we now denote the size of the kernel $K_{B\left( 0,r_{0}\right) }(x,y)$
	as $\frac{1}{s_{y_{1}-x_{1}}}$ where 
	\begin{equation*}
		\frac{1}{s_{y_{1}-x_{1}}}\equiv 
		\left\{
		\begin{array}{cc}
			\frac{1}{r^{n-1}f\left( x_{1}\right) },& 0 <r=y_{1}-x_{1}<\frac{2}{%
				|F^{\prime }(x_{1})|} \\
			\frac{|F^{\prime }(y_{1})|^{n}}{f(y_{1})\left( r\left\vert F^{\prime }\left(
				y_{1}\right) \right\vert \right) ^{\frac{n-2}{2}}},&0
			<r=y_{1}-x_{1}\geq \frac{2}{|F^{\prime }(x_{1})|}
		\end{array}
		\right.	.
	\end{equation*}%
	We are writing the size of $\frac{1}{K_{B\left( 0,r_{0}\right) }(x,y)}$ as $%
	s_{y_{1}-x_{1}}=s_{r}$ since\ the quantity $s_{r}$ can be, roughly speaking,
	thought of a cross sectional volume analogous to the height $h_{r}$ in the
	two dimensional case.
	
	Next, write $\Phi (t)$ as 
	\begin{equation}
	\Phi (t)=t\Psi (t),\ \ \ \ \ \text{for }t>0,  \label{def psiD N}
	\end{equation}%
	where for $t\geq E$, 
	\begin{eqnarray*}
		t\Psi (t) &=&\Phi (t)=t\left( \ln t\right) ^{N} \\
		&\Longrightarrow &\Psi (t)=\left( \ln t\right) ^{N},
	\end{eqnarray*}%
	and for $t<E$,%
	\begin{eqnarray*}
		t\Psi (t) &=&\Phi (t)=t\left( \ln E\right) ^{N} \\
		&\Longrightarrow &\Psi (t)=\left( \ln E\right) ^{N}.
	\end{eqnarray*}
	
	Now temporarily fix $y=\left( y_{1},\mathbf{y}_{2},y_{3}\right) \in
	B_{+}\left( 0,r_{0}\right) \equiv \left\{ x\in B\left( 0,r_{0}\right)
	:x_{1}>0\right\} $. Using the definition of $\widetilde{\Gamma }\left(
	x,r_{0}\right) $ in (\ref{def Gamma tilda}) above, we have for $0<a<b<r_{0}$
	that%
	\begin{align*}
		&\mathcal{I}_{a,b}\left( y\right) \equiv \int_{\left\{ x\in B_{+}\left(
			0,r_{0}\right) :a\leq y_{1}-x_{1}\leq b\right\} \cap \widetilde{\Gamma }%
			^{\ast }\left( y,r_{0}\right) }\Phi \left( K_{B\left( 0,r_{0}\right) }\left(
		x,y\right) \frac{\left\vert B\left( 0,r_{0}\right) \right\vert }{\omega
			\left( r_{0}\right) }\right) \frac{dx}{\left\vert B\left( 0,r_{0}\right)
			\right\vert } \\
		&=\sum_{k:\ a\leq r_{k+1}<r_{k}\leq b}\int_{y_{1}-r_{k}}^{y_{1}-r_{k+1}}\\ 
		&\ \times\left[ \int_{\left\vert \mathbf{x}_{2}-\mathbf{y}_{2}\right\vert \leq \sqrt{%
				r_{k}^{2}-r_{k+1}^{2}}}\left[ \int_{y_{3}-h^{\ast }\left(
			x_{1},r_{k}\right) }^{y_{3}+h^{\ast }\left( x_{1},r_{k}\right) }\Phi \left( 
		\frac{1}{s_{y_{1}-x_{1}}}\frac{%
			\left\vert B\left( 0,r_{0}\right) \right\vert }{\omega \left( r_{0}\right) }%
		\right) dx_{3}\right] d\mathbf{x}_{2}\right] \frac{dx_{1}}{\left\vert
			B\left( 0,r_{0}\right) \right\vert } \\
		&=\sum_{k:\ a\leq r_{k+1}<r_{k}\leq b}\int_{y_{1}-r_{k}}^{y_{1}-r_{k+1}} \\
		&\ \times\left[ 4\left( \sqrt{r_{k}^{2}-r_{k+1}^{2}}\right) ^{n-2}h^{\ast }\left(
		x_{1},r_{k}\right) \right] \Phi \left( \frac{1}{s_{y_{1}-x_{1}}}\frac{%
			\left\vert B\left( 0,r_{0}\right) \right\vert }{\omega \left( r_{0}\right) }%
		\right) \frac{dx_{1}}{\left\vert B\left( 0,r_{0}\right) \right\vert } \\
		&\approx \int_{y_{1}-b}^{y_{1}-a}s_{y_{1}-x_{1}}\Phi \left( \frac{1}{%
			s_{y_{1}-x_{1}}}\frac{\left\vert B\left( 0,r_{0}\right) \right\vert }{\omega
			\left( r_{0}\right) }\right) \frac{dx_{1}}{\left\vert B\left( 0,r_{0}\right)
			\right\vert } \\
		&=\int_{y_{1}-b}^{y_{1}-a}s_{y_{1}-x_{1}}\left( \frac{1}{s_{y_{1}-x_{1}}}%
		\frac{\left\vert B\left( 0,r_{0}\right) \right\vert }{\omega \left(
			r_{0}\right) }\right) \Psi \left( \frac{1}{s_{y_{1}-x_{1}}}\frac{\left\vert
			B\left( 0,r_{0}\right) \right\vert }{\omega \left( r_{0}\right) }\right) 
		\frac{dx_{1}}{\left\vert B\left( 0,r_{0}\right) \right\vert },
	\end{align*}%
	where the approximation in the fourth line above comes from the estimates
	from Lemma \ref{new} and Lemma \ref{B and E} 
	\begin{align*}
		\left( \sqrt{r_{k}^{2}-r_{k+1}^{2}}\right) ^{n-2}h^{\ast }\left(
		x_{1},r_{k}\right) \left( r_{k}-r_{k+1}\right) &=\left\vert \widetilde{E}%
		\left( x,r_{k}\right) \right\vert \approx \left\vert B\left( x,r_{k}\right)
		\right\vert \\
		&\approx s_{r_{k}}\left( r_{k}-r_{k+1}\right) ,
	\end{align*}%
	and
	\begin{equation*}
		s_{r_{k}} \approx s_{y_{1}-x_{1}},\ \ \ \ \ \text{for }r_{k+1}\leq
		y_{1}-x_{1}<r_{k}\ .
	\end{equation*}
	Thus we have%
	\begin{eqnarray*}
		\mathcal{I}_{a,b}\left( y\right) &\approx &\frac{1}{\omega \left(
			r_{0}\right) }\int_{y_{1}-b}^{y_{1}-a}\Psi \left( \frac{1}{s_{y_{1}-x_{1}}}%
		\frac{\left\vert B\left( 0,r_{0}\right) \right\vert }{\omega \left(
			r_{0}\right) }\right) dx_{1} \\
		&=&\frac{1}{\omega \left( r_{0}\right) }\int_{a}^{b}\Psi \left( \frac{1}{%
			s_{y_{1}-x_{1}}}\frac{\left\vert B\left( 0,r_{0}\right) \right\vert }{\omega
			\left( r_{0}\right) }\right) dr.
	\end{eqnarray*}%
	and so 
	\begin{eqnarray*}
		&&\int_{B_{+}\left( 0,r_{0}\right) }\Phi \left( K_{B\left( 0,r_{0}\right)
		}\left( x,y\right) \frac{\left\vert B\left( 0,r_{0}\right) \right\vert }{%
			\omega \left( r_{0}\right) }\right) \frac{dx}{\left\vert B\left(
			0,r_{0}\right) \right\vert } \\
		&=&\mathcal{I}_{0,y_{1}}\left( y\right) \\
		&=&\frac{2}{\omega \left( r_{0}\right) }\int_{0}^{y_{1}}\Psi \left( \frac{1}{%
			s_{r}}\frac{\left\vert B\left( 0,r_{0}\right) \right\vert }{\omega \left(
			r_{0}\right) }\right) dr\ .
	\end{eqnarray*}
	
	It thus suffices to prove%
	\begin{equation}
	\mathcal{I}_{0,y_{1}}=\frac{1}{\omega \left( r_{0}\right) }%
	\int_{0}^{y_{1}}\Psi \left( \frac{1}{s_{r}}\frac{\left\vert B\left(
		0,r_{0}\right) \right\vert }{\omega \left( r_{0}\right) }\right) dr\leq
	C_{N}\varphi \left( r_{0}\right) \left\vert F^{\prime }\left( r_{0}\right)
	\right\vert \ ,  \label{the integral'Dn}
	\end{equation}%
	where $\left\vert B\left( 0,r_{0}\right) \right\vert $ is now the Lebesgue
	measure of the $n$-dimensional ball $B\left( 0,r_{0}\right) =B_{nD}\left(
	0,r_{0}\right) $.
	
	To prove this we divide the interval $\left( 0,y_{1}\right) $ of integration
	in $r$ into three regions as before:
	
	(\textbf{1}): the small region $\mathcal{S}$ where $\frac{|B(0,r_{0})|}{%
		s_{r}\varphi \left( r_{0}\right) }\leq E$,
	
	(\textbf{2}): the big region $\mathcal{R}_{1}$ that is disjoint from $%
	\mathcal{S}$ and where $r=y_{1}-x_{1}<\frac{2}{\left\vert F^{\prime }\left(
		x_{1}\right) \right\vert }$ and
	
	(\textbf{3}): the big region $\mathcal{R}_{2}$ that is disjoint from $%
	\mathcal{S}$ and where $r=y_{1}-x_{1}\geq \frac{2}{\left\vert F^{\prime
		}\left( x_{1}\right) \right\vert }$.
	
	In the small region $\mathcal{S}$ we use that $\Phi $ is linear on $\left[
	0,E\right] $\ to obtain that the integral in the right hand side of (\ref%
	{the integral'Dn}), when restricted to those $r\in \left( 0,r_{0}\right) $
	for which $\frac{|B(0,r_{0})|}{s_{r}\omega \left( r_{0}\right) }\leq E$, is
	bounded by 
	\begin{eqnarray*}
		&&\frac{1}{\omega \left( r_{0}\right) }\int_{0}^{r_{0}}\Psi \left( \frac{%
			\left\vert B\left( 0,r_{0}\right) \right\vert }{s_{r}\omega \left(
			r_{0}\right) }\right) dr \\
		&=&\frac{1}{\omega \left( r_{0}\right) }\int_{0}^{r_{0}}\frac{\Phi (E)}{E}dr=%
		\frac{1}{\omega \left( r_{0}\right) }\left( \ln E\right) ^{N}r_{0} \\
		&\leq &C\ t_{N}\ r_{0}\left\vert F^{\prime }\left( r_{0}\right) \right\vert
		\leq C_{N}\varphi (r_{0})\left\vert F^{\prime }\left( r_{0}\right)
		\right\vert \ ,
	\end{eqnarray*}%
	since $\omega \left( r_{0}\right) =\frac{1}{t_{N}\left\vert F^{\prime
		}\left( r_{0}\right) \right\vert }$, and for the last inequality we used $%
	r_{0}\leq \varphi (r_{0})$ which follows from (\ref{phi_incr N}) and
	assumtion $(1)$ on the geometry.
	
	We now turn to the first big region $\mathcal{R}_{1}$ where we have $%
	s_{y_{1}-x_{1}}\approx r^{n-1}f(x_{1})$. We have using the definition of $%
	\Psi $ (\ref{def psiD N}) 
	\begin{equation*}
		\int_{\mathcal{R}_{1}}\Phi \left( K_{B\left( 0,r_{0}\right) }\left(
		x,y\right) \frac{\left\vert B\left( 0,r_{0}\right) \right\vert }{\omega
			\left( r_{0}\right) }\right) \frac{dy}{\left\vert B\left( 0,r_{0}\right)
			\right\vert }\lesssim \frac{1}{\omega \left( r_{0}\right) }%
		\int\limits_{0}^{y_{1}}\ln \left( \frac{c(r_{0})}{f(y_{1}-r)r^{n-1}}\right) ^{N}dr,
	\end{equation*}%
	where we have used the notation 
	\begin{equation}  \label{c_def_new}
	c(r_{0})\equiv \frac{\left\vert B\left( 0,r_{0}\right) \right\vert }{\omega
		\left( r_{0}\right) }\approx \frac{f(r_{0})r_{0}^{\frac{n}{2}-1}}{|F^{\prime
		}(r_{0})|^{\frac{n}{2}}}\lesssim f(r_{0})r_{0}^{n-1},
	\end{equation}%
	where we used Lemma \ref{new} and for the last inequality, property $(4)$ of
	geometry $F$. Using this, we have 
	\begin{align*}
		\int_{0}^{y_{1}}&\ln \left( \frac{c(r_{0})}{f(y_{1}-r)r^{n-1}}\right)
		^{N}dr\leq \int_{0}^{r_{0}}\ln\left( \frac{f(r_{0})r_{0}^{n-1}}{%
			f(y_{1}-r)r^{n-1}}\right) \\
		&\leq Cr_{0}+
		\int_{0}^{r_{0}}\left(F(x_{1})-F(r_{0})\right)^{N}dx_{1}
		+\int_{0}^{r_{0}}%
		\left( \ln \frac{r_{0}}{r}\right) ^{N}dr.
	\end{align*}%
	For the second integral we have 
	\begin{equation*}
		\int_{0}^{r_{0}}\left( \ln \frac{r_{0}}{r}\right) ^{N}dr\lesssim r\left( \ln 
		\frac{r_{0}}{r}\right) ^{N}{\Big |_{r=0}^{r_{0}}=Cr_{0},}
	\end{equation*}
	which can be absorbed into the first term. To estimate the first integral we
	write 
	\begin{align*}
		I\equiv\int_{0}^{r_{0}}\left(F(x_{1})-F(r_{0})\right)^{N}dx_{1}&=%
		\int_{0}^{r_{0}}\left(\int_{x_{1}}^{r_{0}}-F^{\prime }(s)ds\right)^{N}dx_{1}
		\\
		&=\int_{0}^{r_{0}}\left(\int_{0}^{r_{0}}\chi_{[x_{1},r_{0}]}(s)(-F^{\prime
		}(s))ds\right)^{N}dx_{1}
	\end{align*}
	and we will use Minkowski integral inequality to estimate this. More
	precisely, we have 
	\begin{align*}
		&\left(\int_{0}^{r_{0}}\left(\int_{0}^{r_{0}}\chi_{[x_{1},r_{0}]}(s)(-F^{%
			\prime }(s))ds\right)^{N}dx_{1}\right)^{\frac{1}{N}} \\
		&\qquad\qquad\qquad\leq
		\int_{0}^{r_{0}}\left(\int_{0}^{r_{0}}\left(\chi_{[x_{1},r_{0}]}(s)(-F^{%
			\prime }(s))\right)^{N}dx_{1}\right)^{\frac{1}{N}}ds \\
		&\qquad\qquad\qquad=\int_{0}^{r_{0}}(-F^{\prime
		}(s))\left(\int_{0}^{r_{0}}\chi_{[x_{1},r_{0}]}(s)dx_{1}\right)^{\frac{1}{N}%
		}ds \\
		&\qquad\qquad\qquad=\int_{0}^{r_{0}}|F'(s)|s^{\frac{1}{N}}ds.
	\end{align*}
	Finally, integrating by parts and using (\ref{geom_extra_cond}) we obtain 
	\begin{align*}
		\int_{0}^{r_{0}}|F^{\prime }(s)|s^{\frac{1}{N}}ds&=|F^{\prime }(s)|s^{\frac{1}{N}+1}\frac{1}{\frac{1}{N}+1}\Big|_{r=0}^{r_{0}}+\frac{1}{\frac{1}{N}+1}\int F^{\prime \prime}(s)s^{\frac{1}{N}+1}ds\\
		\left(\frac{1}{N}+1\right)\int_{0}^{r_{0}}|F^{\prime }(s)|s^{\frac{1}{N}}ds&=|F^{\prime }(s)|s^{\frac{1}{N}+1}\Big|_{r=0}^{r_{0}}
		+\int F^{\prime \prime}(s)s^{\frac{1}{N}+1}ds\\
		&\leq |F^{\prime }(s)|s^{\frac{1}{N}+1}\Big|_{r=0}^{r_{0}}+\left(\frac{1-\varepsilon}{N}+1\right)\int_{0}^{r_{0}}|F'(s)|s^{\frac{1}{N}}ds\\
		\int_{0}^{r_{0}}|F^{\prime }(s)|s^{\frac{1}{N}}ds&\leq \frac{N}{\varepsilon}|F^{\prime }(s)|s^{\frac{1}{N}+1}\Big|_{r=0}^{r_{0}}=\frac{N}{\varepsilon}|F^{\prime }(r_{0})|r_{0}^{\frac{1}{N}+1},
	\end{align*}
	where in the last inequality we used (\ref{phi_incr N}). This implies 
	\begin{equation*}
		I\leq C|F^{\prime }(r_{0})|^{N}r_{0}^{N+1},
	\end{equation*}
	and therefore using $r_{0}\leq C|F^{\prime }(r_{0})|^{N}r_{0}^{N+1}$ which
	follows from the assumption $(4)$ on the geometry, we conclude (\ref{the
		integral'Dn}) for region $\mathcal{R}_{1}$.
	
	We now turn to second region $\mathcal{R}_{2}$, where $r=y_{1}-x_{1}\geq 
	\frac{2}{\left\vert F^{\prime }\left( x_{1}\right) \right\vert }$, and so $%
	s_{y_{1}-x_{1}}\approx \frac{f(y_{1})r^{\frac{n}{2}-1}}{|F^{\prime
		}(y_{1})|^{\frac{n}{2}+1}}$. The integral to be estimated becomes 
	\begin{equation*}
		I_{\mathcal{R}_{2}}\equiv \frac{1}{\omega \left( r_{0}\right) }%
		\int_{0}^{y_{1}}\ln \left( \frac{c(r_{0})|F^{\prime }(y_{1})|^{\frac{n}{2}+1}%
		}{f(y_{1})r^{\frac{n}{2}-1}}\right) ^{N}dr,
	\end{equation*}%
	where $c(r_{0})$ as in (\ref{c_def_new}). Similar to the estimate for region 
	$\mathcal{R}_{1}$ we have 
	\begin{align*}
		\int_{0}^{y_{1}}\ln \left( \frac{c(r_{0})|F^{\prime }(y_{1})|^{\frac{n}{2}+1}%
		}{f(y_{1})r^{\frac{n}{2}-1}}\right) ^{N}dr \lesssim Cr_{0}&+\int_{0}^{y_{1}}
		\left( \ln \frac{r_{0}}{r}\right)
		^{N}dr\\
		&+\int_{0}^{y_{1}}\left(F(y_{1})-F(r_{0})\right)^{N}dr\\
		&+\int_{0}^{y_{1}}
		\left( \ln \frac{|F^{\prime }(y_1)|}{|F^{\prime }(r_{0})|}\right) ^{N}dr.
	\end{align*}
	The first integral was estimated above and can be bounded by $Cr_{0}$. For
	the second integral we can use the trivial estimate 
	\begin{equation*}
		\int_{0}^{y_{1}}\left(F(y_{1})-F(r_{0})\right)^{N}dr\leq
		\int_{0}^{r_{0}}\left(F(r)-F(r_{0})\right)^{N}dr,
	\end{equation*}
	which reduces it to the integral I arising in region $\mathcal{R}_{1}$. It
	remains to estimate the third integral which we denote as 
	\begin{equation*}
		\mathcal{F}(y_{1})\equiv y_{1}\left( \ln \frac{|F^{\prime }(y_1)|}{%
			|F^{\prime }(r_{0})|}\right) ^{N}.
	\end{equation*}
	First note that from the doubling condition on $|F^{\prime }|$ which is
	assumption $(3)$, it follows that $\mathcal{F}(0)=0$. Also, clearly $%
	\mathcal{F}(r_{0})=0$ and thus $\mathcal{F}(y_{1})$ achieves a maximum
	inside the interval $(0,r_{0})$. Differentiating $\mathcal{F}(y_{1})$ and
	setting the derivative equal to zero, we obtain the following implicit
	expression for $y_{1}^{\ast} $ maximizing $\mathcal{F}(y_{1})$: 
	\begin{equation*}
		\ln \frac{|F^{\prime }(y_{1}^{\ast})|}{|F^{\prime }(r_{0})|}=N y_{1}^{\ast}%
		\frac{F^{\prime \prime }(y_{1}^{\ast})}{|F^{\prime }(y_{1}^{\ast})|}.
	\end{equation*}
	Substituting back into the definition of $\mathcal{F}$ and using assumption $%
	(5)$ on the geometry, we get 
	\begin{equation*}
		\mathcal{F}(y_{1}^{\ast})=y_{1}^{\ast}\left(N y_{1}^{\ast}\frac{F^{\prime
				\prime }(y_{1}^{\ast})}{|F^{\prime }(y_{1}^{\ast})|}\right)^{N}\leq
		Cy_{1}^{\ast}\leq Cr_{0}.
	\end{equation*}
	
	Parts (2) and (3) of the theorem follow easily from part (1).
\end{proof}

\section{Weak Sobolev inequality}

Now we turn to the global Sobolev inequality needed for the maximum principle, namely
\begin{equation}
\left\Vert w\right\Vert _{L^{\Phi }\left( \Omega \right) }\leq C
\left( \Omega \right) \left\Vert \nabla _{A}w\right\Vert _{L^{1}\left(
	\Omega \right) },\ \ \ \ \ w\in \left( W_{A}^{1,1}\right) _{0}\left( \Omega
\right)  \label{Sob_gen_geo}
\end{equation}%
where $\Phi =\Phi_N $ with $N>1$ as defined in (\ref{bump_N}).
\begin{proposition}\label{sob_global}
	Assume that for some $C>0$ the function 
	\begin{equation*}
		\varphi (r)\equiv C|F^{\prime }\left( r\right) |^{N}r^{N+1}
	\end{equation*}%
	satisfies $\sup_{0<s<r_{0}}\varphi(s)<\infty$ $\forall r_0>0$. Then the global Sobolev inequality (\ref{Sob_gen_geo}) holds with geometry $F$ and $C(\Omega)=\sup_{0<s<d(0,\Omega)+diam(\Omega)}\varphi(s)$.
\end{proposition}

\begin{proof}
	This is a consequence of Proposition \ref{sob_nd}. First, assume that for $r_0$ as in Proposition {\ref{sob_nd}} we have  $\Omega \subset B(0,r_{0})$. Extending $w$ to be $0$
	outside $\Omega $ we may assume $w\in \left( W_{A}^{1,2}\right) _{0}\left(
	B(0,r_{0})\right) $ and the result follows from part (2) of Proposition {\ref{sob_nd}}. More generally, for any bounded domain $\Omega$, we construct a finite partition of unity $\{\eta_k\}_{k=1}^{K}$ consisting of Lipschitz functions $\eta_k$ supported in metric balls $B_k\equiv B((x_1^k,\mathbf{x}_2^k,x_3^k),r_0/4)$. By translation invariance of the problem in $(\mathbf{x}_2,x_3)$ variables we may assume $(\mathbf{x}_2^k,x_3^k)=(\mathbf{0},0)$. Now if $|x_1^k|<3r_0/4$ then $B_k\subseteq B(0,r_0)$ and we can apply (\ref{Phi N norm}) to $w\eta_k$ in $B(0,r_0)$. Otherwise, $|x_1^k|\geq 3r_0/4$, and $\forall x\in B_k$ we have $|x|>r_0/2$ so the matrix $A$ is elliptic there and the Orlicz bump Sobolev inequality follows from the classical Sobolev inequality. Finally writing $w=\sum_{k=1}^{K}w\eta_k$ we obtain the result.
\end{proof}

\chapter{Geometric Theorems}

For convenience we recall from the introduction our two geometric theorems
dealing with respectively local boundedness and the maximum principle.

\begin{theorem}
	\label{geom loc bdd}Suppose that $\Omega \subset \mathbb{R}^{n}$ is a domain
	in $\mathbb{R}^{n}$ with $n\geq 2$ and that 
	\begin{equation*}
		\mathcal{L}u\equiv \func{div}\mathcal{A}\left( x,u\right) \nabla u,\ \ \ \ \
		x=\left( x_{1},...,x_{n}\right) \in \Omega ,
	\end{equation*}%
	where$\ \mathcal{A}\left( x,z\right) \sim \left[ 
	\begin{array}{cc}
	I_{n-1} & 0 \\ 
	0 & f\left( x_{1}\right) ^{2}%
	\end{array}%
	\right] $, $I_{n-1}$ is the $\left( n-1\right) \times \left( n-1\right) $
	identity matrix, $\mathcal{A}$ has bounded measurable components, and the
	geometry $F=-\ln f$ satisfies the structure conditions in Definition \ref%
	{structure conditions}.
	
	\begin{enumerate}
		\item If $F\leq F_{\sigma }$ for some $0<\sigma <1$, then every weak
		subsolution to $\mathcal{L}u=\phi $ with $A$-admissible $\phi $ is locally
		bounded in $\Omega $.
		
		\item On the other hand, if $n\geq 3$ and $\sigma \geq 1$, then there exists
		a locally unbounded weak solution $u$ in a neighbourhood of the origin in $%
		\mathbb{R}^{n}$ to the equation $Lu=0$ with geometry $F=F_{\sigma }$.
	\end{enumerate}
\end{theorem}

\begin{theorem}
	\label{geom max princ}Suppose that $F$ satisfies the geometric structure
	conditions in Definition \ref{structure conditions} and $F\leq F_{\sigma }$ for some $0<\sigma <1$. Assume that $u$ is a weak subsolution
	to $\mathcal{L}u=\phi $ in a domain $\Omega \subset \mathbb{R}^{n}$ with $%
	n\geq 2$, where $\mathcal{L}$ has degeneracy $F$ and $\phi $ is $A$%
	-admissible. Moreover, suppose that $u$ is bounded in the weak sense on the
	boundary $\partial \Omega $. Then $u$ is globally bounded in $\Omega $ and
	satisfies%
	\begin{equation*}
		\sup_{\Omega }u\leq \sup_{\partial \Omega }u+C\left\Vert \phi \right\Vert
		_{X\left( \Omega \right) }\ 
	\end{equation*}
	with the constant $C$ depending only on $\Omega$.
\end{theorem}

\section{Proofs of sufficiency}
%
%The geometric maximum principle in Theorem \ref{geom max princ} requires no
%restriction on the geometry $F$ other than the basic structure conditions in
%Definition \ref{structure conditions}, and it follows immediately from the
%abstract maximum principle Theorem \ref{max} and the\ weak Orlicz-Sobolev
%inequality (\ref{Sob_gen}).

The first part of the geometric Theorem \ref{geom loc bdd} follows from the
abstract Theorem \ref{bound_gen_thm} together with the Orlicz-Sobolev
inequality in Proposition \ref{sob_nd}. Indeed, in the special
case of the degenerate geometry $F_{\sigma }=\frac{1}{r^{\sigma }}$, we have 
\begin{equation*}
	|F_{\sigma }^{\prime }(r)|=\frac{\sigma }{r^{\sigma +1}},\quad F_{\sigma
	}^{\prime \prime }(r)=\frac{\sigma \left( \sigma +1\right) }{r^{\sigma +2}},
\end{equation*}%
and it is easy to check that the conditions of Proposition \ref{sob_nd} are
satisfied iff $\sigma <\frac{1}{N}$. Thus the superradius $\varphi \left(
r\right) $ is given by 
\begin{equation*}
	\varphi \left( r\right) =Cr^{-\sigma N+1}<\infty ,
\end{equation*}%
and so the Inner Ball Inequality (\ref{Inner ball inequ}) in Proposition \ref%
{DG} shows that weak subsolutions $u$ to the inhomogeneous equation $%
\mathcal{L}u=\phi $ are locally bounded above, and hence that solutions $u$
are locally bounded.

Similarly, the geometric Theorem \ref{geom max princ} follows from the abstract Theorem \ref{2D max} and Proposition \ref{sob_global}.

The second part of the geometric Theorem \ref{geom loc bdd} follows from the counterexample constructed in the next section.

\section{An unbounded weak solution}

In this final section of the final chapter of Part 3, we demonstrate that
weak solutions to our degenerate equations can fail to be locally bounded.
We modify an example of Morimoto \cite{Mor} that was used to provide an
alternate proof of a result of Kusuoka and Strook \cite{KuStr}.

\begin{theorem}
\label{actual}Suppose that $g\in C^{\infty }\left( \mathbb{R}\right) $
satisfies $g\left( x\right) \geq 0$, $g\left( 0\right) =0$ and the decay
condition%
\begin{equation}
\liminf_{x\rightarrow 0}\left\vert x\ln \frac{1}{g\left( x\right) }%
\right\vert \neq 0.  \label{decay}
\end{equation}%
Then for some $\varepsilon >0$, the operator%
\begin{equation*}
L\equiv \frac{\partial ^{2}}{\partial x^{2}}+g\left( x\right) \frac{\partial
^{2}}{\partial y^{2}}+\frac{\partial ^{2}}{\partial t^{2}}
\end{equation*}%
\textbf{fails} to be $W_{A}^{1,2}\left( \mathbb{R}^{2}\right) $-hypoelliptic
in an open subset $\left( -1,1\right) \times \mathbb{R}\times \left(
-\varepsilon ,\varepsilon \right) $ of $\mathbb{R}^{3}$ containing the
origin, where $\nabla _{A}\equiv \left( \frac{\partial }{\partial x},\sqrt{%
g\left( x\right) }\frac{\partial }{\partial y},\frac{\partial }{\partial t}%
\right) $ is the degenerate gradient associated with $\mathcal{L}$. In fact,
we construct a weak solution $u$ of the homogeneous equation $\mathcal{L}u=0$
in $\left( -1,1\right) \times \mathbb{T}\times \left( -\varepsilon
,\varepsilon \right) $ with $\left\Vert u\right\Vert _{L^{\infty }\left(
\left( -\delta ,\delta \right) \times \mathbb{T}\times \left( -\varepsilon
^{\prime },\varepsilon ^{\prime }\right) \right) }=\infty $ for all $\delta
>0$ and $0<\varepsilon ^{\prime }<\varepsilon $.
\end{theorem}

\begin{proof}
For $a,\eta >0$ we follow Morimoto \cite{Mor}, who in turn followed Bouendi
and Goulaouic \cite{BoGo}, by considering the second order operator $L_{\eta
}\equiv -\frac{\partial ^{2}}{\partial x^{2}}+g\left( x\right) \eta ^{2}$
and the eigenvalue problem%
\begin{eqnarray*}
L_{\eta }v\left( x,\eta \right) &=&\lambda \ v\left( x,\eta \right) ,\ \ \ \
\ x\in I_{a}\equiv \left( -a,a\right) , \\
v\left( a,\eta \right) &=&v\left( -a,\eta \right) =0.
\end{eqnarray*}%
The least eigenvalue is given by the Rayleigh quotient formula%
\begin{eqnarray*}
\lambda _{0}\left( a,\eta \right) &=&\inf_{f\left( \neq 0\right) \in
C_{0}^{\infty }\left( I_{a}\right) }\frac{\left\langle L_{\eta
}f,f\right\rangle _{L^{2}}}{\left\Vert f\right\Vert _{L^{2}}^{2}} \\
&=&\inf_{f\left( \neq 0\right) \in C_{0}^{\infty }\left( I_{a}\right) }\frac{%
\int_{-a}^{a}f^{\prime }\left( x\right) ^{2}dx+\int_{-a}^{a}g\left( x\right)
\eta ^{2}f\left( x\right) ^{2}dx}{\left\Vert f\right\Vert _{L^{2}}^{2}},
\end{eqnarray*}%
from which it follows that%
\begin{equation}
\lambda _{0}\left( a,\eta \right) \leq \lambda _{0}\left( a_{0},\eta \right) 
\text{ if }a\geq a_{0}.  \label{follows that}
\end{equation}

The decay condition (\ref{decay}) above is equivalent to the existence of $%
\delta _{0}>0$ such that $g\left( x\right) \leq e^{-\frac{\delta _{0}}{%
\left\vert x\right\vert }}$ for $x$ small. So we may suppose $g\left(
x\right) \leq Ce^{-\frac{\delta _{0}}{\left\vert x\right\vert }}$ for $x\in
I\equiv \left[ -1,1\right] $ where $C\geq 1$, and then take $\left\vert \eta
\right\vert $ sufficiently large that with%
\begin{equation*}
a\left( \eta \right) \equiv \frac{\delta _{0}}{\ln C+2\ln \left\vert \eta
\right\vert },
\end{equation*}%
we have both $a\left( \eta \right) \leq 1$ and%
\begin{equation*}
g\left( x\right) \eta ^{2}\leq Ce^{-\frac{\delta _{0}}{\left\vert
x\right\vert }}\eta ^{2}\leq e^{\ln C-\frac{\delta _{0}}{a\left( \eta
\right) }+2\ln \left\vert \eta \right\vert }=1,\ \ \ \ \ x\in I_{a\left(
\eta \right) }.
\end{equation*}

Now let $\mu _{0}\left( a\left( \eta \right) \right) $ denote the least
eigenvalue for the problem%
\begin{eqnarray*}
\left\{ -\frac{\partial ^{2}}{\partial x^{2}}+1\right\} v\left( x,\eta
\right) &=&\mu \ v\left( x,\eta \right) ,\ \ \ \ \ x\in I_{a\left( \eta
\right) }=\left( -a\left( \eta \right) ,a\left( \eta \right) \right) , \\
v\left( a\left( \eta \right) ,\eta \right) &=&v\left( -a\left( \eta \right)
,\eta \right) =0,
\end{eqnarray*}%
and note that%
\begin{eqnarray*}
\mu _{0}\left( a\left( \eta \right) \right) &=&\inf_{f\left( \neq 0\right)
\in C_{0}^{\infty }\left( I_{a\left( \eta \right) }\right) }\frac{%
\left\langle -\frac{\partial ^{2}f}{\partial x^{2}}+f,f\right\rangle
_{L^{2}\left( I_{a\left( \eta \right) }\right) }}{\left\Vert f\right\Vert
_{L^{2}\left( I_{a\left( \eta \right) }\right) }^{2}} \\
&=&\inf_{f\left( \neq 0\right) \in C_{0}^{\infty }\left( I_{a\left( \eta
\right) }\right) }\frac{\int_{-a\left( \eta \right) }^{a\left( \eta \right)
}f^{\prime }\left( x\right) ^{2}dx+\int_{-a\left( \eta \right) }^{a\left(
\eta \right) }f\left( x\right) ^{2}dx}{\left\Vert f\right\Vert _{L^{2}\left(
I_{a\left( \eta \right) }\right) }^{2}}.
\end{eqnarray*}%
It follows that%
\begin{equation*}
\lambda _{0}\left( a\left( \eta \right) ,\eta \right) \leq \mu _{0}\left(
a\left( \eta \right) \right) ,\ \ \ \ \ \text{for }\left\vert \eta
\right\vert \text{ sufficiently large}.
\end{equation*}%
Now an easy classical calculation using exact solutions to $\left\{ -\frac{%
\partial ^{2}}{\partial x^{2}}+1-\mu \right\} v=0$ shows that%
\begin{equation*}
\mu _{0}\left( a\left( \eta \right) \right) =C_{1}\frac{1}{a\left( \eta
\right) ^{2}}+1,
\end{equation*}%
for some constant $C_{1}$ independent of $\eta $, and hence combining this
with (\ref{follows that}), we have%
\begin{eqnarray}
0 &<&\lambda _{0}\left( 1,\eta \right) \leq \lambda _{0}\left( a\left( \eta
\right) ,\eta \right) \leq \mu _{0}\left( a\left( \eta \right) \right)
\label{lambda bound} \\
&=&C_{1}\left( \frac{\ln C+2\ln \left\vert \eta \right\vert }{\delta _{0}}%
\right) ^{2}+1\leq C_{2}\left( \ln \left\vert \eta \right\vert \right)
^{2},\ \ \ \ \ \text{for }\left\vert \eta \right\vert \text{ sufficiently
large}.  \notag
\end{eqnarray}

Now let $v_{0}\left( x,\eta \right) $ be an eigenfunction on the interval $%
I=-I_{1}=\left[ -1,1\right] $ associated with $\lambda _{0}\left( 1,\eta
\right) $ and normalized so that 
\begin{equation}
\left\Vert v_{0}\left( \cdot ,\eta \right) \right\Vert _{L^{2}\left(
I\right) }=1.  \label{normalized}
\end{equation}%
Choose a sequence $\left\{ a_{n}\right\} _{n=-\infty }^{\infty }$ satisfying%
\begin{equation}
\left\vert a_{n}\right\vert \leq \frac{1}{1+\left\vert n\right\vert ^{\alpha
}}\rho _{n},  \label{suitable decay}
\end{equation}%
for some $\alpha >0$ where $\left\Vert \left\{ \rho _{n}\right\} _{n\in 
\mathbb{Z}}\right\Vert _{\ell ^{2}}=1$. For $y\in I_{\pi }=\left[ -\pi ,\pi %
\right] $, identified with the unit circle $\mathbb{T}$ upon identifying $%
-\pi $ and $\pi $, we formally define%
\begin{eqnarray*}
w\left( x,y\right)  &\equiv &\sum_{n\in \mathbb{Z}}e^{iyn}v_{0}\left(
x,n\right) a_{n}; \\
w_{N}\left( x,y\right)  &\equiv &\sum_{n\in \mathbb{Z}}\lambda _{0}\left(
1,n\right) ^{N}e^{iyn}v_{0}\left( x,n\right) a_{n}; \\
u\left( x,y,t\right)  &\equiv &\sum_{N=0}^{\infty }\frac{t^{2N}}{\left(
2N\right) !}w_{N}\left( x,y\right) .
\end{eqnarray*}%
We now claim that%
\begin{equation*}
w_{N}\left( x,y\right) =\left\{ -\frac{\partial ^{2}}{\partial x^{2}}%
-g\left( x\right) \frac{\partial ^{2}}{\partial y^{2}}\right\} ^{N}w\left(
x,y\right) .
\end{equation*}%
Indeed, assuming this holds for $N$, and using 
\begin{equation*}
-\frac{\partial ^{2}}{\partial x^{2}}v_{0}\left( x,n\right) =\left[ \lambda
_{0}\left( 1,n\right) -g\left( x\right) n^{2}\right] v_{0}\left( x,n\right) ,
\end{equation*}%
we obtain that%
\begin{eqnarray*}
\left\{ -\frac{\partial ^{2}}{\partial x^{2}}-g\left( x\right) \frac{%
\partial ^{2}}{\partial y^{2}}\right\} ^{N+1}w\left( x,y\right)  &=&\left\{ -%
\frac{\partial ^{2}}{\partial x^{2}}-g\left( x\right) \frac{\partial ^{2}}{%
\partial y^{2}}\right\} w_{N}\left( x,y\right)  \\
&=&-\sum_{n\in \mathbb{Z}}\lambda _{0}\left( 1,n\right) ^{N}e^{iyn}\frac{%
\partial ^{2}}{\partial x^{2}}v_{0}\left( x,n\right) a_{n} \\
&&+\sum_{n\in \mathbb{Z}}\lambda _{0}\left( 1,n\right) ^{N}e^{iyn}g\left(
x\right) n^{2}v_{0}\left( x,n\right) a_{n} \\
&=&\sum_{n\in \mathbb{Z}}\lambda _{0}\left( 1,n\right) ^{N}e^{iyn}\lambda
_{0}\left( 1,n\right) v_{0}\left( x,n\right) a_{n} \\
&&-\sum_{n\in \mathbb{Z}}\lambda _{0}\left( 1,n\right) ^{N}e^{iyn}g\left(
x\right) n^{2}v_{0}\left( x,n\right) a_{n} \\
&&+\sum_{n\in \mathbb{Z}}\lambda _{0}\left( 1,n\right) ^{N}e^{iyn}g\left(
x\right) n^{2}v_{0}\left( x,n\right) a_{n} \\
&=&\sum_{n\in \mathbb{Z}}\lambda _{0}\left( 1,n\right)
^{N+1}e^{iyn}v_{0}\left( x,n\right) a_{n}=w_{N+1}\left( x,y\right) .
\end{eqnarray*}%
It follows that%
\begin{equation*}
\left\{ -\frac{\partial ^{2}}{\partial x^{2}}-g\left( x\right) \frac{%
\partial ^{2}}{\partial y^{2}}\right\} w_{N}\left( x,y\right) =w_{N+1}\left(
x,y\right) ,
\end{equation*}%
and so formally we get%
\begin{eqnarray*}
Lu\left( x,y,t\right)  &=&\left\{ -\frac{\partial ^{2}}{\partial x^{2}}%
-g\left( x\right) \frac{\partial ^{2}}{\partial y^{2}}-\frac{\partial ^{2}}{%
\partial t^{2}}\right\} \sum_{N=0}^{\infty }\frac{t^{2N}}{\left( 2N\right) !}%
w_{N}\left( x,y\right)  \\
&=&\sum_{N=0}^{\infty }\frac{t^{2N}}{\left( 2N\right) !}w_{N+1}\left(
x,y\right) -\sum_{N=0}^{\infty }\frac{\partial ^{2}}{\partial t^{2}}\frac{%
t^{2N}}{\left( 2N\right) !}w_{N}\left( x,y\right)  \\
&=&\sum_{N=0}^{\infty }\frac{t^{2N}}{\left( 2N\right) !}w_{N+1}\left(
x,y\right) -\sum_{N=1}^{\infty }\frac{t^{2N-2}}{\left( 2N-2\right) !}%
w_{N}\left( x,y\right) =0.
\end{eqnarray*}

Now we show that $u\left( x,y,t\right) $ is well defined as an $L^{2}\left(
I\times \mathbb{T}\right) $-valued analytic function of $t$ for $t$ in some
small neighbourhood of $0$ provided $\left\{ a_{n}\right\} _{n\in \mathbb{Z}}
$ is in $\ell ^{2}\left( \mathbb{Z}\right) $ with suitable decay at $\infty $%
, namely (\ref{suitable decay}). Here $I=\left[ -1,1\right] $. Indeed, using
Plancherel's formula in the $y$ variable, and then Fubini's theorem, we have%
\begin{eqnarray*}
\left\Vert w_{N}\right\Vert _{L^{2}\left( I\times \mathbb{T}\right) }^{2}
&=&\int_{-1}^{1}\left\{ \int_{-\pi }^{\pi }\left\vert \sum_{n\in \mathbb{Z}%
}\lambda _{0}\left( 1,n\right) ^{N}e^{iyn}v_{0}\left( x,n\right)
a_{n}\right\vert ^{2}dy\right\} dx \\
&=&\int_{-1}^{1}\left\{ \sum_{n\in \mathbb{Z}}\left\vert \lambda _{0}\left(
1,n\right) ^{N}v_{0}\left( x,n\right) a_{n}\right\vert ^{2}\right\} dx \\
&=&\sum_{n\in \mathbb{Z}}\left\{ \int_{-1}^{1}\left\vert v_{0}\left(
x,n\right) \right\vert ^{2}dx\right\} \left\vert \lambda _{0}\left(
1,n\right) ^{N}a_{n}\right\vert ^{2} \\
&=&\sum_{n\in \mathbb{Z}}\left\vert \lambda _{0}\left( 1,n\right)
^{N}a_{n}\right\vert ^{2}.
\end{eqnarray*}%
Now from (\ref{lambda bound}) we have the bound $\lambda _{0}\left(
1,n\right) \leq C_{2}\left( \ln n\right) ^{2}$ for $n$ sufficiently large,
and hence from (\ref{suitable decay}),%
\begin{equation*}
\left\vert a_{n}\right\vert \leq \frac{1}{1+\left\vert n\right\vert ^{\alpha
}}\rho _{n},
\end{equation*}%
where $\left\Vert \left\{ \rho _{n}\right\} _{n\in \mathbb{Z}}\right\Vert
_{\ell ^{2}}=1$, we obtain%
\begin{eqnarray*}
\left\Vert w_{N}\right\Vert _{L^{2}\left( I\times \mathbb{T}\right) } &\leq
&C_{3}\sqrt{\sum_{n\in \mathbb{Z}}\left\vert \left( \ln n\right)
^{2N}a_{n}\right\vert ^{2}} \\
&\leq &C_{3}\sqrt{\sum_{n\in \mathbb{Z}}\left\vert \left( \ln n\right)
^{2N}\left( \frac{1}{1+e^{\alpha \ln n}}\right) \right\vert ^{2}\left\vert
\rho _{n}\right\vert ^{2}} \\
&\leq &C_{4}\sqrt{N}\alpha ^{-2N}\left( 2N\right) !\sqrt{\sum_{n\in \mathbb{Z%
}}\left\vert \rho _{n}\right\vert ^{2}}=C_{4}\frac{1}{\sqrt{N}}\alpha
^{-2N}\left( 2N\right) !
\end{eqnarray*}%
since the maximum value of $s^{2N}e^{-\alpha s}$ occurs at\thinspace $s=%
\frac{2N}{\alpha }$, and then by Stirling's formula, 
\begin{equation*}
\frac{s^{2N}}{1+e^{\alpha s}}\leq s^{2N}e^{-\alpha s}\leq \left( \frac{2N}{%
\alpha }\right) ^{2N}e^{-\alpha \frac{2N}{\alpha }}=\left( \frac{2N}{e}%
\right) ^{2N}\alpha ^{-2N}\leq \frac{1}{\sqrt{N}}\alpha ^{-2N}\left(
2N\right) !
\end{equation*}%
Thus we conclude that%
\begin{equation*}
\left\Vert u\left( x,y,t\right) \right\Vert _{L_{x,y}^{2}\left( I\times 
\mathbb{T}\right) }\leq \sum_{N=0}^{\infty }\frac{t^{2N}}{\left( 2N\right) !}%
\left\Vert w_{N}\right\Vert _{L^{2}\left( I\times \mathbb{T}\right) }\leq
C_{4}\sum_{N=0}^{\infty }\frac{1}{\sqrt{N}}\left( \frac{t}{\alpha }\right)
^{2N}<\infty 
\end{equation*}%
for $t\in \left( -\alpha ,\alpha \right) $, and it follows that $u\left(
x,y,t\right) $ is a well-defined $L^{2}\left( I\times \mathbb{T}\right) $%
-valued analytic function of $t\in \left( -\alpha ,\alpha \right) $ that
satisfies the homogeneous equation $\mathcal{L}u\left( x,y,t\right) =0$ for $%
\left( x,y,t\right) \in I\times \mathbb{T}\times \left( -\alpha ,\alpha
\right) $.

Next, we show that $u\in W_{A}^{1,2}\left( I\times \mathbb{T}\times \left(
-\varepsilon ,\varepsilon \right) \right) $ for some $\varepsilon >0$. We
first compute that%
\begin{eqnarray*}
&&\left\Vert \frac{\partial }{\partial x}w_{N}\right\Vert _{L^{2}\left(
I\times \mathbb{T}\right) }^{2}+\left\Vert \sqrt{g\left( x\right) }\frac{%
\partial }{\partial y}w_{N}\right\Vert _{L^{2}\left( I\times \mathbb{T}%
\right) }^{2} \\
&=&\left\langle \left\{ -\frac{\partial ^{2}}{\partial x^{2}}-g\left(
x\right) \frac{\partial ^{2}}{\partial y^{2}}\right\} w_{N}\left( x,y\right)
,w_{N}\left( x,y\right) \right\rangle _{L^{2}\left( I\times \mathbb{T}%
\right) } \\
&=&\left\langle w_{N+1}\left( x,y\right) ,w_{N}\left( x,y\right)
\right\rangle _{L^{2}\left( I\times \mathbb{T}\right) } \\
&\leq &\left\Vert w_{N+1}\right\Vert _{L^{2}\left( I\times \mathbb{T}\right)
}\left\Vert w_{N}\right\Vert _{L^{2}\left( I\times \mathbb{T}\right) } \\
&\leq &C_{4}\frac{1}{\sqrt{N+1}}\alpha ^{-2N-2}\left( 2N+2\right) !C_{4}%
\frac{1}{\sqrt{N+1}}\alpha ^{-2N}\left( 2N\right) ! \\
&\leq &C_{5}^{2}N\left[ \left( 2N\right) !\right] ^{2}\alpha ^{-4N-2}\ ,
\end{eqnarray*}%
which shows in particular that $w_{N}\in W_{A}^{1,2}\left( I\times \mathbb{T}%
\right) $ for each $N\geq 1$ with the norm estimate%
\begin{equation*}
\left\Vert \frac{w_{N}}{\left( 2N\right) !}\right\Vert _{W_{A}^{1,2}\left(
I\times \mathbb{T}\right) }\leq C_{5}\sqrt{N}\alpha ^{-2N-1}.
\end{equation*}%
Thus the $W_{A}^{1,2}\left( I\times \mathbb{T}\right) $-valued analytic
function $u\left( t\right) =\sum_{N=0}^{\infty }\frac{w_{N}}{\left(
2N\right) !}t^{2N}$ is $W_{A}^{1,2}\left( I\times \mathbb{T}\right) $%
-bounded in the complex disk $B\left( 0,\alpha \right) $ centered at the
origin with radius $\alpha $. Then we use Cauchy's estimates for the $%
W_{A}^{1,2}\left( I\times \mathbb{T}\right) $-valued analytic function $%
u\left( t\right) $ to obtain that $\frac{\partial }{\partial t}u\left(
t\right) $ is $W_{A}^{1,2}\left( I\times \mathbb{T}\right) $-bounded in any
complex disk $B\left( 0,\varepsilon \right) $ with $0<\varepsilon <\alpha
=\beta -\frac{1}{2}$, which shows that $\frac{\partial }{\partial t}u\in
L^{2}\left( I\times \mathbb{T}\times \left( -\varepsilon ,\varepsilon
\right) \right) $ for $0<\varepsilon <\beta -\frac{1}{2}$. This completes
the proof that $u\in W_{A}^{1,2}\left( I\times \mathbb{T}\times \left(
-\varepsilon ,\varepsilon \right) \right) $ for some $\varepsilon >0$.

Finally, we note that with $\rho _{n}=\frac{1}{1+\left\vert n\right\vert
^{\beta }}$ where $\frac{1}{2}<\beta \leq \frac{3}{2}-\alpha $, then $u$ is 
\emph{not} smooth near the origin since%
\begin{eqnarray*}
\left\Vert \frac{\partial }{\partial y}w\right\Vert _{L^{2}\left( I\times 
\mathbb{T}\right) }^{2} &=&\int_{-1}^{1}\left\{ \int_{-\pi }^{\pi
}\left\vert \sum_{n\in \mathbb{Z}}ine^{iyn}v_{0}\left( x,n\right)
a_{n}\right\vert ^{2}dy\right\} dx \\
&=&\int_{-1}^{1}\left\{ \sum_{n\in \mathbb{Z}}\left\vert v_{0}\left(
x,n\right) na_{n}\right\vert ^{2}\right\} dx \\
&=&\sum_{n\in \mathbb{Z}}\left\{ \int_{-1}^{1}\left\vert v_{0}\left(
x,n\right) \right\vert ^{2}dx\right\} \left\vert na_{n}\right\vert ^{2} \\
&=&\sum_{n\in \mathbb{Z}}\left\vert na_{n}\right\vert ^{2}=\sum_{n\in 
\mathbb{Z}}\left\vert \frac{n}{1+\left\vert n\right\vert ^{\alpha }}\rho
_{n}\right\vert ^{2}=\infty .
\end{eqnarray*}%
This is essentially the example of Morimoto \cite{Mor}. However, we need
more - namely, we must construct an essentially unbounded weak solution $u$
in some neighbourhood $\left( -\delta ,\delta \right) \times \mathbb{T}%
\times \left( -\varepsilon ,\varepsilon \right) $.

To accomplish this, we first derive the additional property (\ref{lower
bound}) below of the least eigenfunction $v_{n}\left( x\right) \equiv
v_{0}\left( x,n\right) $ that satisfies the equation%
\begin{eqnarray*}
\left\{ -\frac{\partial ^{2}}{\partial x^{2}}+g\left( x\right) n^{2}\right\}
v_{n}\left( x\right)  &=&\lambda _{0}\left( 1,n\right) \ v_{n}\left(
x\right) , \\
v_{n}\left( -1\right)  &=&v_{n}\left( 1\right) =0.
\end{eqnarray*}%
We claim that $v_{n}\left( x\right) $ is even on $\left[ -1,1\right] $ and
decreasing from $v_{n}\left( 0\right) $ to $0$ on the interval $\left[ 0,1%
\right] $. Indeed, the least eigenfunction $v_{n}$ minimizes the Rayleigh
quotient%
\begin{equation*}
\frac{\int_{-1}^{1}v_{n}^{\prime }\left( x\right)
^{2}dx+\int_{-1}^{1}g\left( x\right) n^{2}v_{n}\left( x\right) ^{2}dx}{%
\left\Vert v_{n}\right\Vert _{L^{2}}^{2}}=\inf_{f\left( \neq 0\right) \in
C_{0}^{\infty }\left( I_{a}\right) }\frac{\int_{-1}^{1}f^{\prime }\left(
x\right) ^{2}dx+\int_{-1}^{1}g\left( x\right) n^{2}f\left( x\right) ^{2}dx}{%
\left\Vert f\right\Vert _{L^{2}}^{2}},
\end{equation*}%
and since the radially decreasing rearrangement $v_{n}^{\ast }$ of $v_{n}$
on $\left[ -1,1\right] $ satisfies both 
\begin{equation*}
\int_{-1}^{1}v_{n}^{\ast \prime }\left( x\right) ^{2}dx\leq
\int_{-1}^{1}v_{n}^{\prime }\left( x\right) ^{2}dx\text{ and }%
\int_{-1}^{1}g\left( x\right) n^{2}v_{n}^{\ast }\left( x\right) ^{2}dx\leq
\int_{-1}^{1}g\left( x\right) n^{2}v_{n}\left( x\right) ^{2}dx,
\end{equation*}%
as well as $\left\Vert v_{n}^{\ast }\right\Vert _{L^{2}}^{2}=\left\Vert
v_{n}\right\Vert _{L^{2}}^{2}$, we conclude that $v_{n}=v_{n}^{\ast }$. The
only simple consequence we need from this is that%
\begin{equation}
2v_{n}\left( 0\right) ^{2}\geq \int_{-1}^{1}v_{n}\left( x\right) ^{2}dx=1,\
\ \ \ \ n\geq 1,  \label{lower bound}
\end{equation}%
where the equality follows from our normalizing assumption $\left\Vert
v_{n}\right\Vert _{L^{2}}=1$ in (\ref{normalized}).

Now recall $\alpha >0$ from (\ref{suitable decay}) above, and choose $%
0<\alpha <\alpha ^{\prime }\leq \frac{1}{4}$ and define 
\begin{equation}
a_{n}=\left\{ 
\begin{array}{ccc}
\frac{1}{n^{\frac{1}{2}+\alpha ^{\prime }}} & \text{ for } & n\geq 1 \\ 
0 & \text{ for } & n\leq 0%
\end{array}%
\right. .  \label{def an}
\end{equation}%
Then for each $\left( x,t\right) \in I\times \left( -\alpha ,\alpha \right) $%
, we have with 
\begin{equation*}
B_{n}\left( x,t\right) \equiv \left[ \sum_{N=0}^{\infty }\frac{t^{2N}}{%
\left( 2N\right) !}\lambda _{0}\left( 1,n\right) ^{N}\right] v_{n}\left(
x\right) a_{n},
\end{equation*}%
that 
\begin{eqnarray*}
u\left( x,y,t\right) &=&\sum_{N=0}^{\infty }\frac{t^{2N}}{\left( 2N\right) !}%
\sum_{n=1}^{\infty }\lambda _{0}\left( 1,n\right) ^{N}e^{iyn}v_{n}\left(
x\right) a_{n}=\sum_{n=1}^{\infty }e^{iyn}B_{n}\left( x,t\right) \ , \\
w\left( x,y\right) ^{2} &=&\left( \sum_{n=1}^{\infty }e^{iyn}B_{n}\left(
x,t\right) \right) ^{2}=\sum_{n=2}^{\infty }\left\{
\sum_{k=1}^{n-1}B_{n-k}\left( x,t\right) B_{k}\left( x,t\right) \right\}
e^{iyn}\ ,
\end{eqnarray*}%
and so by Plancherel's theorem,%
\begin{equation*}
\left\Vert u\left( x,\cdot ,t\right) \right\Vert _{L^{4}\left( \mathbb{T}%
\right) }^{4}=\left\Vert u\left( x,\cdot ,t\right) ^{2}\right\Vert
_{L^{2}\left( \mathbb{T}\right) }^{2}=\sum_{n=2}^{\infty }\left\vert
\sum_{k=1}^{n-1}B_{n-k}\left( x,t\right) B_{k}\left( x,t\right) \right\vert
^{2}.
\end{equation*}%
In particular we have from (\ref{lower bound}) that%
\begin{eqnarray*}
\left\Vert u\left( 0,\cdot ,0\right) \right\Vert _{L^{4}\left( \mathbb{T}%
\right) }^{4} &=&\sum_{n=2}^{\infty }\left\vert
\sum_{k=1}^{n-1}B_{n-k}\left( 0,0\right) B_{k}\left( 0,0\right) \right\vert
^{2} \\
&=&\sum_{n=2}^{\infty }\left\vert \sum_{k=1}^{n-1}v_{n-k}\left( 0\right)
a_{n-k}v_{k}\left( 0\right) a_{k}\right\vert ^{2}\geq \frac{1}{2}%
\sum_{n=2}^{\infty }\left\vert \sum_{k=1}^{n-1}a_{n-k}a_{k}\right\vert ^{2},
\end{eqnarray*}%
and now we obtain that $\left\Vert w\left( 0,\cdot \right) \right\Vert
_{L^{4}\left( \mathbb{T}\right) }^{4}=\infty $ from the estimates%
\begin{eqnarray*}
\sum_{k=1}^{n-1}a_{n-k}a_{k} &=&\sum_{k=1}^{n-1}\frac{1}{\left( n-k\right) ^{%
\frac{1}{2}+\alpha ^{\prime }}}\frac{1}{k^{\frac{1}{2}+\alpha ^{\prime }}}%
\gtrsim \frac{1}{\left( n-\frac{n}{2}\right) ^{\frac{1}{2}+\alpha ^{\prime }}%
}\sum_{k=1}^{\frac{n}{2}}\frac{1}{k^{\frac{1}{2}+\alpha ^{\prime }}} \\
&\gtrsim &\frac{1}{\left( \frac{n}{2}\right) ^{\frac{1}{2}+\alpha ^{\prime }}%
}\left( \frac{n}{2}\right) ^{\frac{1}{2}-\alpha ^{\prime }}\approx \frac{1}{%
n^{2\alpha ^{\prime }}}
\end{eqnarray*}%
and%
\begin{equation*}
\left\Vert w\left( 0,\cdot \right) \right\Vert _{L^{4}\left( \mathbb{T}%
\right) }^{4}\geq \frac{1}{2}\sum_{n=2}^{\infty }\left\vert
\sum_{k=1}^{n-1}a_{n-k}a_{k}\right\vert ^{2}\gtrsim \sum_{n=2}^{\infty }%
\frac{1}{n^{4\alpha ^{\prime }}}=\infty ,\ \ \ \ \ \text{\ for }\alpha
^{\prime }\leq \frac{1}{4}.
\end{equation*}

Now we note that each eigenfunction $v_{n}\left( x\right) $ is continuous in 
$x$ since it solves an elliptic second order equation on the interval $\left[
-1,1\right] $. Moreover it now follows that $B_{n}\left( x,t\right) $ is
jointly continuous on the rectangle $\left( -1,1\right) \times \left(
-\alpha ,\alpha \right) $. Then we write%
\begin{equation*}
\sum_{n=2}^{\infty }\left\vert \sum_{k=1}^{n-1}B_{n-k}\left( x,t\right)
B_{k}\left( x,t\right) \right\vert ^{2}=\sum_{\substack{ \beta =\left( \beta
_{1},\beta _{2},\beta _{3},\beta _{4}\right) \in \mathbb{N}^{4}  \\ \beta
_{1}+\beta _{2}=n=\beta _{3}+\beta _{4}}}B_{\beta _{1}}\left( x,t\right)
B_{\beta _{2}}\left( x,t\right) B_{\beta _{3}}\left( x,t\right) B_{\beta
_{4}}\left( x,t\right) ,
\end{equation*}%
and apply Fatou's lemma to conclude that%
\begin{eqnarray*}
\infty &=&\sum_{n=2}^{\infty }\left\vert \sum_{k=1}^{n-1}B_{n-k}\left(
0,0\right) B_{k}\left( 0,0\right) \right\vert ^{2}=\sum_{\substack{ \beta
=\left( \beta _{1},\beta _{2},\beta _{3},\beta _{4}\right) \in \mathbb{N}%
^{4}  \\ \beta _{1}+\beta _{2}=n=\beta _{3}+\beta _{4}}}B_{\beta _{1}}\left(
0,0\right) B_{\beta _{2}}\left( 0,0\right) B_{\beta _{3}}\left( 0,0\right)
B_{\beta _{4}}\left( 0,0\right) \\
&=&\sum_{\substack{ \beta =\left( \beta _{1},\beta _{2},\beta _{3},\beta
_{4}\right) \in \mathbb{N}^{4}  \\ \beta _{1}+\beta _{2}=n=\beta _{3}+\beta
_{4}}}\liminf_{\left( x,t\right) \rightarrow \left( 0,0\right) }\left\{
B_{\beta _{1}}\left( x,t\right) B_{\beta _{2}}\left( x,t\right) B_{\beta
_{3}}\left( x,t\right) B_{\beta _{4}}\left( x,t\right) \right\} \\
&\leq &\liminf_{\left( x,t\right) \rightarrow \left( 0,0\right) }\sum 
_{\substack{ \beta =\left( \beta _{1},\beta _{2},\beta _{3},\beta
_{4}\right) \in \mathbb{N}^{4}  \\ \beta _{1}+\beta _{2}=n=\beta _{3}+\beta
_{4}}}B_{\beta _{1}}\left( x,t\right) B_{\beta _{2}}\left( x,t\right)
B_{\beta _{3}}\left( x,t\right) B_{\beta _{4}}\left( x,t\right) \\
&=&\liminf_{\left( x,t\right) \rightarrow \left( 0,0\right)
}\sum_{n=2}^{\infty }\left\vert \sum_{k=1}^{n-1}B_{n-k}\left( x,t\right)
B_{k}\left( x,t\right) \right\vert ^{2}=\liminf_{\left( x,t\right)
\rightarrow \left( 0,0\right) }\left\Vert u\left( x,\cdot ,t\right)
\right\Vert _{L^{4}\left( \mathbb{T}\right) }^{4}.
\end{eqnarray*}%
Thus we have $\lim_{\left( x,t\right) \rightarrow \left( 0,0\right)
}\left\Vert u\left( x,\cdot ,t\right) \right\Vert _{L^{4}\left( \mathbb{T}%
\right) }^{4}=\infty $, which implies that $\left\Vert u\right\Vert
_{L^{\infty }\left( \left( -\delta ,\delta \right) \times \mathbb{T}\times
\left( -\varepsilon ^{\prime },\varepsilon ^{\prime }\right) \right)
}=\infty $ for all $0<\delta <1$ and $0<\varepsilon ^{\prime }<\varepsilon
<\alpha $. Thus we have shown that 
\begin{equation*}
u\left( x,y,t\right) =\sum_{N=0}^{\infty }\frac{t^{2N}}{\left( 2N\right) !}%
w_{N}\left( x,y\right) =w\left( x,y\right) +\sum_{N=1}^{\infty }\frac{t^{2N}%
}{\left( 2N\right) !}w_{N}\left( x,y\right)
\end{equation*}%
is a weak solution of $Lu=0$ in $I\times \mathbb{T}\times \left(
-\varepsilon ,\varepsilon \right) $ with $\left\Vert u\right\Vert
_{L^{\infty }\left( \left( -\delta ,\delta \right) \times \mathbb{T}\times
\left( -\varepsilon ^{\prime },\varepsilon ^{\prime }\right) \right)
}=\infty $ provided $0<\delta <1$ and $0<\varepsilon ^{\prime }<\varepsilon
<\alpha <\alpha ^{\prime }\leq \frac{1}{4}$. This completes the proof that $%
u $ fails to be essentially bounded on $\left( -\delta ,\delta \right)
\times \mathbb{T}\times \left( -\varepsilon ^{\prime },\varepsilon ^{\prime
}\right) $.
\end{proof}

\end{document}